\documentclass[11pt]{amsart}

\usepackage{evearticle}
\usepackage{longtable}

\makeatletter
\renewcommand*\env@matrix[1][\arraystretch]{
	\edef\arraystretch{#1}
	\hskip -\arraycolsep
	\let\@ifnextchar\new@ifnextchar
	\array{*\c@MaxMatrixCols c}}
\makeatother

\usepackage{etoolbox}
\newcounter{stepc}
\newcounter{stepcanchor}
\renewcommand{\thestepc}{\arabic{stepc}}

\crefname{stepc}{Step}{Steps}
\Crefname{stepc}{Step}{Steps}
\AtBeginEnvironment{proof}{
	\setcounter{stepc}{0}
}
\newcommand{\stepx}[2]{
	\stepcounter{stepcanchor}
	\refstepcounter{stepc}
	\label{#1}
	\emph{Step \thestepc. #2.}
}

\begin{document}

\title[The conformally invariant metric on $\CLE_4$ I]{The conformally invariant metric on $\CLE_4$ I:\\ subsequential limits of the non-simple CLE graph metric}
\author[E.~Kammerer, K.~Kavvadias, J.~Miller, and Y.~Tian]{Emmanuel Kammerer, Konstantinos Kavvadias, Jason Miller, and Yi Tian}

\begin{abstract}
	We consider the conformal loop ensemble (CLE) with the parameter $\kappa=4$, the critical value at or below which the loops are simple and do not intersect each other or the domain boundary. We show that the loops of a $\CLE_4$ uniquely determine a conformally invariant, local, and geodesic metric so that the metric ball growth from the domain boundary coincides with the uniform exploration of Werner and Wu. This metric was previously constructed in unpublished work of Sheffield, Watson, and Wu. Our approach differs in that we show that the metric arises as the renormalized limit of the graph metric on $\CLE_\kappa$ loops as $\kappa \downarrow 4$. In this first paper in a series of three, we prove that the subsequential limits exist and define a non-trivial conformally invariant metric on $\CLE_4$ which is local and such that the metric ball growth from the boundary is given by the uniform exploration of Werner and Wu. In subsequent work, we will show that the subsequential limit exists as a true limit. 
\end{abstract}

\date{\today}

\maketitle

\setlength{\parindent}{0pt}
\setlength{\parskip}{0\baselineskip plus 1pt minus 1pt}

\setcounter{tocdepth}{1}
\tableofcontents

\begin{figure}[h]
	\includegraphics[width=0.495\textwidth]{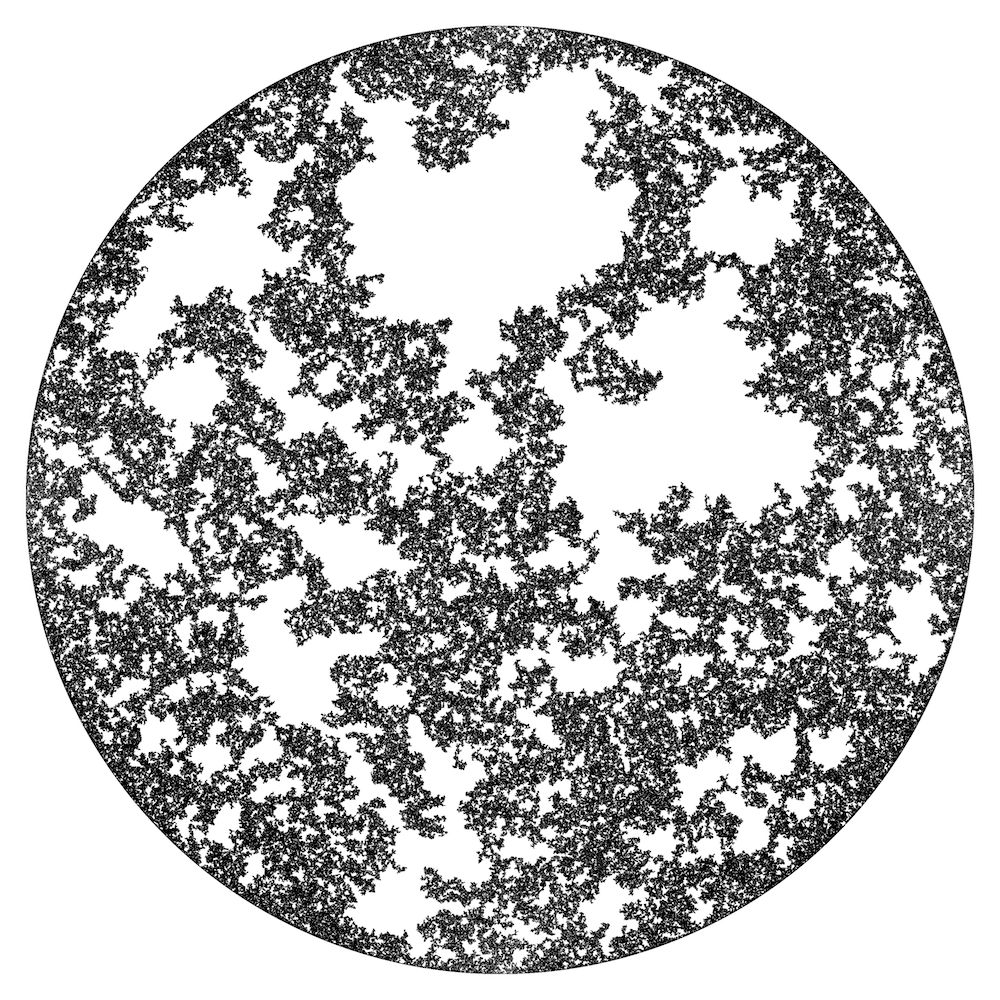}
	\includegraphics[width=0.495\textwidth]{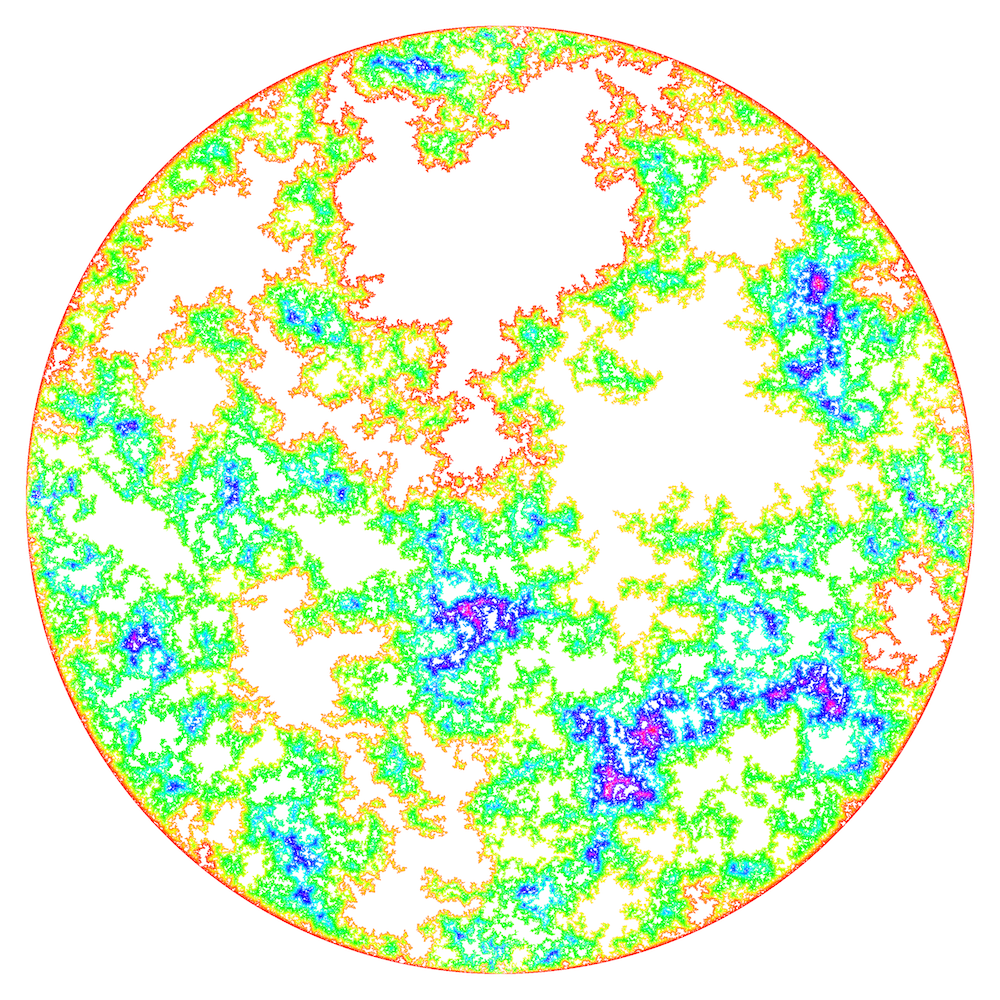}
	\caption{\label{fig:simulations}
		{{\bf Left:} A (discrete approximation of a) $\CLE_4$ in the unit disk $\BD$. {\bf Right:} The loops are colored according to their distance to $\partial \BD$. Equivalently, loops are colored according to the time at which they are discovered by the uniform exploration of Werner and Wu~\cite{CoInCLEExpl}.}}
\end{figure}

\setlength{\parindent}{0pt}
\setlength{\parskip}{0.5\baselineskip plus 1pt minus 1pt}

\section{Introduction}\label{sec:intro}

\subsection{Overview}\label{subsec:overview}

The \emph{conformal loop ensembles ($\CLE_\kappa$ for $\kappa \in (8/3, 8)$)} are random collections of non-crossing loops within simply connected domains in $\BC$ \cite{TreeCLE,CLE}. They are conformally invariant, meaning their laws are preserved under conformal automorphisms of the domain. CLEs are the natural loop variants of the Schramm--Loewner evolution (SLE)~\cite{s2000sle}. SLEs and CLEs are conjectured (and in several cases rigorously shown) to describe the scaling limits of interfaces in critical statistical mechanics models on planar lattices~\cite{s2001percolation,lsw2004lerw,ss2009dgff,s2010ising} as well as on random planar maps~\cite{s2016hc,lsw2017schnyder,gkmw2018active,kmsw2019bipolar,gm2021saw,gm2021percolation}.

Just like for $\SLE$, the value $\kappa=4$ is critical for CLE. For $\kappa \le 4$, the loops are simple and do not intersect each other or the domain boundary. On the other hand, for $\kappa \in (4,8)$ the loops are self-intersecting and intersect each other; each loop moreover intersects the domain boundary with positive probability~\cite{rs2005basic}. In this paper, we focus on the critical case $\kappa=4$. For $\kappa \in (4,8)$, since the loops intersect, there is a natural metric that one can define on the loops of a $\CLE_\kappa$. It is simply the graph metric where the vertices in the graph are the loops of the $\CLE_\kappa$ and there is an edge precisely when two distinct loops intersect. Since a $\CLE_\kappa$ is conformally invariant and this metric is a function of a $\CLE_\kappa$, it is trivially conformally invariant. This construction does not make sense when $\kappa = 4$ since the loops do not intersect each other. The main aim of this paper is to show that nevertheless such a conformally invariant metric on the loops of a $\CLE_4$ exists. See~\Cref{fig:simulations} for numerical simulations of the metric.

The program of constructing the $\CLE_4$ metric was initiated by Werner and Wu~\cite{CoInCLEExpl}, who defined the distance from any loop in the $\CLE_4$ to the domain boundary. Equivalently, they gave a definition of the growth of a metric ball from the domain boundary, where one explores the loops according to their distance from the domain boundary. They coined this the \emph{uniform exploration} (cf.~also~\cite{LevelLineGFFI,TVSGFF}). The reason for this terminology is that as one grows the uniform exploration from the domain boundary toward a target point $z \in \BD$, the $\CLE_4$ loops are discovered in a Poissonian way where, when a new loop is discovered, it is rooted at a point of the boundary of the component of the unexplored region containing $z$ just before this time, chosen according to the harmonic measure seen from $z$; equivalently, uniformly on $\partial \BD$ once that component is mapped conformally to $\BD$ with $z$ sent to $0$. It was not proved in~\cite{CoInCLEExpl} that the uniform exploration corresponds to the metric ball growth from the boundary of a metric defined on the entire $\CLE_4$. The aim of this paper is to prove this conjecture and to show that the metric satisfies a number of important properties; we will spell these out just below in~\Cref{subsec:setup}.  It was also conjectured in~\cite{CoInCLEExpl} that the uniform exploration is a.s.\ determined by the $\CLE_4$. That this holds will be a consequence of the sequels~\cite{kkmt2026cle4_part2,kkmt2026cle4_part3} to the present paper.  Let us already give a statement of the main result of this paper, which will be made more precise in \Cref{subsec:main_results}.

	\begin{theorem}[See~\Cref{thm:convergence_nonsimple_cle} for a more detailed statement]
		For all $\kappa\in(4,8)$, consider a $\CLE_\kappa$ in the unit disk $\BD$ equipped with the graph distance $D^\kappa$, where two distinct loops are adjacent if and only if they intersect. Then, for every sequence $\kappa_n \downarrow 4$, there exists a sequence $\ka_{\kappa_n} \uparrow \infty$ such that, along a subsequence, 
		\[
		\left( \ka_{\kappa_n}^{-1}D^{\kappa_n} (\SCL^{\kappa_n}(x), \SCL^{\kappa_n}(y))\right)_{x,y \in \BD \cap \BQ^2} \mathop{\longrightarrow}\limits_{n\to \infty}^{(\mathrm{d})}
		\left( D (\SCL(x), \SCL(y))\right)_{x,y \in \BD \cap \BQ^2},
		\]
		for the product topology, where $D$ is a distance on the set of loops of a $\CLE_4$ in $\BD$, and where $\SCL^{\kappa_n}(x)$ (resp.\ $\SCL(x)$) is the loop of the $\CLE_{\kappa_n}$ (resp.\ $\CLE_4$) that surrounds $x$. Moreover, the above convergence also holds for the distances to $\partial \BD$, and the distances $D(\partial \BD, \SCL(x))$ for $x \in \BD \cap \BQ^2$ have the law of the times at which the uniform exploration discovers each loop $\SCL(x)$. Finally, the limiting metric $D$ depends on the $\CLE_4$ in a conformally invariant and local manner.
	\end{theorem}

The reason that we have to introduce a normalization factor $\ka_\kappa^{-1}$ in the statement of the theorem is that, as $\kappa \downarrow 4$, the number of adjacent loops necessary to get from the boundary to the loop surrounding, e.g., the origin tends to $\infty$. In fact, $\ka_\kappa^{-1}$ is defined in~\Cref{subsec:main_results} as the probability that the loop surrounding the origin touches the boundary; the number of loops in such a shortest path is then geometric (starting at $1$) with success probability $\ka_\kappa^{-1}$, of mean $\ka_\kappa$.

The uniqueness of the limit is established in the subsequent papers~\cite{kkmt2026cle4_part2,kkmt2026cle4_part3}, so that the above scaling limit actually holds as $\kappa \downarrow 4$, not only along subsequences. The metric that we construct in this article together with~\cite{kkmt2026cle4_part2,kkmt2026cle4_part3} was previously constructed by Sheffield, Watson, and Wu in unpublished work using an approach which is completely distinct from the one taken here. In particular, we will show that the metric arises as an appropriately renormalized limit as $\kappa \downarrow 4$ of the adjacency graph metric of $\CLE_\kappa$ loops as opposed to the indirect approach taken by Sheffield, Watson, and Wu.

\subsubsection*{Random planar map motivation}

We anticipate that the $\CLE_4$ metric will have applications in the study of random planar maps. Recall that a \emph{planar map} is a graph together with an embedding into the plane so that no two edges cross. Two planar maps are considered to be equivalent if it is possible to deform one into the other using an orientation-preserving homeomorphism of the plane. In recent years, the limiting behavior of random planar maps has received a large amount of attention. Major results include the Gromov--Hausdorff--Prokhorov convergence of rescaled quadrangulations toward the Brownian sphere~\cite{MM06, Mie13, LG13} and, for Boltzmann planar maps with large faces, the subsequential scaling limits obtained in~\cite{LGM11} together with the convergence toward the $\alpha$-stable carpet (for $\alpha \in [3/2,2)$) or gasket (for $\alpha \in (1,3/2)$) established in~\cite{StabCarGas}. In this general framework, we expect that our random metric space describes the scaling limit of $3/2$-stable maps. These are random planar maps where vertices of degree $k\ge1$ are assigned a weight of order $k^{-2}$, first studied in the local limit in~\cite{BCM18CauchyMaps}. Although these maps do not admit a scaling limit in the sense of Gromov--Hausdorff or Gromov--Prokhorov~\cite{Kam25large32maps}, one of the authors obtained in~\cite{Kam25distancesonCLE4and32maps} a non-trivial scaling limit of the distances to the root, identified with a distance from the $\CLE_4$ loops to the boundary introduced in~\cite{BECriLQG}, which will be shown in future work to agree with our metric. As shown in~\cite{Kam26GasketsO2maps} (and in~\cite{ADSH26volumeOn} in the particular case of quadrangulations, see also~\cite{DSHPW25} for the integrable fully-packed case on triangulations), the gasket of a critical $O(2)$ loop-decorated planar map, obtained by forgetting the parts of the map that are inside the outermost loops, is an instance of the $3/2$-stable map. This makes our expectation agree with the conjecture that the scaling limit of critical $O(2)$ loop-decorated maps is described by the $\CLE_4$ together with an independent critical Liouville quantum gravity (see, e.g.,~\cite[Conjecture 2.1]{HL24weightedLQG} for a precise conjecture).

As mentioned above, in the dual setting it has been shown in~\cite{StabCarGas} that the scaling limit of random maps with large faces, also called $\alpha$-stable maps for $\alpha \in (1,2)$, with respect to the Gromov--Hausdorff topology, is given by a random metric space known as the $\alpha$-stable carpet (or gasket when $\alpha<3/2$). Consequently, our metric space can be understood as the ``dual'' of the $3/2$-stable carpet. We also expect that it is possible to construct analogous duals for the $(1/2 + 4/\kappa)$-stable carpets, where $\kappa \in (8/3, 4)$, corresponding to a metric on the loops of a $\CLE_\kappa$, though they would be conformally covariant instead of conformally invariant.

\subsubsection*{Relationship to other work}  In recent years, there have been many works aimed at constructing random metrics. In addition to the Brownian sphere~\cite{MM06, Mie13, LG13} and $\alpha$-stable carpet and gasket~\cite{StabCarGas} mentioned above, these include the $\sqrt{8/3}$-Liouville quantum gravity (LQG) metric~\cite{ms2019lqg_bm1,ms2021lqg_bm2,ms2021lqg_bm3}, the $\gamma$-LQG metric for $\gamma \in (0,2)$~\cite{TightLFPP,ExUniLQG}, and the CLE metric in the simple~\cite{TightSimCLE,GeoCLECarp} and non-simple regimes~\cite{TightNonsimCLE,ExUniCoCoGeoMetNonsimCLEGas}. The present work together with its sequels~\cite{kkmt2026cle4_part2,kkmt2026cle4_part3} differs in several essential ways from these works primarily because the metric on $\CLE_4$ is conformally \emph{invariant} rather than \emph{covariant}. In particular, arbitrarily small regions in the Euclidean sense can have arbitrarily large size in the $\CLE_4$-metric sense. This will manifest itself in particular in the proof of the uniqueness~\cite{kkmt2026cle4_part3}, which we emphasize will not proceed by showing that different subsequential limits are bi-Lipschitz equivalent as in~\cite{ExUniLQG,GeoCLECarp,ExUniCoCoGeoMetNonsimCLEGas}. We will instead obtain a representation of the $\CLE_4$ metric in terms of the uniform exploration which we expect will have future applications in relating the $\CLE_4$ metric to the scaling limit of $3/2$-stable maps.

\subsection{Setup}
\label{subsec:setup}

Let $U \subsetneq \BC$ be a deterministic simply connected domain. Let $\Gamma_U$ be a non-nested $\CLE_4$ in $U$. For each open subset $V \subseteq U$, we shall write 
\begin{equation}\label{eq:def-V-star}
	V^\star \defeq V \setminus \overline{\bigcup_{\SCL \in \Gamma_U: \SCL \not\subseteq V} \mathop{\mathrm{int}}(\SCL)}.
\end{equation}
Here $\mathop{\mathrm{int}}(\SCL)$ denotes the region surrounded by $\SCL$. Note that $V^\star$ is random, depending on $\Gamma_U$, a dependence we suppress from the notation; if $V$ is simply connected, so is each component of $V^\star$. For deterministic $z \in U$, we shall set $\SCL(z)$ to be the a.s.\ unique loop of $\Gamma_U$ that surrounds $z$.

Let $D$ be a metric on $\Gamma_U$. For $\SCL \in \Gamma_U$ and $t \ge 0$, we shall write 
\begin{equation*}
	\SCB_t(\SCL; D) \quad \text{(resp.~} \SCB_t^-(\SCL; D)\text{)}
\end{equation*}
(or simply $\SCB_t(\SCL)$ (resp.~$\SCB_t^-(\SCL)$), when there is no danger of confusion) for the closure of the union of the domains surrounded by $\SCL^\prime$ for $\SCL^\prime \in \Gamma_U$ with $D(\SCL, \SCL^\prime) \le t$ (resp.~$D(\SCL, \SCL^\prime) < t$). Note that $\SCB_0^-(\SCL; D) = \emptyset$. 
Note that $\SCB_t(\SCL; D) = \SCB_t^-(\SCL; D)$ unless there exists $\SCL^\prime \in \Gamma_U$ with $D(\SCL, \SCL^\prime) = t$. 

Let $A, B \subseteq \overline U$ be subsets. Then we shall write
	\begin{itemize}
		\item $D(\SCL, A) \defeq \inf\{t \ge 0: \SCB_t(\SCL; D) \cap A \neq \emptyset\}$ for all $\SCL \in \Gamma_U$; 
		\item $\SCB_t(A; D)$ (resp.~$\SCB_t^-(A; D)$) for the closure of the union of $A$ and the domains surrounded by $\SCL$ for $\SCL \in \Gamma_U$ with $D(\SCL, A) \le t$ (resp.~$D(\SCL, A) < t$) for all $t \ge 0$, except for $\SCB_0^-(A; D)\defeq \emptyset$; 
		\item $D(A, B) \defeq \inf\{t \ge 0 : \SCB_t(A; D) \cap B \neq \emptyset\}$.
	\end{itemize}
	(with the convention that $\inf\emptyset = \infty$). Note that $D(A,B) = D(B,A)$ (since one can see that $D(A, B)$ is the limit as $\varepsilon \downarrow 0$ of the infimum of $D(\SCL, \SCL')$ over pairs of loops $\SCL, \SCL' \in \Gamma_U$ such that $\mathrm{dist}(A, \mathrm{int}(\SCL))\le \varepsilon$ and $\mathrm{dist}(B, \mathrm{int}(\SCL'))\le \varepsilon$). We also write $D(\SCL, z)\defeq D(\SCL, \{z\})$ for all $z \in \overline{U}$.

We shall write $U_\BQ \defeq U \cap \BQ^2$ and $\BN \defeq \{1,2,\ldots\}$. By abuse of notation, we shall also denote by $D$ the mapping
\begin{equation*}
	D \colon U_\BQ \times U_\BQ \to \BR \colon (x, y) \mapsto D(\SCL(x), \SCL(y)). 
\end{equation*}
The same abuse of notation is used for the distance $D^U_{\Gamma^\kappa_U}$ that is defined in the next subsection.
We shall equip $\BR^{U_\BQ \times U_\BQ}$ with the product topology. Note that $\BR^{U_\BQ \times U_\BQ}$ is a Polish space. 

\subsection{Main results}
\label{subsec:main_results}

Let $\kappa \in (4, 8)$. Let $\Gamma_U^\kappa$ be a non-nested $\CLE_\kappa$ in $U$. We shall write $D_{\Gamma_U^\kappa}^U \colon \Gamma_U^\kappa \times \Gamma_U^\kappa \to \BR$ for the graph distance associated with $\Gamma_U^\kappa$, where distinct $\SCL_1, \SCL_2 \in \Gamma_U^\kappa$ are adjacent if and only if $\SCL_1 \cap \SCL_2 \neq \emptyset$. For deterministic $z \in U$, let $\SCL^\kappa(z)$ be defined in the same manner as $\SCL(z)$ but with $\Gamma_U^\kappa$ in place of $\Gamma_U$. We shall write
\begin{equation*}
    1/\ka_\kappa \defeq \BP\lbrack\SCL^\kappa(0) \cap \partial \BD \neq \emptyset\rbrack = 1 - \frac{\sin(\pi(\kappa/4 + 8/\kappa))}{\sin(\pi(\kappa/4 - 1))}
\end{equation*}
(cf.~\cite{BdryTouchingNonsimCLE}; we remark that the precise formula for $\ka_\kappa$ will not be important for this paper or the sequels~\cite{kkmt2026cle4_part2,kkmt2026cle4_part3}). Here $\SCL^\kappa(0)$ is the loop of a non-nested $\CLE_\kappa$ in $\BD$ surrounding the origin; by conformal invariance the right-hand side is unchanged if $(\BD, 0)$ is replaced by any $(U, x)$ with $x \in U$.

\begin{definition}\label{def:weak_axioms}
	We define a \emph{weak geodesic CLE$_4$ metric coupling} to be a family of couplings
	\begin{equation*}
		D= \left\{(\Gamma_U, D_{\Gamma_U}^U)\right\}_U
	\end{equation*}
	where $U \subsetneq \BC$ ranges over all simply connected domains, and for each $U$, $\Gamma_U$ is a non-nested CLE$_4$ in $U$, such that the following conditions are satisfied for every such $U$ and all deterministic choices of the auxiliary data ($V$, $I$, $\phi$, arcs) below:
	\begin{enumerate}[label=(\Roman*), ref=\Roman*]
		\item\label{it:weak_axiom_geodesic} {\bfseries (Weak geodesic metric)} $D_{\Gamma_U}^U$ is a.s.\ a metric on $\Gamma_U$ satisfying the following conditions:
		\begin{enumerate}[label=(\roman*), ref=\roman*]
			\item\label{it:weak_axiom_geodesic_0} Let $\SCL \in \Gamma_U$ and $t \ge 0$. Then the metric ball $\SCB_t(\SCL; D_{\Gamma_U}^U)$ is connected.
			\item\label{it:weak_axiom_geodesic_1} Let $\SCL \in \Gamma_U$ and $t \ge 0$. Let $\phi\colon U \to \BD$ be a conformal mapping. Then
			\begin{equation*}
				\lim_{t^\prime \downarrow t} \phi (\SCB_{t^\prime}(\SCL; D_{\Gamma_U}^U) )= \phi (\SCB_t(\SCL; D_{\Gamma_U}^U)) \quad \text{and} \quad \lim_{t^\prime \uparrow t} \phi (\SCB_{t^\prime}(\SCL; D_{\Gamma_U}^U)) = \phi (\SCB_t^-(\SCL; D_{\Gamma_U}^U))
			\end{equation*}
			with respect to the Hausdorff metric, the second limit being required only for $t > 0$; see \Cref{subsec:setup} for the convention used when a ball is empty.
			\item\label{it:weak_axiom_geodesic_2} Let $\SCL_1, \SCL_2 \in \Gamma_U$ and $t \ge 0$. Suppose that $\SCB_t(\SCL_1; D_{\Gamma_U}^U) \cap \SCL_2 = \emptyset$. Then
			\begin{equation*}
				D_{\Gamma_U}^U(\SCL_1, \SCL_2) = t + D_{\Gamma_U}^U(\SCB_t(\SCL_1; D_{\Gamma_U}^U), \SCL_2). 
			\end{equation*}
			\item\label{it:weak_axiom_geodesic_4} Let $\SCL_1, \SCL_2 \in \Gamma_U$. Then the closed set $K_{\SCL_1, \SCL_2}$ defined as the closure of $\{ z \in \overline U : D_{\Gamma_U}^U(\SCL_1, z)+ D_{\Gamma_U}^U(\SCL_2,z)= D_{\Gamma_U}^U(\SCL_1, \SCL_2) \}$ has a connected component containing $\SCL_1 \cup \SCL_2$.
		\end{enumerate}
		Items~\eqref{it:weak_axiom_geodesic_0}--\eqref{it:weak_axiom_geodesic_4} are also required with a deterministic connected arc (resp.\ two disjoint such arcs) of $\partial U$ in place of $\SCL$ or $\SCL_1$ (resp.\ $\SCL_1$ and $\SCL_2$).
		\item\label{it:weak_axiom_locality} {\bfseries (Locality)}
		Let $V \subset U$ be a deterministic simply connected subdomain. Let $I \subset \partial U$ be a deterministic connected arc. Let $\{V_j\}_j$ be the connected components of $V^\star$. Then there exists a coupling $(\Gamma_U, D_{\Gamma_U}^U, \{D_{\Gamma_U|_{V_j}}^{V_j}\}_j)$ such that the following hold: 
			\begin{itemize}
				\item Conditionally on the $\sigma$-algebra generated by
				\begin{multline}\label{eq:weak_axiom_locality}
					\bigl\{\SCL \in \Gamma_U : \SCL \not\subset V\bigr\}, \quad \bigl\{\SCB_t(\SCL; D_{\Gamma_U}^U) : \SCL \in \Gamma_U, \ \SCL \not\subset V, \ t \in [0, D_{\Gamma_U}^U(\SCL, \partial V^\star)]\bigr\}, \\
					\text{and} \quad \bigl\{\SCB_t(I; D_{\Gamma_U}^U) : t \in [0, D_{\Gamma_U}^U(I, \partial V^\star)]\bigr\}
				\end{multline}
				the $(\Gamma_U|_{V_j}, D_{\Gamma_U|_{V_j}}^{V_j})$'s are independent and their conditional laws are those of $(\Gamma_{V_j}, D_{\Gamma_{V_j}}^{V_j})$, respectively. 
				\item For all $j$, for all $\SCL_1, \SCL_2 \in \Gamma_U|_{V_j}$, we have $D_{\Gamma_U}^U(\SCL_1, \SCL_2)\le D_{\Gamma_U|_{V_j}}^{V_j}(\SCL_1, \SCL_2)$.
				\item For all $j$, for all $\SCL \in \Gamma_U|_{V_j}, \ t \in [0, D_{\Gamma_U}^U(\SCL, \partial V_j)]$, we have $\SCB_t(\SCL; D_{\Gamma_U}^U) = \SCB_t(\SCL; D_{\Gamma_U|_{V_j}}^{V_j})$.
		\end{itemize} 
		\item\label{it:weak_axiom_conformal_invariance} {\bfseries (Conformal invariance)} Let $\phi \colon U \to \phi(U)$ be a deterministic conformal mapping. Then,
		\begin{equation*}
				\left(\phi(\Gamma_U), \left( D_{\Gamma_U}^U\left(\phi^{-1}(\SCL_1), \phi^{-1}(\SCL_2) \right) \right)_{\SCL_1, \SCL_2 \in \phi(\Gamma_U)} \right)
				\overset{(\mathrm{d})}{=} \left( \Gamma_{\phi(U)}, \left(D_{\Gamma_{\phi(U)}}^{\phi(U)}(\SCL_1, \SCL_2) \right)_{\SCL_1, \SCL_2 \in \Gamma_{\phi(U)}} \right).
		\end{equation*}
		Here both sides are viewed as random elements of the space of pairs consisting of a loop ensemble in $\phi(U)$ and a symmetric function on its pairs of loops; equivalently, by the reformulation of \Cref{subsec:setup}, as random elements of $\BR^{\phi(U)_\BQ \times \phi(U)_\BQ}$ together with the ensemble.
		\item\label{it:weak_axiom_uniform_exploration} {\bfseries (Uniform exploration)} The collection $\{(\SCL, D_{\Gamma_U}^U(\SCL, \partial U))\}_{\SCL \in \Gamma_U}$ has the law of a uniform exploration of $\Gamma_U$ (cf.~\cite{CoInCLEExpl}; see \Cref{subsec:labeled_cle_4}). That is, the pair consisting of the ensemble and its labels has the law of the labeled CLE$_4$ in $U$, namely of a CLE$_4$ decorated by its uniform exploration.
	\end{enumerate}
\end{definition}

Note that Axiom~\eqref{it:weak_axiom_geodesic}\eqref{it:weak_axiom_geodesic_2} can equivalently be written as follows, using that $D_{\Gamma_U}^U$ is symmetric. For all $t\ge 0$ such that $\SCB_t(\SCL_1; D_{\Gamma_U}^U) \cap \SCL_2 = \emptyset$, we have
\[
D_{\Gamma_U}^U(\SCL_1, \SCL_2) = t + \inf\{s \ge 0: \SCB_s(\SCL_2;D_{\Gamma_U}^U) \cap \SCB_t(\SCL_1;D_{\Gamma_U}^U) \neq \emptyset \}.
\]
In particular, taking $t=0$ and noting that $\SCB_0(\SCL_1; D_{\Gamma_U}^U) \cap \SCL_2 = \emptyset$ a.s.\ for distinct loops of a non-nested $\CLE_4$, we have
\begin{equation*}
	D_{\Gamma_U}^U(\SCL_1, \SCL_2) = \inf\{t \ge 0: \SCB_t(\SCL_1; D_{\Gamma_U}^U) \cap \SCL_2 \neq \emptyset\}, \quad \forall \SCL_1 \neq \SCL_2 \in \Gamma_U.
\end{equation*}
Next, we state the main results of the paper. The first main result states that there exists a weak geodesic $\CLE_4$ metric coupling that can be obtained as a subsequential limit in law as $\kappa \downarrow 4$ of the graph distance on $\Gamma_U^{\kappa}$ (appropriately re-scaled).

\begin{theorem}[Convergence of nonsimple CLE graph distance]\label{thm:convergence_nonsimple_cle}
For any sequence $\kappa_n \downarrow 4$, there exist a subsequence and a weak geodesic $\CLE_4$ metric coupling $D$ such that for every simply connected domain $U \subsetneq \BC$, the random metrics $\ka_{\kappa_n}^{-1} D_{\Gamma_U^{\kappa_n}}^U$ converge in law along this subsequence to $D_{\Gamma_U}^U$ in $\BR^{U_\BQ \times U_\BQ}$ equipped with the product topology.
\end{theorem}

\begin{remark}
	The above result actually holds in a stronger sense.
	\begin{itemize}
		\item Thanks to the uniqueness proved in the subsequent papers, the above result implies that $\ka_{\kappa}^{-1} D_{\Gamma_U^{\kappa}}^U$ converges in law to $D_{\Gamma_U}^U$ as $\kappa \downarrow 4$.
		\item Jointly with the above convergence, we also prove the convergence of $\ka_{\kappa_n}^{-1} D_{\Gamma_U^{\kappa_n}}^U(\partial U, \SCL^{\kappa_n}(x))$ to $D_{\Gamma_U}^U(\partial U, \SCL(x))$ for $x \in U_\BQ$. See, e.g.,~\Cref{prop subsequential limit boundary} for more details on the convergence.
	\end{itemize}
\end{remark}

\subsection{Tightness of random metric spaces}
Let us give here some general background on subsequential limits of sequences of random metric spaces. In this subsection, we also explain why our approach differs from the classical approaches to random metrics.

\emph{When there is no natural embedding.} A general approach for scaling limits of random metric spaces which are not equipped with a particular embedding in some reference metric space is to prove in some sense the convergence in distribution of the matrix of distances between random points toward the matrix of distances in the limiting metric space. This sense of convergence can be made precise using the so-called Gromov--Hausdorff, Gromov--Prokhorov, or Gromov--Hausdorff--Prokhorov topologies. Proving the existence of subsequential limits in this context often requires precise concentration estimates and relies on chaining techniques~\cite{Tal05}. See for example the foundational work~\cite{Ald91} for convergence of random trees toward the Brownian continuum random tree (CRT). For random planar maps, the convergence of a process that encodes the metric space can be useful to obtain the tightness; see~\cite{LG07} for the existence of subsequential scaling limits of $2p$-angulations and~\cite{LGM11} for the analogous result for random maps with large faces. In our case, we do not work with such topologies but with the product topology where the loops are labeled by the points with rational coordinates that they surround. This is why we follow a different approach.

\emph{When there is a natural embedding in the plane.} A case that is in some sense closer to this work is when the metric space has a natural embedding in a domain of the Euclidean plane. A famous example is the LQG metric; see~\cite{IntroLQGMet} for an introduction to this metric. When such an embedding exists, one can hope to compare the metric to the Euclidean metric and use some geometric arguments. For instance, in~\cite{TightLFPP}, the authors prove the tightness of an approximation scheme for the LQG metric using Russo--Seymour--Welsh arguments as well as independence across scales. To get the convergence, Ding, Dub\'edat, Dunlap, and Falconet~\cite{TightLFPP} also rely on an Efron--Stein-type argument which involves resampling the metric in some box. The uniqueness of the metric was then obtained in~\cite{ExUniLQG, UniCriSupercriLQGMet}, which characterize the limit using a list of axioms. The tightness of the chemical distance on the carpet of simple CLEs was obtained in~\cite{TightSimCLE} and the limiting metric was later characterized in~\cite{GeoCLECarp} by a list of axioms using closely related techniques. The same program was also achieved on non-simple CLEs in~\cite{TightNonsimCLE, ExUniCoCoGeoMetNonsimCLEGas} relying on a different resampling operation. In our case, in addition to the above-mentioned conformal invariance that complicates the analysis, the distance $D^U_{\Gamma_U}$ cannot be extended to a continuous function on $\overline{U} \times \overline{U}$, so that one cannot hope for a tightness result with respect to uniform convergence, in contrast to~\cite{TightLFPP}.

The object of study of this paper is somewhat in-between these two broad classes of random metric spaces since, even though the $\CLE_4$ is naturally embedded in a planar domain, the associated distance cannot be compared easily with the Euclidean distance. The fact that we take a different approach will be even more apparent in the next two papers~\cite{kkmt2026cle4_part2, kkmt2026cle4_part3}.
\subsection{Techniques and strategy of proof}\label{subsec:techniques}

Recall that the purpose of this paper is to construct a weak $\CLE_4$ metric coupling $D$ (see~\Cref{def:weak_axioms}) as the scaling limit, along a subsequence, of the renormalized graph distance $\ka_\kappa^{-1} D^U_{\Gamma^\kappa_U}$ as $\kappa \downarrow 4$. 

In the literature, the construction of the limit of a sequence of random metric spaces is often split into two parts. The first one is establishing tightness, which by Prokhorov's theorem yields the existence of subsequential limits, and the second part consists in proving the uniqueness of such a limit. Tightness often requires careful work, depending on the topology, while the uniqueness part relies on a deeper understanding of the geometry of the limit. 

We emphasize at the outset that proving tightness of the mapping from $U_\BQ \times U_\BQ$ to $\BR_{\ge 0}$ for the product topology is not the most difficult part of this paper. The main difficulty in the proof of~\Cref{thm:convergence_nonsimple_cle} is to check that the subsequential limit indeed satisfies the axioms of~\Cref{def:weak_axioms}.

In order to construct $D$, we will first construct a related metric $D^\partial$, defined on
\begin{equation*}
	\widehat{\Gamma}_U \defeq \Gamma_U \cup \{\partial U\},
\end{equation*}
i.e., $\partial U$ (the domain boundary) is treated as a distinguished point of the metric space. The metric $D^\partial$ can be recovered from $D$ via
\begin{equation*}
	D^\partial(\SCL_1, \SCL_2) = D(\SCL_1, \SCL_2) \wedge \bigl(D(\SCL_1, \partial U) + D(\partial U, \SCL_2)\bigr), \qquad D^\partial(\SCL, \partial U) = D(\SCL, \partial U).
\end{equation*}
We will nevertheless begin by constructing $D^\partial$ rather than $D$, since its construction is simpler and because it will play a key role in the subsequent construction of $D$. The construction of $D^\partial$, as a subsequential scaling limit of the graph distance on $\Gamma_U^\kappa \cup \{\partial \BD\}$, will be achieved by first obtaining the scaling limit of the distances to the boundary using results of Aru, Holden, Powell, and Sun~\cite{BECriLQG}, who prove the convergence of the branching $\SLE_{\kappa}(\kappa-6)$ exploration to the uniform exploration of the $\CLE_4$ as well as the convergence of the loops of the $\CLE_\kappa$ surrounding points with rational coordinates toward the loops of the $\CLE_4$. Then, we describe the metric balls from the loops using inversion invariance of the CLE (see~\cite{CoInCLERiemSph} for $\kappa>4$ and~\cite{SimCLEDblConnDom} for $\kappa=4$).

Using the convergence to $D^\partial$ and some tightness arguments, we construct $D$ and the associated metric balls as a subsequential limit of the renormalized graph distance $\ka_\kappa^{-1} D^U_{\Gamma^\kappa_U}$ on $\Gamma_U^\kappa$. The associated metric balls are actually constructed as subsequential scaling limits of metric balls on the $\CLE_\kappa$. We then prove that they are indeed given by the closure of the union of the domains surrounded by the loops at distance at most $t$ for all $t\ge 0$, thus matching the definition given in~\Cref{subsec:setup}.

The conformal invariance of the subsequential limit (Axiom~\eqref{it:weak_axiom_conformal_invariance}) comes as a by-product of the conformal invariance of the $\CLE_\kappa$ for $\kappa>4$. Moreover, Axiom~\eqref{it:weak_axiom_uniform_exploration} will stem from the fact that the $D$-distance to the boundary is equal to the $D^\partial$-distance to the boundary, which will be given by the uniform exploration.

Thus, the only axioms which will require some work are the first two Axioms~\eqref{it:weak_axiom_geodesic} and~\eqref{it:weak_axiom_locality}.

The first one, Axiom~\eqref{it:weak_axiom_geodesic} (weak geodesic metric), will play a key role in the next paper~\cite{kkmt2026cle4_part2} to prove the existence of geodesics. It describes some properties of metric balls: namely their connectedness~\eqref{it:weak_axiom_geodesic}\eqref{it:weak_axiom_geodesic_0} and the fact that metric balls are c\`adl\`ag with respect to the radius~\eqref{it:weak_axiom_geodesic}\eqref{it:weak_axiom_geodesic_1}. The two points~\eqref{it:weak_axiom_geodesic}\eqref{it:weak_axiom_geodesic_0} and~\eqref{it:weak_axiom_geodesic}\eqref{it:weak_axiom_geodesic_1}, as well as~\eqref{it:weak_axiom_geodesic}\eqref{it:weak_axiom_geodesic_4} which describes in some weak sense the points in a geodesic between two loops, will be consequences of the subsequential limit and of the tightness of the balls for the Skorokhod topology. Let us stress that~\eqref{it:weak_axiom_geodesic}\eqref{it:weak_axiom_geodesic_2} is not a simple consequence of the triangle inequality due to the definition of the metric balls in~\Cref{subsec:setup}. Surprisingly,~\eqref{it:weak_axiom_geodesic}\eqref{it:weak_axiom_geodesic_2} will require substantial work.

The second one, Axiom~\eqref{it:weak_axiom_locality} (locality), is also difficult to check. At first glance, it seems to be a simple consequence of the restriction property for the $\CLE_\kappa$ with $\kappa>4$. Yet, the Hausdorff convergence of the $\CLE_\kappa$ loops surrounding rational points to the loops of the $\CLE_4$ is not enough to get the convergence of the connected components of the set $V^\star$ obtained by removing from $V$ the closure of the union of the interiors of the loops which do not lie in $V$. Indeed, it may happen for instance that there is a large but very thin loop of the $\CLE_\kappa$. 

One key intermediate result toward checking Axiom~\eqref{it:weak_axiom_locality} is a uniform bound on the fatness of $\CLE_\kappa$ loops when $\kappa \downarrow 4$. We prove in~\Cref{prop:thin-loops} that for all $\varepsilon>0$, there exists $\delta>0$ such that, uniformly as $\kappa \downarrow 4$, with probability at least $1-\varepsilon$ all $\CLE_\kappa$ loops with diameter at least $\varepsilon$ surround a Euclidean ball of radius $\delta$. We also prove the same result when $\kappa \uparrow 4$ in~\Cref{prop:big_loops_are_fat_simple_cle}. Such a statement is clearly true for a fixed $\kappa$, and one can show that it is true locally uniformly in $\kappa \in (4, 8)$ using the space-filling SLE. However, when $\kappa \downarrow 4$, the space-filling SLE degenerates. We prove~\Cref{prop:thin-loops} using couplings from~\cite{CLEPerc} and using the analogous statement for $\kappa \uparrow 4$ and the convergence of $\CLE_\kappa$ loops as $\kappa \uparrow 4$ in terms of closed paths. The analogous statement for $\kappa \uparrow 4$ is obtained via the coupling with Brownian loop soups.

Let us conclude this subsection with the following observation. One can note that, by conformal invariance, it suffices to consider the case where $U = \BD$. We record here the following equivalent formulation of the axioms of~\Cref{def:weak_axioms}, which will be more convenient to use in this paper. We chose to first present~\Cref{def:weak_axioms} before the following axioms inasmuch as~\Cref{def:weak_axioms} will be used in the further work~\cite{kkmt2026cle4_part2}.

\begin{lemma}\label{def:weak_axioms2}
	Taking images under conformal maps is a bijection, with inverse the restriction to $U=\BD$, from the couplings $(\Gamma,D)$, $\Gamma$ a non-nested $\CLE_4$ in $\BD$ and $D$ a metric on $\Gamma$, satisfying the following conditions onto the weak geodesic $\CLE_4$ metric couplings of~\Cref{def:weak_axioms}:
	\begin{enumerate}[label=(\Roman*), ref=\Roman*]
		\item\label{it:weak_axiom_geodesic2} {\bfseries (Weak geodesic metric)} Recall that the metric ball $\SCB_t(\SCL)=\SCB_t(\SCL; D)$ (resp.~$\SCB_t^-(\SCL)=\SCB_t^-(\SCL; D)$) denotes the closure of the union of the domains surrounded by $\SCL^\prime$ for $\SCL^\prime\in\Gamma$ with $D(\SCL,\SCL^\prime) \le t$ (resp.~$D(\SCL,\SCL^\prime) < t$). For a general subset $A\subset\overline\BD$, define
		\begin{equation*}
			D(\SCL, A) \defeq \inf\bigl\{t \ge 0: \SCB_t(\SCL) \cap A \neq \emptyset\bigr\}, 
		\end{equation*}
		and in particular $D(\SCL, z) \defeq D(\SCL, \{z\})$ for all $z \in \overline{\BD}$. Also write $\SCB_t(A)=\SCB_t(A;D)$ (resp.~$\SCB_t^-(A)=\SCB_t^-(A;D)$) for the closure of the union of $A$ and of the domains surrounded by $\SCL$ for $\SCL\in\Gamma$ with $D(\SCL,A)\le t$ (resp.~$D(\SCL,A)<t$). Then:
		\begin{enumerate}[label=(\roman*), ref=\roman*]
			\item\label{it:weak_axiom_geodesic_02} Let $\SCL \in \Gamma$ and $t \ge 0$. Then the metric ball $\SCB_t(\SCL)$ is connected. The same is true with a deterministic connected arc of $\partial \BD$ in place of $\SCL$.
			\item\label{it:weak_axiom_geodesic_12} Let $\SCL \in \Gamma$ and $t \ge 0$. Then
			\begin{equation*}
				\lim_{t^\prime \downarrow t} \SCB_{t^\prime}(\SCL) = \SCB_t(\SCL) \quad \text{and} \quad \lim_{t^\prime \uparrow t} \SCB_{t^\prime}(\SCL) = \SCB_t^-(\SCL)
			\end{equation*}
			with respect to the Hausdorff metric, the second limit being required only for $t > 0$; see \Cref{subsec:setup} for the convention used when a ball is empty. The same is true with a deterministic connected arc of $\partial \BD$ in place of $\SCL$.
			\item\label{it:weak_axiom_geodesic_22} Let $\SCL_1, \SCL_2 \in \Gamma$ and $t \ge 0$. Suppose that $\SCB_t(\SCL_1) \cap \SCL_2 = \emptyset$. Then
			\begin{equation*}
				D(\SCL_1, \SCL_2) = t + D( \SCL_2, \SCB_t(\SCL_1)). 
			\end{equation*}
			The same is true with a deterministic connected arc (resp.~two deterministic and disjoint connected arcs) of $\partial \BD$ in place of $\SCL_1$ (resp.~$\SCL_1$ and $\SCL_2$).
			\item\label{it:weak_axiom_geodesic_42} Let $\SCL_1, \SCL_2 \in \Gamma$. Then, the compact set $K_{\SCL_1, \SCL_2}$ defined as the closure of $ \{ z \in \overline{\BD}, \ D(\SCL_1, z)+ D(\SCL_2,z)= D(\SCL_1, \SCL_2) \}$ contains a connected component $\widetilde{K}_{\SCL_1, \SCL_2} $ that contains both $\SCL_1$ and $\SCL_2$. The same is true with a deterministic connected arc (resp.~two deterministic and disjoint connected arcs) of $\partial\BD$ in place of $\SCL_1$ (resp.~$\SCL_1$ and $\SCL_2$). 
		\end{enumerate}
		\item\label{it:weak_axiom_locality2} {\bfseries (Locality)} Let $V \subseteq \BD$ be a deterministic simply connected subdomain. Let $\{V_j\}_j$ be the connected components of $V^\star$ (cf.~\eqref{eq:def-V-star}). Let $I$ be a deterministic connected arc of $\partial \BD$. For each $j$, let $\phi_j$ be a conformal mapping from $V_j$ onto $\BD$ that is a.s.\ determined by
		\begin{multline}\label{eq:weak_axiom_locality2}
			\bigl\{\SCL \in \Gamma: \SCL \not\subseteq V\bigr\} \quad \text{and} \quad \bigl\{\SCB_t(\SCL): \SCL \in \Gamma, \ \SCL \not\subseteq V, \ t \in [0, D(\SCL, \partial V^\star)]\bigr\}, \\
			\text{and} \quad \bigl\{\SCB_t(I): t \in [0, D(I, \partial V^\star)]\bigr\}.
		\end{multline}
		There exists a coupling of $(\Gamma, D)$ with some $(\phi_j(\Gamma\vert_{V_j}), D_j)$'s such that 
		\begin{itemize}
			\item Conditionally on~\eqref{eq:weak_axiom_locality2}, the $(\phi_j(\Gamma\vert_{V_j}), D_j)$'s are independent and their conditional laws are those of $(\Gamma, D)$.
			\item For all $j$, for all $\SCL_1, \SCL_2 \in \Gamma\vert_{V_j}$, we have $D(\SCL_1, \SCL_2)\le D_j(\phi_j(\SCL_1), \phi_j(\SCL_2))$.
			\item For all $j$, for all $\SCL \in \Gamma|_{V_j}, \ t \in [0, D(\SCL, \partial V_j)]$, the closure of the image under $\phi_j^{-1}$ of the intersection with $\BD$ of the metric ball of radius $t$ for $D_j$ centered at $\phi_j(\SCL)$ is $\SCB_t(\SCL)$.
		\end{itemize}       
		\item\label{it:weak_axiom_conformal_invariance2} {\bfseries (Conformal invariance)} Let $\phi \colon \BD \to \BD$ be a deterministic conformal automorphism of $\BD$. Then $(\Gamma, D)$ and $(\phi(\Gamma), \phi_\ast(D))$ have the same law, where $\phi_\ast(D)$ is defined by setting $\phi_\ast (D)(\phi(\SCL_1), \phi(\SCL_2)) = D(\SCL_1, \SCL_2)$.  
		\item\label{it:weak_axiom_uniform_exploration2} {\bfseries (Uniform exploration)} The collection $\{(\SCL, D(\SCL, \partial\BD))\}_{\SCL \in \Gamma}$ has the law of a uniform exploration of $\Gamma$ (see~\Cref{subsec:labeled_cle_4} for the definition of the uniform $\CLE_4$ exploration).
	\end{enumerate}
\end{lemma}
\begin{proof}
	Let $U$ be a simply connected domain of $\BC$ which is different from $\BC$. One can define the coupling $(\Gamma_U, D^U_{\Gamma_U})$ from the above coupling $(\Gamma, D)$ by taking the image under a conformal map $\varphi$ from $\BD$ to $U$. Such a coupling does not depend on the choice of $\varphi$ by conformal invariance. Then, the family of couplings $\{(\Gamma_U, D^U_{\Gamma_U})\}_U$ satisfies the axioms of~\Cref{def:weak_axioms}. Conversely, such a family is determined by its restriction to $U=\BD$, so the two constructions are inverse to each other, and it suffices to check that each axiom of~\Cref{def:weak_axioms} for all $U$ is equivalent to its counterpart below for $U=\BD$. This is immediate for~\eqref{it:weak_axiom_conformal_invariance2},~\eqref{it:weak_axiom_uniform_exploration2} and the geodesic axioms, whose statements involve only the loops, the metric balls and the conformally invariant notions of connectedness and of surrounding. For locality, $\varphi$ maps the loops not contained in $V$ to those not contained in $\varphi(V)$, hence $V^\star$ and its components to $\varphi(V)^\star$ and its components, and the metric balls of $D$ to those of $D^U_{\Gamma_U}$.
\end{proof}

\subsection{Outline}

Let us give an outline of the content of the rest of the paper.

\smallskip
\noindent\emph{\Cref{section:preliminaries}: Preliminaries.} We will review the necessary background material on the Gaussian free field (GFF), Schramm--Loewner evolutions (SLE) with force points, conformal loop ensembles (CLE); couplings between the GFF, $\CLE_4$, and boundary conformal loop ensembles (BCLE); the uniform exploration; and Carath\'eodory convergence.

\smallskip
\noindent\emph{\Cref{sec:boundary_is_a_point}: When the boundary is a point.} We construct the subsequential scaling limit of the $\CLE_\kappa$ graph distance when $\kappa \downarrow 4$ in the case where we identify the boundary of the domain with a point of the metric space. Building on results of~\cite{BECriLQG}, we prove that balls from the boundary correspond to the uniform exploration. Using the inversion invariance of the CLE in the annulus, we also describe metric balls from any loop in terms of the uniform exploration. This construction is useful in order to construct the metric in the case where the boundary is not identified with a point.

\smallskip
\noindent\emph{\Cref{sec:boundary not a point}: When the boundary is not a point.} In this section, we construct the subsequential limiting metric in the case that the boundary of the domain is not a single point; this is in fact the case in which we are primarily interested.  We first show that the rescaled graph distances $\ka_\kappa^{-1} D_{\Gamma_U^\kappa}^U$ for $\kappa \in (4, 8)$ are tight as $\kappa \downarrow 4$. Then, we construct the associated metric balls and prove that they are c\`adl\`ag via a tightness argument and a comparison between these metric balls and the metric balls obtained in the previous section.

\smallskip
\noindent\emph{\Cref{sec:big non simple CLE loops are fat}: Large non-simple CLE loops are fat.} Next, in order to check that our metric satisfies the axioms of~\Cref{def:weak_axioms}, we establish a stronger convergence of loops than the Hausdorff convergence obtained in~\cite{BECriLQG}. More precisely, we prove that the closure of the union of loops intersecting a given connected compact set converges. This is a consequence of the fact that uniformly in $\kappa \downarrow 4$, large $\CLE_\kappa$ loops surround a large Euclidean ball. To prove this fact, we will first prove an analogous statement when $\kappa \uparrow 4$ using the coupling with Brownian loop soups and prove the convergence of simple loops as $\kappa \uparrow 4$ in terms of closed paths. Then, relying on the couplings from~\cite{CLEPerc}, we show the desired result for large non-simple loops. On a first reading, one can skip the proofs in this section. They are, however, of independent interest and may have other applications in further works on CLE.

\smallskip
\noindent\emph{\Cref{sec: distance between two segments}: On the distance between two segments.} In this section, we prove that for all $\varepsilon>0$ one can find two disjoint segments on the boundary of the domain whose $\CLE_\kappa$ graph distance is at most $\varepsilon \ka_\kappa$ with probability at least $1-\varepsilon$, uniformly in $\kappa \downarrow 4$. This technical result will be useful to check the axioms.

\smallskip
\noindent\emph{\Cref{sec:checking_the_axioms}: Checking the axioms.} We finally check that our subsequential limiting metric satisfies the axioms of~\Cref{def:weak_axioms}, concluding the proof of~\Cref{thm:convergence_nonsimple_cle}.

\subsection*{Acknowledgements}

E.K.~acknowledges the support of a Research Fellowship from Emmanuel College, Cambridge. J.M.~received support from ERC consolidator grant ARPF (Horizon Europe UKRI G120614). Y.T.~was supported by a Cambridge International Scholarship from the Cambridge Trust.  K.K.~was supported by the Simons Collaboration Grant
\emph{Probabilistic Paths to Quantum Field Theory}. We thank Wendelin Werner for stimulating discussions at an early stage of this work. E.K.~also thanks Juhan Aru and Ellen Powell for insightful discussions on this topic before starting this work.

\section{Preliminaries}\label{section:preliminaries}

\subsection{Gaussian free fields}
\label{subsec:gff}

We will now briefly recall some of the basics of the Gaussian free field (GFF); see~\cite{ss2007gff} for a more in-depth review. Let $G \subseteq \BC$ be a simply connected domain with harmonically non-trivial boundary, and let $H_0^1(G)$ denote the Hilbert space closure of $C_0^{\infty}(G)$ with respect to the Dirichlet inner product 
\begin{align*}
(f,g)_{\nabla} = \frac{1}{2\pi} \int_{G} \nabla f(z) \cdot \nabla g(z) \,\mathrm{d}z.
\end{align*}
The zero-boundary Gaussian free field (GFF) $h$ is the random distribution defined by (the series converging in the sense explained below)
\begin{align}\label{eqn:gff_series}
h = \sum_{n \geq 1} \alpha_n \phi_n,
\end{align}
where $(\phi_n)_{n \geq 1}$ is a $(\cdot, \cdot)_{\nabla}$-orthonormal basis of $H_0^1(G)$ and $(\alpha_n)_{n \geq 1}$ is a sequence of independent $\mathcal{N}(0,1)$-distributed random variables; the law of $h$ does not depend on the choice of $(\phi_n)_{n \ge 1}$. 

A GFF on $G$ is not a function but rather a random variable in the space of distributions on $G$, since the series in~\eqref{eqn:gff_series} does not converge in $H_0^1(G)$ a.s. Nevertheless, it converges a.s.\ in $H^{-1}(G)$, the dual space of $H_0^1(G)$, when $G$ is bounded, and in the space of distributions on $G$ in general (by conformal invariance). We say that a GFF on $G$ has boundary conditions given by a bounded measurable function $f$ defined on $\partial G$ if it can be expressed as the sum of a zero-boundary GFF on $G$ plus the harmonic extension of $f$ from $\partial G$ to $G$.

The following are two important properties of the GFF that we will repeatedly use.
\begin{itemize}
\item \emph{The Markov property.} If $h$ is a zero-boundary GFF on $G$ and $U \subseteq G$ is open, then we can write $h = h_U^0 + h_U$, where $h_U^0$ is a zero-boundary GFF on $U$ and $h_U$ is a distribution on $G$ which is harmonic on $U$. Moreover, $h_U^0$ and $h_U$ are independent.
\item \emph{Conformal invariance.} If $\phi: G \to G'$ is a conformal transformation and $h$ is a GFF on $G$, then $h \circ \phi^{-1}$ is a GFF on $G'$. 
\end{itemize}
There is also a variant of the Markov property where one thinks of conditioning the values of the GFF on a set which is random, in analogy with the strong Markov property and stopping times for Brownian motion. Specifically, a random closed set $A \subseteq G$ coupled with $h$ is said to be \emph{local} for $h$ \cite{MR3101840} if we can write $h = h_A^0 + h_A$ where $h_A$ is a distribution on $G$ which is harmonic on $G \setminus A$ and, given the $\sigma$-algebra generated by $(A,h_A)$, $h_A^0$ has the law of a GFF on $G \setminus A$ with zero boundary conditions.

\subsection{Chordal $\SLE_{\kappa}(\underline{\rho})$ processes}
\label{subsec:chordal_sle}

The chordal $\SLE_\kappa(\underline{\rho})$ processes are variants of $\SLE_\kappa$ where one keeps track of extra marked points~\cite{CoRestr}. To define them, we fix $\kappa>0$, let $\underline{x}_L = (x_{\ell,L},\ldots,x_{1,L})$ and $\underline{x}_R = (x_{1,R},\ldots,x_{r,R})$, where $x_{\ell,L} < \cdots < x_{1,L} \leq 0 \leq x_{1,R} < \cdots < x_{r,R}$, and let $\underline{\rho}_L = (\rho_{1,L},\ldots,\rho_{\ell,L})$ and $\underline{\rho}_R = (\rho_{1,R},\ldots,\rho_{r,R})$, where $\rho_{j,q} \in \BR$ for $q \in \{L,R\}$ and $j \in [1, N_q]_\BZ$ with $N_L = \ell$, $N_R = r$, where $[a,b]_\BZ \defeq [a,b] \cap \BZ$. Let $(g_t)$ denote the solution to the ODE:
\begin{align}\label{eqn:loewner_ode}
\partial_t g_t(z) = \frac{2}{g_t(z) - W_t}, \quad g_0(z) = z \in \BH,
\end{align}
where $W$ is a solution to:
\begin{align}\label{eqn:multiforce_point_sde}
&\mathrm{d} W_t = \sum_{j=1}^{\ell} \frac{\rho_{j,L}}{W_t - V_t^{j,L}} \mathrm{d}t + \sum_{j=1}^r \frac{\rho_{j,R}}{W_t - V_t^{j,R}} \mathrm{d}t + \sqrt{\kappa} \mathrm{d}B_t, \quad W_0 = 0,\notag \\
&\mathrm{d}V_t^{j,q} = \frac{2}{V_t^{j,q} - W_t} \mathrm{d}t, \quad V_0^{j,q} = x_{j,q}, \  j \in [1, N_q]_\BZ, \ q \in \{L,R\},
\end{align}
for a standard Brownian motion $B$.

For any value of $\underline{\rho}_L$ and $\underline{\rho}_R$, it is clear that~\eqref{eqn:multiforce_point_sde} has a unique strong solution (when $x_{1,L} < 0 < x_{1,R}$) until the first time $t$ such that $W_t=V^{j,q}_t$ for some $q \in \{L,R\}$ and $j \in [1, N_q]_\BZ$. Actually, it was shown in~\cite{IG1,LevelLineGFFI} that, when $\sum_{j=1}^k \rho_{j,L} > -2$ for all $k \in [1,\ell]_\BZ$ and $\sum_{j=1}^k \rho_{j,R} > -2$ for all $k \in [1,r]_\BZ$, there exists a solution to~\eqref{eqn:multiforce_point_sde} defined for all times $t\geq 0$ (also for force points at $0^-$, $0^+$), and for this solution the set of $t \geq 0$ for which $W_t=V^{j,q}_t$ for some $q \in \{L,R\}$ and $j \in [1, N_q]_\BZ$ a.s.\ has zero Lebesgue measure. The uniqueness in law of such a solution is also shown in~\cite{IG1}. Moreover, it was shown in~\cite{IG1,LevelLineGFFI} that there a.s.\ exists a continuous curve $\eta$ such that the domain $\BH_t$ of $g_t$ is given by the unbounded connected component of $\BH \setminus \eta([0,t])$, for all $t\ge 0$. The curve $\eta$ is called the chordal $\SLE_{\kappa}(\underline{\rho})$ in $\BH$ from~$0$ to~$\infty$. (These processes are defined below for $\rho \le -2$, for which the above condition fails; see~\cite{CLEPerc,ms2019lightcone} for their continuity.)

An $\SLE_{\kappa}(\underline{\rho})$ process in a simply connected domain $G \subsetneq \BC$ is defined as the image of an $\SLE_{\kappa}(\underline{\rho})$ on $\BH$ under a conformal transformation mapping $\BH$ onto $G$ which sends the starting point $0$ to the desired starting point on $\partial G$ and the target point $\infty$ to the desired target point on $\partial G$ (the law of the image does not depend on the choice of the map, by the scale invariance of $\SLE_\kappa(\underline\rho)$). The force points on $\partial G$ are then defined as the images of the corresponding force points on $\partial \BH$.

One can also define a variant of $\SLE_{\kappa}(\rho)$ in which the side of the marked point is chosen afresh for each excursion by an independent (possibly biased) coin toss.  In particular,  we fix $\kappa>0,  \rho \in (-2-(\kappa/2), (\kappa/2) -2) \setminus \{-2\}$,  and $\beta \in [-1,1]$,  and let $X$ be a Bessel process starting from $0$ with dimension
\begin{align*}
\delta = \delta(\kappa,\rho)\defeq1+\frac{2(\rho+2)}{\kappa}.
\end{align*}
Then,  we modify the process $\sqrt{\kappa} X$ by tossing an independent coin for each excursion of $\sqrt{\kappa} X$ away from $0$ to decide its sign,  where the sign of each excursion is positive with probability $(1+\beta)/2$ and negative with probability $(1-\beta)/2$.  We let $Z$ denote the process obtained by concatenating the above modified excursions and set 
\begin{align*}
W_t\defeq Z_t -2 \int_0^t \frac{1}{Z_s} \mathrm{d}s
\end{align*}
in the case that $\rho>-2$.

If $\rho\le-2$,  we can still define $W_t$ as above except that we use a compensation to make sense of the integral of $1/Z$ (see \cite[Section~3]{TreeCLE} and \cite[Section~3]{CLEPerc} for more details).  In both cases, we define the $\SLE_{\kappa}^{\beta}(\rho)$ process in $\BH$ from $0$ to $\infty$ as the collection of hulls generated by $W$ via the Loewner equation~\eqref{eqn:loewner_ode}.  It was shown in \cite[Theorem~10.3]{CLEPerc} that $\SLE_{\kappa}^{\beta}(\rho)$ processes are a.s.  generated by continuous curves for all $\kappa \in (2,4),  \beta \in [-1,1]$,  and all $\rho \in (-2-(\kappa/2),  (\kappa/2) -4)$.  Moreover,  it was shown in \cite[Theorem~10.2]{CLEPerc} that $\SLE_{\kappa'}^{\beta}(\rho')$ processes are a.s.  generated by continuous curves for all $\kappa'>4,  \beta \in [-1,1]$,  and all $\rho' \in (-2-(\kappa'/2) ,  (\kappa'/2) -2) \setminus \{-2\}$.

For a general simply connected domain $G \subsetneq \BC$,  we define an $\SLE_{\kappa}^{\beta}(\rho)$ process in $G$ between two distinct marked points $x,y \in \partial G$ to be given by the image of an $\SLE_{\kappa}^{\beta}(\rho)$ in $\BH$ from $0$ to $\infty$ under a conformal transformation mapping $\BH$ onto $G$ which maps $0$ to $x$ and $\infty$ to $y$.

\subsection{Conformal loop ensembles}
\label{subsec:cle}

Now, we will briefly review CLE and refer the reader to~\cite{CLE,CLEPerc,TreeCLE} for more details.

A CLE in $\BD$ is a random countable collection $\Gamma$ of non-nested loops $(\gamma_j)_{j \in J}$ in $\overline{\BD}$ satisfying the following properties.
\begin{enumerate}
\item\label{it:cle_conf_inv} \textbf{Conformal invariance}: For any M\"obius transformation $\phi: \BD \to \BD$, the laws of $\Gamma$ and $\phi(\Gamma)$ are the same. This makes it possible to define CLE in any proper simply connected domain $G \subseteq \BC$ as the image of a CLE in $\BD$ under a conformal transformation mapping $\BD$ onto $G$.
\item\label{it:cle_dmp} \textbf{Domain Markov property}: For any simply connected domain $U \subseteq \BD$, we let $U^\star$ denote the set obtained by removing from $U$ the closure of the union of the interiors of the loops in $\Gamma$ not contained in $U$. Then, conditionally on the loops of $\Gamma$ not contained in $U$ (which determine $U^\star$), the restrictions of $\Gamma$ to the components $V$ of $U^\star$ are independent, and each has the law of $\Gamma$ conformally mapped to $V$.
\end{enumerate}

It was shown in~\cite{CLE,TreeCLE} that for each $\kappa \in (8/3,8)$ there exists a $\CLE_\kappa$ where the loops can be constructed using a branching variant of $\SLE_\kappa$ (specifically, a tree built out of a collection of $\SLE_\kappa(\kappa-6)$ processes; we will describe this construction in more detail in~\Cref{subsec:branching_sle_kappa_minus_6}). The loops in the $\CLE_\kappa$ are then $\SLE_\kappa$-type curves. From the perspective of the branching tree construction, it is not obvious that the loops of the $\CLE_\kappa$ do not depend on the choice of root, are continuous curves, and are locally finite. In the case that $\kappa \in (8/3,4]$, this was proved in~\cite{CLE} as a consequence of the representation of the $\CLE_\kappa$ loops as the outer boundaries of clusters of the Brownian loop soup; see also~\cite{CLEPerc} for another argument. In the case that $\kappa \in (4,8)$, this was proved as a consequence of the continuity of $\SLE_\kappa(\kappa-6)$ \cite{IG1}, the reversibility of $\SLE_\kappa$ \cite{IG3}, and the continuity of space-filling $\SLE_\kappa$ \cite{IG4}.

If $\kappa \in (8/3,4]$, we have that $\CLE_{\kappa}$ consists of simple loops which do not intersect the domain boundary or each other a.s. However, if $\kappa \in (4,8)$, we have that the loops in the $\CLE_{\kappa}$ are non-simple and can intersect each other and the domain boundary. In the present work, we will focus on the $\kappa = 4$ case, for which the loops do not intersect each other or the domain boundary but \emph{nearly} do. It was also shown in~\cite{CLE} that if one has an ensemble of loops which satisfies~\eqref{it:cle_conf_inv} and~\eqref{it:cle_dmp}, together with the further assumptions that the loops are simple and do not intersect the boundary or each other, and that the ensemble is locally finite, then it must be a $\CLE_\kappa$ for some $\kappa \in (8/3,4]$.

\subsection{Level lines of the GFF}
\label{subsec:level_lines_of_gff}

Let $\lambda = \pi / 2$ and let $h$ be an instance of the GFF on $\BH$ with boundary conditions given by $-\lambda$ (resp.\ $\lambda$) on $\BR_-$ (resp.\ $\BR_+$). Then, it was shown in~\cite{MR3101840} that a zero level line $\eta$ of $h$ from $0$ to $\infty$ can be rigorously defined and has the law of a chordal $\SLE_4$. Moreover, $\eta$ is characterized by the property that for all $t \geq 0$, the conditional law of $h$ given $\eta|_{[0,t]}$ is that of a GFF on $\BH \setminus \eta([0,t])$ with boundary conditions given by $-\lambda$ on the left side of $\eta([0,t])$ and on $\BR_-$, and by $\lambda$ on the right side of $\eta([0,t])$ and on $\BR_+$.

The results in~\cite{MR3101840} were generalized in~\cite{LevelLineGFFI} as follows. Fix a collection of weights $(\underline{\rho}_L; \underline{\rho}_R)$ satisfying $\sum_{i=1}^{k} \rho_{i,L} > -2$ for all $k \in [1,\ell]_\BZ$ and $\sum_{i=1}^{k} \rho_{i,R} > -2$ for all $k \in [1,r]_\BZ$, and force points 
$(\underline{x}_L; \underline{x}_R)$, and let $h$ be a GFF on $\BH$ with boundary conditions given by $-\lambda( 1+ \sum_{i=0}^j \rho_{i,L})$ on $[x_{j+1,L}, x_{j,L})$ and $\lambda (1+\sum_{i=0}^j \rho_{i,R})$ on $[x_{j,R}, x_{j+1,R})$ for $j \in [0,\ell]_\BZ$ (resp.\ $j \in [0,r]_\BZ$), where $\rho_{0,L} = \rho_{0,R} = 0$, $x_{0,L} = 0^-$, $x_{\ell + 1,L} = -\infty$, $x_{0,R} = 0^+$, and $x_{r+1,R} = \infty$.  Then, the zero level line of $h$ from $0$ to $\infty$ can be rigorously defined. Moreover, it has the law of a chordal $\SLE_4(\underline{\rho}_L; \underline{\rho}_R)$ and it is a measurable function of $h$. Furthermore, the level line is characterized by the property that, for all $t \ge 0$, the conditional law of $h$ given $\eta|_{[0,t]}$ is that of a GFF on $\BH \setminus \eta([0,t])$ with boundary conditions given by $-\lambda$ (resp.\ $\lambda$) on the left (resp.\ right) side of $\eta([0,t])$ and the same boundary values as $h$ on $\partial \BH$.

For all $u \in \BR$, the level line of $h$ with height $u$ from $0$ to $\infty$ is given by the level line of $h-u$ from $0$ to $\infty$. Then, we have the following interaction rules for level lines with different heights proved in~\cite{LevelLineGFFI}.

\begin{theorem}[{\cite[Theorem~1.1.4 and Remark~1.1.5]{LevelLineGFFI}}]
Suppose that $h$ is a GFF on $\BH$ with piecewise constant boundary conditions that change only finitely many times. For all $u \in \BR$ and $x \in \partial \BH$, we let $\gamma_u^x$ be the level line of $h$ with height $u$ starting from $x$ and ending at $\infty$ (when it exists). Fix $x_2 \leq x_1$.
\begin{enumerate}
\item If $u_2 < u_1$, then $\gamma_{u_2}^{x_2}$ a.s.\ stays to the left of $\gamma_{u_1}^{x_1}$ (possibly touching it).
\item If $u_2 = u_1$, then $\gamma_{u_2}^{x_2}$ may intersect $\gamma_{u_1}^{x_1}$ and, upon intersecting, the two curves merge and never separate.
\item If $u_1 - u_2 \geq 2 \lambda$, then $\gamma_{u_1}^{x_1}$ and $\gamma_{u_2}^{x_2}$ do not intersect each other a.s.
\end{enumerate}
\end{theorem}

\subsection{Labeled $\CLE_4$}
\label{subsec:labeled_cle_4}

Labeled $\CLE_4$ in $\BD$, also called the uniform exploration of the $\CLE_4$ in $\BD$, was introduced in~\cite{CoInCLEExpl} and corresponds to a Markovian exploration of $\CLE_4$ loops, with a label on each loop that keeps track of the time when each loop is discovered. We will briefly describe the construction and refer to~\cite{CoInCLEExpl} and~\cite{LevelLineGFFI} for more details.

To construct the exploration, we define the measure $M $ as the image measure of $\mu \otimes \omega$ under the mapping $(\ell, x) \mapsto x \cdot \ell$ (the rotation of $\ell$ by $x$),  where $\omega$ denotes the harmonic measure on $\partial \BD$ as seen from $0$, and $\mu$ denotes the (infinite) $\SLE_4$ bubble measure on $\BD$ pinned at $1$ introduced in~\cite{CLE}, normalized so that the time at which the loop surrounding $0$ is discovered is an exponential random variable of parameter one. We also let $(\ell_t)_{t \geq 0}$ denote a Poisson point process (PPP) with intensity measure given by $M$ times the Lebesgue measure on $\BR_+$ ($\ell_t$ being defined only for the countably many atom times $t$). This process defines an ordering on the (pinned) loops (also called bubbles), where each loop $\ell$ is equipped with time label $t_{\ell}$.

The exploration targeted at $0$ is defined as follows. Let $\tau_0$ denote the first time that the PPP discovers a loop that surrounds $0$. For all $\varepsilon \in (0,1)$, we let $\ell_{t_1^{\varepsilon}},\ldots, \ell_{t_{n_{\varepsilon}}^{\varepsilon}}$ denote the loops discovered strictly before time $\tau_0$ with diameter at least $\varepsilon$, such that $t_1^{\varepsilon}<\cdots<t_{n_{\varepsilon}}^{\varepsilon}$. 
 We define $D_0 = \BD$ and inductively let $D_i$ denote the connected component of $D_{i-1} \setminus \phi_i(\ell_{t_i^{\varepsilon}})$ containing $0$, where $\phi_i$ is the conformal mapping from $\BD$ onto $D_{i-1}$ such that $\phi_i(0) = 0$ and $\phi_i'(0) > 0$. Then, $D_{n_{\varepsilon}}$ converges in the Carath\'eodory sense as $\varepsilon \to 0$ to the interior of a loop which, by~\cite{CoInCLEExpl}, has the same law as the loop in a $\CLE_4$ surrounding $0$. The loop surrounding $0$ is then equipped with the time label $\tau_0$ and we say that it is discovered at time $\tau_0$. More generally, if we set $D^\varepsilon_t = D_{i-1}$ when $t_{i-1}^\varepsilon \le t<t_i^\varepsilon$, where $t^\varepsilon_0 \defeq 0$ and $t^\varepsilon_{n_\varepsilon+1} \defeq \tau_0$, then for all $t < \tau_0$, $D^\varepsilon_t$ decreases as $\varepsilon \downarrow 0$ and hence converges a.s. in the Carath\'eodory sense to a simply connected open domain $C_t(0)$, which is called the connected component of the unexplored region containing the origin.

It was shown in \cite[Lemma~6]{CoInCLEExpl} that the measure $M$ is invariant under M\"obius transformations of $\BD$. Therefore, by conformal invariance, we can define the exploration targeted at any fixed point $z \in \BD$ by mapping the exploration targeted at $0$ via a M\"obius transformation of $\BD$ which maps $0$ to $z$. The transformation is unique up to rotation, but by the rotational invariance of $M$ the law of the exploration does not depend on the choice.

The labeled $\CLE_4$ exploration is then defined by coupling the exploration processes targeted at any two fixed points $z,w \in \BD$ together so that they agree up until the first time that they are separated, and then evolve independently. For all $z \in \BD$, we denote by $C_t(z)$ the connected component of the unexplored region containing $z$. The explored region is the compact set 
\[\SCB_t(\partial \BD) \defeq \overline{\BD} \setminus \bigcup_{z \in \BD_\BQ} C_t(z),\]
where $\BD_\BQ \defeq \BD \cap \BQ^2$.
Since every rational point is encircled by a loop (the loop $\SCL(x)$ surrounding $x \in C_t(z)$ being discovered after time $t$, inside $C_t(z)$) and since $C_t(z)$ is open, we can also write $\overline{C_t(z)} = \overline{\bigcup_{x \in C_t(z) \cap \BQ^2} \mathrm{int}(\SCL(x)) }$.

\subsection{Two-valued sets of the GFF}\label{subsec:tvs_gff}

Next, we introduce the two-valued sets of a GFF $h$ on $\BD$ that we are going to consider. Recall the definition of a local set of $h$ given in~\cite{MR3101840}. We refer to~\cite{aru2019bounded} and~\cite{TVSGFF} for more detailed expositions of two-valued sets of the GFF.

\begin{definition}\label{def TVS}
Fix $a,b>0$ and let $h$ be a zero-boundary GFF in a simply connected domain $G \subsetneq \BC$. We say that $\mathbb{A}_{-a,b}$ is a two-valued local set of $h$ of levels $-a$ and $b$ if it is a local set of $h$ and satisfies the following properties.
\begin{enumerate}
\item \label{it:boundary_conditions}
If $h_{\mathbb{A}_{-a,b}}$ denotes the harmonic part of $h$ on $G \setminus \mathbb{A}_{-a,b}$ (as in~\Cref{subsec:gff}), then we have that $h_{\mathbb{A}_{-a,b}} \in \{-a,b\}$ a.s.
\item \label{it:thin_local_set}
The set $\mathbb{A}_{-a,b}$ is a \emph{thin} local set, i.e., for any smooth test function $f \in C_0^{\infty}(G )$, we have that
\begin{align*}
(h,f) = (h-h_{\mathbb{A}_{-a,b}}, f) + \int_{G \setminus \mathbb{A}_{-a,b}} h_{\mathbb{A}_{-a,b}}(x) f(x) \mathrm{d}x
\end{align*}
a.s.
\item \label{it:finitely_many_components} The set $\mathbb{A}_{-a,b} \cup \partial G$ has a finite number of connected components.
\end{enumerate}
\end{definition}

It was shown in~\cite{aru2019bounded} that if $a,b \in \BR$ are such that $-a<0<b$ and $a+b \geq 2\lambda$, then such a set $\mathbb{A}_{-a,b}$ exists and it is a.s.\ determined by $h$. Moreover, $\mathbb{A}_{-a,b}$ is unique in the sense that if $A'$ is another local set of $h$ satisfying~\eqref{it:boundary_conditions}, \eqref{it:thin_local_set} and~\eqref{it:finitely_many_components}, we have that $A' = \mathbb{A}_{-a,b}$ a.s.

Let us now briefly review the construction of $\mathbb{A}_{-a,-a+2\lambda}$ for $a \in (0,2\lambda)$, which will be needed for our purposes. We will follow \cite[Section~3.2]{TVSGFF}. 

To begin with, we let $\eta$ denote the level line of $h+a-\lambda$ from $-\ri $ to $\ri$ (level lines in $\BD$ being defined by conformal invariance) and set $A^1 = \eta([0,\infty])$. In each connected component $O$ of $\BD \setminus A^1$ lying to the left of $\eta$, we let $x$ and $y$ be the two endpoints of the (a.s.\ nondegenerate) arc $\partial O \cap \partial \BD$, and we assume that $(-\ri,x,y)$ are in counterclockwise order. We then let $\eta^O$ denote the level line of $h|_{O} + a -\lambda$ from $x$ to $y$.

As for the connected components $O$ of $\BD \setminus A^1$ lying to the right of $\eta$, we perform an analogous procedure except that we explore the level line of $h|_O + a - \lambda$ from $y$ to $x$. We denote those level lines by $\eta^O$.

In the next step of the exploration, we let $A^2$ denote the closure of the union of $A^1$ with the $\eta^O$'s for each connected component $O$ of $\BD \setminus A^1$. In the connected components of $\BD \setminus A^2$ whose boundaries are contained in $A^2$, we stop the exploration. Note that the boundary conditions of $h$ on each of the aforementioned components are a.s.\ constant and given by either $-a$ or $2\lambda - a$. In the rest of the components, we iterate as before to obtain $A^n$. We note that if $V$ is a connected component of $\BD \setminus A^2$ such that $\partial V \cap \partial \BD \neq \emptyset$, then we have that the boundary conditions of $h|_{V}$ are given by $0$ on $\partial V \cap \partial \BD$, and $-a$ (resp.\ $2\lambda-a$) on $\partial V \setminus \partial \BD$ if $V$ is contained in a component of $\BD \setminus A^1$ lying to the right (resp.\ left) of $\eta$. Then, we set $\mathbb{A}_{-a,-a+2\lambda}\defeq\overline{\cup_{n \in \BN} A^n}$.

\begin{remark}\label{rem:target_invariance}
The uniqueness of $\mathbb{A}_{-a,b}$ implies that the starting and ending points of the level lines and the order in which they were sampled to produce $\mathbb{A}_{-a,b}$ do not matter. Alternatively, this can also be seen from the reversibility of the $\SLE_4(\underline{\rho})$ processes~\cite[Theorem~1.1.6]{LevelLineGFFI}.
\end{remark}

\begin{remark}\label{remark coupling TVS BCLE}
The set $\mathbb{A}_{-a,-a+2\lambda}$ constructed above is equal to the range of a clockwise boundary conformal loop ensemble with $\kappa=4$ and parameter $\rho$ (i.e., a $\cwBCLE_4(\rho)$ process) with $\rho = -a/\lambda$ introduced in \cite[Section~7]{CLEPerc} (the definition is recalled in~\Cref{sec:cle-percolation}). By varying $a \in (0,2\lambda)$, we obtain $\cwBCLE_4(\rho)$ for all $\rho \in (-2,0)$. More precisely, the loops with boundary condition $2\lambda -a$ in $\mathbb{A}_{-a,-a+2\lambda}$ correspond to the true loops of the $\cwBCLE_4(\rho)$ and the loops with boundary condition $-a$ in $\mathbb{A}_{-a,-a+2\lambda}$ correspond to the false loops of the $\cwBCLE_4(\rho)$.
\end{remark}

\subsection{Coupling between the GFF and the labeled $\CLE_4$}
\label{subsec:gff_labeled_cle_4}

Now, we are ready to construct the coupling that we will use between a zero-boundary GFF $h$ on $\BD$ and a labeled $\CLE_4$ $\{(\ell, t_{\ell})\}_{\ell \in \Gamma}$. We will follow the construction given in \cite[Section~6]{TVSGFF}.

Fix $r \in (0,2\lambda)$. For all $j \in \BN$, we define sets $\mathbb{B}_r^j$ iteratively as follows. First, we set $\mathbb{B}_r^1\defeq\mathbb{A}_{-r,-r+2\lambda}$. In each connected component $O$ of $\BD \setminus \mathbb{B}_r^1$ such that the boundary conditions of $h|_O$ are given by $-r$, we explore the set $\mathbb{A}_{-r,-r+2\lambda}(O)$ which is defined in the same way as $\mathbb{A}_{-r,-r+2\lambda}$, but with the zero-boundary GFF $h|_O - h_{\mathbb{B}_r^1}$ in place of $h$, where we recall that $h_{\mathbb{B}_r^1}$ is as in~\Cref{def TVS}. Then, we let $\mathbb{B}_r^2$ denote the closure of the union of the sets explored and note that the boundary conditions of $h$ on the connected components of $\BD \setminus \mathbb{B}_r^2$ are constant and lie in $\{2\lambda-r,2\lambda-2r,-2r\}$.

Suppose that we have constructed $\mathbb{B}_r^j$ for $j \in \BN$. Then, in each connected component $O$ of $\BD \setminus \mathbb{B}_r^j$ such that the boundary conditions of $h|_O$ are given by $-jr$, we explore the set $\mathbb{A}_{-r,-r+2\lambda}(O)$ defined in the same way as $\mathbb{A}_{-r,-r+2\lambda}$ but with the zero-boundary GFF $h|_O - h_{\mathbb{B}_r^j}$ in place of $h$. We define $\mathbb{B}_r^{j+1}$ as the closure of the union of $\mathbb{B}_r^j$ and the sets explored.

We set $\mathbb{B}_r\defeq\overline{\cup_{j \in \BN} \mathbb{B}_r^j}$ and $\mathbb{B}_0\defeq\overline{\cup_{n \in \BN} \mathbb{B}_{2^{-n}}}$. For every loop $\ell$ of $\mathbb{B}_0$ (i.e., whose interior $\text{int}(\ell)$ is a connected component of $\BD \setminus \mathbb{B}_0$), we define a label $t_{\ell}$ such that the boundary conditions of $h|_{\text{int}(\ell)}$ are given by $2\lambda - t_{\ell}$ (their a.s.\ constancy on each connected component of $\BD \setminus \mathbb{B}_0$ is part of the construction of the two-valued sets in~\cite[Section~6]{TVSGFF}). Then, we have the following.

\begin{proposition}[{\cite[Proposition~6.6]{TVSGFF}}]
\label{prop coupling TVS uniform explo}
Suppose that we have the setup described above. Then, the collection of loops in $\mathbb{B}_0$ has the law of a non-nested $\CLE_4$ $\Gamma$ in $\BD$ and $\{(\ell,t_{\ell})\}_{\ell \in \Gamma}$ has the law of a labeled $\CLE_4$.
\end{proposition}
Let us provide more details on this coupling. The first lemma shows an inclusion of loops with boundary condition $-2^{-k}$.
\begin{lemma}\label{lemma link TVS BCLE}
	Let $r \in (0,2\lambda)$. For all $m \in \BZ_{\ge 0}$, the loops with boundary condition $-r$ of $\mathbb{B}^1_{r}$ are loops with boundary condition $-r$ in $\mathbb{B}^{2^{m}}_{2^{-m}r}$. In particular, for $r = 2^{-k}$, $n\ge k$ and $m = n-k$, the false loops of the $\cwBCLE_4(-2^{-k}/\lambda)$ (from~\Cref{remark coupling TVS BCLE}) are loops with boundary condition $-2^{-k}$ in $\mathbb{B}^{2^{n-k}}_{2^{-n}}$.
\end{lemma}
\begin{proof}
	By \cite[Subsection 6.2.4]{TVSGFF}, we know that for all $r \in (0, 2\lambda)$, all loops with boundary condition $-r$ of $\mathbb{B}^1_r$ are loops with boundary condition $-r$ in $\mathbb{B}^2_{r/2}$. The desired result follows by induction as detailed below. The case $m=0$ is trivial.
	
	Let $m \in \BZ_{\ge 0}$. Let us assume that for all $r \in (0, 2\lambda)$, for all zero-boundary GFFs $h$ on $\BD$ (equivalently, by conformal invariance, on any simply connected domain $U \subsetneq \BC$), all the loops with boundary condition $-r$ in $\mathbb{B}^{1}_{r}$ (constructed with $h$) are loops with boundary condition $-r$ in $\mathbb{B}^{2^m}_{2^{-m}r}$. Let $r \in (0,2\lambda)$ and fix a zero-boundary GFF $h$ on $\BD$. Let $O$ be the interior of a loop with boundary condition $-r $ in $\mathbb{B}^{1}_{r}$. By the beginning of the proof, we know that $O$ is also the interior of a loop with boundary condition $-r$ in $\mathbb{B}^{2}_{r/2}$. But, by the iterative construction of $\mathbb{B}^{2}_{r/2}$, we see that $O$ is included in the interior $O'$ of a loop of $\mathbb{B}^1_{r/2}$ with boundary condition $-r/2$ (in the sense that the boundary conditions of $h\vert_{O'}$ are given by $-r/2$) and that $O$ is the interior of a loop of $\mathbb{A}_{-r/2, -r/2+2\lambda}(O')$ with boundary condition $-r/2$.
	Moreover, using that $\mathbb{B}_{r/2}^1$ is a local set for the GFF $h$, conditionally on $\mathbb{B}_{r/2}^1$ and $h_{\mathbb{B}_{r/2}^1}$, the field $h|_{O'} - h_{\mathbb{B}_{r/2}^1}$ is a zero-boundary GFF on $O'$. Note that $\mathbb{A}_{-r/2, -r/2+2\lambda}(O')$ can therefore be seen as $\mathbb{B}^1_{r/2}(O')$, in the sense that it is defined in the same way as $\mathbb{B}^1_{r/2}$ except that $\BD, h$ are replaced by $O',h|_{O'} - h_{\mathbb{B}_{r/2}^1}$.
	By performing the exact same iterative construction in $O'$ (instead of $\BD$) with the field $h|_{O'} - h_{\mathbb{B}_{r/2}^1}$ (instead of $h$), we define the set $\mathbb{B}^j_{r/2}(O')$ for all $j\ge 1$. By induction hypothesis, we deduce that $O$ is the interior of a loop with boundary condition $-r/2$ of $\mathbb{B}^{2^m}_{2^{-m-1}r}(O')$ and that $O'$ is the interior of a loop with boundary condition $-r/2$ of $\mathbb{B}^{2^m}_{2^{-m-1}r}$.
	Finally, note that by the iterative construction of $\mathbb{B}^{2^{m+1}}_{2^{-m-1}r}$, we can build $\mathbb{B}^{2^{m+1}}_{2^{-m-1}r}$ by first taking $\mathbb{B}^{2^{m}}_{2^{-m-1}r}$ and then by exploring $\mathbb{B}^{2^{m}}_{2^{-m-1}r}(O'')$ in all the interiors $O''$ of the loops with boundary condition $-r/2$ of $\mathbb{B}^{2^{m}}_{2^{-m-1}r}$. In particular, we conclude that $O$ is the interior of a loop with boundary condition $-r$ in $\mathbb{B}^{2^{m+1}}_{2^{-m-1}r}$. This ends the induction step.
\end{proof}

Next, let us identify the explored region $\SCB_{2^{-k}}(\partial \BD)$ in terms of the increasing union of the loops with boundary condition $-2^{-k}$.
\begin{lemma}\label{lemma link TVS uniform exploration}
	Let $k \in \BZ_{\ge 0}$. Then, in the coupling of~\Cref{prop coupling TVS uniform explo}, the increasing union over $n$ of the regions encircled by the loops of $\mathbb{B}^{2^{n-k}}_{2^{-n}}$ with boundary condition $-2^{-k}$ is $\BD \setminus \SCB_{2^{-k}}(\partial \BD)$.
\end{lemma}
\begin{proof}
	Recall that the loops of $\mathbb{B}^{2^{n-k}}_{2^{-n}}$ have boundary conditions $2\lambda- 2^{-n}, 2\lambda -  2\cdot 2^{-n}, 2\lambda -  3\cdot 2^{-n},\ldots, 2 \lambda -  2^{-k}$ and $-2^{-k}$. When we iterate the construction in the domains encircled by the loops with boundary condition $-2^{-k}$ as explained at the beginning of the current subsection, the loops that are drawn have boundary conditions given by either $-s$ or $2\lambda- s$ for some $s>2^{-k}$. Moreover, by~\Cref{lemma link TVS BCLE}, the loops with boundary condition $-2^{-k}$ in $\mathbb{B}^{2^{n-k}}_{2^{-n}}$ are still loops with boundary condition $-2^{-k}$ in $\mathbb{B}^{2^{n+1-k}}_{2^{-n-1}}$. After taking the increasing union as $n\to \infty$, this shows that the loops $\ell \in \Gamma$ that are in the domains encircled by the loops of $\mathbb{B}^{2^{n-k}}_{2^{-n}}$ with boundary condition $-2^{-k}$ have a label $t_{\ell} \ge 2^{-k}$.
	
	Recall that by construction of the uniform exploration in~\Cref{subsec:labeled_cle_4}, we have that 
	\begin{align*}
		\overline{C_t(z)} = \overline{\bigcup_{x \in C_t(z) \cap \BQ^2} \mathrm{int}(\SCL(x)) }
	\end{align*}
	for all $t \geq 0$ and all $z \in \BD_{\BQ}$. By taking~\Cref{prop coupling TVS uniform explo} into account, this implies that the increasing union over $n$ of the regions encircled by the loops of $\mathbb{B}_{2^{-n}}^{2^{n-k}}$ with boundary condition $-2^{-k}$ is contained in $\BD \setminus \SCB_{2^{-k}}(\partial \BD)$.
	
	Conversely, let $O = \text{int}(\SCL(x))$ for some $x \in \BQ^2$ be a connected component of boundary condition $2\lambda-t$ for some $t>2^{-k}$ in $\mathbb{B}_0$. By definition of $\mathbb{B}_0$, for all $n \in \BN$, the component $O$ is included in the interior $O_n(x)$ of a loop in $\mathbb{B}_{2^{-n}}$ with boundary condition $2 \lambda - j_n 2^{-n}$ where $j_n2^{-n} \to t$ as $n\to \infty$ (due to \cite[Lemma 6.13]{TVSGFF}). For $n$ large enough, we have $j_n 2^{-n}>2^{-k}$. Therefore, by the iterative definition of $\mathbb{B}^j_{2^{-n}}$'s for $j\ge 1$, we deduce that $O_n(x)$ is included in the interior of a loop of $\mathbb{B}^{2^{n-k}}_{2^{-n}}$ with boundary condition $-2^{-k}$. Thus, $O$ is included in the interior of a loop of $\mathbb{B}^{2^{n-k}}_{2^{-n}}$ with boundary condition $-2^{-k}$. This concludes the proof of the lemma since almost surely, there is no loop of $\mathbb{B}_0$ with boundary condition exactly $2\lambda-2^{-k}$.
\end{proof}

\begin{remark}\label{rem:counterclockwise_loops_bcle_4}
	The coupling defined in~\cite{TVSGFF}, performed with local sets whose loops have boundary conditions determined by dyadic numbers $j2^{-k}$ for $j,k \in \BZ_{\ge 0}$, can be performed in the exact same way by multiplying these dyadic numbers by a fixed real number $r \in (0,2\lambda)$, so that the resulting levels $jr2^{-k}$ again range over $(0, 2\lambda)$. Hence, by taking~\Cref{prop coupling TVS uniform explo},~\Cref{lemma link TVS BCLE} and~\Cref{lemma link TVS uniform exploration} (with $k=0$) into account, we see that for all $t\in (0,2\lambda)$, we can couple the uniform exploration of the $\CLE_4$ with a $\cwBCLE_4(-t/\lambda)$ so that the interior of the false $\cwBCLE_4(-t/\lambda)$ loop containing the origin (on the event that the origin is encircled by a false loop) is the connected component of the unexplored region containing $0$ at time $t$. More generally, one can couple the uniform exploration with a $\cwBCLE_4(-t/\lambda)$ so that the interiors of all the false $\cwBCLE_4(-t/\lambda)$ loops are connected components of the unexplored region at time $t$.
\end{remark}

Let us emphasize that, in order to get all the connected components of the unexplored region $\BD \setminus \SCB_t(\partial \BD)$, one needs to look at the union described in~\Cref{lemma link TVS uniform exploration}: the interiors of the false $\cwBCLE_4(-t/\lambda)$ loops do not give all the connected components of $\BD \setminus \SCB_t(\partial \BD)$. However, we can state the following lemma which will be useful in~\cite{kkmt2026cle4_part3}.

\begin{lemma}\label{lem:bcle_uniform_exploration}
	Consider the uniform exploration $(\SCB_s(\partial \BD))_{s\ge 0}$, let $r \in (0, 2\lambda)$ and consider the $\cwBCLE_4(-r/\lambda)$ coupled with $(\SCB_s(\partial \BD))_{s\ge 0}$ as in~\Cref{rem:counterclockwise_loops_bcle_4}. Then, the interiors of the false $\cwBCLE_4(-r/\lambda)$ loops are exactly the connected components of $\BD \setminus \SCB_r(\partial \BD)$ whose boundary intersects $\partial \BD$.
\end{lemma}
\begin{proof}
Let $r \in (0, 2\lambda)$. By \cite[Lemma 3.8]{TVSGFF}, we know that every loop $\ell$ of $\mathbb{A}_{-r,2\lambda-r/2}$ with boundary condition $-r$ intersects $\partial \BD$ if and only if $\ell$ is a loop of $\mathbb{A}_{-r,2\lambda-r}$ with boundary condition $-r$. Moreover, as explained in \cite[Subsection 6.2.4]{TVSGFF}, the loops of $\mathbb{B}_{r/2}^2$ with boundary condition $-r$ are loops of $\mathbb{A}_{-r,2\lambda - r/2}$ with boundary condition $-r$, and the loops of $\mathbb{A}_{-r,2\lambda-r}$ with boundary condition $-r$ that intersect $\partial \BD$ are loops of $\mathbb{B}^2_{r/2}$. We obtain that the collection of loops in $\mathbb{B}_{r/2}^2$ with boundary condition $-r$ which intersect $\partial \BD$ is precisely the family of loops in $\mathbb{A}_{-r,2\lambda-r}$ with boundary condition $-r$. 
	
Let us show by induction that for all $m \in \BN$, for all $r \in (0, 2 \lambda)$, for all zero-boundary GFFs $h$ on $\BD$ (or, by conformal invariance, on any simply connected domain), the family of loops in $\mathbb{B}_{2^{-m} r}^{2^m}$ (defined using $h$) with boundary condition $-r$ that intersect $\partial \BD$ is precisely the family of loops in $\mathbb{B}^1_r=\mathbb{A}_{-r,2\lambda-r}$ with boundary condition $-r$. We proved the result for $m=1$ in the first paragraph.
		
Assume that the result holds for some $m \in \BN$. Let $r \in (0, 2 \lambda)$ and let $h$ be a zero-boundary GFF on $\BD$. Let $O$ be the interior of a loop of $\mathbb{B}^{2^{m+1}}_{2^{-m-1}r}$ with boundary condition $-r$ such that $\overline{O} \cap \partial \BD \neq \emptyset$. As in the proof of~\Cref{lemma link TVS BCLE}, there exists a loop  of $\mathbb{B}^{2^{m}}_{2^{-m-1}r}$ with boundary condition $-r/2$, whose interior $O'$ contains $O$ and such that $O$ is the interior of a loop of $\mathbb{B}^{2^{m}}_{2^{-m-1}r}(O')$ with boundary condition $-r/2$. Note that $\overline{O} \cap \partial O' \neq \emptyset$ and $\overline{O'} \cap \partial \BD \neq \emptyset $. Then, by induction hypothesis, we know that $O'$ is the interior of a loop of $\mathbb{B}^1_{r/2}$ with boundary condition $-r/2$ and that $O$ is the interior of a loop of $\mathbb{B}^1_{r/2}(O')$ with boundary condition $-r/2$. In particular, by definition of $\mathbb{B}^2_{r/2}$, we see that $O$ is the interior of a loop of $\mathbb{B}^2_{r/2}$ with boundary condition $-r$. By the case $m=1$, we conclude that $O$ is the interior of a loop of $\mathbb{B}^1_r$ with boundary condition $-r$. Moreover, by~\Cref{lemma link TVS BCLE}, we know that all the loops with boundary condition $-r$ of $\mathbb{B}^1_r$ are loops of $\mathbb{B}^{2^{m+1}}_{2^{-m-1}r}$. Recalling that the loops of $\mathbb{B}^1_r$ are loops of a $\cwBCLE_4(-r/\lambda)$ and thus intersect $\partial \BD$, this ends the induction step.

Furthermore, by~\Cref{lemma link TVS uniform exploration} (combined with~\Cref{rem:counterclockwise_loops_bcle_4}), the collection of the connected components of $\BD \setminus \SCB_r(\partial \BD)$ is the collection of the interiors of the loops in $\mathbb{B}_{r 2^{-n}}^{2^n}$ with boundary condition $-r$ as $n \in \BN$ varies. Therefore, the collection of the interiors of the false $\cwBCLE_4(-r/\lambda)$ loops is the collection of the connected components of $\BD \setminus \SCB_r(\partial \BD)$ whose boundary intersects $\partial \BD$.
\end{proof}

Next, let us state another consequence of~\cite{TVSGFF}.
\begin{lemma}\label{lemma density of loops}
	Almost surely, for all $t>0$, the explored region $\SCB_t(\partial \BD)$ is the closure of the union of the domains encircled by loops of the $\CLE_4$ discovered up to time $t$.
\end{lemma}
\begin{proof}
	The statement when $t>0$ is dyadic is a consequence of the definition of the sets $\mathbb{B}^t_0$ for $t\ge 0$ (the superscript in $\mathbb{B}^t_0$, imported from~\cite{TVSGFF}, is a label cutoff, not a number of iterations as in the notation $\mathbb{B}^j_r$ of~\Cref{subsec:tvs_gff}) as the closure of the union of the domains encircled by loops of the $\CLE_4$ with boundary condition at least $2\lambda-t$ in \cite[Definition 6.2]{TVSGFF} and of the identification of $(\mathbb{B}^t_0)_{t\ge 0}$ with the uniform exploration $(\SCB_t(\partial \BD))_{t\ge 0}$ made in \cite[Section 6.4.3]{TVSGFF}.
	
	For general $t>0$, we view $(\SCB_t(\partial \BD))_{t\ge 0}$ as a c\`adl\`ag process taking its values in the space of the compact subsets of $\overline{\BD}$ equipped with the Hausdorff distance (left limits exist by monotonicity, and right-continuity holds since the loops discovered in a component of the unexplored region shortly after time $t$ are small and pinned at its boundary). When $t>0$ is a time of left continuity of $(\SCB_t(\partial \BD))_{t\ge 0}$, we obtain directly the result by approximating $\SCB_t(\partial \BD)$ using dyadic times (the dyadic case holding a.s.\ simultaneously for all dyadic $t$). When $t$ is a jump time, then it is a time of a loop discovery in the PPP of loops (the unexplored components change only at such times), so that $\SCB_t(\partial \BD)$ is obtained as the union of the left limit $\SCB_t^-(\partial \BD)\defeq \overline{\bigcup_{s<t} \SCB_s(\partial \BD)}$ together with the closure of the domain encircled by the loop discovered at time $t$. This concludes the proof of the lemma.
\end{proof}

\subsection{Background on Carath\'eodory convergence}
\label{subsec:caratheodory}

We gather here some equivalent definitions of Carath\'eodory convergence. Let $(U_n)_{n\ge 0}$ be a sequence of simply connected domains $U_n \subsetneq \BC$. Let $U \subsetneq \BC$ be a simply connected domain. Let $z \in U$. When $z \in U_n$, let $f_n$ (resp.\ $f$) be the unique conformal map from $\BD$ to $U_n$ (resp.\ $U$) such that $f_n(0)=z$ (resp.\ $f(0)=z$) and $f_n'(0)>0$ (resp.\ $f'(0)>0$). We say that $(U_n,z)$ converges to $(U,z)$ as $n\to \infty$ in the Carath\'eodory sense if $z \in U_n$ for $n$ large enough and $f_n$ converges to $f$ uniformly on compact subsets of $\BD$. The convergence does not depend on the choice of $z\in U$, so we will often simply say that $U_n$ converges to $U$ in the Carath\'eodory sense.

Let us also recall the original definition of Carath\'eodory convergence. 
\begin{definition}
	Let $U\subset \BC$ be a domain, let $z \in U$, and let $(U_n)_{n\ge 0}$ be a sequence of domains in $\BC$ containing $z$. The sequence $(U_n)_{n\ge 0}$ converges in the Carath\'eodory sense toward $U$ with respect to $z$ if: 
	\begin{itemize}
		\item For every compact set $K\subset U$, we have $K\subset U_n$ for all $n$ large enough;
		\item For every connected open set $V\ni z$, if $V \subset U_n$ for infinitely many $n\ge 0$, then $V \subset U$.
	\end{itemize}
\end{definition}
The above definition is more general since it does not assume that the domains are simply connected. The equivalence between the two definitions in the case of simply connected domains is the content of the Carath\'eodory kernel theorem.

Furthermore, using the Arzelà--Ascoli theorem, the convergence in the Carath\'eodory sense implies the uniform convergence of $f_n^{-1}$ toward $f^{-1}$ on every compact subset of $U$. 

We shall repeatedly use the following elementary fact. Let $\SCL_n \to \SCL$ in the Hausdorff topology, where $\SCL$ is a Jordan curve in $\BD$ surrounding a fixed point $x$, as does each $\SCL_n$. Then every compact $K \subseteq \mathrm{int}(\SCL)$ (resp.\ $K \subseteq \overline{\BD} \setminus \overline{\mathrm{int}(\SCL)}$) satisfies $K \subseteq \mathrm{int}(\SCL_n)$ (resp.\ $K \cap \overline{\mathrm{int}(\SCL_n)} = \emptyset$) for $n$ large. 
In particular $\overline{\mathrm{int}(\SCL_n)} \to \overline{\mathrm{int}(\SCL)}$. This applies to $\CLE_4$ loops, which are Jordan curves.\label{rem:interior membership}

We now define the space of domains. We denote by $\SD$ the space of decreasing families of simply connected domains $(D_t)_{t\ge 0}$ such that for all $t\ge 0$, we have $0 \in D_t$ and the conformal radius of $D_t$ seen from zero is $e^{-t}$: in other words, if $f_t$ is the unique conformal map from $\BD$ to $D_t$ sending $0$ to $0$ with $f'_t(0)>0$, then $f'_t(0) = e^{-t}$.

The space $\SD$ is equipped with an extension of the Carath\'eodory topology: if $(D_t)_{t\ge 0}, (\widetilde{D}_t)_{t\ge 0} \in \SD$ and $(f_t)_{t \geq 0}$, $(\widetilde{f}_t)_{t \geq 0}$ are the associated conformal maps, then we set
\begin{equation}\label{eq space of domains}
	d_\SD \left( (D_t)_{t\ge 0}, (\widetilde{D}_t)_{t\ge 0}\right) \defeq \sum_{j,k \in \BN}
	2^{-j-k} \min\left\{\sup_{s \in [0, 2^k]} \sup_{z \in (1-1/j)\BD} \vert f_s(z)- \widetilde{f}_s(z) \vert, 1\right \}.
\end{equation}
The metric space $(\SD, d_\SD)$ is complete and separable. See, e.g.,~\cite[Section 6.1]{MS16QLE}.

\subsection{Properties of SLEs with force points and BCLEs}\label{subsec:properties_sles}
Recall that for each $t \geq 0$, $C_t(0)$ is the connected component of the unexplored region of the uniform exploration containing the origin.  This subsection is aimed at proving the following result.
\begin{proposition}\label{prop BCLE rho to zero}
	 As $t \to 0$, $C_t(0)$ converges in probability in the Carath\'eodory sense to $\BD$ (i.e.\ $\BP[K \subseteq C_t(0)] \to 1$ for every compact $K \subseteq \BD$) 
	 and the maximum of the diameters of the connected components of $\BD \setminus \overline{C_t(0)}$ converges in probability to zero. In particular, for all $\delta>0$, with probability $1-o(1)$ as $t \to 0$, for all $a \in \partial \BD$, there exists $x \in \partial C_t(0) \cap \partial \BD$ such that $\vert a - x \vert \le \delta$.
\end{proposition}

The key result for establishing~\Cref{prop BCLE rho to zero} is the following.
	\begin{proposition}\label{prop diameter of arcs}
		For each $a>0$ and $\rho \in (-2,0)$, let $d_\rho(a)$ be the supremum of the diameters of the excursions away from $\BR$ (i.e.\ the closures of the connected components of $\gamma \setminus \BR$) of a chordal $\SLE_4(\rho; -\rho-2)$ $\gamma$ in $\BH$ from $0$ to $\infty$, with force points at $0^-$ and $0^+$, which have an endpoint in $[0,a]$. Then,
		\[
		d_\rho(a) \mathop{\longrightarrow}\limits_{\rho \uparrow 0}^{(\BP)} 0.
		\]
	\end{proposition}
\begin{proof}
This is a consequence of \cite[Proposition~3.1]{DKM24}. We note that this result is stated and proved for $\kappa \in (0,4)$ using the tools from~\cite{IG1}. The same argument, however, works for $\kappa=4$ using the results from~\cite{LevelLineGFFI}. The statement of \cite[Proposition~3.1]{DKM24} does not explicitly mention that the diameter of the components tends to $0$ as $\rho \uparrow 0$; however, this is implicit in the proof of \cite[Lemma~3.8]{DKM24}.
\end{proof}

\Cref{prop BCLE rho to zero} follows from the definition of BCLEs using SLEs with force points and from the coupling of~\Cref{subsec:gff_labeled_cle_4}. 

\begin{proof}[Proof of~\Cref{prop BCLE rho to zero} assuming~\Cref{prop diameter of arcs}]
Let $h$ denote the zero-boundary GFF used to generate the uniform $\CLE_4$ exploration as in~\Cref{subsec:gff_labeled_cle_4}. Let $\gamma$ be a chordal $\SLE_4(\rho;-2-\rho)$ curve in $\BD$ from $-1$ to $1$ which is the branch of the exploration tree of a $\cwBCLE_4(\rho)$ process $\Gamma$ in $\BD$. Let $\gamma'$ be the chordal $\SLE_4(\rho;-2-\rho)$ from $1$ to $-1$ in $\BD$ of the $\ccwBCLE_4(-2-\rho)$ process $\Gamma'$ in $\BD$ where $\Gamma$, $\Gamma'$ are coupled together so that the true (resp.\ false) loops of $\Gamma$ coincide with the false (resp.\ true) loops of $\Gamma'$.

Combining~\Cref{prop diameter of arcs} with the reversibility of $\SLE_4(\rho; -\rho -2)$ processes \cite[Theorem~1.1.6]{LevelLineGFFI} (which transfers to the arcs near the target point the control near the starting point), the maximum of the diameters of the arcs of $\gamma$ touching $\{e^{\ri\theta}: \theta \in [\pi,2\pi]\}$ and of the arcs of $\gamma'$ touching $\{e^{\ri\theta}: \theta \in [0,\pi]\}$ converges to zero in probability as $\rho \uparrow 0$.  But \cite[Lemma~3.6]{TVSGFF} combined with the level line interaction rules \cite[Theorem~1.1.4]{LevelLineGFFI} implies that every arc of the $\cwBCLE_4(\rho)$ loop around the origin is then one of these arcs so that the maximum of the diameters of the arcs of the $\cwBCLE_4(\rho)$ false loop around the origin tends to zero in probability as $\rho \uparrow 0$. Moreover, another consequence of~\cite{TVSGFF} is that with probability going to $1$ as $\rho \uparrow 0$, the loop of the $\cwBCLE_4(\rho)$ around the origin is a false loop (see, e.g.,~\cite[Theorem~1.10]{LSYZ24} for more details).

Next, by~\Cref{rem:counterclockwise_loops_bcle_4}, for all $t \in (0,2\lambda)$, $C_t(0)$ is the interior of the false $\cwBCLE_4(-t/\lambda)$ loop around the origin on an event of probability $1-o(1)$ as $t \to 0$. Hence, for all $\delta>0$, with probability $1-o(1)$, all the arcs of $\partial C_t(0)$ away from $\partial\BD$ have diameter at most $\delta$. On this event, $\partial C_t(0)$ lies within distance $\delta$ of $\partial \BD$, so $B_{1-\delta}(0) \subseteq C_t(0)$; every $a \in \partial\BD$ is within distance $2\delta$ of $\partial C_t(0) \cap \partial\BD$ (consider the point of $\partial C_t(0)$ on $[0,a]$ and an endpoint of the arc containing it); and every component of $\BD \setminus \overline{C_t(0)}$, bounded by one such arc and the arc of $\partial\BD$ between its endpoints, has diameter at most $3\delta$. This proves the three assertions.
\end{proof}

 \section{When the boundary is identified with a point}
\label{sec:boundary_is_a_point}
 In this section, we construct the metric $D^\partial$ as a subsequential scaling limit of the graph distance on the non-simple CLE when the boundary $\partial \BD$ is also seen as a point of the metric space. In particular, we prove that the scaling limits of the balls from the boundary are given by the balls of the uniform exploration.

\subsection{The graph distance on non-simple CLEs}\label{subsec:graph_distance}

{Let $\kappa \in (4, 8)$.} Let $\Gamma^\kappa$ be a non-nested $\CLE_\kappa$ in $\BD$. We shall write $\widehat\Gamma^\kappa \defeq \Gamma^\kappa \cup \{\partial \BD \}$. We shall write $D^{\kappa, \partial} \colon \widehat\Gamma^\kappa \times \widehat\Gamma^\kappa \to \BR$ for the graph distance associated with $\widehat\Gamma^\kappa$, where distinct $\SCL_1, \SCL_2 \in \widehat\Gamma^\kappa$ are adjacent if and only if $\SCL_1 \cap \SCL_2 \neq \emptyset$. For all deterministic $x \in \BC$, we shall write $\SCL^\kappa(x)$ for the loop of $\widehat\Gamma^\kappa$ that surrounds $x$ (with the convention that $\SCL^\kappa(x) = \partial\BD$ for $x \notin \BD$), which exists a.s. We note that if $x \in \BD$, then the random variable $D^{\kappa, \partial}(\SCL^\kappa(x), \partial \BD)$ is geometric (starting at $1$) with success probability given by the quantity already introduced in~\Cref{subsec:main_results}, namely
\begin{equation}\label{eq touching proba}
	1/\ka_\kappa = \BP\lbrack\SCL^\kappa(x) \cap \partial \BD \neq \emptyset\rbrack = 1 - \frac{\sin(\pi(\kappa/4 + 8/\kappa))}{\sin(\pi(\kappa/4 - 1))}.
\end{equation}
The second equality is a consequence of the main result of~\cite{BdryTouchingNonsimCLE}; the particular formula for $\ka_\kappa$ is not important for this paper, however, as we just need that $\ka_\kappa \to \infty$ as $\kappa \downarrow 4$, which also follows directly from the fact that $\CLE_4$ loops a.s.\ do not touch the domain boundary. The assertion made just before~\eqref{eq touching proba} follows from the domain Markov property of the $\CLE_\kappa$. Indeed, it follows from \cite[Section~5.1]{TreeCLE} that if we condition on the closure $\chi$ of the union of the $\CLE_{\kappa}$ loops which intersect $\partial \BD$, then the law of the remaining loops is given by an independent $\CLE_{\kappa}$ in each connected component of $\BD \setminus \chi$. So, if $\SCL^{\kappa}(0) \cap \partial \BD = \emptyset$, then we can iterate inside the connected component $U_1$ of $\BD \setminus \chi$ containing $0$. By conformal invariance, we still have that the conditional probability that $\SCL^{\kappa}(0) \cap \partial U_1 \neq \emptyset$ is equal to $1 / \ka_{\kappa}$. If $\SCL^{\kappa}(0) \cap \partial U_1 = \emptyset$, then we can iterate the above procedure and at each step, we have that the conditional probability that $\SCL^{\kappa}(0)$ intersects the domain boundary is equal to $1/\ka_{\kappa}$. It follows that $D^{\kappa,\partial}(\SCL^{\kappa}(0), \partial \BD)$ has the law of a geometric random variable with success probability given by~\eqref{eq touching proba}. By conformal invariance, the same is true for the loop surrounding a point $x \in \BD$.

For all $t \ge 0$ and $\SCL \in \widehat{\Gamma}^\kappa$, let us denote by $\SCB_t^{\kappa, \partial}(\SCL)$ the closure of the union of the domains encircled by the loops $\SCL' \in \widehat{\Gamma}^\kappa$ which are at $D^{\kappa, \partial}$-distance at most $t$ from $\SCL$. For all pairs of loops $\SCL_1, \SCL_2 \in \widehat{\Gamma}^\kappa$, let $K^{\kappa, \partial}_{\SCL_1, \SCL_2}$ be the union of all the loops which lie on some shortest path between $\SCL_1$ and $\SCL_2$, so that $K^{\kappa, \partial}_{\SCL_1, \SCL_2}$ is a well-defined random compact set and $K^{\kappa, \partial}_{\SCL_1, \SCL_2}=K^{\kappa, \partial}_{\SCL_2, \SCL_1}$.

\subsection{Tightness of distances and geodesics}

\begin{proposition}\label{prop subsequential limit}
	Let $\kappa_n \downarrow 4$. Then there exists a subsequence of $(\kappa_n)_{n\ge 0}$, again denoted by $(\kappa_n)_{n\ge 0}$, such that we have the joint convergence in law:
	\begin{equation}\label{eq cv subsequence}
		\begin{pmatrix}
			(\SCL^{\kappa_n}(x))_{x \in \BQ^2} \\
			(\ka_{\kappa_n}^{-1}D^{\kappa_n, \partial}(\SCL^{\kappa_n}(x), \SCL^{\kappa_n}(y)))_{x,y \in \BQ^2}			
			\\
			(K^{\kappa_n, \partial}_{\SCL^{\kappa_n}(x), \SCL^{\kappa_n}(y)})_{x, y \in \BQ^2}			
		\end{pmatrix}
		\mathop{\longrightarrow}\limits_{n\to \infty}^{(\mathrm{d})}
		\begin{pmatrix}
			(\SCL(x))_{x \in \BQ^2} \\
			(D^\partial(\SCL(x), \SCL(y)))_{x,y \in \BQ^2}
			\\
			(K^\partial_{\SCL(x), \SCL(y)})_{x, y \in \BQ^2}
		\end{pmatrix},
	\end{equation}
	where the loops $(\SCL(x))_{x \in \BQ^2} $ have the same law as the loops of a $\CLE_4$ $\Gamma$ in $\BD$ (and will be identified with the loops of $\Gamma$), where we write $\widehat{\Gamma} \defeq \Gamma \cup \{\partial \BD\}$, where $D^\partial$ is a random pseudometric on $\widehat{\Gamma}$, and where $K^\partial_{\SCL(x), \SCL(y)}$ 
	are random compact subsets of $\overline{\BD}$. The above convergence holds for the product topology, and the compact sets converge in the Hausdorff topology.
\end{proposition}
\begin{proof}
	 The convergence of the first coordinate is a direct consequence of \cite[Proposition~3.12]{BECriLQG}. 
	 
	 Recall from~\eqref{eq touching proba} that $D^{\kappa, \partial}(\SCL^\kappa(x), \partial \BD)$ is a geometric random variable with success probability $1/\ka_{\kappa}$. In particular, $\ka_{\kappa}^{-1}D^{\kappa, \partial}(\SCL^\kappa(x), \partial \BD)$ converges in distribution to an exponential random variable of parameter $1$ as $\kappa \downarrow 4$. As a result, the tightness (and almost sure finiteness of the limit) of the second coordinate is a simple consequence of the triangle inequality. The function $D^\partial$ is a.s.\ a pseudometric as a limit of pseudometrics. The tightness of the last coordinate stems from the fact that the space of compact subsets of $\overline{\BD}$ equipped with the Hausdorff distance is compact since $\overline{\BD}$ is compact. Since the three coordinates are tight,~\eqref{eq cv subsequence} follows by extracting a subsequence along which they converge jointly.
\end{proof}

\subsection{Convergence to the uniform exploration}
\label{subsec:convergence_to_uniform_exploration}

This subsection is devoted to the proof of the following result, showing that the scaling limit of the graph distance to the boundary is given by the uniform exploration of the $\CLE_4$. 

\begin{proposition}\label{prop cv explo unif}
	For all $y \in \BQ^2$, let $C^{\kappa_n}_t(y)$ be the connected component of $\BD \setminus \SCB^{\kappa_n, \partial}_{\ka_{\kappa_n}t}(\partial \BD)$ containing $y$ when it exists and the empty set otherwise (these components are simply connected, $\SCB^{\kappa_n, \partial}_{\ka_{\kappa_n}t}(\partial \BD)$ being a connected closed subset of $\overline{\BD}$ containing $\partial \BD$). Then, jointly with the convergence~\eqref{eq cv subsequence}, we have
	\begin{equation}\label{eq cv ball}
        (C^{\kappa_n}_t(y))_{t \in \BQ_{\ge 0}, y \in \BQ^2}
        \mathop{\longrightarrow}\limits_{n\to \infty}^{(\mathrm{d})}
		(C_t(y))_{t \in \BQ_{\ge 0}, y \in \BQ^2}
	\end{equation}
	for the product topology where the space of simply connected domains containing $y$ is equipped with the Carath\'eodory topology viewed from $y$, to which we adjoin $\emptyset$ as the limit of domains whose conformal radius seen from $y$ tends to $0$, 
	where ${C_t(y)}$ is the connected component of $\BD \setminus \SCB_t^\partial(\partial \BD)$ containing $y$, and where $(\SCB_t^\partial(\partial \BD))_{t\ge 0}$ is the explored region at time $t$ in the uniform exploration of the $\CLE_4$ (denoted $\SCB_t(\partial \BD)$ in~\Cref{subsec:labeled_cle_4}) (where the time parametrization is that of the PPP of SLE bubbles).
\end{proposition}

\subsubsection{Background on branching $\SLE_\kappa(\kappa-6)$}
\label{subsec:branching_sle_kappa_minus_6}

Let us recall from~\cite{TreeCLE, BECriLQG} the radial $\SLE_\kappa(\kappa-6)$ exploration of the $\CLE_\kappa$ and its approximation. Let $B$ be a standard Brownian motion and let $\theta^\kappa$ be the unique process taking values in $[0,2\pi]$ starting from $0$ that is adapted to the natural filtration of $B$, is reflected instantaneously at $0$ and at $2\pi$, and solves the SDE
\begin{equation}\label{eq diffusion theta}
\rd \theta^\kappa_t = \sqrt{\kappa} \, \rd B_t + \left(\frac{\kappa-4}{2}\right)\cot(\theta^\kappa_t/2) \, \rd t
\end{equation}
on time intervals for which $\theta^\kappa_t \notin \{0, 2\pi\}$. By \cite[Proposition~4.2]{TreeCLE}, we know that the above process $\theta^\kappa$ is uniquely defined. 
The driving function $(W^\kappa_t)_{t\ge 0}$ is then defined by setting for all $t\ge 0$,
\begin{equation*}
W^\kappa_t \defeq \exp\left(\ri \theta^\kappa_t -\ri \int_0^t \cot(\theta_s^\kappa/2) \, \rd s\right).
\end{equation*}
The integral defining $W^\kappa$ is a.s.\ absolutely convergent: $\cot(\theta/2)$ blows up like $2/\theta$ at $0$ and like $-2/(2\pi-\theta)$ at $2\pi$, while near either endpoint $\theta^\kappa/\sqrt{\kappa}$ evolves as a Bessel process of dimension $3-8/\kappa>1$, for which $\int_0^t \rd s/X_s < \infty$ a.s.\ (see, e.g.,~\cite[Chapter~XI]{RY05}).
Next, let $g_t^\kappa(z)$ for $z \in \BD$ be the solution of the radial Loewner equation
\begin{equation}\label{eq Loewner radial}
g^\kappa_0 (z)=z \qquad \text{and} \qquad \partial_t g^\kappa_t (z) = g^\kappa_t(z) \frac{W_t^\kappa + g^\kappa_t (z)}{W_t^\kappa - g^\kappa_t (z)},
\end{equation}
which is well-defined until time $t_z^\kappa\defeq \inf \{ t \ge 0: g^\kappa_t(z) = W^\kappa_t\}$.

Let $(D^\kappa_t)_{t\ge 0}$ be the radial Loewner chain associated with $W^\kappa$, defined by $D^\kappa_t \defeq \{z \in \BD: t_z^\kappa >t\}$. It follows from~\cite{IG4} that there exists a.s.\ a continuous non-self-crossing curve $\eta_0^\kappa: [0, \infty) \to \overline{\BD}$ such that for all $t\ge 0$, the set $D^\kappa_t$ is the connected component of $\BD \setminus \eta_0^\kappa([0,t])$ which contains the origin. The time $\tau^\kappa_{\mathrm{ccw}} \defeq \inf\{t\ge 0: \theta_t^\kappa = 2 \pi\}$ (which is a.s.\ finite) is the first time that $ \eta_0^\kappa$ draws a counterclockwise loop around the origin. By taking the unique conformal mapping $\phi \colon \BD \to \BD$ with $\phi(0)=z$ and $\phi(1)=1$, one can define the $\SLE_\kappa(\kappa-6)$ branch $\eta_z^\kappa \defeq \phi \circ \eta^\kappa_0$ toward $z$. By \cite[Proposition~3.13]{TreeCLE}, $\eta_z^\kappa$ has the law of a radial $\SLE_\kappa(\kappa-6)$ in $\BD$ from $1$ targeted at $z$. Moreover, by \cite[Proposition~3.14]{TreeCLE}, one can couple $\eta_z^\kappa$ with $\eta_w^\kappa$ for $z \neq w \in \BD$ in such a way that the two curves coincide up to a time change until the time at which $z$ and $w$ are disconnected by the curve. By performing such a coupling for all $z \in \BD_\BQ$, we construct the branching $\SLE_\kappa(\kappa-6)$. Recall that $\tau^\kappa_{\mathrm{ccw}} <\infty$ a.s.. We denote by $\SCL^\kappa(0)$ the non-simple loop obtained as follows. Let $\acute{\tau}^\kappa_{\mathrm{ccw}} \defeq \sup \{ t\le \tau^\kappa_{\mathrm{ccw}}: \theta^\kappa_t =0 \}$ and let $\widetilde{\eta}$ be the branch $\eta^\kappa_{\eta^\kappa_0(\acute{\tau}^\kappa_{\mathrm{ccw}} )}$ reparameterized so that $\widetilde{\eta}\vert_{ [0,\tau^\kappa_{\mathrm{ccw}}]}  = \eta^\kappa_0 \vert_{ [0,\tau^\kappa_{\mathrm{ccw}}]} $. The loop $\SCL^\kappa(0)$ is then defined as $\widetilde{\eta}\vert_{ [\acute{\tau}^\kappa_{\mathrm{ccw}}, \infty]} $, where $\widetilde{\eta}(\infty) \defeq \lim_{t \to \infty}\widetilde{\eta}(t) = \eta^\kappa_0(\acute{\tau}^\kappa_{\mathrm{ccw}})$ is the target of the branch, so that this path is indeed a loop. Similarly, for all $z \in \BD$, let us denote by $\SCL^\kappa(z)$ the loop of the $\CLE_\kappa$ which is drawn by the curve $\eta^\kappa_z$ around $z$ when it exists. We note that for arbitrary $z \in \BD$, the loop $\SCL^{\kappa}(z)$ exists if and only if the curve $\eta^{\kappa}_z$ makes a counterclockwise loop around $z$.

The $\CLE_\kappa$ can then be constructed as the countable collection of (non-simple) loops $\SCL^\kappa(z)$ for $z \in \BD_\BQ$; that this collection satisfies the axioms of~\Cref{subsec:cle} is the content of~\cite{CLE,TreeCLE}. It is made of non-simple loops which can touch each other as well as the boundary of $\BD$, and we equip it with the graph distance of~\Cref{subsec:graph_distance}. By \cite[Lemma~5.2]{TreeCLE}, we know that the loop $\SCL^\kappa(z)$ touches $\partial \BD$ if and only if $\eta^\kappa_z$ has not drawn a clockwise loop around $z$ before drawing the counterclockwise loop around $z$. Moreover, when $\eta^\kappa_z$ draws its first clockwise loop around $z$, the encircled region is either a connected component of the complement of a loop of the  $\CLE_\kappa$ or a connected component of the complement of a \emph{chain of loops} of the $\CLE_\kappa$, i.e.\ of a finite sequence of loops each of which intersects the next. Since this loop or chain was drawn before that first time, it touches $\partial \BD$. The same applies to each subsequent clockwise loop, with $\partial \BD$ replaced by the boundary of the region encircled by the previous one. See~\Cref{fig:explo_CLE_kappa} for an illustration.

\begin{figure}[h]
	\centering
	\includegraphics[scale=0.8]{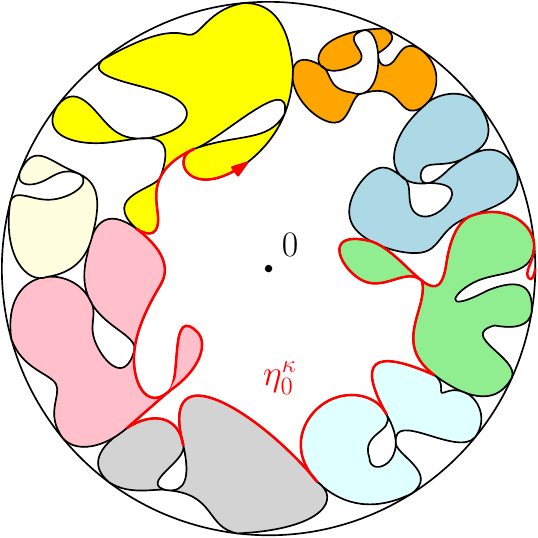}
	\caption{Illustration of the radial (totally asymmetric) $\SLE_\kappa(\kappa-6)$ exploration of the $\CLE_\kappa$ for $\kappa>4$ targeted at $0$. Here, the curve $\eta^{\kappa}_0$ has not yet drawn a clockwise loop around $0$, so the loops it discovers are still at graph distance one from $\partial \BD$.}
	\label{fig:explo_CLE_kappa}
\end{figure}

Using the fact that the times $\tau^\kappa_1, \tau^\kappa_2, \ldots$ at which $\eta^\kappa_z$ draws a clockwise loop around $z$ are renewal times, in the sense that if we conformally map the encircled region to $\BD$ sending $z$ to $0$ and the tip of $\eta^\kappa_z$ to~$1$, then the image of the remainder of the curve is independent of what has been drawn before and has the same law as $\eta^\kappa_0$, we can see by induction that for all $k\ge 1$, the loop $\SCL^\kappa(z)$ is at graph distance $k$ from $\partial \BD$ if and only if the counterclockwise loop around $z$ is drawn by $\eta^\kappa_z$ between times $\tau^\kappa_{k-1}$ and $\tau^\kappa_k$, where by convention we set $\tau^\kappa_0=0$.

Next, let us recall the following approximation of the $\SLE_\kappa(\kappa-6)$ from~\cite{BECriLQG}. Let $n\ge 1$. As in~\cite{BECriLQG}, let $T^{\kappa,n}_0\defeq 0$ and for all $k\ge 1$, let 
\begin{equation}\label{eq excursion intervals}
	\left\lbrace
	\begin{matrix}
	R^{\kappa,n}_k &\defeq &\inf\bigl\{ t\ge T^{\kappa, n}_{k-1}: \theta^\kappa_t \ge 2^{-n}\bigr\}, \\
	S^{\kappa,n}_k & \defeq  &\sup\bigl\{t \le R^{\kappa,n}_k: \theta^\kappa_t =0\bigr\}, \\
	T^{\kappa,n}_k & \defeq &\inf\bigl\{ t\ge R^{\kappa,n}_k: \theta^\kappa_t =0\bigr\},
	\end{matrix}
	\right.
\end{equation}
so that the intervals $[S^{\kappa,n}_k, T^{\kappa,n}_k]$ for $k\ge 1$ are the intervals where $\theta^\kappa$ makes an excursion $e^{\kappa,n}_k$ of height at least $2^{-n}$. Let $\Lambda^{\kappa,n} \defeq \max \{k \ge 1: S^{\kappa,n}_k \le \tau^\kappa_{\mathrm{ccw}} \}$. Let $l^{\kappa, n}_k \defeq T^{\kappa, n}_k - S^{\kappa, n}_k$ for $k< \Lambda^{\kappa, n}$, let $l^{\kappa,n}_{\Lambda^{\kappa,n}} \defeq \tau^\kappa_{\mathrm{ccw}} - S^{\kappa,n}_{\Lambda^{\kappa,n}}$ and $L^{\kappa, n}_k \defeq \sum_{j=1}^k l_j^{\kappa,n}$ for all $0 \le k\le \Lambda^{\kappa,n}$ (so that $L^{\kappa,n}_0 = 0$). We also write $\tau^{\kappa,n}\defeq L^{\kappa,n}_{\Lambda^{\kappa,n}}$. For all $k\le \Lambda^{\kappa,n}$, for all $t \in [L_{k-1}^{\kappa, n}, L^{\kappa,n}_k)$, we write
\begin{equation*}
W^{\kappa,n}_t \defeq W^\kappa_{S^{\kappa,n}_k+ t-L^{\kappa,n}_{k-1}}.
\end{equation*}
Let $(D^{\kappa, n}_t)_{t\ge 0}$ be the Loewner chain driven by $W^{\kappa,n}$.

By the proof of \cite[Lemma~2.10]{BECriLQG}, we know that $(D^{\kappa, n}_t)_{t\ge 0}$ converges in probability to $(D^{\kappa}_t)_{t\ge 0}$ as $n\to \infty$ uniformly in $\kappa \in (4,8)$ in the metric space $\SD$ defined in~\eqref{eq space of domains}, i.e., for all $\varepsilon>0$,
\begin{equation}\label{eq approx D kappa}
	\sup_{\kappa \in (4,8)} \BP\Bigl\lbrack d_{\SD}\bigl((D^{\kappa, n}_t)_{t\ge 0}, (D^{\kappa}_t)_{t\ge 0}\bigr) + \bigl\lvert \tau^{\kappa,n}- \tau^\kappa_{\mathrm{ccw}}\bigr\rvert \ge \varepsilon \Bigr\rbrack \to 0 \quad \text{as } n \to \infty.
\end{equation}
Indeed, by \cite[Equation (2.9)]{BECriLQG}, we have  $\sup_{\kappa \in (4,8)} \BP[  \bigl\lvert \tau^{\kappa,n}- \tau^\kappa_{\mathrm{ccw}}\bigr\rvert \ge \varepsilon ]{\to 0 \quad \text{as } n \to \infty}$. Then, note that by \cite[Proposition~6.1]{MS16QLE} {(see also \cite[Remark~2.1]{BECriLQG})}, the mapping from the space of driving measures equipped with the topology of weak convergence to $\SD$, which sends the driving measure to its corresponding Loewner chain, is continuous, together with the fact that excising the intervals of the excursions of height less than $2^{-n}$ makes the corresponding driving measures converge weakly, in probability and uniformly in $\kappa$ (where we equip the space of driving measures with the Lévy--Prokhorov distance). This proves~\eqref{eq approx D kappa}.

\subsubsection{Approximation of the uniform exploration}\label{subsection approximation of the uniform exploration}

Similarly, in the case $\kappa=4$, the uniform exploration of the $\CLE_4$ can be approximated as follows. Define $R^n_k, S^n_k, T^n_k, l^n_k, L^n_k, \Lambda^n$ as above using a reflected Brownian motion $\theta_t$ with speed $4$, i.e., the process starting at zero, satisfying the SDE $\rd \theta_t = 2 \, \rd B_t$ on time intervals for which $\theta_t \not\in \{0,2\pi\}$, and reflected at $\{0,2\pi\}$. The time $\tau^\kappa_\mathrm{ccw}$ is replaced here by $\tau\defeq \inf\{t\ge 0: \theta_t = 2\pi\}$ and we also define $\tau^n\defeq L^n_{\Lambda^n}$. Independently, for all $k\ge 1$, let $X^n_k$ be a uniform random variable on $\partial \BD$. Then, define for all $k\ge 1$, for all $t \in [L_{k-1}^{ n}, L^{n}_k)$,
\begin{equation*}
W^{n}_t \defeq X_k^n \exp \left(\ri \left(\theta_{S^n_k+ t - L^n_{k-1}} - \int_{S^n_k}^{S^n_k + t - L^n_{k-1}} \cot(\theta_s/2) \, \rd s\right)\right).
\end{equation*}
Let $(D^n_t)_{t\ge 0}$ be the Loewner chain driven by $W^n$. Using the same argument as in \cite[Section 4]{CoInCLEExpl}, we know that 
\begin{equation*}
(D^n_t)_{t\ge 0} \mathop{\longrightarrow}\limits_{n\to \infty}^{(\mathrm{d})} (D_t)_{t\ge 0}
\end{equation*}
in the space $\SD$, where $(D_t)_{t\ge 0}$ is known as the uniform exploration of the $\CLE_4$ targeted at zero (this process is constructed below as the Loewner chain driven by $W$, the convergence stated here being the content of~\Cref{lemme approx explo unif}) and parametrized by minus the logarithm of the conformal radius. Let us emphasize that $(D_t)_{t\ge 0}$ is not exactly the same as $(C_t(0))_{t\ge 0}$ defined in \Cref{subsec:labeled_cle_4}. More precisely, every excursion of the reflected Brownian motion $\theta$ giving the time parametrization of $(D_t)_{t\ge 0}$ is associated to an $\SLE_4$ bubble in the Poisson point process defining $(C_t(0))_{t\ge 0}$, so that $(C_t(0))_{t\ge 0}$ is obtained from $(D_t)_{t\ge 0}$ via a time-change. This will be made precise after the following definitions and notation.

The above convergence can be upgraded to a convergence in probability by choosing the $X^n_k$'s in a way that for all $n<n'$, for all $k,k'$ such that $S^n_k= S^{n'}_{k'}$ we have $X^n_k=X^{n'}_{k'}$, i.e., we take the same uniform random variable in the unit circle for the same excursion as will be done in~\Cref{lemme approx explo unif}.

Let $L^0_t(\theta) $ for $t\ge 0$ be the local time at zero of the reflected Brownian motion $\theta$, which can be defined as the limit in probability of $2^{1-n}\#\{k\ge 1: S^n_k\le t\}$ as $n\to \infty$ by \cite[Theorem~(1.10) Chapter VI]{RY05}. Let us denote by $(\tau_u)_{u\ge 0}$ its right-continuous inverse defined by $\tau_u= \inf\{t\ge 0: L^0_t(\theta) > u\}$. For all $u\ge 0$, for all $t\ge 0$, let $e_u(t)= \theta_{(\tau_{u-}+t)\wedge \tau_u}$ be the excursion of $\theta$ away from zero on the time interval $(\tau_{u-}, \tau_u)$. We then define the measurable function $W$ taking its values in $\partial \BD$ by setting, for all $u\ge 0$, for all $t \in [\tau_{u-}, \tau_u)$,
\begin{align*}
    W_t &\defeq X_u \exp\left(\ri \left(e_u(t-\tau_{u-}) - \int_0^{t-\tau_{u-}}    	
    	\cot(e_u(s)/2) \, \rd s \right) \right) \\
    &= X_u \exp\left(\ri \left(\theta_t - \int_{\tau_{u-}}^{t} \cot(\theta_s/2) \, \rd s\right)\right),
\end{align*}
where the $X_u$ are i.i.d.~uniform random variables on $\partial \BD$ indexed by the countably many $u$ with $\tau_{u-} < \tau_u$, so that $(X_u, e_u)$ is a marked Poisson point process. This defines $W$ outside a Lebesgue-null set of times, which is enough since $W$ enters only through the driving measure $\delta_{W_t}(\rd z) \, \rd t$. The same applies to $W^n$ and $W$ for $t<\tau$: there $\theta$ is a reflected Brownian motion, for which $\int_0^t \rd s/\theta_s = \infty$, but each integral above runs over a single excursion of $\theta$ away from $0$, and a Brownian excursion behaves near its endpoints as a Bessel process of dimension $3$. Note that for all $n\ge 1$, for all $k\ge 1$, there is a unique $u\ge 0$ such that $S^n_k= \tau_{u-}$ and $T^n_k = \tau_u$. We thus set $X^n_k \defeq X_u$ and $e^n_k \defeq e_u$. Let $(D_t)_{t\ge 0}$ be the radial Loewner chain driven by $W$, i.e., by the driving measure $\delta_{W_t}(\rd z) \, \rd t$.

\begin{lemma}\label{lemme approx explo unif}
	With the above choice of $X^n_k$'s we have the convergence in probability 
	\begin{equation}\label{eq approx explo unif}
		\left((D^n_t)_{t\ge 0}, \tau^n \right) \mathop{\longrightarrow}\limits_{n\to \infty}^{(\BP)} \left((D_t)_{t\ge 0}, \tau \right)
	\end{equation}
	in the space $\SD \times \BR$.
\end{lemma}
\begin{proof}
	The convergence of $\tau^n$ to $\tau$ stems from the fact that the set of zeros of the Brownian motion has zero Lebesgue measure. For the convergence of $(D^n_t)_{t\ge 0}$, it suffices to notice that the convergence of $\tau^n$ to $\tau$ implies that for all $T>0$, for all continuous bounded functions $F: \partial \BD \times [0,T] \to \BR$,
	\begin{equation*}
	    \int_0^T F(W^n_t, t) \, \rd t  \mathop{\longrightarrow}\limits_{n\to \infty}^{(\BP)} \int_0^T F(W_t, t) \, \rd t
	\end{equation*}
{(see also \cite[Remark~2.1]{BECriLQG})}. One can then apply \cite[Proposition~6.1]{MS16QLE}.
\end{proof}

Let us also recall that $\tau= \inf\{t\ge 0: \theta_t = 2\pi\}$ is the time at which the Loewner curve driven by $(W_t)_{\tau_{L^0_\tau(\theta)-} \le t \le \tau}$ draws a loop around the origin, which is the loop of the $\CLE_4$ surrounding the origin. Note also that the PPP $(X_u, e_u)_{u\ge 0}$ gives rise to the PPP of SLE bubbles defined in~\cite{CoInCLEExpl} in the chordal setting, so that each curve drawn by $(W_t)_{\tau_{u-} \le t \le \tau_u}$ is a loop of the $\CLE_4$ and the time $u$ is the label of that loop, i.e.\ the time at which it is discovered by the exploration. More precisely, we use the normalization of the bubble measure fixed in~\Cref{subsec:labeled_cle_4}. Moreover, note that, by scaling of the Brownian excursion measure, the first time $u\ge 0$ such that the excursion $e_u$ reaches height $2\pi$ is an exponential random variable of parameter $1/(4\pi)$. 
 Coming back to the reflected Brownian motion $(\theta_t)_{t\ge 0}$, for each time $t\in (0,\tau)$ such that $\theta_t>0$, the Loewner chain driven by $W$ is drawing a loop which is at distance $L^0_t(\theta)/(4 \pi)$ from the boundary $\partial \BD$. In other words, we have
 \begin{equation}\label{eq time change uniform exploration}
 	(C_u(0))_{u\ge 0}=(D_{\tau_{4 \pi u}})_{u\ge 0}.
 \end{equation}This follows from the above correspondence between the two Poisson point processes together with the normalization of the bubble measure.

Next, let us recall the convergence result of~\cite{BECriLQG}.

\begin{lemma}[{Consequence of \cite[Lemma~2.11 and Lemma~2.12]{BECriLQG}}]\label{lemme approx kappa to 4}
	Jointly for all $n\ge 1$,
	\begin{equation*}
	\left( (D^{\kappa,n}_t)_{t\ge 0}, \tau^{\kappa,n}, (e^{\kappa,n}_k)_{1 \le k \le \Lambda^{\kappa,n}}\right) \mathop{\longrightarrow}\limits_{\kappa \downarrow 4}^{(\mathrm{d})}
	\left( (D^n_t)_{t\ge 0}, \tau^n, (e^{n}_k)_{1 \le k \le \Lambda^n}\right),
	\end{equation*}
	where the first coordinate converges in the space $\SD$, the second one converges in $\BR$ and where the excursions converge for the topology of uniform convergence on compact subsets.
\end{lemma}
 The fact that the convergence  of $(D^{\kappa,n}_t)_{t\ge 0}$ and $\tau^{\kappa,n}$ holds jointly with $(e^{\kappa,n}_k)_{1 \le k \le \Lambda^{\kappa,n}}$ comes from the proof of \cite[Lemma~2.11]{BECriLQG} and the fact that the convergence holds jointly for all $n\ge 1$ comes from our definition of the excursions $e^n_k$ and the variables $X^n_k$.

But note that since $\tau^{\kappa,n}$ converges in probability to $\tau^\kappa_{\mathrm{ccw}}$ uniformly in $\kappa \in (4,8)$ and since $\tau^n \to \tau$ in probability as $n\to \infty$, we deduce from the convergence of $(e^{\kappa,n}_k)_{1 \le k \le \Lambda^{\kappa,n}}$ the convergence 
\begin{equation}\label{eq cv theta}
	(\theta^\kappa_t)_{0 \le t \le \tau^\kappa_{\mathrm{ccw}}} \mathop{\longrightarrow}\limits_{\kappa\downarrow 4}^{(\mathrm{d})}(\theta_t)_{0 \le t \le \tau}
\end{equation}
jointly with the convergence of~\Cref{lemme approx kappa to 4}. The topology which is used above, for continuous functions of the form $f\colon [0, t]\to \BR$, can be defined via the topology of uniform convergence on compact sets for the functions extended to $\BR_{\ge 0}$ by setting $f(s)=f(t)$ for all $s\ge t$.

From~\Cref{lemme approx kappa to 4} and~\eqref{eq cv theta} we also get jointly the convergence of the excursion times
\begin{equation}\label{eq cv excursion times}
	\left( S^{\kappa,n}_k, R^{\kappa,n}_k, T^{\kappa, n}_k\right)_{1 \le k \le \Lambda^{\kappa,n}}
	\mathop{\longrightarrow}\limits_{\kappa\downarrow 4}^{(\mathrm{d})}
	\left( S^{n}_k, R^{n}_k, T^{n}_k\right)_{1 \le k \le \Lambda^{n}}.
\end{equation}
Here, and in the similar convergences below, we use that the level-hitting times of $\theta^\kappa$ are a.s.\ continuous functionals at the limiting reflected Brownian path, which a.s.\ crosses each fixed level immediately after hitting it.

\subsubsection{Convergence of the rescaled graph distance}

Next, we prove the following result that extends \cite[Proposition~2.5]{BECriLQG}.

\begin{lemma}\label{lemma number of clockwise loops}
	We have the following convergence
	\begin{equation*}
	\left( (D_t^\kappa)_{t\ge 0}, \tau^\kappa_\mathrm{ccw}, \theta^\kappa,\left( \ka_\kappa^{-1} \# \{i\ge 1: \tau^\kappa_i \le t\}\right)_{0 \le t \le \tau^\kappa_\mathrm{ccw}} \right)
	\mathop{\longrightarrow}\limits_{\kappa\downarrow 4}^{(\mathrm{d})}
	\left( (D_t)_{t\ge 0}, \tau, \theta,\left( \frac{1}{4\pi} L^0_t(\theta)\right)_{0 \le t \le \tau} \right),
	\end{equation*}
	where the first coordinate converges in $\SD$, the second one in $\BR$ and the third and fourth ones converge for the topology of uniform convergence on compact subsets of $\BR_{\ge 0}$, using the extension described below \eqref{eq cv theta}.
\end{lemma}
\begin{proof}
	\stepx{step:cw-first}{Convergence of the first coordinates} The convergence of $(D_t^\kappa)_{t\ge 0}, \tau^\kappa_\mathrm{ccw}$, which corresponds to \cite[Proposition~2.5]{BECriLQG}, stems from the convergence of~\Cref{lemme approx kappa to 4}, together with the convergences in probability~\eqref{eq approx D kappa} and~\eqref{eq approx explo unif}.
	
	\stepx{step:cw-decomposition}{Decomposition of the number of clockwise loops} The only convergence that remains to be proven is thus the convergence of the last coordinate. Recall that $\#\{i\ge 1: \tau_i^\kappa< \tau^\kappa_\mathrm{ccw} \}$ is a geometric random variable starting at zero with success probability $1/\ka_\kappa$. Let $n\ge 1$. One can write
	\begin{equation}\label{eq decomposition number of clockwise loops}
		\#\{i\ge 1: \tau_i^\kappa< \tau^\kappa_\mathrm{ccw} \} = \sum_{j=1}^{\Lambda^{\kappa,n}}
		\#\{i\ge 1: T^{\kappa,n}_{j-1} \le \tau^\kappa_i \le S^{\kappa,n}_j\}.
	\end{equation}
	Note that the $\tau^\kappa_i$ are zeros of $\theta^\kappa$, hence lie in no excursion interval $(S^{\kappa,n}_j, T^{\kappa,n}_j)$, which is what makes~\eqref{eq decomposition number of clockwise loops} an identity; moreover $S^{\kappa,n}_{\Lambda^{\kappa,n}} \le \tau^\kappa_{\mathrm{ccw}}$ by definition of $\Lambda^{\kappa,n}$, so the last interval in the sum is complete. Notice that the terms of the sum on the right-hand side are i.i.d.~geometric random variables (starting at zero) since the times when the radial $\SLE_\kappa(\kappa-6)$ draws a clockwise loop are renewal times, and they are independent of $\Lambda^{\kappa,n}$. Let us denote by $p^{\kappa,n}$ their parameter.
	
	\stepx{step:cw-geometric}{Convergence of the geometric variables} Moreover, $\Lambda^{\kappa,n}$ is a geometric random variable starting at $1$ with parameter $\BP[T^{\kappa,n}_1 \geq \tau^\kappa_\mathrm{ccw}]$. By taking the expectation in~\eqref{eq decomposition number of clockwise loops}, one gets
	\begin{equation*}
	    \ka_\kappa-1=(1/p^{\kappa,n} - 1)/\BP[T^{\kappa,n}_1 \geq \tau^\kappa_\mathrm{ccw}],
	\end{equation*}
	so that
	\begin{equation*}
	    p^{\kappa,n} = \bigl(1+ \BP[T^{\kappa,n}_1 \geq \tau^\kappa_\mathrm{ccw}] (\ka_\kappa-1)\bigr)^{-1}.
	\end{equation*}
	By~\eqref{eq cv theta} and~\eqref{eq cv excursion times}, we have
	\begin{align*}
		\BP[T^{\kappa,n}_1 \geq \tau^\kappa_\mathrm{ccw}] &= \BP[(\theta^\kappa_t)_{t\ge R^{\kappa,n}_1}\text{ reaches } 2\pi \text{ before touching }0] \\
		&\mathop{\longrightarrow}\limits_{\kappa \downarrow 4}\BP[(\theta_t)_{t\ge R^{n}_1}\text{ reaches } 2\pi \text{ before touching }0] \\
		&=\BP[T^n_1 \geq \tau].
	\end{align*}
	Therefore, using the fact that $\ka_\kappa \to \infty$ as $\kappa \downarrow 4$, we deduce that
	\begin{equation*}
	\ka_\kappa p^{\kappa,n} \mathop{\longrightarrow}\limits_{\kappa \downarrow 4} 1/\BP[T^n_1 \geq \tau] .
	\end{equation*}
	As a result, jointly with the convergences of~\Cref{lemme approx kappa to 4},~\eqref{eq cv theta} and~\eqref{eq cv excursion times},
	\begin{equation}\label{eq cv geometric exponential}
		\left(\ka_\kappa^{-1} \sum_{j=1}^{k}
		\#\{i\ge 1: T^{\kappa,n}_{j-1} \le \tau^\kappa_i \le S^{\kappa,n}_j\}
		\right)_{1 \le k \le \Lambda^{\kappa,n}}
		\mathop{\longrightarrow}\limits_{\kappa \downarrow 4}^{(\mathrm{d})} 
		\left( \sum_{j=1}^{k} \mathcal{E}^n_j\right)_{1\le k \le \Lambda^n}
	\end{equation}
	where the variables $\mathcal{E}^n_j$ are i.i.d.~exponential random variables of parameter ${1}/\BP[T^n_1\geq \tau]$ and are independent of $\Lambda^n$. By scaling of the Brownian excursion measure, one computes $\BP[T^n_1 \geq \tau] = 2^{-n}/(2\pi)$, so the variables $\mathcal{E}^n_j$ have parameter $\pi 2^{n+1}$. Note that the family of the $\#\{i\ge 1: T^{\kappa,n}_{j-1} \le \tau^\kappa_i \le S^{\kappa,n}_j\}$'s is actually independent of $(e^{\kappa, n}_k)_{1 \le k \le \Lambda^{\kappa, n}}$, so that conditionally on $\Lambda^n$ and on $(e^n_k)_{1\le k \le \Lambda^n}$, the $\mathcal{E}^n_j$ are i.i.d.~exponential random variables of parameter ${1}/\BP[T^n_1\geq \tau]$. 

	\stepx{step:cw-localtime}{Identification of the limit with the local time} Recall that the local time $L^0_t(\theta)$ is such that for all $t\in [0,\tau]$, 
	\begin{equation*}
	\frac{1}{2^n} \# \{k\ge 1: S^n_k \le t\} \mathop{\longrightarrow}\limits_{n\to \infty}^{(\BP)}
	\frac{1}{2} L^0_t(\theta).
	\end{equation*}
	Note that the above convergence holds uniformly in $t$ by P\'olya's theorem (also known as Dini's second theorem), since both sides are non-decreasing in $t$ and the limit is continuous (applied along a subsequence, to convert the convergence in probability into an almost sure one). In particular, for all $n_0\ge 1$ and $k_0\in [1, \Lambda^{n_0}]_\BZ$, 
	\begin{equation*}
		\frac{1}{2^n} \# \{k\ge 1: S^n_k \le S^{n_0}_{k_0}\} \mathop{\longrightarrow}\limits_{n\to \infty}^{(\BP)}
		\frac{1}{2} L^0_{S^{n_0}_{k_0}}(\theta).
	\end{equation*}
	
	Therefore, for all $n_0\ge 1$ and $k_0\in [1, \Lambda^{n_0}]_\BZ$, 
	\begin{equation*}
	\BE\!\left[ \left. \sum_{j\ge 1 \text{ s.t. }S^n_{j+1} \le S^{n_0}_{k_0}} \mathcal{E}^n_j \right\vert 
	\Lambda^n, (e^n_k)_{1 \le k \le \Lambda^n}
	\right] = \frac{1}{\pi 2^{n+1}} \#\{j\ge 1: S^n_{j+1} \le S^{n_0}_{k_0}\}  \mathop{\longrightarrow}\limits_{n\to \infty}^{(\BP)}
	\frac{1}{4\pi} L^0_{S^{n_0}_{k_0}}(\theta).
	\end{equation*}
	Moreover, the conditional variance converges to zero in probability thanks to the conditional independence of the variables $\mathcal{E}^n_j$. Thus, for all $n_0\ge 1$ and $k_0\in [1, \Lambda^{n_0}]_\BZ$, 
	\begin{equation*}
		\sum_{j\ge 1 \text{ s.t.~}S^n_{j+1} \le S^{n_0}_{k_0}} \mathcal{E}^n_j   \mathop{\longrightarrow}\limits_{n\to \infty}^{(\BP)}
		\frac{1}{4\pi} L^0_{S^{n_0}_{k_0}}(\theta).
	\end{equation*}
	Moreover, the set of the $S^{n_0}_{k_0}$ for $n_0\ge 1$ and $k_0 \in [1, \Lambda^{n_0}]$ is dense in the intersection of $[0,\tau]$ with the image of the right-continuous inverse $(\tau_u)_{u\ge 0}$. Recall also that the process $(L^0_t(\theta))_{t\ge 0}$ is continuous and non-decreasing. Consequently, for all $t\in [0,\tau]$,
	\begin{equation}\label{eq cv local time}
		\sum_{j\ge 1 \text{ s.t.~}S^n_{j+1} \le t} \mathcal{E}^n_j   \mathop{\longrightarrow}\limits_{n\to \infty}^{(\BP)}
		\frac{1}{4\pi} L^0_t(\theta),
	\end{equation}
	and the same convergence holds for the sum over $j\ge 1$ such that $T^n_{j-1} \le t$.
	
	\stepx{step:cw-conclusion}{Conclusion} Then, note that
	\begin{align*}
	    \ka_\kappa^{-1} \sum_{\substack{j\ge 1 \\ S^{\kappa,n}_{j+1}\le t}} \#\{i\ge 1: T^{\kappa,n}_{j-1} \le \tau^\kappa_i \le S^{\kappa,n}_j\} &\le \ka_\kappa^{-1} \# \{i\ge 1: \tau^\kappa_i \le t\} \\
        &\le \ka_\kappa^{-1} \sum_{\substack{j\ge 1 \\ T^{\kappa,n}_{j-1} \le t}} \#\{i\ge 1: T^{\kappa,n}_{j-1} \le \tau^\kappa_i \le S^{\kappa,n}_j\}.
	\end{align*}
	By~\eqref{eq cv geometric exponential} and~\eqref{eq cv local time}, the left-hand side and the right-hand side converge in distribution as $\kappa \downarrow 4$ and then $n\to \infty$ toward the same limit $1/(4\pi) L^0_t(\theta)$ (the middle term does not depend on $n$, so it converges to the common value of the two iterated limits) and the convergence holds as processes indexed by $t \in [0, \tau^\kappa_{\mathrm{ccw}}]$ (for the product topology and thus for the topology of uniform convergence on compact sets since the limiting process is continuous and increasing). Consequently, jointly with the convergences of~\Cref{lemme approx kappa to 4},~\eqref{eq cv theta} and~\eqref{eq cv excursion times},
	\begin{equation}\label{eq cv clockwise local time}
		\left(\ka_\kappa^{-1} \# \{i\ge 1: \tau^\kappa_i \le t\}\right)_{t \in [0, \tau^\kappa_{\mathrm{ccw}}]}
		\mathop{\longrightarrow}\limits_{\kappa\downarrow 4}^{(\mathrm{d})}
		\left(
		\frac{1}{4\pi} L^0_t(\theta)
		\right)_{t\in [0,\tau]}.
	\end{equation}
	This concludes the proof.
\end{proof}

\begin{proof}[Proof of~\Cref{prop cv explo unif}]
    By~\Cref{lemma number of clockwise loops}, we deduce the convergence of $(C^{\kappa_n}_t(0))_{t \in \BQ_{\ge 0}}$ toward $(C_t(0))_{t \in \BQ_{\ge 0}}$ since $(C_t(0))_{t\ge 0}$ is obtained by parameterizing $(D_t)_{t\ge 0}$ by the local time $t \mapsto \frac{1}{4\pi} L^0_t(\theta)$ as in \eqref{eq time change uniform exploration}. Then, the joint convergence of the other branches comes from the domain Markov property and conformal invariance of the $\CLE_\kappa$, from the target invariance and conformal invariance of the uniform exploration, and from the convergence of the separation times stated in \cite[Proposition~3.9]{BECriLQG}.
\end{proof}

\subsection{Convergence of balls} 

Next, let us derive the convergence of balls from~\Cref{prop cv explo unif}, starting with balls from the boundary.

\begin{corollary}\label{cor cv Hausdorff ball boundary}
	Jointly with~\eqref{eq cv subsequence} and~\eqref{eq cv ball}, 
	\begin{equation}\label{eq cv Hausdorff ball}
		\left( \SCB^{\kappa_n, \partial}_{\ka_{\kappa_n}t}(\partial \BD)\right)_{t \in \BQ_{\ge 0}}
		\mathop{\longrightarrow}\limits_{n\to \infty}^{(\mathrm{d})}
		\left( \SCB^\partial_t(\partial \BD) \right)_{t \in \BQ_{\ge 0}}
	\end{equation}
	for the product topology in $t \in \BQ_{\ge 0}$, where the space of compact sets is equipped with the Hausdorff topology, where $(\SCB_t^\partial(\partial \BD))_{t\ge 0}$ is as above the explored region in the uniform exploration of the $\CLE_4$.
\end{corollary}

\begin{proof}
	Let $t \in \BQ_{\ge 0}$. Note that the sequence $\SCB^{\kappa_n, \partial}_{\ka_{\kappa_n}t}(\partial \BD)$ is tight for the Hausdorff topology since it is a sequence of compact subsets of the compact space $\overline{\BD}$. Thus, we may assume that 
	\begin{equation}\label{eq preuve cv Hausdorff balls}
		\SCB^{\kappa_n, \partial}_{\ka_{\kappa_n}t}(\partial \BD) \mathop{\longrightarrow}\limits_{n\to \infty}^{(\mathrm{d})}
		\Lambda
	\end{equation}
	for the Hausdorff topology along some subsequence again denoted by $\kappa_n$, where $\Lambda$ is a random compact set. 
	
	Let us assume that the above convergence and~\eqref{eq cv subsequence} and~\eqref{eq cv ball} hold a.s.\ (for a countable set of times including $t$) by Skorokhod's representation theorem. It is enough to show that $\Lambda= \SCB_t^\partial(\partial \BD)$.
	
	Let $x\in \BD
	\setminus \SCB^\partial_t(\partial \BD)$. Since $ \SCB_t^\partial(\partial \BD)$ is closed, there exists $\varepsilon\in \BQ_{>0}$ and $y \in \BD_\BQ$ such that $x \in B_\varepsilon(y)$ and the Euclidean closed ball $\overline{B_\varepsilon(y)}$ is included in $\BD \setminus \SCB_t^\partial(\partial \BD)$. By the Carath\'eodory convergence~\eqref{eq cv ball}, we deduce that a.s.\ for all $n$ large enough, $\overline{B_\varepsilon(y)} \subseteq  \BD \setminus  \SCB^{\kappa_n,\partial}_{\ka_{\kappa_n} t}(\partial \BD)$. Therefore, by~\eqref{eq preuve cv Hausdorff balls}, we conclude that $x \in \BD \setminus \Lambda$. Since $y$ and $\varepsilon$ lie in a countable set, we obtain that a.s.\@ for all $x \in\BD\setminus \SCB^\partial_t(\partial \BD)$, we have $x \in \BD\setminus \Lambda$.
	
	Conversely, take $x \in \BD
	 \setminus \Lambda$. Since $\BD \setminus \Lambda$ is open, there exists $\varepsilon\in \BQ_{>0}$ and $y\in \BD_\BQ$ such that $x\in B_{\varepsilon/2}(y)$ and the closed Euclidean ball $\overline{B_\varepsilon(y)}$ is included in $\BD\setminus \Lambda$. By the Hausdorff convergence~\eqref{eq preuve cv Hausdorff balls}, we see that a.s.\@ for all $n$ large enough, the open ball $B_{\varepsilon/2}(y)$ is included in $\BD \setminus \SCB^{\kappa_n,\partial}_{\ka_{\kappa_n}t}(\partial \BD)$. By the Carath\'eodory convergence of the connected component $C_t(y)$, we conclude that $B_{\varepsilon/2}(y) \subseteq  \BD\setminus \SCB^\partial_t(\partial \BD)$, so that $x \in\BD\setminus \SCB^\partial_t(\partial \BD)$. Since $y$ and $\varepsilon$ lie in a countable set, we obtain that a.s.\@ for all $x \in \BD\setminus \Lambda$, we have $x \in\BD\setminus \SCB^\partial_t(\partial \BD)$. Combining the two previous paragraphs, the open sets $\BD \setminus \Lambda$ and $\BD \setminus \SCB^\partial_t(\partial \BD)$ are equal; both $\Lambda$ and $\SCB^\partial_t(\partial \BD)$ contain $\partial \BD$, so they are equal.
\end{proof}

Since the distances to the boundary are given by the uniform exploration, we can already deduce that the pseudometric $D^\partial$ is actually a metric.
\begin{corollary}\label{corollary D partial is a metric}
	Almost surely, $D^\partial$ is a metric on $\widehat{\Gamma}$.
\end{corollary}
\begin{proof}
	Let $x,y \in \BD_\BQ$ be two distinct points. We already know that $D^\partial(\partial \BD, \SCL(x))>0$ a.s.\ since it is an exponential random variable. Besides, note that a.s., on the event that $\SCL(x) \neq \SCL(y)$, the uniform exploration targeted at $x$ will either encircle $y$ with a loop or separate the unexplored components containing $x$ and $y$ before drawing the loop surrounding $x$.
	
	If the uniform exploration targeted at $x$ encircles $y$ with a loop before drawing the loop surrounding $x$, then we have $D^\partial(\partial \BD, \SCL(y))<D^\partial(\partial \BD, \SCL(x))$.
	
	In the other case, let $C_{t_{x,y}}(x)$ and $C_{t_{x,y}}(y)$ be the connected components of the unexplored region at the time $t_{x,y}$ when the uniform exploration separates $x$ from $y$. Recall that, by definition of the uniform exploration, conditionally on $(t_{x,y},C_{t_{x,y}}(x),C_{t_{x,y}}(y))$, the processes $(C_{t_{x,y}+s}(x))_{s\ge 0}$ and $(C_{t_{x,y}+s}(y))_{s\ge 0}$ are independent uniform explorations in the respective domains $C_{t_{x,y}}(x)$ and $C_{t_{x,y}}(y)$ targeting respectively $x$ and $y$. Therefore, one can write $D^\partial(\partial \BD, \SCL(x))= t_{x,y}+ \mathcal{E}$ and $D^\partial(\partial \BD, \SCL(y))= t_{x,y}+ \mathcal{E}'$ where $\mathcal{E}$ and $\mathcal{E}'$ are two conditionally independent exponential random variables given $t_{x,y}$.
	
	In both cases, we deduce that a.s.\ on the event that $\SCL(x) \neq \SCL(y)$, we have $D^\partial(\partial \BD, \SCL(x)) \neq D^\partial(\partial \BD, \SCL(y)) $, so that $D^\partial(\SCL(x), \SCL(y)) \ge \vert D^\partial(\partial \BD, \SCL(x)) - D^\partial(\partial \BD, \SCL(y)) \vert>0$.
\end{proof}

Furthermore, one can define the uniform exploration started from a loop of the $\CLE_4$ as follows. Let $x \in \BD_\BQ$. Consider the annulus $\SA$ made of $\BD$ minus the closure of the region encircled by the loop $\SCL(x)$. Let $Z$ be a point chosen uniformly at random with respect to the Lebesgue measure in $\SA$ (the point $Z$ serves only to fix the rotational degree of freedom of $\phi$). Let $\phi \colon \SA \to \mathbb{A}_r$ be the unique conformal mapping from $\SA$ to $\mathbb{A}_r\defeq (1/r)\BD \setminus r\overline{\BD}$ such that $\phi^\prime(Z)>0$ and such that $\phi(\partial \BD)= (1/r) \partial \BD$ (in the sense that if $z_n \to z \in \partial \BD$, then $\dist(\phi(z_n), (1/r)\partial \BD) \to 0$), where $r>0$ depends on $\SA$. Note that $\phi(\SCL(x))= r \partial \BD$. By \cite[Proposition~3.5]{SimCLEDblConnDom}, we know that conditionally on $\SCL(x)$, the collection of loops $\phi(\Gamma \setminus \{\SCL(x)\})$ has the law of a $\CLE_4$ in the annulus $\mathbb{A}_r$ (we refer to~\cite{SimCLEDblConnDom} for the definition of $\CLE_4$ in a doubly connected domain). But, conditionally on $\SCL(x)$, by \cite[Proposition~3.2]{SimCLEDblConnDom}, the image of $\phi(\Gamma \setminus \{\SCL(x)\})$ under $\iota \colon z \mapsto 1/z$ has the same law as $\phi(\Gamma \setminus \{\SCL(x)\})$. In particular, $(\phi^{-1} \circ \iota \circ \phi)(\widehat{\Gamma})$ has the same law as $\widehat{\Gamma}$ (conditionally on $\SCL(x)$, the image of $\Gamma \setminus \{\SCL(x)\}$ has the same law as $\Gamma \setminus \{\SCL(x)\}$, while the two sets $\partial \BD$ and $\SCL(x)$ are exchanged; it then remains to integrate over $\SCL(x)$), and $(\phi^{-1} \circ \iota \circ \phi)(\SCL(x))$ now plays the role of $\partial \BD$ while $(\phi^{-1} \circ \iota \circ \phi)(\partial \BD)$ is the loop $\SCL(x)$. One can then sample a uniform exploration of $(\phi^{-1} \circ \iota \circ \phi)(\widehat{\Gamma})$ starting from $\partial \BD$. By taking the image under $(\phi^{-1} \circ \iota \circ \phi)$, we thus obtain a uniform exploration of $\widehat{\Gamma}$ starting from $\SCL(x)$. 

\begin{corollary}\label{cor uniform exploration from x}
	Jointly with the convergences~\eqref{eq cv subsequence} and~\eqref{eq cv Hausdorff ball}, along a subsequence which is again denoted by $(\kappa_n)_{n \ge 0}$,
	\begin{equation}\label{eq cv Hausdorff balls}
		\left( \SCB_{\ka_{\kappa_n}t}^{\kappa_n,\partial}( \SCL^{\kappa_n}(x))\right)_{t \in \BQ_{\ge 0}, x \in \BQ^2}
		\mathop{\longrightarrow}\limits_{n\to \infty}^{(\mathrm{d})}
		\left( \SCB_t^\partial( \SCL(x))\right)_{t \in \BQ_{\ge 0}, x \in \BQ^2}
	\end{equation}
	for the product topology in $x \in \BQ^2$ and $t \in \BQ_{\ge 0}$, where the space of compact sets is equipped with the Hausdorff topology, and where $(\SCB^\partial_t(\SCL(x)))_{t\ge 0}$ has the law of the explored region in the uniform exploration of the $\CLE_4$ started from $\SCL(x)$ (where the time parametrization is that of the PPP of SLE bubbles).
\end{corollary}
\begin{proof}
	Let $x \in \BD_\BQ$. Let $\SA_n$ be the annulus made of $\BD$ minus the closure of the region encircled by $\SCL^{\kappa_n}(x)$ (on the event that $\SCL^{\kappa_n}(x)$ does not touch $\partial \BD$, which occurs with probability tending to $1$ as $n \to \infty$). Let $Z_n$ be a point chosen uniformly at random with respect to the Lebesgue measure in $\SA_n$. Let $\phi_n \colon \SA_n \to \mathbb{A}_{r_n}$ be the unique conformal mapping from $\SA_n$ to $\mathbb{A}_{r_n} \defeq (1/r_n) \BD \setminus r_n \overline{\BD}$ such that $\phi_n^\prime(Z_n)>0$ and such that $\phi_n(\partial \BD) = (1/r_n) \partial \BD$. 
	
	Then, note that $((\phi_n^{-1} \circ \iota \circ \phi_n) (\SCB^{\kappa_n, \partial}_t(\SCL^{\kappa_n}(x))))_{t\ge 0}$ has the same law as $(\SCB^{\kappa_n, \partial}_t(\partial \BD))_{t\ge 0}$. This is a consequence of \cite[Theorem~2.17 and Corollary~2.15]{CoInCLERiemSph} (passing from the conditional to the unconditional statement as in the previous subsection). Indeed, by \cite[Theorem~2.17]{CoInCLERiemSph}, we know that conditionally on $\SCL^{\kappa_n}(x)$, the collection of loops $\phi_n(\Gamma^{\kappa_n}\setminus \{\SCL^{\kappa_n}(x)\} )$ is a $\CLE_{\kappa_n}$ in the annulus $\mathbb{A}_{r_n}$ and by \cite[Corollary~2.15]{CoInCLERiemSph} the $\CLE_{\kappa_n}$ in the annulus is invariant in distribution under inversion.

	As a consequence, by~\eqref{eq cv Hausdorff ball},
	\begin{equation*}
	\left((\phi_n^{-1} \circ \iota \circ \phi_n) (\SCB^{\kappa_n, \partial}_{\ka_{\kappa_n}t}(\SCL^{\kappa_n}(x)))\right)_{t \in \BQ_{\ge 0}}
	\mathop{\longrightarrow}\limits_{n\to \infty}^{(\mathrm{d})}
	\left( (\phi^{-1} \circ \iota \circ \phi) (\SCB^\partial_t(\SCL(x))) \right)_{t \in \BQ_{\ge 0}}.
	\end{equation*}
	Using Skorokhod's representation theorem, let us assume that the above convergence and~\eqref{eq cv subsequence}, \eqref{eq cv Hausdorff ball} hold a.s.\ along some subsequence again denoted by $(\kappa_n)_{n \ge 0}$ for an arbitrary finite collection of times $t_1<\ldots < t_d$
	. Since $Z_n$ is uniformly distributed in $\SA_n$ and $Z$ is uniformly distributed in $\SA$, thanks to the Hausdorff convergence of $\SCL^{\kappa_n}(x)$ toward $\SCL(x)$ we may also assume that $Z_n \to Z$ a.s.\ (here we use $\mathrm{Leb}(\SA_n \,\triangle\, \SA) \to 0$, which follows from the Hausdorff convergence of the filled loops
	).
	
	To conclude, it is enough to prove that a.s.\ $r_n \to r$, that $\phi_n \to \phi$ {a.s.}\ uniformly on compact subsets of $\SA$ and that $\phi^{-1}_n$ also converges uniformly on compact subsets of $\mathbb{A}_{r}$ toward $\phi^{-1}$. Indeed, we will have a.s. 
	\begin{equation*}
	\left( \SCB^{\kappa_n, \partial}_{\ka_{\kappa_n}t_i}(\SCL^{\kappa_n}(x)) \right)_{1 \le i \le d} \mathop{\longrightarrow}\limits_{n\to \infty}
	\left( \SCB_{t_i}^\partial(\SCL(x)) \right)_{1 \le i \le d},
	\end{equation*}
	which will end the proof. 
	
	In order to do so, note that, by~\Cref{rem:interior membership}, the Hausdorff convergence of $\SCL^{\kappa_n}(x)$ toward $\SCL(x)$ implies the almost sure Carath\'eodory convergence of the annulus $(\SA_n, Z_n)$ toward $(\SA,Z)$, i.e., a.s., for all compact subsets $K \subseteq \SA$, we have $K\subseteq \SA_n$ for $n$ large enough, and for all connected open sets $V$ containing $Z$ such that $V\subseteq \SA_n$ for infinitely many $n$, we have $V \subseteq \SA$.
	
	Then, by \cite[Theorem~3.7]{Com13} (Theorem 3.2 in the arXiv version) applied in the particular case of doubly connected domains, we deduce that $r_n \to r$, that $\phi_n \to \phi$ uniformly on compact subsets of $\SA$ and that $\phi^{-1}_n$ also converges uniformly on compact subsets of $\mathbb{A}_{r}$ toward $\phi^{-1}$. This ends the proof.
\end{proof}

\begin{remark}\label{remark balls are balls}
	Note that by~\Cref{cor uniform exploration from x}, the time of discovery of a loop $\SCL(y)$ by the uniform exploration from $\SCL(x)$ is the limit of $\ka_{\kappa_n}^{-1}D^{\kappa_n, \partial}(\SCL^{\kappa_n}(x), \SCL^{\kappa_n}(y))$ as $n \to \infty$, so it corresponds to $D^\partial(\SCL(x), \SCL(y))$. In particular, {a.s.}, a loop $\SCL(y)$ belongs to $\SCB^\partial_t(\SCL(x))$ if and only if $D^\partial(\SCL(x), \SCL(y))\le t$. But since any point $y \in \BQ^2$ is surrounded by a loop $\SCL(y)$, using the definition of the uniform exploration from a loop and~\Cref{lemma density of loops}, one can then see that for all $t \in \BQ_{\ge 0}$, $\SCB^\partial_t(\SCL(x))$ is the closure of the union of the domains encircled by the loops which are at $D^\partial$ distance at most $t$ from $\SCL(x)$. 
\end{remark}

\section{When the boundary is not a point}
\label{sec:boundary not a point}

The purpose of this section is to construct another metric $D$, coupled with the objects $D^\partial, \SCB^\partial, K^\partial$ constructed in~\Cref{sec:boundary_is_a_point}, in which the boundary is not identified with a point of the metric space. We follow the same recipe as in~\Cref{sec:boundary_is_a_point} except that the metric balls are not directly described by the uniform exploration. The corresponding distances, metric balls, and geodesics will be denoted as before, but without the $\partial$ superscript.

\subsection{Tightness of the distances, geodesics and metric balls at rational times}
Let $D^{\kappa}$ be the graph distance on $\Gamma^\kappa$, with the same adjacency relation as in~\Cref{subsec:graph_distance} but without the extra vertex $\partial \BD$. Note that $\Gamma^\kappa$ need not be connected as a graph, so $D^\kappa$ could a priori be infinite; it is a.s.\ finite by the stochastic domination established in the proof of~\Cref{prop subsequential limit boundary}.

For all segments $\alpha\subseteq \partial \BD$, for all $\SCL_1^\kappa \in \Gamma^\kappa$, we also set
\begin{equation*}
D^{\kappa}(\alpha, \SCL_1^\kappa)\defeq 1+ \inf_{\substack{\SCL_2^\kappa \in \Gamma^\kappa\\ \SCL_2^\kappa \cap \alpha \neq \emptyset }}D^{\kappa}(\SCL^\kappa_1, \SCL^\kappa_2).
\end{equation*}
More generally, for non-empty $A, B \subseteq \overline{\BD}$, we let $D^\kappa(A,B)$ be the graph distance between $A$ and $B$ in the graph with vertex set $\Gamma^\kappa \cup \{A,B\}$ in which two distinct vertices are adjacent when the corresponding sets intersect (with $\inf\emptyset = \infty$). A \emph{path of loops} from $A$ to $B$ is the collection of loops of $\Gamma^\kappa$ along such a path, identified with its union; it is \emph{shortest} if its length is minimal, and one exists whenever $D^\kappa(A,B) < \infty$, lengths being positive integers.

For all $t\ge 0$, for all $\SCL \in \Gamma^\kappa$ (resp.~for all segments $\alpha$) we denote by $\SCB_t^{\kappa}(\SCL)$ (resp.~$\SCB^{\kappa}_t(\alpha)$) the closure of the union of $\mathrm{int}(\SCL)$ (resp.\@ $\alpha$) with the domains encircled by the loops which are at $D^{\kappa}$ distance at most $t$ from $\SCL$ (resp.~$\alpha$). For all pairs of loops $\SCL_1, \SCL_2 \in \Gamma^\kappa$, let $K^{\kappa}_{\SCL_1, \SCL_2}$ be the union of all the loops which lie on some shortest path between $\SCL_1$ and $\SCL_2$, as in~\Cref{subsec:graph_distance}, so that $K^{\kappa}_{\SCL_1, \SCL_2}=K^{\kappa}_{\SCL_2, \SCL_1}$. We define similarly $K^\kappa_{\alpha, \SCL_1}$ and $K^\kappa_{\alpha, \beta}$ when $\alpha$ and $\beta$ are two segments of $\partial \BD$. We note that all of these sets are automatically closed since a shortest path of loops connecting any pair of loops, a loop to the boundary, or two boundary segments consists of finitely many loops for $\kappa > 4$.

Let $\SCS$ be the set of arcs of $\partial \BD$ with positive length (those appearing in~\Cref{def:weak_axioms,def:weak_axioms2}; we use ``segment'' and ``arc'' interchangeably) and let $\SCS'$ be a countable subset of $\SCS$. We allow $\alpha = \partial \BD$ below, with the convention that $\partial \BD \in \SCS'$.
\begin{proposition}\label{prop subsequential limit boundary}
	There exists a subsequence of $(\kappa_n)_{n\ge 0}$, again denoted by $(\kappa_n)_{n\ge 0}$, such that jointly with~\eqref{eq cv subsequence} and~\eqref{eq cv Hausdorff balls}, we have the joint convergence in law:
	\begin{equation}\label{eq cv subsequence boundary}
		\begin{pmatrix}[1.3]
			(\SCL^{\kappa_n}(x))_{x \in \BD_\BQ} \\
			({\ka_{\kappa_n}^{-1}}D^{\kappa_n}(\SCL^{\kappa_n}(x), \SCL^{\kappa_n}(y)))_{x,y \in \BD_\BQ} \\
			({\ka_{\kappa_n}^{-1}}D^{\kappa_n}(\alpha, \SCL^{\kappa_n}(y)))_{y \in \BD_\BQ, \alpha \in \SCS'} 
			\\
			((\SCB^{\kappa_n }_{\ka_{\kappa_n} t}(\SCL^{\kappa_n}(x)))_{t \in \BQ_{\ge 0}})_{ x\in \BD_\BQ}
			\\
			((\SCB^{\kappa_n }_{\ka_{\kappa_n} t}(\alpha))_{t \in \BQ_{\ge 0}})_{ \alpha \in \SCS'}
			\\
			(K^{\kappa_n}_{\SCL^{\kappa_n}(x), \SCL^{\kappa_n}(y)})_{x, y \in \BD_\BQ}\\
			{(K^{\kappa_n}_{\alpha, \SCL^{\kappa_n}(x)})_{\alpha \in \SCS', x \in \BD_{\BQ}}} \\
			{(K^{\kappa_n}_{\alpha, \beta})_{\alpha, \beta \in \SCS'}} 
		\end{pmatrix}
		\mathop{\longrightarrow}\limits_{n\to \infty}^{(\mathrm{d})}
		\begin{pmatrix}[1.4]
			(\SCL(x))_{x \in \BD_\BQ} \\
			(D (\SCL(x), \SCL(y)))_{x,y \in \BD_\BQ}
			\\
			(D (\alpha, \SCL(y)))_{y \in \BD_\BQ, \alpha \in \SCS'}
			\\
			((\SCB_{t}(\SCL(x)))_{t \in \BQ_{\ge 0}})_{x\in \BD_\BQ}
			\\
			((\SCB_{t}(\alpha))_{t \in \BQ_{\ge 0}})_{\alpha \in \SCS'}
			\\
			(\widetilde{K}_{\SCL(x), \SCL(y)})_{x, y \in \BD_\BQ }\\
			{(\widetilde{K}_{\alpha, \SCL(x)})_{\alpha \in \SCS', x \in \BD_\BQ}}\\
			{(\widetilde{K}_{\alpha, \beta})_{\alpha, \beta \in \SCS'}}
		\end{pmatrix}
	\end{equation}
	where the loops $(\SCL(x))_{x \in \BD_\BQ} $ are the loops of a $\CLE_4$ $\Gamma$ in $\BD$,  $D$ is a random metric on $\Gamma$, and the $\widetilde{K}_{\SCL(x), \SCL(y)}$'s, $\widetilde{K}_{\alpha, \SCL(x)}$'s, $\widetilde{K}_{\alpha, \beta}$'s, $\SCB_{t}(\SCL(x))$ and $\SCB_t(\alpha)$ are random compact subsets of $\overline{\BD}$. The above convergence holds for the product topology and the compact sets converge in the Hausdorff topology.
\end{proposition}
The limits above are denoted by $\widetilde{K}$ rather than $K^\partial$ because they are limits of the geodesics for $D^{\kappa_n}$ on $\Gamma^{\kappa_n}$, whereas those of~\Cref{prop subsequential limit} are limits for $D^{\kappa_n, \partial}$ on $\widehat{\Gamma}^{\kappa_n}$. Note that since the above result holds for an arbitrary countable subset $\SCS' \subset \SCS$, using the Kolmogorov extension theorem, we may define $D(\alpha, \SCL)$ and $\SCB_t(\alpha)$ for all $\alpha \in \SCS$ jointly. In the above result (and several others later in this section), we restrict ourselves to $\SCS'$ to perform a diagonal argument.

\begin{proof}
	\stepx{step:sub-reduction}{Reduction to two tightness statements} It suffices to show the tightness of $\ka_{\kappa_n}^{-1}D^{\kappa_n}(\SCL^{\kappa_n}(x), \SCL^{\kappa_n}(y)) $ for $x,y \in \BD_\BQ$ and the tightness of $
	\ka_{\kappa_n}^{-1}D^{\kappa_n}(\alpha, \SCL^{\kappa_n}(y))$ for $y \in \BD_\BQ$ and $\alpha \in \SCS'$. The remaining coordinates are automatically tight, the loops and compact sets taking values in compact spaces, and~\eqref{eq cv subsequence boundary} follows by extracting a subsequence along which all the coordinates converge jointly. {The fact that $D$ is a.s.\ a metric will follow from the inequality $D(\SCL(x), \SCL(y)) \ge D^\partial (\SCL(x), \SCL(y))$ (which comes from~\eqref{eq cv subsequence} and~\eqref{eq cv subsequence boundary}) and from~\Cref{corollary D partial is a metric}.}
	
	\stepx{step:sub-segment}{Tightness of the distance from a segment} We start with the tightness of $\ka_{\kappa_n}^{-1}D^{\kappa_n}(\alpha, \SCL^{\kappa_n}(y))$. Let $\phi \colon \BD \to \BD$ {be a conformal mapping} such that $\phi(y) = 0$. Then, by conformal invariance, it suffices to prove the tightness of $\ka_{\kappa_n}^{-1}D^{\kappa_n}(\phi(\alpha), \SCL^{\kappa_n}(0))$. By~\eqref{eq touching proba} and by rotational invariance of the $\CLE_\kappa$, we know that 
	\begin{equation*}
	\BP\bigl[ \SCL^{\kappa_n}(0) \cap  \phi(\alpha) \neq \emptyset\bigr] \ge \frac{\len(\phi(\alpha))}{2\pi \ka_{\kappa_n}},
	\end{equation*}
	where $\len(\phi(\alpha))$ is the length of the arc $\phi(\alpha)$. Indeed, if $\SCL^{\kappa_n}(0)$ touches $\partial \BD$ at a point $p$, then $p \in e^{\ri \theta}\phi(\alpha)$ for a set of $\theta \in [0,2\pi)$ of Lebesgue measure $\len(\phi(\alpha))$, so $\int_0^{2\pi} \BP[\SCL^{\kappa_n}(0) \cap e^{\ri \theta}\phi(\alpha) \neq \emptyset] \, \rd \theta \ge \len(\phi(\alpha))/\ka_{\kappa_n}$, the integrand being independent of $\theta$ by rotational invariance. Moreover, on the event that $\SCL^{\kappa_n}(0) \cap \phi(\alpha) = \emptyset$, i.e., $0 \not\in \SCB^{\kappa_n}_1(\phi(\alpha))$, let $\psi$ be the conformal mapping from the connected component $C$ of $\BD \setminus \SCB^{\kappa_n}_1(\phi(\alpha))$ containing $0$ onto $\BD$ sending $0$ to $0$ and such that $\psi^\prime(0)>0$. Let $\gamma$ be the arc $\psi(\partial C \cap \SCB^{\kappa_n}_1(\phi(\alpha)) )$. The set $\gamma$ is indeed an arc since $ \SCB^{\kappa_n}_1(\phi(\alpha))$ is connected, since $\partial C \cap \SCB^{\kappa_n}_1(\phi(\alpha)) $ is closed.
	
	Then, we have the equality of events
	\begin{equation*}
	\bigl\{ 0 \in \SCB^{\kappa_n }_2(\phi(\alpha))\bigr\} = \bigl\{ \SCL^{\kappa_n}(0) \cap \SCB^{\kappa_n}_1(\phi(\alpha)) \neq \emptyset\bigr\}  = \bigl\{\psi(\SCL^{\kappa_n}(0 )) \cap \gamma \neq \emptyset\bigr\},
	\end{equation*}
	and by the domain Markov property, we know that conditionally on $\SCB^{\kappa_n}_1(\phi(\alpha))$, 
	\begin{equation*}
	\BP\bigl[\psi(\SCL^{\kappa_n}(0 )) \cap \gamma \neq \emptyset\bigr]  \ge \frac{\len(\gamma)}{2\pi \ka_{\kappa_n}}.
	\end{equation*} 
	Furthermore, notice that the probability that a planar Brownian motion started from $0$ and stopped upon reaching $\partial \BD$ hits $\SCB^{\kappa_n}_1(\phi(\alpha))$ is larger than the probability that it hits $\phi(\alpha)$ (here we use that $\phi(\alpha) \subseteq \SCB^{\kappa_n}_1(\phi(\alpha))$
	, so that a Brownian motion exiting through $\phi(\alpha)$ must meet $\SCB^{\kappa_n}_1(\phi(\alpha))$; identifying $\len(\gamma)/(2\pi)$ with the probability of hitting it is the conformal invariance of harmonic measure under $\psi$), hence $\len(\gamma) \ge \len (\phi(\alpha))$. By induction, we see that $D^{\kappa_n}(\phi(\alpha), \SCL^{\kappa_n}(0))$ is stochastically dominated by a geometric random variable with success probability $\len(\phi(\alpha))/(2\pi \ka_{\kappa_n})$, hence the desired tightness.

	Next, let us prove the tightness of $\ka_{\kappa_n}^{-1}D^{\kappa_n}(\SCL^{\kappa_n}(x), \SCL^{\kappa_n}(y))$. Let $x,y \in \BD_\BQ$. Let
	\begin{equation*}
	T^{\kappa_n} \defeq \inf\left\{t\ge 0: \SCB_t^{\kappa_n, \partial}(\SCL^{\kappa_n}(x)) \cap \partial \BD \neq \emptyset\right\} =D^{\kappa_n, \partial} (\SCL^{\kappa_n}(x), \partial \BD)-1.
	\end{equation*}
	Then, by definition of $T^{\kappa_n}$, for all $t {\le}T^{\kappa_n}$,
	\begin{equation*}
	\SCB_t^{\kappa_n}(\SCL^{\kappa_n}(x))=  \SCB_t^{\kappa_n, \partial }(\SCL^{\kappa_n}(x)).
	\end{equation*}
	Indeed, a path from $\SCL^{\kappa_n}(x)$ passing through $\partial \BD$ has length at least $T^{\kappa_n}+2$.
	
	\stepx{step:sub-loops}{Tightness of the distance between two loops} To prove the tightness of the rescaled distance $
	\ka_{\kappa_n}^{-1}D^{\kappa_n}(\SCL^{\kappa_n}(x), \SCL^{\kappa_n}(y))$, it is thus sufficient to work on the event that $D^{\kappa_n}(\SCL^{\kappa_n}(x), \SCL^{\kappa_n}(y))>T^{\kappa_n}$.
	
	Let $C_n(y)$ be the connected component of $\BD \setminus \SCB_{T^{\kappa_n}}^{\kappa_n}(\SCL^{\kappa_n}(x))$ containing $y$. If $\partial C_n(y) \cap \partial \BD = \emptyset$, then we have $D^{\kappa_n}(\SCL^{\kappa_n}(x), \SCL^{\kappa_n}(y))= D^{\kappa_n, \partial}(\SCL^{\kappa_n}(x), \SCL^{\kappa_n}(y))$ (any path using $\partial \BD$ must re-enter $C_n(y)$ through a loop of $\SCB_{T^{\kappa_n}}^{\kappa_n}(\SCL^{\kappa_n}(x))$, and rerouting it through that loop does not increase its length) so that we get directly the tightness.
	
	If $\partial C_n(y) \cap \partial\BD \neq \emptyset$, let $\phi_n \colon \BD \to C_n(y) $ be the unique conformal mapping such that $\phi_n(0)= y$, $\phi_n^\prime(0)>0$. Then the restriction of $\Gamma^{\kappa_n}$ to $C_n(y)$ is a $\CLE_{\kappa_n}$ (the domain Markov property applies here because $\SCB_{T^{\kappa_n}}^{\kappa_n}(\SCL^{\kappa_n}(x))$ is a stopping set for the exploration of $\Gamma^{\kappa_n}$ by graph distance from $\SCL^{\kappa_n}(x)$) and $\phi_n^{-1}(\partial C_n(y) \cap  \partial \SCB_{T^{\kappa_n}}^{\kappa_n }(\SCL^{\kappa_n}(x)))$ is a segment $\alpha_n$. Indeed, $\SCB_{T^{\kappa_n}}^{\kappa_n }(\SCL^{\kappa_n}(x))$ is connected and $\partial C_n(y) \cap  \partial \SCB_{T^{\kappa_n}}^{\kappa_n }(\SCL^{\kappa_n}(x))$ is closed.
	
	Let us show that the length $\len(\alpha_n)$ of $\alpha_n$ is such that $\lim_{\delta \to 0} \liminf_{n\to \infty}\BP\lbrack \len(\alpha_n) >\delta\rbrack=1$. By Skorokhod's representation theorem, let us assume that~\eqref{eq cv subsequence},~\eqref{eq cv Hausdorff balls} hold a.s. The desired result then comes from the fact that $\SCL^{\kappa_n}(x) \subseteq \SCB_{T^{\kappa_n}}^{\kappa_n}(\SCL^{\kappa_n}(x))   $ and $\SCL^{\kappa_n}(x) $ converges in the Hausdorff topology toward $\SCL(x)$, so that the diameter of $\SCL^{\kappa_n}(x)$ is bounded from below, and thus the probability that a planar Brownian motion started from $y$ and stopped when it exits $\BD$ hits $\SCB_{T^{\kappa_n}}^{\kappa_n}(\SCL^{\kappa_n}(x)) $ is a.s.\ bounded from below. Therefore, $\lim_{\delta \to 0} \liminf_{n\to \infty}\BP\lbrack \len(\alpha_n) >\delta\rbrack=1$.
	
	Thus, on $\{\len(\alpha_n) \ge \delta\}$, the conditional law of $D^{\kappa_n}(\SCL^{\kappa_n}(x), \SCL^{\kappa_n}(y)) - T^{\kappa_n}$ is stochastically dominated by that of $1+D^{\kappa_n}(\alpha_\delta, \SCL^{\kappa_n}(0))$ for a deterministic arc $\alpha_\delta$ of length $\delta$. Combining with the tightness of the distance to a segment and with $\lim_{\delta \to 0}\liminf_{n\to \infty}\BP\lbrack \len(\alpha_n) >\delta\rbrack=1$, and letting $\delta \downarrow 0$, $\ka_{\kappa_n}^{-1}D^{\kappa_n}(\SCL^{\kappa_n}(x), \SCL^{\kappa_n}(y))$ is tight.
\end{proof}

\begin{remark}\label{remark removing the boundary increases the distance}
	Note that for all $\SCL_1, \SCL_2 \in \Gamma$, for all $t \in \BQ_{\ge 0}$, for all $\alpha \in \SCS'$, we have $D^\partial(\SCL_1, \SCL_2) \le D(\SCL_1, \SCL_2)$, $D^\partial(\partial \BD, \SCL_2) \le D(\alpha, \SCL_2)$, $\SCB_t(\SCL_1) \subseteq \SCB^\partial_t(\SCL_1)$, and $\SCB_t (\alpha) \subseteq \SCB_t^\partial(\partial \BD)$. Moreover, if $\alpha \subseteq \beta \in \SCS'$, then $\SCB_t(\alpha) \subseteq \SCB_t(\beta)$. These inequalities come from~\eqref{eq cv subsequence},~\eqref{eq cv Hausdorff balls} and~\eqref{eq cv subsequence boundary} and from the analogous inequalities for the graph distance. One can also see that $\SCB^\partial_t (\partial \BD)= \SCB_t(\partial \BD)$ (indeed, a path of loops from $\partial \BD$ realizing either of the two distances never has an interior vertex equal to $\partial \BD$, so that the two graph distances from $\partial \BD$ agree) and that $K^\partial_{\partial \BD, \SCL_1} =\widetilde{K}_{\partial \BD, \SCL_1}$.
\end{remark}

The remainder of the present section is devoted to the study of metric balls: we prove that loops are dense in balls and we extend the definition of metric balls to irrational times.

\subsection{Study of metric balls at rational times} Let us state the following technical lemma which gives a stronger relation between $\SCB_t(\alpha)$, $\SCB_t(\SCL)$, and the uniform exploration. For all $\kappa>4$, for all $s\ge 0$, let $\SCB^{\kappa, \partial}_s(\SCB_t^{\kappa} (\alpha))$ be the union of $\SCB_t^{\kappa} (\alpha)$ with the closure of the union of the interiors of the loops of $\widehat{\Gamma}^\kappa$ at $D^{\kappa, \partial}$ distance at most $s$ from $\SCB_t^{\kappa} (\alpha)$ (the loops of $\widehat{\Gamma}^\kappa$ which touch $\SCB_t^{\kappa}(\alpha)$ but which are not included in $\SCB_t^{\kappa} (\alpha)$ are at distance $1$ from $\SCB_t^{\kappa} (\alpha)$, etc)
. One can write the same definition for $\SCB^{\kappa, \partial}_s(\SCB_t^{\kappa} (\SCL))$ for $\SCL \in \Gamma^\kappa$. Similarly, we also let $K^{\kappa, \partial}_{\SCB_t^{\kappa} (\alpha), \SCL^\prime}$ and $K^{\kappa, \partial}_{\SCB_t^{\kappa} (\SCL), \SCL^\prime}$ be the unions of all the loops lying on some $D^{\kappa, \partial}$ shortest path from $\SCB_t^{\kappa} (\alpha) \cup \partial \BD$ and $\SCB_t^{\kappa} (\SCL)$ respectively to $\SCL^\prime$ (without including $\SCB_t^{\kappa} (\alpha) \cup \partial \BD$ or $\SCB_t^{\kappa} (\SCL)$ in the path), as in~\Cref{subsec:graph_distance}.
\begin{lemma}\label{lem:good_convergence}
	Jointly with~\eqref{eq cv subsequence},~\eqref{eq cv Hausdorff balls} and~\eqref{eq cv subsequence boundary}, for all $t \in \BQ_{\ge 0}$, $x \in \BD_\BQ$, $\alpha \in \SCS'$, along a subsequence again denoted by $(\kappa_n)_{n \ge 0}$, we have
	\begin{equation*}
		\left(
		\SCB^{\kappa_n, \partial}_{\ka_{\kappa_n} s}(\SCB_{\ka_{\kappa_n}t}^{\kappa_n} (\SCL^{\kappa_n}(x))), \SCB^{\kappa_n, \partial}_{\ka_{\kappa_n} s}(\SCB_{\ka_{\kappa_n}t}^{\kappa_n} (\alpha))
		\right)_{s \in \BQ_{\ge 0}}
		\mathop{\longrightarrow}\limits_{n\to \infty}^{(\mathrm{d})}
		\left( 
		\SCB^\partial_{ s}(\SCB_{t} (\SCL(x))), \SCB_{ s}^\partial(\SCB_{t} (\alpha))
		\right)_{s \in \BQ_{\ge 0}},
	\end{equation*}
	for the product topology where each component is equipped with the Hausdorff topology and $(\SCB^\partial_{ s}(\SCB_{t} (\alpha)))_{s\ge 0}$ (resp.~$(\SCB_{ s}^\partial(\SCB_{t} (\SCL(x))))_{s\ge 0}$) has the same law as a collection of independent uniform explorations of the $\CLE_4$, started from the whole boundary of each connected component of $\BD \setminus \SCB_{t}(\alpha)$ (resp.~$\BD \setminus\SCB_{t}(\SCL(x))$). Moreover, jointly with~\eqref{eq cv subsequence},~\eqref{eq cv Hausdorff balls},~\eqref{eq cv subsequence boundary}, and the above convergence, for all $t \in \BQ_{\ge 0}$, $x,y \in \BD_\BQ$, $\alpha \in \SCS'$,
	\begin{equation}
		\label{eq cv geodesics from balls}
		\left(K^{\kappa_n, \partial}_{\SCB_{\ka_{\kappa_n} t}^{\kappa_n} (\alpha), \SCL^{\kappa_n}(y)}, K^{\kappa_n, \partial}_{\SCB_{\ka_{\kappa_n} t}^{\kappa_n} (\SCL^{\kappa_n}(x)), \SCL^{\kappa_n}(y)} \right)
		\mathop{\longrightarrow}\limits_{n\to \infty}^{(\mathrm{d})}
		\left( K^\partial_{\SCB_t(\alpha), \SCL(y)}, K^\partial_{\SCB_t(\SCL(x)), \SCL(y)}
		\right),
	\end{equation}
	for the Hausdorff topology, where $K^\partial_{\SCB_t(\alpha), \SCL(y)}, K^\partial_{\SCB_t(\SCL(x)), \SCL(y)}$ are random compact subsets of $\overline{ \BD}$.
\end{lemma}
\begin{proof}
	This is a direct consequence of the domain Markov property for the $\CLE_\kappa$, of~\eqref{eq cv Hausdorff balls}, and of~\eqref{eq cv subsequence boundary}.
\end{proof}

Note that for all $t \in \BQ_{\ge 0}$, for all $s \in \BQ_{\ge 0}$, for all $\alpha \in \SCS'$, for all $\SCL \in \Gamma$, we have $\SCB_{t+s}(\alpha) \subseteq \SCB^\partial_s(\SCB_t (\alpha)) $ and $\SCB_{t+s}(\SCL) \subseteq \SCB_s^\partial(\SCB_t (\SCL))$. Both follow by passing to the limit in the analogous inclusions for the graph distances (a shortest path realizing $t+s$ from $\alpha$ meets $\SCB^{\kappa_n}_{t}(\alpha)$), using~\eqref{eq cv subsequence} and~\eqref{eq cv Hausdorff balls}.

For all $t \in \BQ_{\ge 0}$, for all $s \in \BQ_{\ge 0}$, let $\widetilde{\SCB}^\partial_s(\SCB_t (\alpha))$ be the closure of the union of $\SCB_t (\alpha)$ with all the connected components $C$ of $ \SCB^\partial_s(\SCB_t (\alpha))\setminus \SCB_t(\alpha)$ such that $C \cap  \overline{\partial \BD\setminus \alpha} = \emptyset$ (see~\Cref{fig:ball_from_a_segment} for a sketch). Similarly, for all $t \in \BQ_{\ge 0}$, for all $s \in \BQ_{\ge 0}$, let $\widetilde{\SCB}_s^\partial (\SCB_t(\SCL))$ be the closure of the union of $\SCB_t(\SCL)$ with all the connected components $C$ of $\SCB_s^\partial(\SCB_t(\SCL))\setminus \SCB_t(\SCL)$ such that $C \cap \overline{\partial \BD \setminus \SCB_t(\SCL)} = \emptyset$. Note that when $t+s < D(\SCL, \partial \BD)$, we have $\widetilde{\SCB}_s^\partial (\SCB_t(\SCL)) = \SCB_{t+s}(\SCL) = \SCB^\partial_{t+s}(\SCL)$. Indeed, none of the balls involved has then reached $\partial \BD$, so that $D$ and $D^\partial$ agree on the loops they contain, and no connected component of $\SCB_s^\partial(\SCB_t(\SCL)) \setminus \SCB_t(\SCL)$ meets $\partial \BD$.

Let us state a technical lemma.
\begin{lemma}\label{lemma one point}
	Almost surely, for all $x \in \BD_\BQ$ the intersection $K^\partial_{\partial \BD, \SCL(x)} \cap \partial \BD $ (where here we use the definition of the geodesic without the endpoints) contains exactly one point. As a result, a.s., for all $t \in \BQ_{\ge 0}$, for all $\alpha \in \SCS'$, for all $y \in \BD_\BQ \setminus  \SCB_t(\alpha)$, the intersection $K^\partial_{\SCB_t(\alpha), \SCL(y)} \cap ( \SCB_t(\alpha)\cup \partial \BD)$ contains exactly one point. The same is true for the intersection $K^\partial_{\SCB_t(\SCL(x)), \SCL(y)} \cap ( \SCB_t(\SCL(x))\cup \partial \BD)$ for all $x \in \BD_\BQ$ on the event that $D(\SCL(x), \partial \BD)<t$.
\end{lemma}
\begin{proof}
	This is a consequence of the fact that $\SCB^\partial_t(\SCL(x))$ is the uniform exploration starting from $\SCL(x)$ defined above~\Cref{cor uniform exploration from x}, so that the boundary $\partial \BD$ corresponds to a loop and each loop is attached to the explored region at exactly one point according to the PPP of $\SLE_4$ bubbles.
	
	More precisely, recall that the annulus $\SA$ is made of $\BD$ minus the closure of the region encircled by the loop $\SCL(x)$, that $Z$ is a point chosen uniformly at random with respect to the Lebesgue measure in $\SA$. Recall that $\phi \colon \SA \to \mathbb{A}_r$ is the unique conformal mapping from $\SA$ to $\mathbb{A}_r\defeq (1/r)\BD \setminus r\overline{\BD}$ such that $\phi^\prime(Z)>0$ and such that $\phi(\partial \BD)= (1/r) \partial \BD$, where $r>0$ depends on $\SA$. Recall that $\left( (\phi^{-1} \circ \iota \circ \phi) (\SCB^\partial_t(\SCL(x))) \right)_{t\ge 0}$ has the law of $(\SCB_t^\partial(\partial \BD))_{t\ge 0}$ and, in this coupling, $(\phi^{-1} \circ \iota \circ \phi)(\partial \BD)$ corresponds to $\SCL(x)$. We then conclude that $K^\partial_{\partial \BD, \SCL(x)} \cap \partial \BD $ contains exactly one point using the facts that each loop of the uniform exploration is attached at exactly one point on the boundary of the unexplored region when it is discovered and that $(\phi^{-1} \circ \iota \circ \phi)$ is a homeomorphism (a conformal map between the annuli, which extends to a homeomorphism of their closures since $\SCL(x)$ is a Jordan curve).
	
	Next, for the case of a segment $\alpha \in \SCS'$, we use the fact that conditionally on $\SCB^{\kappa_n}_{\ka_{\kappa_n}t}(\alpha)$, the restriction of $\Gamma^{\kappa_n}$ to the connected component of $\BD \setminus \SCB^{\kappa_n}_{\ka_{\kappa_n}t}(\alpha)$ containing $y$ has the law of a $\CLE_{\kappa_n}$. Therefore, we can apply the previous case by letting $n\to \infty$. For the last point of the lemma, the same reasoning works.
\end{proof}

 \begin{figure}[h]
	\centering
	\includegraphics[scale=0.7]{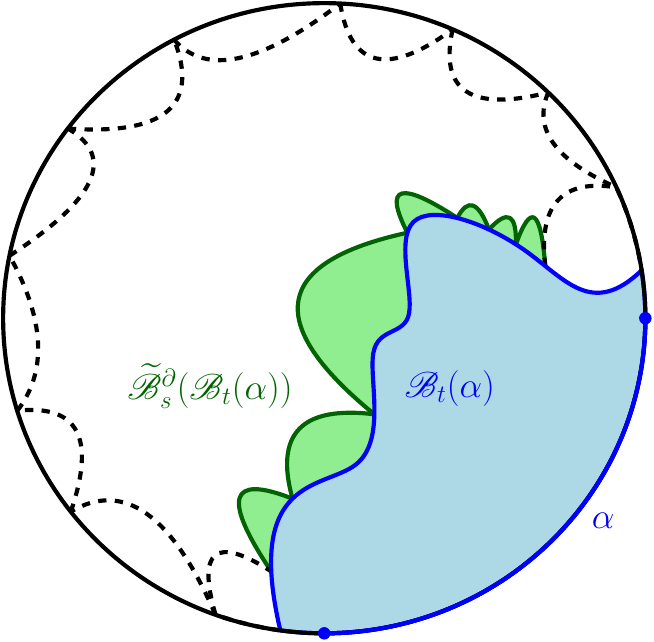}
	\includegraphics[scale=0.7]{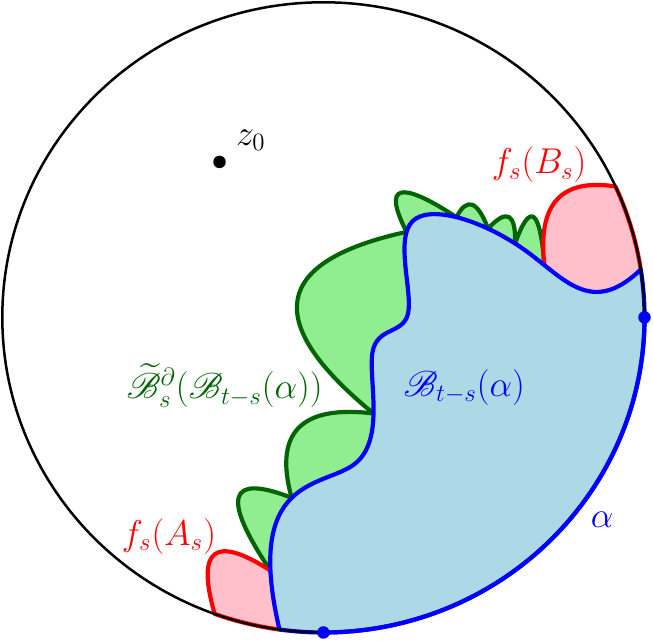}
	\caption{\textbf{Left:} Illustration of the definition of $\widetilde{\SCB}^\partial_s(\SCB_{t}(\alpha))$ which can be informally described as the union of the blue and green regions. Note that none of these components is simply connected; they are drawn as simply connected since we focus on the intersection with $\overline{\partial \BD\setminus \alpha}$. The connected components of $\SCB^\partial_s(\SCB_{t}(\alpha)) \setminus \SCB_{t}(\alpha)$ which are not kept are sketched using dotted lines.  \textbf{Right:} Illustration for the proof of~\Cref{lemma approx balls from segments with unif explo}. Informally, the ball of radius $t-s$ around $\alpha$ is in blue, the connected components of $\SCB^\partial_s(\SCB_{t-s}(\alpha)) \setminus \SCB_{t-s}(\alpha)$ which are kept in $\widetilde{\SCB}^\partial_s(\SCB_{t-s}(\alpha)) \setminus \SCB_{t-s}(\alpha)$ are in green, and the two connected components $f_s(A_s)$ and $f_s(B_s)$ whose diameters go to zero as $s\to 0$ are drawn in pink. More precisely, the blue region is $\BD \setminus C_s$, the green regions are the connected components of $C_s \setminus \overline{{C}'_s}$ whose closures do not touch $\overline{\partial \BD\setminus \alpha}$ and the pink regions are the connected components of $C_s \setminus \overline{{C}'_s}$ whose closures touch both $\overline{\partial \BD\setminus \alpha}$ and $\SCB_{t-s}(\alpha)$.}
\label{fig:ball_from_a_segment}
\end{figure}
\begin{lemma}\label{lemma approx balls from segments with unif explo}
	We have for all $t \in \BQ_{\ge 0}$ and all $\alpha \in \SCS'$, a.s.,
	\begin{equation*}
	\overline{\bigcup_{s\in(0,t) \cap \BQ} \widetilde{\SCB}^\partial_s(\SCB_{t-s} (\alpha)) } = \SCB_t(\alpha),
	\end{equation*}
	where the union is non-decreasing as $s\downarrow 0$. Moreover,
	\begin{equation*}
	\overline{\bigcup_{s\in(0,t) \cap \BQ} \left( \widetilde{\SCB}^\partial_s(\SCB_{t-s} (\alpha)) \setminus  \SCB_{t-s} (\alpha) \right)}  \supset \partial \SCB_t (\alpha).
	\end{equation*}
	{The same is true if we replace the segment $\alpha$ by a loop $\SCL \in \Gamma$.}
\end{lemma}
\begin{proof}
	\stepx{step:approx-first}{The first inclusion} Let us first show that 
	\begin{equation}\label{eq first inclusion approx balls from segments}
		\forall s \in (0,t) \cap \BQ, \quad \widetilde{\SCB}^\partial_s(\SCB_{t-s} (\alpha))  \subseteq \SCB_t(\alpha).
	\end{equation}
	Let $s \in (0,t) \cap \BQ$. Take a loop $\SCL(x) \subseteq C$ with $x \in C \cap \BQ^2$, where $C$ is a connected component of $ \SCB^\partial_s(\SCB_{t-s} (\alpha))\setminus \SCB_{t-s}(\alpha)$ such that $C \cap  \overline{\partial \BD\setminus \alpha} = \emptyset$. 
	
	Then, the compact set $K^\partial_{\SCB_{t-s}(\alpha), \SCL(x)}$ is connected (as a Hausdorff limit of connected sets). It is a subset of $ \overline{\widetilde{\SCB}^\partial_s(\SCB_{t-s} (\alpha)) \setminus \SCB_{t-s} (\alpha) }$. It intersects $\partial \BD \cup \SCB_{t-s} (\alpha)$ and $C$, so it intersects $\partial \SCB_{t-s}(\alpha) $. By~\Cref{lemma one point}, we deduce that the only point of $K^\partial_{\SCB_{t-s}(\alpha), \SCL(x)} \cap (\partial \BD \cup \SCB_{t-s} (\alpha))$ is in $\partial \SCB_{t-s}(\alpha)$ (and is not in $\overline{\partial\BD \setminus \alpha}$). As a result, by~\eqref{eq cv geodesics from balls} we deduce that for $n$ large enough, $K^{\kappa_n, \partial}_{\SCB_{\ka_{\kappa_n} (t-s)}^{\kappa_n} (\alpha), \SCL^{\kappa_n}(x)}$ does not intersect $\overline{\partial\BD \setminus \alpha}$.  

Moreover,  since $\SCB_{\ka_{\kappa_n} s}^{\kappa_n,\partial}(\SCB_{\ka_{\kappa_n} (t-s)}^{\kappa_n}(\alpha))$ converges to $\SCB_s^{\partial}(\SCB_{t-s}(\alpha))$ as $n \to \infty$ in the Hausdorff sense by Lemma~\ref{lem:good_convergence},  we obtain that  $\SCL^{\kappa_n}(x) \subseteq \SCB_{\ka_{\kappa_n} s}^{\kappa_n,\partial}(\SCB_{\ka_{\kappa_n}(t-s)}^{\kappa_n}(\alpha))$ for all $n$ large enough.  In particular,  we have that the $D^{\kappa_n,\partial}$ distance between $\SCL^{\kappa_n}(x)$ and $\alpha$ is at most $\ka_{\kappa_n} t$ for all $n$ large enough.  Therefore,  since $K_{\SCB_{\ka_{\kappa_n}(t-s)}^{\kappa_n,\partial}(\alpha),\SCL^{\kappa_n}(x)}^{\kappa_n,\partial}$ does not intersect $\overline{\partial \BD \setminus \alpha}$ for $n$ large enough,  it follows that the distances between $\SCL^{\kappa_n}(x)$ and $\SCB_{\ka_{\kappa_n}(t-s)}^{\kappa_n}(\alpha)$ with respect to the metrics $D^{\kappa_n,\partial}$ and $D^{\kappa_n}$ are the same for such $n$.  Combining,  we obtain that $\SCL^{\kappa_n}(x)$ is at distance at most $\ka_{\kappa_n} t$ from $\alpha$ with respect to $D^{\kappa_n}$ for all $n$ large enough. Thus, $\SCL(x) \subseteq  \SCB_t(\alpha)$, hence~\eqref{eq first inclusion approx balls from segments}.
	
\stepx{step:approx-second}{The reverse inclusion} Next, let us show the other inclusion. Let $\SCB^{ -}_t(\alpha)$ be the closure of the increasing union of the $\SCB_{t-s}(\alpha)$ as $s \downarrow 0$. Let $z_0 \in  \BD_\BQ \setminus \SCB^{ -}_t(\alpha)$ and let $C$ be the connected component of $\BD \setminus \SCB^{ -}_t(\alpha)$ containing $z_0$. Similarly, let $C_{s}$ (resp.\ ${C}'_s$) be the connected component of $\BD \setminus \SCB_{t-s}(\alpha)$ (resp.\ $\BD \setminus \widetilde{\SCB}^\partial_{s}(\SCB_{t-s}(\alpha))$) containing $z_0$.
	
	\stepx{step:approx-small}{Smallness of the added components} To conclude the proof of the two points of the lemma, it is enough to show that any point in $\SCB_t(\alpha) \cap \overline{C}$ is in 
	\begin{equation*}
	\overline{\bigcup_{s\in(0,t) \cap \BQ} \left( \widetilde{\SCB}^\partial_s(\SCB_{t-s} (\alpha)) \setminus  \SCB_{t-s} (\alpha) \right)}.
	\end{equation*}
	To see this, it is enough to see that the diameters of the (at most two) connected components of 
	\begin{equation*}
	C_{s} \setminus \overline{{C}'_s}
	\end{equation*}
	whose closures touch both $\overline{\partial \BD \setminus \alpha}$ and $\partial \SCB_{t-s}(\alpha)$ go to zero. Let $f_{s} \colon \BD \to C_{s}$ be the unique conformal mapping such that $f_{s}(0)=  z_0$ and $f'_{s}(0)>0$. Let us write
	\begin{equation*}
	\beta_s \defeq f_s^{-1}(\overline{\partial \BD \setminus \alpha} \cap \overline{C_s}), 	
	\end{equation*}
	and let $a_s, b_s \in \partial \BD$ be the two endpoints of the segment $\beta_s$.
	
	Then, the two connected components of $C_{s} \setminus \overline{C'_s}$ whose closures touch both $\overline{\partial \BD \setminus \alpha}$ and $\partial \SCB_{t-s}(\alpha)$ are the images under $f_s$ of the two connected components $A_s$ and $B_s$ of $\BD \setminus \overline{\check{C}_s(0)}
	$ such that $a_s \in \overline{A_s} $ and $b_s \in \overline{B_s}$, where $\check{C}_s(0)
	$ is the branch targeting the origin of a uniform exploration starting from $\partial \BD$. By \Cref{prop BCLE rho to zero}, we know that the diameters of $A_s$ and $B_s$ go to zero as $s \downarrow 0$. So, the probability that a planar Brownian motion starting at zero and killed in $\partial \BD$ hits $A_s$ or $B_s$ goes to zero as $s\downarrow 0$. Using the fact that $C$ is the interior of the decreasing intersection of the $C_s$'s as $s\downarrow 0$ we see that the diameters of $f_{s}(A_s)$ and $f_{s}(B_s)$ go to zero as $s\downarrow 0$ in probability (here we use that $f_s(A_s)$ and $f_s(B_s)$ are connected sets whose closures meet $\partial \BD$, so that by a Beurling-type estimate their diameters are controlled by their harmonic measure seen from $z_0$). Indeed, since $C\subseteq C_s$ the (conditional) probability that a planar Brownian motion started at $z_0$ and stopped when it exits $C$ hits $f_{s}(A_s)$ or $f_{s}(B_s)$ is smaller than or equal to the probability that the Brownian motion stopped when it exits $C_s$ hits $f_{s}(A_s)$ or $f_{s}(B_s)$, and the latter one goes to zero as $s\downarrow 0$ by conformal invariance of the Brownian motion. See~\Cref{fig:ball_from_a_segment}. 
	
	\stepx{step:approx-loop}{The case of a loop} In the case where $\alpha$ is replaced by a loop $\SCL \in \Gamma$, we first use the fact that $\SCB_t(\SCL) = \SCB^\partial_t(\SCL)$ for all $t {<}D(\SCL, \partial \BD)$. (The case $t = D(\SCL, \partial \BD)$ needs no treatment: $t$ being a fixed rational number and $D(\SCL, \partial \BD)$ having a continuous law, it a.s.\ does not occur.) {For $t > D(\SCL, \partial \BD)$, the ball $\SCB_t(\SCL)$ has already hit $\partial \BD$, so we have to do the same reasoning as above. More precisely, the exact same reasoning as above shows the analog of~\eqref{eq first inclusion approx balls from segments}, i.e., that
	\[
	\forall s \in (0,t)\cap \BQ, \quad \widetilde{\SCB}^\partial_s(\SCB_{t-s} (\SCL))  \subseteq \SCB_t(\SCL).
	\]
	Then, to show the other inclusion, we define $\SCB^-_t(\SCL)$ as the closure of the increasing union of the $\SCB_{t-s}(\SCL)$ as $s\downarrow0$. We take as before $z_0 \in \BD_\BQ \setminus \SCB_t^-(\SCL)$ and we look at the connected component $C$ of $\BD \setminus \SCB_t^-(\SCL)$ containing $z_0$. Similarly, let $C_{s}$ (resp.\ $C'_s$) be the connected component of $\BD \setminus \SCB_{t-s}(\SCL)$ (resp.\ $\BD \setminus \widetilde{\SCB}^\partial_s(\SCB_{t-s}(\SCL))$) containing $z_0$. It is enough to show that any point in $\SCB_t(\SCL) \cap \overline{C}$ is in
	\begin{equation*}
		\overline{\bigcup_{s\in(0,t) \cap \BQ} \left( \widetilde{\SCB}^\partial_s(\SCB_{t-s} (\SCL)) \setminus  \SCB_{t-s} (\SCL) \right)}.
	\end{equation*}
	Because $t>D(\SCL, \partial \BD)$, we know that for $s<t-D(\SCL, \partial \BD)$, the set $C_s$ is simply connected, and that its boundary is either included in $\partial \SCB_{t-s}(\SCL)$ or is the union of two arcs, one of them being included in $\partial \SCB_{t-s}(\SCL)$ and the other in $\partial \BD$. Thus, as before, it is enough to see that the diameters of the (at most two) connected components of 
	\begin{equation*}
		C_{s} \setminus \overline{C'_s}
	\end{equation*}
	whose closures touch both $\overline{\partial \BD \setminus \SCB_{t-s}(\SCL)}$ and $\partial \SCB_{t-s}(\SCL)$ go to zero. The exact same reasoning as in the case of a segment $\alpha$ yields this result.} This ends the proof of the lemma.
\end{proof}

\subsection{Definition of metric balls for all times: Skorokhod convergence} \label{subsec:skorokhod}
\begin{lemma}\label{lemma Skorokhod tightness}
	For all $\alpha \in \SCS'$, the sequence of the processes $(\SCB^{\kappa_n}_{\ka_{\kappa_n} t}(\alpha))_{t\ge 0}$ for $n\ge 1$ is tight in the space of c\`adl\`ag functions taking their values in the space of compact subsets of $\overline{\BD}$ with the Hausdorff topology, this space of c\`adl\`ag functions being itself equipped with the J$_1$ Skorokhod topology. The same is true for $(\SCB^{\kappa_n}_{\ka_{\kappa_n} t}(\SCL^{\kappa_n}(x)))_{t\ge 0}$ for all $x \in \BD_\BQ$.
\end{lemma}
\begin{proof}
	\stepx{step:skoro-aldous}{Aldous' tightness criterion} Let $X^n_t \defeq  \SCB^{\kappa_n}_{\ka_{\kappa_n} t}(\alpha)$. In order to prove tightness, we use Aldous' tightness criterion. See \cite[Theorem~23.11]{Kal21}. Note that the space of compact subsets of $\overline{\BD}$ equipped with the Hausdorff distance $d_\rH$ is compact. So we only need to check that for all sequences of stopping times $\tau_n$ (for the natural filtration of $(X_t^n)_{t \geq 0}$), for all sequences $\delta_n \downarrow 0$, we have
	\begin{equation}\label{eq tightness Aldous}
	d_\rH(X^n_{\tau_n+ \delta_n}, X^n_{\tau_n}) \mathop{\longrightarrow}\limits_{n\to \infty}^{(\BP)} 0.
	\end{equation}

	\stepx{step:skoro-components}{Convergence of the components at the stopping times} Since the space of compact subsets of $\overline{\BD}$ equipped with the Hausdorff topology $d_\rH$ is compact, we may take a subsequence so that $X^n_{\tau_n}$ converges in distribution to some random compact set $K$; using Skorokhod's representation theorem, we may and do assume that this convergence holds a.s. Let $\delta\in \BQ_{>0}$.
	
	For all $z \in \BD_\BQ$, let $C_n(z)$ be the connected component of $\BD \setminus X^n_{\tau_n}$ containing $z$ and $C(z)$ be the connected component of $\BD \setminus K$ containing $z$ (when it exists). Let $\phi_n^z \colon \BD \to C_n(z)$ (resp.~$\phi^z \colon \BD \to C(z)$) be the unique conformal mapping such that $\phi^z_n(0)=z$ and $(\phi^z_n)^\prime(0)>0$ (resp.~$\phi^z(0)=z$ and $(\phi^z)^\prime(0)>0$). By the convergence of $X^n_{\tau_n}$ to $K$, we get the convergence jointly in $z \in \BD_\BQ$
	\begin{equation}\label{eq cv mapping out functions at a stopping time}
		\phi^z_n \mathop{\longrightarrow}\limits_{n\to \infty}^{(\mathrm{d})} \phi^z,
	\end{equation}
	for the topology of uniform convergence on compact subsets of $\BD$.
	
	Note that $X^n_{\tau_n+ \delta}\subseteq \SCB^{\kappa_n, \partial}_{\ka_{\kappa_n} \delta +1} (X^n_{\tau_n})$, since a shortest path from $\alpha$ to a loop of $X^n_{\tau_n+\delta}$ passes through a loop of $X^n_{\tau_n}$, its remaining portion having length at most $\ka_{\kappa_n}\delta+1$, where the $+1$ accounts for the rounding of the radii to integers. Moreover, by the domain Markov property of the $\CLE_\kappa$ and because $\tau_n$ is a stopping time for the filtration generated by $(X^n_t)_{t\ge 0}$, we see that for all $z \in \BD_\BQ$, conditionally on $X^n_{\tau_n}$, the closure in $\overline{\BD}$ of the set $(\phi^z_n)^{-1}(C_n(z) \cap \SCB^{\kappa_n, \partial}_{\ka_{\kappa_n} \delta +1} (X^n_{\tau_n}))$ has the same law as $\SCB^{\kappa_n, \partial}_{\ka_{\kappa_n} \delta +1}(\partial \BD)$ (the source of the growth is all of $\partial \BD$, since $X^n_{\tau_n}$ touches $\partial \BD$ and $D^{\kappa_n,\partial}$ treats $\partial \BD$ as a single vertex), which converges in distribution to $\SCB_\delta^\partial(\partial \BD)$ as $n\to \infty$ in the Hausdorff topology. 
	
	\stepx{step:skoro-rightcont}{Right-continuity of the exploration at time zero} Moreover, {a.s.}, in the Hausdorff topology,
	\begin{equation}\label{eq continuite a droite explo unif}
		\partial \BD= \SCB^\partial_0(\partial \BD)= \lim_{t\downarrow 0} \SCB^\partial_t(\partial \BD).
	\end{equation}
	Indeed, first we note that $(\SCB^{\partial}_t(\partial \BD))_{t \geq 0}$ has the law of the collection of the explored regions in the uniform $\CLE_4$ exploration. Then, $\SCB^{\partial}_t(\partial \BD)$ is contained in the closure of the union of the connected components of $\BD \setminus \overline{C_t(0)}$, each of which touches $\partial \BD$, so every point of it lies within the maximal diameter of these components from $\partial \BD$; since $\partial \BD \subseteq \SCB^{\partial}_t(\partial \BD)$,~\Cref{prop BCLE rho to zero} implies that $\SCB^{\partial}_t(\partial \BD)$ converges to $\partial \BD$ in probability as $t \downarrow 0$ (with respect to $d_\rH$). Thus, the claim follows since the family of sets $(\SCB^{\partial}_t(\partial \BD))_{t \geq 0}$ is non-decreasing in $t$.
	
	\stepx{step:skoro-conclusion}{Conclusion} As a result, by~\eqref{eq continuite a droite explo unif}, we have for all $\varepsilon>0$,
	\begin{equation*}
	\lim_{\delta \to 0} \sup_{n\ge 1} \BP\!\left[ d_\rH\left(\overline{(\phi^z_n)^{-1}(C_n(z) \cap \SCB^{\kappa_n, \partial}_{\ka_{\kappa_n} \delta +1} (X^n_{\tau_n}))}, \partial \BD\right) \ge \varepsilon\right] =0.
	\end{equation*}
	Therefore, 
	\begin{equation*}
	{d_\rH}\left(\overline{ \partial \BD \cup (\phi^z_n)^{-1}(C_n(z) \cap (X^n_{\tau_n+\delta_n}))}, \partial \BD\right) 
	\mathop{\longrightarrow}\limits_{n\to \infty}^{(\BP)} 0.
	\end{equation*}
	Note that if a connected component $C$ of $\BD \setminus K$ is not included in $B_{\varepsilon}(K)$, then there exists $x \in C$ such that $\dist(x,K) \geq \varepsilon$, and $C$ then contains $B_\varepsilon(x) \cap \BD$, whose area is bounded from below by a constant depending only on $\varepsilon$. Since the area of $\BD$ is finite, the number of such components is finite. Therefore, combining everything, we obtain that
\begin{align*}
d_\rH(X^n_{\tau_n + \delta_n}, X^n_{\tau_n}) \to 0 \quad \text{as} \quad n \to \infty
\end{align*}
in probability. This proves~\eqref{eq tightness Aldous} and completes the proof of the lemma. The same reasoning works if we replace $\alpha$ by a loop.
\end{proof}

The remainder of the present subsection aims to prove that loops are dense in metric balls. We start with metric balls at rational times.

\begin{lemma}\label{lemma right-continuity at rational times and density of loops}
	For all $t \in \BQ_{\ge 0}$ and for all $\alpha \in \SCS'$, a.s., for all $\SCL \in \Gamma$, the compact set $\SCB_t(\alpha)$ (resp.~$\SCB_t(\SCL)$) is the closure of the union of the regions encircled by the loops of $\Gamma$ at $D$-distance at most $t$ from $\alpha$ (resp.~$\SCL$).
\end{lemma}
\begin{proof}
	First, note that by~\eqref{eq cv subsequence boundary}, every loop which is included in $\SCB_t(\alpha)$ is at $D$-distance at most $t$ from $\alpha$. Moreover, again by~\eqref{eq cv subsequence boundary}, every loop which is at distance strictly less than $t$ from $\alpha$ is included in $\SCB_t(\alpha)$. But by right-continuity (consequence of~\Cref{lemma Skorokhod tightness}), we deduce that all the loops which are at distance at most $t$ from $\alpha$ are included in $\SCB_t(\alpha)$. As a result, the closure of the union of the domains encircled by the loops at $D$-distance at most $t$ from $\alpha$ is included in $\SCB_t(\alpha)$.
	
	It remains to prove the other inclusion. One can first see that for all $x $ in the interior of $\SCB_t(\alpha)$, there exists a sequence $(x_k)_{k\ge 1} $ in $\BQ^2 $ in the interior of $\SCB_t(\alpha)$ such that  $x_k \to x$ as $k\to \infty$. Then, $\SCL(x_k) \subseteq \SCB_t(\alpha)$ for all $k$ (because otherwise we cannot have $x_k \in \SCB_t(\alpha)$ by~\eqref{eq cv subsequence boundary}). Moreover, $x$ is arbitrarily close to the union of the regions encircled by the $\SCL(x_k)$'s. Thus, the interior of $\SCB_t(\alpha)$ is included in the closure of the union of the domains encircled by the loops at distance at most $t$ from $\alpha$.
	
	To conclude, it remains to check that any point on the boundary can be approached by loops at distance at most $t$ from $\alpha$. This is a direct consequence of the two points of~\Cref{lemma approx balls from segments with unif explo} and of the analogous property for the uniform exploration (see~\Cref{remark balls are balls}). The exact same argument works to prove the claim of the lemma with a loop $\SCL$ in place of $\alpha$, since~\Cref{remark balls are balls} and~\Cref{lemma approx balls from segments with unif explo,lemma Skorokhod tightness} still hold with $\SCL$ in place of $\alpha$.	
\end{proof}
We denote by $(\SCB_t(\alpha))_{t\ge 0}$ and $(\SCB_t(\SCL))_{t\ge 0}$ the right-continuous processes
\begin{equation*}
\SCB_t(\alpha) \defeq \bigcap_{s\ge t: s \in \BQ} \SCB_s (\alpha) \qquad \text{and}
\qquad
\SCB_t(\SCL) \defeq \bigcap_{s\ge t: s \in \BQ} \SCB_s (\SCL),
\end{equation*}

Combining the tightness of~\Cref{lemma Skorokhod tightness} with the convergence of the finite-dimensional marginals
, which identifies the J$_1$ limit, we obtain that jointly with~\eqref{eq cv subsequence},~\eqref{eq cv Hausdorff balls},~\eqref{eq cv subsequence boundary}, jointly in $\alpha \in \SCS'$ and $x \in \BD_\BQ$, 
\begin{equation}\label{eq cv Skorokhod}
	\left(\SCB^{\kappa_n}_{\ka_{\kappa_n} t}(\alpha)\right)_{t\ge 0}\mathop{\longrightarrow}\limits_{n\to \infty}^{(\mathrm{d})} \left(\SCB_t(\alpha)\right)_{t\ge 0}
	\qquad
	\text{and}
	\qquad
	\left(\SCB^{\kappa_n}_{\ka_{\kappa_n} t}(\SCL^{\kappa_n}(x))\right)_{t\ge 0}\mathop{\longrightarrow}\limits_{n\to \infty}^{(\mathrm{d})} \left(\SCB_t(\SCL(x))\right)_{t\ge 0}
\end{equation}
for the J$_1$ Skorokhod topology.
Then, by construction, $\left(\SCB_t(\alpha)\right)_{t\ge 0}$ and $\left(\SCB_t(\SCL)\right)_{t\ge 0}$ are right-continuous. It remains to check that ``loops are dense'' in $\SCB_t(\alpha)$ and $\SCB_t(\SCL)$ for all $t$.

\begin{lemma}\label{lem:density_of_loops_in_balls}
	For all $\alpha \in \SCS'$, a.s., for all $t \ge 0$, the compact set $\SCB_t(\alpha)$ is the closure of the union of the regions encircled by the loops at $D$-distance at most $t$ from $\alpha$. The same result holds if we replace $\alpha$ by a loop $\SCL \in \Gamma$.
\end{lemma}

\begin{proof}
	\stepx{step:dens-cont}{Continuity times} For continuity times we can approximate $\SCB_t(\alpha)$ from below with rational times and apply~\Cref{lemma right-continuity at rational times and density of loops}. More precisely, let $t \ge 0$ be a continuity time of the process $\left(\SCB_t(\alpha)\right)_{t\ge 0}$. Then, we know that
	\[
	\SCB_t(\alpha) = \overline{\bigcup_{s \in [0,t] \cap \BQ} \SCB_s(\alpha)}.
	\]
	By~\Cref{lemma right-continuity at rational times and density of loops}, we deduce that
	\begin{equation}\label{eq continuity time}
	\SCB_t(\alpha) = \overline{\bigcup_{s \in [0,t] \cap \BQ} \overline{\bigcup_{\substack{\SCL\in \Gamma\\ D(\alpha, \SCL) \le s}} \mathrm{int}(\SCL)}} = \overline{\bigcup_{\substack{\SCL \in \Gamma \\ D(\alpha, \SCL) < t}} \mathrm{int}(\SCL)}.
	\end{equation}
	Moreover, there is no loop $\SCL$ such that $D(\alpha, \SCL)=t$. Indeed, assume by contradiction that there exists such a loop $\SCL \in \Gamma$. By Skorokhod's representation theorem, we may assume that~\eqref{eq cv subsequence boundary} and \eqref{eq cv Skorokhod} hold a.s. Then, a.s., there exists $\SCL^{\kappa_n} \in \Gamma^{\kappa_n}$ such that $\ka_{\kappa_n}^{-1} D^{\kappa_n}(\alpha, \SCL^{\kappa_n}) \to D(\alpha, \SCL)$. In particular, for all $\varepsilon>0$, we have $\mathrm{int}(\SCL) \subseteq \SCB_{t+ \varepsilon}(\alpha)$. By right-continuity, we deduce that $\mathrm{int}(\SCL) \subseteq \SCB_t(\alpha)$. By~\eqref{eq continuity time}, the latter is the closure of the union of the interiors of the loops at $D$-distance strictly less than $t$ from $\alpha$, each disjoint from $\mathrm{int}(\SCL)$; the non-empty open set $\mathrm{int}(\SCL)$ would then lie in the complement of the union of the interiors of all the loops of $\Gamma$, which has zero Lebesgue measure, hence empty interior. This is absurd. We thus conclude that 
	\[
	\SCB_t(\alpha) = \overline{\bigcup_{\substack{\SCL \in \Gamma \\ D(\alpha, \SCL) \le t}} \mathrm{int}(\SCL)}.
	\]
	The same reasoning works when $\alpha$ is replaced by a loop of $\Gamma$.
	
	\stepx{step:dens-jump}{Reduction to the jump times} Next, in order to prove the result for jump times, we will actually show that each jump is obtained by attaching at most one loop in each connected component of the unexplored region.
	
	The J$_1$ Skorokhod convergence gives the convergence of the jumps as a point process (convergence in the J$_1$ topology implies that of the jump times and of the corresponding jumps; see, e.g.,~\cite{Kal21}) and of the metric balls just before the jumps.
	
	More precisely, for all $t\ge 0$, let us denote by $\SCB^{\kappa_n,-}_{\ka_{\kappa_n} t}(\alpha)$ (resp.~$\SCB^{ -}_t(\alpha)$) the closure of the increasing union of $\SCB^{\kappa_n}_{\ka_{\kappa_n} s}(\alpha)$ (resp.~$\SCB_s(\alpha))$ as $s \uparrow t$. For all $z \in \BD_\BQ$, let $C^n_t(z)$ (resp.~$C_t(z)$) be the connected component of $\BD \setminus\SCB^{\kappa_n,-}_{\ka_{\kappa_n} t}(\alpha)$ (resp.~$\BD \setminus \SCB^{ -}_t(\alpha)$) containing $z$ (we set $C^n_t(z) = \emptyset $ if $ z\in \SCB^{\kappa_n, -}_{\ka_{\kappa_n} t}(\alpha)$ and $C_t(z)= \emptyset$ if $z \in \SCB^{-}_t(\alpha)$ by convention). When possible, we let $\phi^{n,z}_{t} \colon \BD \to C^n_t(z)$ (resp.~$\phi^z_t \colon \BD \to C_t(z)$) be the unique conformal mapping such that $\phi^{n,z}_{t}(0) = z$ and $(\phi^{n,z}_{t})^\prime(0)>0$ (resp.~$\phi^z_t (0)= z$ and $(\phi^z_t)^\prime(0)>0$).
	
	To end the proof of the lemma, it is enough to show that a.s., for all $t \ge 0$ and $z \in \BD_\BQ$, the set $C_t(z) \cap \SCB_t(\alpha)$ is the domain encircled by a loop of $\Gamma$ whenever $C_t(z) \neq \emptyset$. Note also that any loop in $\Gamma$ which is contained in $\SCB_t(\alpha)$ has $D$-distance from $\alpha$ at most $t$. This is the purpose of the remainder of the proof.
	
	\stepx{step:dens-conv}{Convergence of the unexplored components} By Skorokhod's representation theorem, let us assume that the convergences~\eqref{eq cv subsequence},~\eqref{eq cv subsequence boundary},~\eqref{eq cv Skorokhod} hold a.s. 
	
	By~\eqref{eq cv Skorokhod}, we know that there exist increasing (random) homeomorphisms $\lambda_n: \BR_{\ge 0} \to \BR_{\ge 0}$ such that $\lambda_n$ converges uniformly toward $x \mapsto x$ and 
	\begin{equation}\label{eq cv unif balls segments}
	\left(\SCB^{\kappa_n}_{\ka_{\kappa_n} \lambda_n(t)}(\alpha)\right)_{t\ge 0}\mathop{\longrightarrow}\limits_{n\to \infty}^{(\mathrm{a.s.})} \left(\SCB_t(\alpha)\right)_{t\ge 0}
	\end{equation}
	uniformly on compact sets of $\BR_{\ge 0}$. In particular, a.s., for all $z \in \BD_\BQ$, the function $(t,w) \mapsto \phi^{n,z}_{\lambda_n(t)}(w)$ converges uniformly on compact subsets of $\BR_{\ge 0} \times\BD$ towards $(t,w) \mapsto \phi^z_t(w)$. 
	Indeed, let us work on the event that the convergence \eqref{eq cv unif balls segments} occurs. Let $(t_n, w_n) \to (t, w) \in \BR_{\ge 0}\times \BD$. Then, the components $C_{\lambda_n(t_n)}(z)$ converges to $C_t(z)$ in the sense of Carathéodory due to \eqref{eq cv unif balls segments}. Moreover, recall that the convergence $C_{\lambda_n(t_n)}(z)$ to $C_t(z)$ in the sense of Carathéodory is equivalent to the convergence of $ \phi^{n,z}_{\lambda_n(t_n)}$ to $\phi^z_t$ uniformly on compact subsets of $\BD$. Therefore, $\phi^{n,z}_{\lambda_n(t_n)}(w_n) \to \phi^z_t(w)$ as $n\to \infty$.
	
	Let $\varepsilon\in (0,1)$. For all $j\in \BZ_{\ge 1}$, let $t_j$ (resp.\ $t^n_j$) be the $j$-th time $t$ such that $(\phi^z_t)^{-1}\left( C_t(z) \cap \SCB_{t}(\alpha)\right) \cap (1-\varepsilon)\BD \neq \emptyset$ (resp.\ $(\phi^{n,z}_{t})^{-1} \left( C^n_t(z)\cap \SCB^{\kappa_n}_{\ka_{\kappa_n} t}(\alpha) \right) \cap (1-\varepsilon)\BD \neq \emptyset$). By convention, this time is infinite when it does not exist. By \eqref{eq cv unif balls segments} and by the convergence of $\phi^{n,z}_{\lambda_n(t)}$, we get the almost sure convergence of $t^n_j$ to $t_j$ as $n\to \infty$.
	
	\stepx{step:dens-domination}{Stochastic domination by the balls from the boundary} Furthermore, let $\alpha^n_k \defeq (\phi^{n,z}_{k/\ka_{\kappa_n}})^{-1} ( \partial C^n_{k/\ka_{\kappa_n}}(z) \cap \SCB^{\kappa_n}_{k-1}(\alpha))$ for all $k\ge 1$ and $\alpha^n_0= \alpha$ by convention. By the domain Markov property of the $\CLE_\kappa$, we know that for all integers $k\ge 0$, the conditional law of $(\phi^{n,z}_{k/\ka_{\kappa_n}})^{-1} ( C^n_{k/\ka_{\kappa_n}}(z) \cap \SCB^{\kappa_n}_{k}(\alpha) )$ given $\SCB^{\kappa_n}_{k-1}(\alpha)$ is that of $\SCB^{\kappa_n}_1(\alpha^n_k) \setminus \partial \BD$. In particular, it is stochastically dominated by a copy $B^n_k$ of $\SCB^{\kappa_n, \partial}_1(\partial \BD) \setminus \partial \BD = \SCB^{\kappa_n}_1 (\partial \BD) \setminus \partial \BD$. Note that by induction, we can construct the $B^n_k$'s so that they are independent. Therefore, there is a coupling of the $B^n_k$'s with the $(\phi^{n,z}_{k/\ka_{\kappa_n}})^{-1} ( C^n_{k/\ka_{\kappa_n}}(z) \cap \SCB^{\kappa_n}_{k}(\alpha) )$'s for $k \in \BZ_{\ge 0}$ such that the $B^n_k$'s are i.i.d.\ with the same law as $\SCB^{\kappa_n, \partial}_1(\partial \BD) \setminus \partial \BD = \SCB^{\kappa_n}_1 (\partial \BD) \setminus \partial \BD$ and such that all the $\CLE_{\kappa_n}$ loops forming $(\phi^{n,z}_{k/\ka_{\kappa_n}})^{-1} ( C^n_{k/\ka_{\kappa_n}}(z) \cap \SCB^{\kappa_n}_{k}(\alpha) )$ are loops of $B^n_k$. For all $t \ge 0$ such that $\ka_{\kappa_n} t \notin \BZ$, we set $B^n_{\ka_{\kappa_n} t} = \emptyset$.
	
	\stepx{step:dens-ppp}{Identification of the jumps via the bubble process} Besides, observe that the convergence~\eqref{eq cv Skorokhod} holds in the particular case $\alpha=\partial \BD$ and that this yields the convergence in distribution for all $r \in (0,1)$
	\begin{equation}\label{eq cv PPP}
		\left( r\overline{\BD} \cap B^n_{\ka_{\kappa_n}t}\right)_{t\ge 0}  \mathop{\longrightarrow}\limits_{n\to \infty}^{(\mathrm{d})}
		\left( r\overline{\BD} \cap \overline{\mathop{\mathrm{int}}(\gamma_t)} \right)_{t\ge 0},
	\end{equation}
	where the point process takes its values in the set of closed subsets of $\BD$, equipped with the Hausdorff topology, where $(\gamma_t)_{t\ge 0}$ is the PPP of $\SLE_4$ bubbles rooted uniformly at random on $\partial \BD$ and where the empty set is seen as a cemetery point (with the convention that $\gamma_t = \emptyset$ off the countably many atom times). Indeed, the PPP on the right-hand side comes from the fact that $\SCB^\partial_t(\partial \BD) = \SCB_t(\partial \BD)$ is the uniform exploration of the $\CLE_4$. 
	
	As a result, combining the above remark with the Hausdorff convergences of the loops stated in \eqref{eq cv subsequence boundary}, we deduce that for all $j\ge 1$ for all $r \in (1-\varepsilon, 1)$, with probability $1-o(1)$ as $n\to \infty$, there is one loop of $(\phi^{n,z}_{t^n_j})^{-1} \left( \Gamma^{\kappa_n}\vert_{C^n_{t^n_j}(z)\cap \SCB^{\kappa_n}_{\ka_{\kappa_n} t^n_j}(\alpha) }\right)$ that intersects $r\overline{\BD}$ and the image of this loop by $\phi^{n,z}_{t^n_j}$ converges a.s.\ to one loop $\SCL_j$ of $\Gamma$. In particular, we get that $\overline{\mathrm{int}(\SCL_j)} \subseteq  \overline{C_{t_j}(z)} \cap \SCB_t(\alpha)$.
		
	Finally, \eqref{eq cv PPP} and the stochastic domination described above imply that $\overline{C_{t_j}(z)} \cap \SCB_{t_j}(\alpha) = \overline{\mathrm{int}(\SCL_j)}$. This concludes the proof of the lemma.
\end{proof}

\section{Large non-simple CLE loops are fat}
\label{sec:big non simple CLE loops are fat}

This section can be skipped on a first reading. Its goal is to establish a strengthened convergence for the loops of a non-simple CLE as $\kappa' \downarrow 4$ (\Cref{lemma cv Hausdorff union of loops}), which will be a crucial tool in the subsequent sections for verifying that our subsequential limiting metric satisfies the axioms of~\Cref{def:weak_axioms}. Throughout this section, we will denote the parameter of the non-simple CLE by $\kappa'\in (4,8)$ and we will write $\kappa=16/\kappa'$ for the parameter of the corresponding simple CLE. The sequence which was denoted by $\kappa_n$ in the previous sections will be denoted by $\kappa'_n$ here. We already know from~\eqref{eq cv subsequence boundary} that the loop surrounding any given rational point converges in the Hausdorff topology as $\kappa' \downarrow 4$. We will strengthen this convergence by showing that the closure of the union of loops intersecting a given compact subset of $\overline{\BD}$ also converges in the Hausdorff topology as $\kappa' \downarrow 4$. Proving this strengthening is highly non-trivial because it requires us to rule out the existence of macroscopic but ``skinny'' loops uniformly in $\kappa'$ as $\kappa' \downarrow 4$. The bulk of this section is therefore devoted to proving~\Cref{prop:thin-loops}, a result of independent interest which states that large non-simple CLE loops are ``fat'', meaning they must contain a large Euclidean ball.

\begin{lemma}
	\label{lemma cv Hausdorff union of loops}
	Let $K \subset \overline{\BD}$ be a connected compact subset of $\overline{\BD}$. For each $n\ge 1$, let $F_n$ (resp.\ $F$) be the closure of the union of $K$ and the domains encircled by the loops of $\Gamma^{\kappa'_n}$ (resp.\ $\Gamma^4$) which intersect $K$. Then, jointly with~\eqref{eq cv subsequence boundary} and~\eqref{eq cv Skorokhod}, we have
	\[
	F_n \mathop{\longrightarrow}\limits_{n\to \infty}^{(\mathrm{d})} F,
	\]
	in the Hausdorff topology.
\end{lemma}
The main ingredient of the proof of the above lemma is the following result which says that large loops of $\Gamma^{\kappa_n'}$ surround, with probability close to $1$ uniformly in $n$, Euclidean balls whose radius is bounded from below. 

\begin{proposition}\label{prop:thin-loops}
		For all $\varepsilon, \varepsilon'>0$, there exists $\delta>0$ such that for all $n$, with probability at least $1-\varepsilon'$, all the loops of $\Gamma^{\kappa'_n}$ with diameter at least $\varepsilon$ surround a Euclidean ball of radius $\delta$.
\end{proposition}    
The above proposition is not easy to prove and the main goal of this section is to prove it. Let us explain how it implies~\Cref{lemma cv Hausdorff union of loops}.
\begin{proof}[Proof of~\Cref{lemma cv Hausdorff union of loops} using~\Cref{prop:thin-loops}]
	Let $\varepsilon>0$ be such that a.s.\ no loop of $\Gamma^4$ has diameter exactly $\varepsilon$ (see~\Cref{prop:big_loops_are_fat_simple_cle}). Let $F_n^\varepsilon$ (resp.\ $F^\varepsilon$) be the closure of the union of $K$ and of the domains encircled by loops of $\Gamma^{\kappa'_n}$ (resp.\ $\Gamma^4$) intersecting $K$ that have diameter at least $\varepsilon$. Note that 
	\begin{equation}\label{eq approx F n loops with diameter at least epsilon}d_{\mathrm{H}}(F_n^\varepsilon, F_n) \le \varepsilon \qquad \text{and}
	\qquad d_{\mathrm{H}}(F^\varepsilon, F)\le \varepsilon. 
	\end{equation}
	Let $\delta>0$ be given by~\Cref{prop:thin-loops} applied with $\varepsilon' = \varepsilon$ (taken rational, so that the grid below lies in $\BD_\BQ$). By possibly taking $\delta$ smaller, by the local finiteness of $\CLE_4$, we also know that with probability at least $1-\varepsilon$, all the loops of $\Gamma^4$ with diameter at least $\varepsilon$ surround a Euclidean ball of radius $\delta$. Let $F_n^{\varepsilon, \delta}$ (resp.\ $F^{\varepsilon, \delta}$) be the closure of the union of $K$ and of the domains encircled by loops of $\Gamma^{\kappa'_n}$ (resp.\ $\Gamma^4$) intersecting $K$ that have diameter at least $\varepsilon$ and surround a point in $(\delta/2)\BZ^2$. By Skorokhod's representation theorem, we may assume that the convergences~\eqref{eq cv subsequence boundary} and~\eqref{eq cv Skorokhod} hold a.s. 
	
	Since the set $(\delta/2)\BZ^2 \cap \BD$ is finite, by~\eqref{eq cv subsequence boundary} and~\Cref{rem:interior membership}, we get the a.s.\ convergence of $F_n^{\varepsilon, \delta}$ to $F^{\varepsilon, \delta}$. Indeed, the selection by ``intersects $K$'' stabilizes by~\Cref{rem:interior membership}, since a.s.\ no loop of $\Gamma^4$ is tangent to $K$, i.e.\ every loop $\SCL$ meeting $K$ is such that $\mathrm{int}(\SCL)$ intersects $K$. Indeed, using the $\SLE_4(-2)$ exploration tree, this comes from the analogous fact that on the event that a chordal $\SLE_4$ $\eta$ from $0$ to $\infty$ intersects a deterministic connected compact set $K\subset \BH$, there are points of $K$ in both connected components of $\BH\setminus \eta$. Such a result can be seen as a consequence of \cite[Lemmas 4.4 and 4.5]{miller2017intersections}: let $f_t=g_t-W_t$ be the associated centered Loewner flow and consider the stopping times $\tau_k = \inf\{ t\ge 0: \mathrm{dist}(\eta(t), K) \le 2^{-k} \}$ for all $k\ge 1$. By choosing $\delta$ small, a.s.\ for infinitely many $k$, we have that $f_{\tau_k}(K) \cap \BD$ is included in $\{z \in \BH: \arg(z) \in (\delta, \pi-\delta)\}$, so that the curve $\eta$ will have points of $K$ on its left and right sides. 
	
	But, note that by~\Cref{prop:thin-loops}, with probability at least $1- \varepsilon$, we have $F_n^\varepsilon= F_n^{\varepsilon, \delta}$. The same is true for $F^\varepsilon$ and $F^{\varepsilon,\delta}$ by our choice of $\delta$. By taking~\eqref{eq approx F n loops with diameter at least epsilon} into account, this concludes the proof.
\end{proof}
The same proof shows that the lemma remains true if the loop surrounding a fixed $z_0 \in \BD_\BQ$ is excluded from $F_n$ and $F$: this only removes from $(\delta/2)\BZ^2 \cap \BD$ the points surrounded by $\SCL^{\kappa_n'}(z_0)$, and for each fixed such $z$, by~\Cref{rem:interior membership}, for $n$ large $\SCL^{\kappa_n'}(z_0)$ surrounds $z$ if and only if $\SCL^4(z_0)$ does. We use this with $z_0 = 0$ in~\Cref{sec: distance between two segments}.

The proof of~\Cref{prop:thin-loops} will be achieved by combining the following three steps.
\begin{enumerate}
		\item \label{it:first_step} 
		We will prove the same result as in the statement of~\Cref{prop:thin-loops} but with $\CLE_{\kappa_n'}$ replaced by $\CLE_{\kappa_n}$ in $\BD$ with $\kappa_n =16/\kappa'_n$. This result is much easier because $\CLE_\kappa$ for $\kappa \in (8/3,4]$ can be coupled on a common probability space so as to be monotone in $\kappa$ \cite{CLE}.
		\item \label{it:second_main_step}
		We will use a specific coupling of  $\CLE_{\kappa_n'}, \CLE_{\kappa_n}$, and a GFF on $\BD$ which comes from the results of~\cite{CLEPerc}, which ensures that a.s.\ whenever a $\CLE_{\kappa_n'}$ loop intersects a $\CLE_{\kappa_n}$ loop, it has to surround it.
		\item \label{it:third_main_step}
		Under the coupling from~\eqref{it:second_main_step}, we have that the following holds with probability tending to $1$ as $n \to \infty$. Every $\CLE_{\kappa_n'}$ loop with large Euclidean diameter has to intersect a $\CLE_{\kappa_n}$ loop with large Euclidean diameter. We will prove this by showing that large $\CLE_{\kappa_n'}$ loops cannot stay close to $\partial \BD$ and therefore must cross a macroscopic annulus strictly contained in $\BD$. We will then explore the simple $\CLE_{\kappa_n}$ loops crossing this annulus to form a mesh of conformal rectangles, and show that it is highly unlikely for a non-simple loop to cross such a rectangle without intersecting its simple $\CLE$ boundaries.
\end{enumerate}

The claim in item~\eqref{it:first_step} will be shown in~\Cref{sec:big_simple_loops_are_fat}. In~\Cref{sec:cle-percolation}, we will describe the coupling between the $\CLE_{\kappa_n'}$, the $\CLE_{\kappa_n}$, and the GFF on $\BD$ that we are going to use. Finally in~\Cref{sec:proof_of_main_result}, we will prove the claim in item~\eqref{it:third_main_step} and hence complete the proof of~\Cref{prop:thin-loops}.

\subsection{Big simple CLE loops are fat}
\label{sec:big_simple_loops_are_fat}

In this section, we will prove that with probability tending to $1$ as $\kappa \uparrow 4$, we have that every loop in a non-nested $\CLE_{\kappa}$ in $\BD$ with large Euclidean diameter surrounds a Euclidean ball with large Euclidean diameter. 

For all $\kappa \in (8/3, 4]$, let $\Gamma^\kappa$ be a non-nested $\CLE_\kappa$ in $\BD$. In this subsection, we have two main results. The first is the following.
\begin{proposition}
	\label{prop:big_loops_are_fat_simple_cle}
		Let $(\kappa_n)_{n \ge 1}$ be a sequence in $(8/3,4)$ with $\kappa_n \uparrow 4$. For all $\varepsilon>0$ such that a.s.\ no loop of $\Gamma^4$ has diameter exactly $\varepsilon$ (all but countably many $\varepsilon$, since $\varepsilon \mapsto \BE[\#\{\SCL \in \Gamma^4 \colon \diam(\SCL) \ge \varepsilon\}]$ is finite and non-increasing, hence has countably many discontinuities), the number of loops of $\Gamma^{\kappa_{n}}$ with diameter at least $\varepsilon$ converges in law as $n \to \infty$ to the number of loops of $\Gamma^4$ with diameter at least $\varepsilon$. Moreover, for all $\varepsilon, \varepsilon'>0$, there exists $\delta>0$ (depending on $\varepsilon$, $\varepsilon'$, and on the sequence $(\kappa_n)_{n \ge 1}$) such that for all $n\ge 1$, with probability at least $1-\varepsilon'$, all the loops of $\Gamma^{\kappa_n}$ with diameter at least $\varepsilon$ encircle a Euclidean ball of radius $\delta$.
\end{proposition}

Let us now turn to state our second main result. Let us first introduce some notation. Fix $0 <s<t$ and $z \in \BD$ such that the annulus $A=A_{s,t}(z)= B_t(z) \setminus \overline{B_s(z)} \subseteq \BD$. Given $\SCL \in \Gamma^\kappa$, the \emph{segments} of $\SCL$ in $A$ are the closures of the connected components of $\SCL \cap A$; each has both endpoints on $\partial A$. Such a segment is a \emph{crossing} if its endpoints lie on different components of $\partial A$, and an \emph{excursion from the outer (resp.\ inner) boundary} if they both lie on $\partial B_t(z)$ (resp.\ $\partial B_s(z)$); $\SCL$ \emph{crosses} $A$ if one of its segments is a crossing. A \emph{conformal rectangle} is a simply connected domain with four marked prime ends, splitting its boundary into its top, bottom, left and right boundaries. Let $\Gamma_A^{\kappa,\mathrm{out}}$ (resp.~$\Gamma_A^{\kappa, \mathrm{in}}$) denote the collection of segments of loops in $\Gamma^{\kappa}$ which make an excursion from the outer (resp.~inner) boundary of $A$. Let also $\Gamma^\kappa_A$ denote the collection of loops in $\Gamma^\kappa$ which cross $A$.

Next, for all $\SCL \in \Gamma^{\kappa, \mathrm{out}}_A \cup \Gamma^{\kappa, \mathrm{in}}_A$, we denote by $\mathrm{int}(\SCL)$ the intersection of the region encircled by the associated loop and $A$. In particular, distinct segments of one and the same loop have the same $\mathrm{int}$. We will prove the following result that we are going to need in~\Cref{sec:proof_of_main_result}.
\begin{proposition}
\label{prop:conformal_rectangles_tight}
	We have the following convergences:
	\begin{enumerate}[label=(\roman*)]
		\item The number $N^\kappa$ of crossings of $A$ by loops in $\Gamma^\kappa_A$ converges in law to the number $N^4$ of crossings of $A$ by loops in $\Gamma^4_A$ as $\kappa \uparrow 4$. More precisely, for all $\kappa_n \uparrow 4$ there exists an ordering of the crossings $\eta_1^n, \ldots, \eta^n_{N^{\kappa_n}}$ of $\Gamma^{\kappa_n}_A$ for which the random vectors $(N^{\kappa_n}, \eta^n_1, \ldots, \eta^n_{N^{\kappa_n}})$ converge in law to $(N^4, \eta_1, \ldots, \eta_{N^4})$, where $\eta_1, \ldots, \eta_{N^4}$ are the crossings of $A$ by loops in $\Gamma^4_A$, for the topology of Hausdorff convergence.
		\item Jointly with the above convergence, the compact sets $\overline{\bigcup_{\SCL \in \Gamma^{\kappa_n, \mathrm{out}}_A} \mathrm{int}(\SCL)}$ and $\overline{\bigcup_{\SCL \in \Gamma^{\kappa_n, \mathrm{in}}_A} \mathrm{int}(\SCL)}$ converge in distribution for the Hausdorff distance to $\overline{\bigcup_{\SCL \in \Gamma^{4, \mathrm{out}}_A} \mathrm{int}(\SCL)}$ and $\overline{\bigcup_{\SCL \in \Gamma^{4, \mathrm{in}}_A} \mathrm{int}(\SCL)}$.
		\item In particular, the conformal rectangles whose left and right sides are given by the crossings of $A$ by loops in $\Gamma^{\kappa_n}_A$ and whose top and bottom sides are respectively included in $\overline{\bigcup_{\SCL \in \Gamma^{\kappa_n, \mathrm{out}}_A} \SCL}$ and $\overline{\bigcup_{\SCL \in \Gamma^{\kappa_n, \mathrm{in}}_A} \SCL}$ converge in distribution as $n \to \infty$ in the Carath\'eodory sense (jointly with their corners) to the conformal rectangles whose left and right sides are given by the crossings of $A$ by loops in $\Gamma^{4}_A$ and whose top and bottom sides are respectively included in $\overline{\bigcup_{\SCL \in \Gamma^{4, \mathrm{out}}_A} \SCL}$ and $\overline{\bigcup_{\SCL \in \Gamma^{4, \mathrm{in}}_A} \SCL}$, and their moduli are tight in $(0,\infty)$.
	\end{enumerate}
\end{proposition}	

To prove the above results, let us use an appropriate coupling. Let $(S^\kappa)_{\kappa \in (8/3,4]}$ be an increasing coupling of Brownian loop soups on $\BD$ such that $S^\kappa$ has intensity $c_\kappa = (3\kappa-8)(6-\kappa)/(2\kappa)$. Recall from \cite[Theorem~1.6]{CLE} that the outer boundaries of the clusters of the loops of $S^\kappa$ form a non-nested $\CLE_{\kappa}$ denoted by $\Gamma^{\kappa}$. Note that in this coupling, for all $8/3<\kappa_1\le \kappa_2\le 4$, every loop of $\Gamma^{\kappa_1}$ is surrounded by a loop of $\Gamma^{\kappa_2}$.

In order to give the proof of~\Cref{prop:conformal_rectangles_tight}, we will first need to prove a number of preparatory lemmas.
\begin{enumerate}[label=(\roman*)]
\item In~\Cref{subsubsec:loops_converge}, we will prove a result related to the convergence of the origin-surrounding loop, as a path, as $\kappa_m \uparrow 4$.
\item In~\Cref{subsubsec:big_loops_convergence}, we will prove a result related to the convergence of the large loops as $\kappa_m \uparrow 4$. This will, in particular, give us~\Cref{prop:big_loops_are_fat_simple_cle}.
\item In~\Cref{subsubsec:annulus_segments}, we will prove the convergence as $\kappa_m \uparrow 4$ of the crossings of an annulus. This will, in particular, give us~\Cref{prop:conformal_rectangles_tight}.
\end{enumerate}

\subsubsection{Convergence of the origin-surrounding loop as a path}
\label{subsubsec:loops_converge}
Next, let us describe the convergence of loops. We start by recalling the Hausdorff convergence of loops surrounding the origin, which was obtained in~\cite{AG23} (see also \cite[Lemma A.3]{ACSW24} for more details or \cite[Lemma 3.5]{Kam25distancesonCLE4and32maps} for a closely related result).

\begin{lemma}[{\cite[Lemma 3.6]{AG23}}]
\label{lemma cv Hausdorff simple loops}
	The origin-containing loop $\SCL^\kappa(0)$ of $\Gamma^\kappa$ converges in distribution in the Hausdorff topology toward the origin-containing loop $\SCL^4(0)$ of $\Gamma^4$.
\end{lemma}
Let us improve this convergence in the following convergence of curves:
\begin{lemma}\label{lemma improved convergence loops}
	For all $\varepsilon>0$, for all sequences $\kappa_m \uparrow 4$, there exist a subsequence, again denoted by $(\kappa_m)_{m \ge 1}$, and a coupling of the $\Gamma^{\kappa_m}$, $m \ge 1$, with random variables $Z_m \in \SCL^{\kappa_m}(0)$ and random simple curves $\gamma_m:[0,1] \to \overline{\BD}$ for $m\ge 1$ such that $(Z_m,\gamma_m)$ converges a.s.\ in $\overline{\BD} \times \SCC^0([0,1],\overline{\BD})$, where $\SCC^0([0,1], \overline{\BD})$ is the space of continuous functions on $[0,1]$ equipped with the topology of uniform convergence, and a.s.\ for all $m$ large enough,
	\[\gamma_m([0,1]) \subseteq \SCL^{\kappa_m}(0)
	\subset B_\varepsilon(Z_m) \cup \gamma_m([0,1]).
	\]
\end{lemma}
In the statement of~\Cref{lemma improved convergence loops}, the point $Z_m$ will be the first place on the loop $\SCL^{\kappa_m}(0)$ visited by the trunk of the $\SLE_{\kappa_m}(\kappa_m-6)$ process and $\gamma_m$ will be given by $\SCL^{\kappa_m}(0)$ outside of a small neighborhood of $Z_m$.

\begin{remark}
	By conformal invariance, the above statements hold if we replace $0$ by any point $z \in \BD$.
\end{remark}
\begin{remark}
	We expect that~\Cref{lemma improved convergence loops} and~\Cref{lemma cv Hausdorff simple loops} imply the convergence of the loops as simple closed paths, but we do not pursue this, as we do not need it.
\end{remark}
Before proving~\Cref{lemma improved convergence loops}, let us show the following result. Let $n\ge 1$. Recall from~\eqref{eq excursion intervals} that the intervals $[S^{\kappa,n}_k, T^{\kappa,n}_k]$ for $k\ge 1$ are the intervals where $\theta^\kappa$ makes an excursion $e^{\kappa,n}_k$ of height at least $2^{-n}$ and that $R^{\kappa, n}_k$ is the first time where $e^{\kappa,n}_k$ reaches $2^{-n}$. Now, let $K^{\kappa, n}$ be the first $k$ such that $e^{\kappa,n}_k$ reaches $2\pi$.

\begin{lemma}
	\label{lemma cv number of excursion}
	For all $n\ge 1$, the random variable $K^{\kappa, n}$ converges in distribution as $\kappa \uparrow 4$ to a geometric random variable with success probability $2^{-n}/(2\pi)$.
\end{lemma}
\begin{proof}
	Note that $K^{\kappa, n}$ follows a geometric distribution (by the strong Markov property of $\theta^\kappa$ at the times $T^{\kappa,n}_k$) of success probability $\BP_\varepsilon[\theta^\kappa \text{ hits } 2\pi \text{ before } 0]$, where under $\BP_\varepsilon$, the process $\theta^\kappa$ is the diffusion~\eqref{eq diffusion theta} starting from $\varepsilon$ with $\varepsilon= 2^{-n}$. In this proof, we kill $\theta^\kappa$ when it hits $\{0,2\pi\}$. Therefore, to conclude it suffices to prove that $\BP_\varepsilon[\theta^\kappa \text{ hits } 2\pi \text{ before } 0]$ converges to a limit in $(0,1)$ as $\kappa \uparrow 4$.
	
	By \cite[Chapter~VII,  Proposition~3.2]{RY05}, there is a function $s^\kappa$, called a scale function, such that
	\[
	\BP_\varepsilon[\theta^\kappa \text{ hits } 2\pi \text{ before } 0] = \frac{s^\kappa(\varepsilon)-  s^\kappa(0)}{s^\kappa(2\pi)-s^\kappa(0)}.
	\]
	By \cite[Chapter~VII,  Exercise (3.20)]{RY05}, the scale function $s^\kappa$ associated with the diffusion $\theta^\kappa$ on $(0,2\pi)$ is
	\[
	s^\kappa(x) = \int_\pi^x \exp\left(-  \int_\pi^y 2\frac{\kappa-4}{2} \cot(z/2)/\kappa \mathrm{d}z  \right) \mathrm{d}y = \int_\pi^x \exp\left(\frac{2(4-\kappa)}{\kappa} \log (\sin(y/2))\right) \mathrm{d}y.
	\]
	
	The formula is only stated in~\cite{RY05} for diffusions on $\BR$, but the reader may check that $(s^\kappa(\theta^\kappa_t))_{t\ge 0}$ is a bounded martingale and apply the optional stopping theorem.
	
	By letting $\kappa \uparrow 4$ and applying the dominated convergence theorem, we see that $s^\kappa(x)\to x-\pi$ so that
	\[
	\BP_\varepsilon[\theta^\kappa \text{ hits } 2\pi \text{ before } 0] \mathop{\longrightarrow}\limits_{\kappa \uparrow 4} \frac{\varepsilon}{2\pi}.
	\]
	This yields the desired result.
\end{proof}

Let us end this subsection by proving~\Cref{lemma improved convergence loops}.
\begin{proof}[Proof of~\Cref{lemma improved convergence loops}]
	Consider the radial $\SLE_\kappa(\kappa-6)$ targeted at $0$ constructed in~\Cref{subsec:branching_sle_kappa_minus_6}, i.e.\ the totally asymmetric one ($\beta=1$); we fix this choice throughout the present proof, since $Z^\kappa$ and the curves $\gamma_m$ below depend on it. Let $Z^\kappa$ be the first point of $\SCL^\kappa(0)$ which is hit by the trunk of the exploration (in the sense of~\cite{CLEPerc}). Let $n\ge 1$. Recall from~\eqref{eq diffusion theta} the definition of $\theta^\kappa$ (the definition was stated for $\kappa>4$ but the same definition works for $\kappa<4$). 
	
	Recall that $g_t^\kappa$ is the solution of the radial Loewner equation~\eqref{eq Loewner radial} and that $g^\kappa_t\colon D^\kappa_t \to \BD$ is a conformal mapping such that $(g_t^\kappa)'(0)>0$ and $g_t^\kappa(0)=0$. Recall that $\eta^\kappa$ is the $\SLE_\kappa(\kappa-6)$ curve associated with the exploration.
	
	\stepx{step:ilc-far}{Subsequential limit of the curve far from the endpoints} Let $n\ge 1$ and $k\ge 1$. By Montel's theorem one can see that the family of random conformal maps defined on the unit disk $((g^\kappa_{R^{\kappa,n}_k})^{-1})_{8/3<\kappa \le 4}$ is tight as $\kappa \uparrow 4$ for the topology of uniform convergence on compact subsets of~$\BD$.
	
	Furthermore, since $R^{\kappa,n}_k$ is a stopping time for $\theta^\kappa$, by \cite[Proposition~3.12]{TreeCLE}, we know that conditionally on $(\theta^\kappa_s)_{0 \le s \le R^{\kappa,n}_k}$, the curve $g^\kappa_{R^{\kappa,n}_k}(\eta^\kappa|_{[R^{\kappa, n}_k, T^{\kappa, n}_k]})$ is a chordal $\SLE_\kappa$ from $W^\kappa_{R^{\kappa,n}_k}$ to the associated force point that we denote by $O^\kappa_{R^{\kappa,n}_k}$. Recall from the discussion above \cite[Equation (4.1)]{TreeCLE} that $\theta^\kappa_{R^{\kappa,n}_k} = \arg (W^\kappa_{R^{\kappa,n}_k})- \arg(O^\kappa_{R^{\kappa,n}_k})$. Furthermore, we know that $\theta^\kappa_{R^{\kappa,n}_k} = 2^{-n}$. Let $\check{\gamma}_k^{\kappa,n}:[0, \infty]\to \overline{\BD}$ be the curve $g^\kappa_{R^{\kappa,n}_k}(\eta^\kappa|_{[R^{\kappa, n}_k, T^{\kappa, n}_k]})$ parameterized by half-plane capacity (by conformally mapping $\BD$ to $\BH$, the starting point to $0$ and the endpoint to $\infty$ and the midpoint on the clockwise arc between the starting and ending points to $1$). By \cite[Theorem~1.10]{KS17}, we know that a chordal $\SLE_\kappa$ curve from $0$ to $\infty$ in $\BH$ parameterized by half-plane capacity converges in distribution to a chordal $\SLE_4$ curve from $0$ to $\infty$ in $\BH$ in the topology of uniform convergence on compact subsets of $\BR_{\ge 0}$. Therefore, $(W^\kappa_{R^{\kappa,n}_k})^{-1} \cdot \check{\gamma}_k^{\kappa,n}$ converges in distribution uniformly on compact subsets as $\kappa \uparrow 4$ toward a chordal $\SLE_4$ in $\BD$ from $1$ to $e^{-\ri 2^{-n}}$ denoted by $\widetilde{\gamma}^n_k$. Moreover, the convergence is uniform on $[0,\infty]$: by scaling, $\BP[\check\gamma([T,\infty]) \not\subseteq B_\delta(\text{target})]$ is the probability that the $\BH$-curve returns to $B_{R/\sqrt{T}}(0)$ after time $1$ for some $R=R(\delta)$, which tends to $0$ as $T\to\infty$ uniformly in $\kappa \in (8/3,4]$.
	
	Let $\kappa_m \uparrow 4$. By Prokhorov's theorem and Skorokhod's representation theorem, we may assume that along a subsequence again denoted by $(\kappa_m)_{m\ge 1}$, for all $n,k\ge 1$, as $m \to \infty$, a.s.\ $W^{\kappa_m}_{R_k^{\kappa_m, n}} \to X_k^n$ for some random variable $X_k^n$ in $\partial \BD$, but also $(g^{\kappa_m}_{R^{\kappa_m, n}_k})^{-1} \to f_k^n$ uniformly on compact subsets of $\BD$ for some random conformal map $f^n_k$, as well as $\SCL^{\kappa_m}(0) \to \SCL^4(0)$ in the Hausdorff topology by~\Cref{lemma cv Hausdorff simple loops}, $K^{\kappa_m,n} \to K^{n}$ for some random variable $K^{n}$ in $\BZ_{\ge0}$ by~\Cref{lemma cv number of excursion} and
	\begin{equation}\label{eq cv SLE kappa to SLE 4}
		(W^{\kappa_m}_{R^{\kappa_m,n}_k})^{-1} \cdot \check{\gamma}_k^{\kappa_m,n} \mathop{\longrightarrow}\limits_{m \to \infty}  \widetilde{\gamma}_k^n,
	\end{equation}
	uniformly on $[0,\infty]$. Note that, as a consequence of Hurwitz' theorem for sequences of holomorphic functions, the holomorphic function $f^n_k$ is either injective or constant as a limit of injective holomorphic functions, but since $(g^{\kappa_m}_{R^{\kappa_m, n}_k})^{-1}(\BD)$ contains the domain encircled by $\SCL^{\kappa_m}(0)$, hence a fixed ball around $0$ for $m$ large (\Cref{rem:interior membership}), $f^n_k$ is not constant. Hence, $f^n_k$ is injective.
	
	Furthermore, the convergences stated above imply that a.s., uniformly on $[0,\infty]$,
	\begin{equation}\label{eq uniform convergence SLE}
		\check{\gamma}_k^{\kappa_m, n} \mathop{\longrightarrow}\limits_{m \to \infty} \check{\gamma}^n_k \defeq X^n_k \cdot \widetilde{\gamma}_k^n.
	\end{equation}
	Since $\check{\gamma}_k^n$ is a chordal $\SLE_4$, it does not touch the boundary except at its endpoints. Therefore, we have a.s.\ for all $n,k\ge 1$
	\begin{equation}\label{eq cv unif on compact sets simple loop}
		\left((g^{\kappa_m}_{R^{\kappa_m, n}_k})^{-1} \circ  \check{\gamma}_k^{\kappa_m, n}\right) \mathop{\longrightarrow}\limits_{m\to \infty}
		\gamma^n_k \defeq (f^n_k \circ  \check{\gamma}^n_k)
	\end{equation}
	uniformly on compact subsets of $(0,\infty)$. Note that the curve $\gamma^n_k$ is a simple curve since $\check{\gamma}^n_k$ is an $\SLE_4$ curve and since $f^n_k$ is injective.
	
	Now, recall that $K^{\kappa_m, n}$ is the first $k$ such that the corresponding excursion reaches $2\pi$. In other words, the loop $\eta^{\kappa_m}|_{[S^{\kappa_m, n}_{K^{\kappa_m,n}}, T^{\kappa_m, n}_{K^{\kappa_m,n}}]}$ is $\SCL^{\kappa_m}(0)$. Since $K^{\kappa_m,n} \to K^{n}$ a.s.\ as $m\to \infty$, we know that a.s.\ for all $m$ large enough, $K^{\kappa_m, n}=K^{n}$.

	By the Hausdorff convergence of $\SCL^{\kappa_m}(0)$ toward $\SCL^{4}(0)$ and~\eqref{eq cv unif on compact sets simple loop}, we have $\gamma^n_{K^n}(\BR_{\ge 0}) \subseteq \SCL^4(0)$.	Note that because $2^{-n}$ is decreasing, for all $m\ge 1$ and $n\ge 1$, we have $R^{\kappa_m, n+1}_{K^{\kappa_m,n+1}}\le R^{\kappa_m, n}_{K^{\kappa_m,n}}\le T^{\kappa_m, n}_{K^{\kappa_m,n}} = T^{\kappa_m, n+1}_{K^{\kappa_m,n+1}}$, so that  $\eta^{\kappa_m}([R^{\kappa_m, n}_{K^{\kappa_m,n}}, T^{\kappa_m, n}_{K^{\kappa_m,n}}]) \subseteq \eta^{\kappa_m}([R^{\kappa_m, n+1}_{K^{\kappa_m,n+1}}, T^{\kappa_m, n+1}_{K^{\kappa_m,n+1}}])$. 
	
	\stepx{step:ilc-all}{The subsequential limit far from the endpoints gives all of the loop} Let us prove that there exist two sequences of random variables $0<\delta_n<A_n$ such that a.s., the probability that a planar Brownian motion starting from zero hits $\partial\BD$ before hitting $\gamma^n_{K^n}([\delta_n, A_n])$ goes to zero as $n\to \infty$.
	
	To construct the sequence $((\delta_n, A_n))_{n\ge 1}$, we reason as follows. Note that the $\SLE_4$ curve $\widetilde{\gamma}^n_{K^n}\colon [0, \infty] \to \overline{\BD}$ is continuous at $0$ and $\infty$, and recall that $\widetilde{\gamma}^n_{K^n}(0) = 1$ and $\widetilde{\gamma}^n_{K^n}(\infty)= e^{-\ri 2^{-n}}$. Moreover, $\widetilde{\gamma}^n_{K^n}$ separates $0$ from the long arc of $\partial \BD$ between $1$ and $e^{-\ri 2^{-n}}$ (the prelimit curves do, and the limit a.s.\ does not pass through $0$), so the only way for a planar Brownian motion starting at zero to hit $\partial \BD$ before hitting $\widetilde{\gamma}^n_{K^n}$ is to hit the short arc $[e^{-\ri 2^{-n}}, 1]$ before $\widetilde{\gamma}^n_{K^n}$. Since the length of this arc goes to zero as $n\to \infty$, the probability that a planar Brownian motion starting at $0$ hits $\partial \BD$ before hitting $\widetilde{\gamma}^n_{K^n}$ goes to zero as $n\to \infty$. By continuity of $\widetilde{\gamma}^n_{K^n}\colon [0, \infty] \to \overline{\BD}$ at $0$ and $\infty$, we can thus choose some (random) $0<\delta_n<A_n$ so that the probability that a planar Brownian motion starting at $0$ hits $\partial \BD$ before hitting $\widetilde{\gamma}^n_{K^n}([\delta_n, A_n])$ converges to zero as $n\to \infty$. By conformal invariance of planar Brownian motion and by definition of $\gamma^n_{K^n}$ (noting that $f^n_{K^n}(0)=0$), we get that the probability that a planar Brownian motion starting at $0$ hits $\partial f^n_{K^n}(\BD)$ before hitting $\gamma^n_{K^n}([\delta_n, A_n])$ goes to zero as $n\to \infty$. Since $f^n_{K^n}(\BD) \subseteq \BD$, we have that the probability of hitting $\partial \BD$ before hitting $\gamma^n_{K^n}([\delta_n, A_n])$ also goes to zero as $n\to \infty$.

	Therefore, recalling once more that $\gamma^n_{K^n}([\delta_n, A_n])\subset \SCL^4(0)$, we see that a.s.,
	\begin{equation}\label{eq cv Hausdorff gamma loop}
		\gamma^n_{K^n}([\delta_n, A_n]) \mathop{\longrightarrow}\limits_{n\to \infty} \SCL^4(0)
	\end{equation}
	in the Hausdorff topology. Indeed, if it is not the case then there is $\varepsilon>0$ such that for infinitely many $n$, there is an $\varepsilon$-ball centered at a point $x_n$ of $\SCL^4(0)$ which does not intersect $\gamma^n_{K^n}([\delta_n, A_n])$. By compactness, we can assume that $x_n \to x \in \SCL^4(0)$ and $B_{\varepsilon/2}(x) \cap \gamma^n_{K^n}([\delta_n,A_n]) = \emptyset$ for infinitely many $n$; since a planar Brownian motion started at $0$ crosses $\SCL^4(0)$ inside $B_{\varepsilon/2}(x)$ and then hits $\partial \BD$ before returning to $\SCL^4(0)$ with probability $p(x,\varepsilon)>0$, this contradicts the choice of $(\delta_n,A_n)$.
	
	Next, let us prove that $\mathrm{diam}(\gamma^n_{K^n}([ A_n, \infty))\cup\gamma^n_{K^n}([ 0, \delta_n])) \to 0$ a.s.\ as $n\to \infty$. Recall that the loop $\SCL^4(0)$ is simple, so that $\SCL^4(0)\setminus \gamma^n_{K^n}([\delta_n, A_n])$ is a connected simple path. 
	More precisely, let us show that a.s.,
	\begin{equation}\label{eq the complement in the loop goes to zero}
		\mathrm{diam}(\SCL^4(0)\setminus \gamma^n_{K^n}([\delta_n, A_n]))\mathop{\longrightarrow}\limits_{n\to \infty}0.
	\end{equation}
	The loop $\SCL^4(0)$ being a Jordan curve, the inverse of its parameterization is uniformly continuous, so if the arc $\SCL^4(0)\setminus\gamma^n_{K^n}([\delta_n,A_n])$ had diameter at least $c>0$ for infinitely many $n$, its midpoint would be at distance bounded below from $\gamma^n_{K^n}([\delta_n,A_n])$, contradicting~\eqref{eq cv Hausdorff gamma loop}. This shows \eqref{eq the complement in the loop goes to zero}. In particular $|\gamma^n_{K^n}(\delta_n)-\gamma^n_{K^n}(A_n)|\to0$.
	
	\stepx{step:ilc-endpoints}{The parts close to the endpoints of the loop are negligible} Let us show that a.s.,
	\begin{align}
	&\limsup_{m \to \infty} \mathrm{diam}(\eta^{\kappa_m}([S^{\kappa_m,n}_{K^n}, R^{\kappa_m, n}_{K^n}]))\mathop{\longrightarrow}\limits_{n\to \infty}0 \quad\text{and} \label{eq diameter of small parts in the loop}\\
		&\limsup_{m\to \infty} \mathrm{diam}\left(\left((g^{\kappa_m}_{R^{\kappa_m, n}_{K^n}})^{-1} \circ  \check{\gamma}_{K^n}^{\kappa_m, n} \right)\left([0, \delta_n]  \right)\right) \notag\\
		&\quad+\limsup_{m\to \infty} \mathrm{diam}\left(\left((g^{\kappa_m}_{R^{\kappa_m, n}_{K^n}})^{-1} \circ  \check{\gamma}_{K^n}^{\kappa_m, n} \right) \left( [A_n, \infty) \right)\right)
		\mathop{\longrightarrow}\limits_{n\to \infty}0. \label{eq diameter of small parts in the loop2}
	\end{align}
	Let us start with~\eqref{eq diameter of small parts in the loop}. Since $\eta^{\kappa_m}([S^{\kappa_m,n}_{K^n}, R^{\kappa_m, n}_{K^n}])\subset \SCL^{\kappa_m}(0)$, any subsequential limit in the Hausdorff topology of $\eta^{\kappa_m}([S^{\kappa_m,n}_{K^n}, R^{\kappa_m, n}_{K^n}])$ as $m\to \infty$ must be a connected subset of $\SCL^4(0)$. 
	Let $x_n, y_n$ be the two endpoints of a subsequential limit $\SCK_n$ (which can be seen as a closed interval of $\SCL^4(0)$) of $\eta^{\kappa_m}([S^{\kappa_m,n}_{K^n}, R^{\kappa_m, n}_{K^n}])$ for all $n\ge 1$ (if~\eqref{eq diameter of small parts in the loop} failed, such limits could be chosen with $\mathrm{diam}(\SCK_n) \ge c>0$ along a subsequence in $n$). Note that $x_n, y_n \in \SCL^4(0)$. By compactness, we may assume that $x_n \to x \in \SCL^4(0)$ and $y_n \to y\in \SCL^4(0)$ along a subsequence. But, by~\eqref{eq cv Hausdorff gamma loop}, we see that there exists $t_n, t'_n \in [\delta_n, A_n]$ such that $\gamma^n_{K^n}(t_n) \to x$ as $n\to \infty$ and $\gamma^n_{K^n}(t'_n) \to y$ as $n\to \infty$. Assume for contradiction that $x\neq y$. Then, by~\eqref{eq the complement in the loop goes to zero}, there are two cases. 
	\begin{itemize}
		\item If the interval $[x,y]$ in $\SCL^4(0)$ (corresponding to the subsequential limit of $\SCK_n$) is the Hausdorff limit of $\gamma^n_{K^n}([t_n\wedge t'_n, t_n \vee t'_n])$, then this means that for $n$ large enough, there exist $t^n_1< t^n_2$ such that $\gamma^n_{K^n}([t^n_1, t^n_2])\subset \SCK_n$ and has diameter at least $\vert x-y\vert/2$ ($\SCK_n$ and $\gamma^n_{K^n}([t_n\wedge t'_n,t_n\vee t'_n])$ being sub-arcs of the same Jordan curve $\SCL^4(0)$ with endpoints converging to $x$ and $y$, their intersection contains such a sub-arc). This is absurd since $\gamma^n_{K^n}$ is a simple curve which is included (except its endpoints $\gamma^n_{K^n}(0)$, $\gamma^n_{K^n}(\infty)$) in $f^n_{K^n}(\BD)$ which does not intersect $\SCK_n$. Indeed, since $(g^{\kappa_m}_{R^{\kappa_m, n}_{K^n}})^{-1} \to f_{K^n}^n$ uniformly on compact sets, the domain $D^{\kappa_m}_{R^{\kappa_m, n}_{K^n}}$ converges in the Carath\'eodory sense toward $f_{K^n}^n(\BD)$. So if we had $f^n_{K^n}(\BD)\cap \SCK_n \neq \emptyset$, we would have a point $w$ of $\SCK_n$ and $\delta>0$ such that $\overline{B_\delta(w)} \subset f^n_{K^n}(\BD)$, but then this compact set $\overline{B_\delta(w)}$ would be in $D^{\kappa_m}_{R^{\kappa_m, n}_{K^n}}$ for all $m$ large enough, hence contradicting the fact that $\SCK_n$ is a subsequential limit of  $\eta^{\kappa_m}([S^{\kappa_m,n}_{K^n}, R^{\kappa_m, n}_{K^n}])$ (which does not intersect $D^{\kappa_m}_{R^{\kappa_m, n}_{K^n}}$).
		\item Otherwise, the interval $[x,y]$ is the Hausdorff limit of $\gamma^n_{K^n}([\delta_n,t_n \wedge t'_n]) \cup \gamma^n_{K^n}([t_n \vee t'_n, A_n))$ as $n\to \infty$. Then for $n$ large enough, there exist $t^n_1<t^n_2$ such that $\gamma^n_{K^n}([t^n_1, t^n_2]) \subset \SCK_n$ and has diameter at least $\vert x-y\vert/4$ (the denominator $4$ comes from the fact that in this case we cover $[x,y]$ using two intervals, so at least one of them has diameter at least $\vert x-y\vert/4$). We conclude as in the first case.
	\end{itemize}
	Next, let us prove~\eqref{eq diameter of small parts in the loop2}. Let us deal with $[0, \delta_n]$, the same reasoning works for $[A_n, \infty)$ by reversibility of the $\SLE_4$. For all $n$, let $\SCK_n$ be a subsequential limit of $\left((g^{\kappa_m}_{R^{\kappa_m, n}_{K^n}})^{-1} \circ  \check{\gamma}_{K^n}^{\kappa_m, n} \right)\left([0, \delta_n]  \right)$ as $m\to \infty$. Let $x_n, y_n$ be the two endpoints of $\SCK_n$, seen as an interval of $\SCL^4(0)$ and take some subsequential limits $x_n \to x$ and $y_n \to y$. As in the above paragraph, assume by contradiction that $x\neq y$. Then, by~\eqref{eq cv Hausdorff gamma loop}, there exists $t_n, t'_n \in [\delta_n, A_n]$ such that $\gamma^n_{K^n}(t_n) \to x$ as $n\to \infty$ and $\gamma^n_{K^n}(t'_n) \to y$ as $n\to \infty$. Again, we have the same two cases as above. 
	\begin{itemize}
		\item If the interval $[x,y]$ in $\SCL^4(0)$ (corresponding to the subsequential limit of $\SCK_n$) is the Hausdorff limit of $\gamma^n_{K^n}([t_n\wedge t'_n, t_n \vee t'_n])$, then for all $n$ large enough, there exist $t^n_1<t^n_2$ such that $\gamma^n_{K^n}([t_1^n, t_2^n]) \subset \SCK_n$ and  $\mathrm{diam}(\gamma^n_{K^n}([t_1^n, t_2^n])) $ is at least $\vert x-y \vert /2$. But since $\mathrm{diam}(\gamma^n_{K^n}([ A_n, \infty))\cup\gamma^n_{K^n}([ 0, \delta_n])) \to 0$ by~\eqref{eq the complement in the loop goes to zero}, we may assume that $\delta_n \le t^n_1 < t^n_2 \le A_n$. 
		
		Let us take such $n$ large enough. Note that $\SCK'_n \coloneqq \gamma^n_{K^n}([t_1^n, t_2^n]) \subset \SCK_n$ is a compact subset of $f^n_{K^n}(\BD)$ by~\eqref{eq cv unif on compact sets simple loop}. Let $z_n, w_n \in \SCK'_n$ such that $\vert z_n- w_n \vert \ge \vert x-y\vert /2$. Since $\SCK'_n \subseteq \SCK_n$, we have, along a subsequence, some $u^m_n, v^m_n \in [0, \delta_n]$ such that 
		\[
		\left((g^{\kappa_m}_{R^{\kappa_m, n}_{K^n}})^{-1} \circ  \check{\gamma}_{K^n}^{\kappa_m, n}\right) (u^m_n)
		\mathop{\longrightarrow}\limits_{m\to \infty} z_n \qquad \text{and}\qquad
		\left((g^{\kappa_m}_{R^{\kappa_m, n}_{K^n}})^{-1} \circ  \check{\gamma}_{K^n}^{\kappa_m, n}\right) (v^m_n)
		\mathop{\longrightarrow}\limits_{m\to \infty} w_n.
		\]
		By the uniform convergence of $g^{\kappa_m}_{R^{\kappa_m, n}_{K^n}}$ toward $(f^n_{K^n})^{-1}$ on compact subsets of $f^n_{K^n}(\BD)$ and by the convergence of $W^{\kappa_m}_{R^{\kappa_m,n}_{K^n}}$ to $X^n_{K^n}$, we deduce that along a subsequence,
		\begin{align*}
		(W^{\kappa_m}_{R^{\kappa_m,n}_{K^n}})^{-1} \cdot  \check{\gamma}_{K^n}^{\kappa_m, n}(u^m_n)
		&= 
		(W^{\kappa_m}_{R^{\kappa_m,n}_{K^n}})^{-1} \cdot g^{\kappa_m}_{R^{\kappa_m, n}_{K^n}}\left(\left((g^{\kappa_m}_{R^{\kappa_m, n}_{K^n}})^{-1} \circ  \check{\gamma}_{K^n}^{\kappa_m, n}\right) (u^m_n)\right)\\
		&\mathop{\longrightarrow}\limits_{m\to \infty} (X^n_{K^n})^{-1} \cdot (f^n_{K^n})^{-1}( z_n).
		\end{align*}
		But since $z_n, w_n \in \SCK'_n$, there exist $s_1^n, s_2^n \in [t^n_1, t^n_2]$ such that $z_n = \gamma^n_{K^n}(s_1^n)$ and $w_n = \gamma^n_{K^n}(s_2^n)$. By definition of $\gamma^n_{K^n}$, 
		\[
		(X^n_{K^n})^{-1} \cdot (f^n_{K^n})^{-1}( z_n)= (X^n_{K^n})^{-1} \cdot \check{\gamma}^n_{K^n}(s^n_1)  = \widetilde{\gamma}^n_{K^n}(s^n_1),
		\] 
		and similarly
		\[
		(X^n_{K^n})^{-1} \cdot (f^n_{K^n})^{-1}( w_n)= (X^n_{K^n})^{-1} \cdot \check{\gamma}^n_{K^n}(s^n_2)  = \widetilde{\gamma}^n_{K^n}(s^n_2).
		\] 
		Moreover, since $u^m_n, v^m_n \in [0, \delta_n]$, along a (sub-)subsequence, we may assume that $u^m_n \to u_n$  and $v^m_n \to v_n$ as $m \to \infty$. Thus, by~\eqref{eq cv SLE kappa to SLE 4}, along this (sub-)subsequence,
		\[
		(W^{\kappa_m}_{R^{\kappa_m,n}_{K^n}})^{-1} \cdot  \check{\gamma}_{K^n}^{\kappa_m, n}(u^m_n) \mathop{\longrightarrow}\limits_{m\to \infty} \widetilde{\gamma}^n_{K^n}(u_n)
		\qquad \text{and} \qquad
		(W^{\kappa_m}_{R^{\kappa_m,n}_{K^n}})^{-1} \cdot  \check{\gamma}_{K^n}^{\kappa_m, n}(v^m_n) \mathop{\longrightarrow}\limits_{m\to \infty} \widetilde{\gamma}^n_{K^n}(v_n).
		\]
		Since $\widetilde{\gamma}^n_{K^n}$ is a simple curve in $\BD$ from $1$ to $e^{-\ri 2^{-n}}$, we have necessarily $s_1^n=u_n = \delta_n$ and $s_2^n=v_n= \delta_n$. This is absurd since it implies that $z_n= w_n$.
		\item Otherwise, the interval $[x,y]$ is the Hausdorff limit of $\gamma^n_{K^n}([\delta_n,t_n \wedge t'_n]) \cup \gamma^n_{K^n}([t_n \vee t'_n, A_n))$ as $n\to \infty$. Then for $n$ large enough, there exist $t^n_1<t^n_2$ such that $\gamma^n_{K^n}([t^n_1, t^n_2]) \subset \SCK_n$ and has diameter at least $\vert x-y\vert/4$. We conclude as in the first case.
	\end{itemize}
	This ends the proof of~\eqref{eq diameter of small parts in the loop2}.

	\stepx{step:ilc-construction}{Construction of $\gamma_m$} We take a (random) integer $n$, chosen large enough below, and set for all $t \in [0, 1]$, \[\gamma_m(t)\defeq\left((g^{\kappa_m}_{R^{\kappa_m, n}_{K^n}})^{-1} \circ  \check{\gamma}_{K^n}^{\kappa_m, n}\right) (	\delta_n+ (A_n-\delta_n)t).\]
	By~\eqref{eq cv unif on compact sets simple loop}, we know that $\gamma_m$ converges uniformly on $[0,1]$ as $m \to \infty$ a.s. Furthermore, we have
	\begin{align*}
		\SCL^{\kappa_m}(0) =& \eta^{\kappa_m}([S^{\kappa_m, n}_{K^{\kappa_m, n}}, T^{\kappa_m,n}_{K^{\kappa_m, n}}]) \\
		=&  \eta^{\kappa_m}([S^{\kappa_m, n}_{K^{\kappa_m, n}}, R^{\kappa_m,n}_{K^{\kappa_m, n}}]) \cup \left((g^{\kappa_m}_{R^{\kappa_m, n}_{K^{\kappa_m, n}}})^{-1} \circ  \check{\gamma}_{K^{\kappa_m, n}}^{\kappa_m, n}\right)([0, \delta_n])\\
		&\cup \left((g^{\kappa_m}_{R^{\kappa_m, n}_{K^{\kappa_m, n}}})^{-1} \circ  \check{\gamma}_{K^{\kappa_m, n}}^{\kappa_m, n}\right)([ A_n, \infty))\cup
		\gamma_m([0,1]) 
	\end{align*}
	By~\eqref{eq diameter of small parts in the loop} and~\eqref{eq diameter of small parts in the loop2}, we know that the diameters of the first three terms of the union are a.s.\ small, provided $m$ is large enough and $n$ is chosen large enough (possibly random).	Notice that their union is connected (using $\eta^{\kappa_m}(T^{\kappa_m,n}_{K^{\kappa_m,n}}) = \eta^{\kappa_m}(S^{\kappa_m,n}_{K^{\kappa_m,n}})$), so that the diameter of the union of the first three terms is a.s.\ small, provided $m$ is large enough and $n$ is chosen large enough (possibly random). Set $Z_m \defeq Z^{\kappa_m}$. Since $Z_m$ is the first point of the loop hit by the trunk, it is the root of the loop \cite{CLEPerc}, meaning $Z_m = \eta^{\kappa_m}(S^{\kappa_m, n}_{K^{\kappa_m, n}})$. Consequently, $Z_m$ belongs to the first term of the union above. Because the union of the first three terms is connected, contains $Z_m$, and has an arbitrarily small diameter, it is entirely contained in $B_\varepsilon(Z_m)$ for all $m$ large enough and $n$ chosen large enough. Thus, $\SCL^{\kappa_m}(0) \subset B_\varepsilon(Z_m) \cup \gamma_m([0,1])$, and $(Z_m,\gamma_m)$ converges a.s.\ by~\eqref{eq cv unif on compact sets simple loop}. This ends the proof.
\end{proof}

\subsubsection{Convergence of large loops}
\label{subsubsec:big_loops_convergence}
Next, let us show the convergence of the loops of diameter larger than $\varepsilon>0$. Before doing so, we rule out the existence of thin loops in the following lemma, in which we write $B_\varepsilon(A)$ for the $\varepsilon$-neighborhood of a set $A$.
\begin{lemma}\label{lemma no thin simple loop}
	For all $\kappa_m \uparrow 4$, for all $x \in \BD$, for all $\beta>0$,
	\[
	\lim_{\varepsilon \to 0} \sup_{m\ge 1} \BP[\exists \SCL \in \Gamma^{\kappa_m} \setminus \{\SCL^{\kappa_m}(x)\}, \ \SCL \subset B_\varepsilon(\SCL^{\kappa_m}(x)) \text{ and } \mathrm{diam}(\SCL)\ge \beta]=0.
	\]
\end{lemma}
\begin{proof}
	\stepx{step:thin-reduction}{Reduction to the case $m_n \to \infty$} Let $\mathcal{A}^m(\varepsilon)$ be the above event. Assume for a contradiction that there exist $\delta>0$, $\varepsilon_n\downarrow 0$ and $m_n\ge 1$ such that for all $n\ge 1$,
	\[
	\BP[\mathcal{A}^{m_n}(\varepsilon_n)]\ge \delta.
	\]
	By taking a subsequence, we may assume that $m_n \to m \in \BZ_{\ge 1} \cup \{\infty\}$. If $m<\infty$, then $m_n=m$ for all $n$ large enough. We then get a contradiction since the loops of the $\CLE_{\kappa_m}$ do not touch each other and $\CLE_{\kappa_m}$ is locally finite a.s.
	
	\stepx{step:thin-inversion}{The inverted configuration} If $m_n \to \infty$, then we recall from~\Cref{lemma cv Hausdorff simple loops} that the loop $\SCL^{\kappa_{m_n}}(x)$ converges in law in the Hausdorff topology to $\SCL^4(x)$. Let $K_n$ (resp.\ $K$) be the closure of the domain encircled by $\SCL^{\kappa_{m_n}}(x)$ (resp.\ $\SCL^4(x)$). Moreover, the restriction of $\Gamma^{\kappa_{m_n}}$ to $\BD \setminus K_n$ has the law of a $\CLE_{\kappa_{m_n}}$ in $\BD \setminus K_n$ in the sense of~\cite{SimCLEDblConnDom}: by the Poissonian structure of the loop soup, given $K_n$ the remaining loops form a loop soup in $\BD \setminus K_n$ conditioned on no cluster surrounding $K_n$.
	
	Let $Z_n$ be an independent uniform random point in $\BD \setminus K_n$. By taking a subsequence again denoted by $m_n$ and applying Skorokhod's representation theorem, we may assume that $K_n\to K$ a.s.\ in the Hausdorff topology (by~\Cref{rem:interior membership}) and that $Z_n \to Z$ a.s.\ (using $\mathrm{Leb}(K_n \,\triangle\, K) \to 0$, as in the proof of~\Cref{cor uniform exploration from x}), where $Z$ is uniform in $\BD \setminus K$ conditionally on $K$. Let $f_n \colon \BD  \setminus K_n \to ({1}/{r_n})\BD \setminus r_n \overline{\BD}$ be the unique conformal map (for some unique $r_n\in (0, 1)$) such that $f_n'(Z_n)>0$ and $f_n(\partial K_n)=r_n\partial \BD$. Since $K_n\to K$ in the Hausdorff topology, we deduce that $\BD \setminus K_n \to \BD \setminus K$ in the Carath\'eodory sense. Note that $\BD \setminus K$ is a.s.\ a topological annulus. By \cite[Theorem~3.7]{Com13}, we deduce that $r_n \to r$ and $f_n \to f$ uniformly on compact sets of $\BD \setminus K$, where $f\colon \BD \setminus K \to({1}/{r}) \BD \setminus r \overline{\BD}$ is the unique conformal map such that $f(\partial K)=r\partial \BD$ and $f'(Z)>0$.
	
	Recall that we denote by $\iota \colon z \mapsto 1/z$ the inversion. Note that conditionally on $K_n$, by inversion invariance of the CLE \cite{SimCLEDblConnDom}, $\iota(f_n(\Gamma^{\kappa_{m_n}} \vert_{\BD \setminus K_n}))$ has the law of a CLE$_{\kappa_{m_n}}$ in $(1/r_n)\BD \setminus r_n \BD$. In particular, the set 
		\[
		\widetilde{\Gamma}^{\kappa_{m_n}} \defeq \{\SCL^{\kappa_{m_n}}(x) \} \cup (f_n)^{-1} \left(\iota(f_n(\Gamma^{\kappa_{m_n}} \vert_{\BD \setminus K_n}))\right)
		\]
		has the same law as $\Gamma^{\kappa_{m_n}}$.
	
	\stepx{step:thin-diameter}{The inverted loop has a macroscopic diameter} Conditionally on $K$, let $p$ be the infimum over connected subsets $\eta \subset \partial K=\SCL^4(x)$ with diameter at least $\beta$ of the probability that an independent planar Brownian motion starting from $Z$ and stopped when it first hits $\partial \BD \cup K$ hits $\eta$ ($p>0$ a.s.: the harmonic measure of a sub-arc of the Jordan curve $\partial K$ seen from $Z$ is a continuous positive function of its endpoints, and the set of sub-arcs of diameter at least $\beta$ is compact). On the event $\mathcal{A}^{m_n}(\varepsilon_n)$, let $\SCL_n$ be a loop of diameter at least $\beta$ which is included in the $\varepsilon_n$-neighborhood of $K_n$. Conditionally on $K$ and $\SCL_n$, let $p_n$ be the probability that a planar Brownian motion starting from $Z_n$ stopped when it first hits $\partial \BD \cup K_n$ hits $\SCL_n$. Note that since $Z_n \to Z$ and $K_n \to K$, and since $\SCL_n$ is included in the $\varepsilon_n$-neighborhood of $K_n$, for all $\varepsilon'>0$, a.s., for all $n$ large enough, 
	\begin{equation*}
		p_n {\bf 1}_{\mathcal{A}^{m_n}(\varepsilon_n)} \ge (p-\varepsilon'){\bf 1}_{\mathcal{A}^{m_n}(\varepsilon_n)}.
	\end{equation*}
	By conformal invariance of the planar Brownian motion, the probability $p_n$ is the probability that a planar Brownian motion starting from $f_n(Z_n)$ and stopped at $\partial ((1/r_n)\BD) \cup r_n\overline{\BD}$ hits $f_n(\SCL_n)$. Again by conformal invariance, $p_n$ is the probability that a planar Brownian motion starting from $(f_n)^{-1}(\iota(f_n(Z_n)))$ and stopped at $\partial \BD \cup K_n$ hits $(f_n)^{-1}(\iota(f_n(\SCL_n)))$. Moreover, $f_n^{-1} \circ \iota \circ f_n \to f^{-1} \circ \iota \circ f$ uniformly on compact subsets of $\BD \setminus K$ so that $(f_n)^{-1}(\iota(f_n(Z_n)))$ converges to $f^{-1}(\iota(f(Z)))$. In particular, a.s., there exists $\varepsilon''>0$ such that for all $n$ large enough, $\mathrm{dist}((f_n)^{-1}(\iota(f_n(Z_n))), (f_n)^{-1}(\iota(f_n(\SCL_n))))\ge \varepsilon''$.
	
	Therefore, on the event $\mathcal{A}^{m_n}(\varepsilon_n)$, the diameter of the loop $(f_n)^{-1}(\iota(f_n(\SCL_n)))$ of $\widetilde{\Gamma}^{\kappa_{m_n}} $ stochastically dominates a random variable $Y>0$ (which does not depend on $n$) (a set at distance at least $\varepsilon''$ from the starting point that is hit with probability at least $p/2$ has diameter bounded below in terms of $p$ and $\varepsilon''$). Moreover, for all $\varepsilon>0$ and $n$ large, on $\mathcal{A}^{m_n}(\varepsilon_n)$ the loop $(f_n)^{-1}(\iota(f_n(\SCL_n)))$ lies in $B_\varepsilon(\partial\BD)$: since $\SCL_n$ is within distance $\varepsilon_n$ of the connected set $K_n$, by Beurling's estimate and conformal invariance a Brownian motion started from any point of $(f_n)^{-1}(\iota(f_n(\SCL_n)))$ hits $K_n$ before $\partial\BD$ with probability at most $C\sqrt{\varepsilon_n}$, whereas from any point at distance at least $\varepsilon$ from $\partial\BD$ this probability is bounded below (by that of hitting a fixed ball $B \subset \mathrm{int}(K)$, which lies in $K_n$ for $n$ large).
	
	\stepx{step:thin-contradiction}{Conclusion} Hence, if $\BP[Y\ge y]\ge 1-\delta/4$ with $y>0$, then for every $\varepsilon>0$ and $n$ large, with probability at least $\delta/2$, $\widetilde\Gamma^{\kappa_{m_n}}$, which has the law of $\Gamma^{\kappa_{m_n}}$, contains a loop of diameter at least $y$ inside $B_\varepsilon(\partial\BD)$. In the loop-soup coupling with a $\CLE_4$, such a loop is surrounded by a $\CLE_4$ loop of diameter at least $y$ meeting $B_\varepsilon(\partial\BD)$, whose probability does not depend on $n$ and tends to $0$ as $\varepsilon\to0$ (a.s.\ finitely many $\CLE_4$ loops have diameter at least $y$, and none of them touches $\partial\BD$). This is a contradiction.
\end{proof}
We are now able to show the convergence of the large loops. 

\begin{proof}[Proof of~\Cref{prop:big_loops_are_fat_simple_cle}]
	Recall that we couple the $\Gamma^\kappa$'s with increasing Brownian loop soups. In particular, any loop of $\Gamma^\kappa$ with diameter larger than $\varepsilon$ is encircled by a loop of $\Gamma^4$ with diameter larger than $\varepsilon$. Let $\kappa_n \uparrow 4$.
	
	By~\Cref{lemma cv Hausdorff simple loops} and~\Cref{lemma improved convergence loops} (combined with conformal invariance), we know that for all $x\in \BD$, the loop $\SCL^{\kappa_n}(x)$ converges in distribution to $\SCL^4(x)$ in the Hausdorff topology, but also in the more precise sense of~\Cref{lemma improved convergence loops} as a continuous closed curve. Using the fact that in our coupling, $\mathrm{int}(\SCL^{\kappa_n}(x))$ is increasing in $n$ for inclusion, we can identify the law of the limit (the increasing union $U$ of the $\mathrm{int}(\SCL^{\kappa_n}(x))$ is contained in $\mathrm{int}(\SCL^4(x))$ and its conformal radius seen from $x$ is the limit of those of the $\mathrm{int}(\SCL^{\kappa_n}(x))$, hence has the law of that of $\mathrm{int}(\SCL^4(x))$ by~\Cref{lemma cv Hausdorff simple loops} and~\Cref{rem:interior membership}; by the Schwarz lemma, $U=\mathrm{int}(\SCL^4(x))$ a.s.) so that a.s., for all $x \in \BD_\BQ$,
	\begin{equation}\label{eq increasing union loop}
		\bigcup_{n\ge 1} \mathrm{int}(\SCL^{\kappa_n}(x))= \mathrm{int}(\SCL^4(x)),
	\end{equation}
	where the union is increasing.
	
	Let $X_1, \ldots, X_N$ in $\BD_\BQ$ be such that the loops of $\Gamma^4$ with diameter larger than $\varepsilon$ are $\SCL^4(X_1),\allowbreak \ldots,\allowbreak \SCL^4(X_N)$. Let $N_n$ be the number of loops of $\Gamma^{\kappa_n}$ with diameter larger than $\varepsilon$. Note that by~\eqref{eq increasing union loop}, a.s.\ for all $n$ large enough, we have $N_n\ge N$.
	
	Let us show the converse inequality. Let $\SCL_n$ be a sequence of loops of $\Gamma^{\kappa_n}$ with diameter larger than $\varepsilon$. For all $n\ge 1$, let $i_n \in [1, N]_\BZ$ be the unique $i$ such that $\SCL_n$ is encircled by $\SCL^4(X_i)$. By~\eqref{eq increasing union loop} and~\Cref{rem:interior membership} (the loop $\SCL_n$ lies in $\mathrm{int}(\SCL^4(X_{i_n}))\setminus\mathrm{int}(\SCL^{\kappa_n}(X_{i_n}))$, hence within distance $\varepsilon'$ of $\SCL^4(X_{i_n})$ for $n$ large), for all $\varepsilon'>0$, a.s.\ for all $n$ large enough, the loop $\SCL_n$ is in the $2\varepsilon'$-neighborhood of $\SCL^{\kappa_n}(X_{i_n})$. By~\Cref{lemma no thin simple loop} (we take $X_i$ to be the first point of $\mathrm{int}(\SCL^4(X_i))$ in a fixed enumeration of $\BD_\BQ$, so that $X_1,\ldots,X_N$ lie in a fixed finite set with probability close to $1$, and apply the lemma to each point of that set), the probability that there exists a loop $\SCL_n \neq \SCL^{\kappa_n}(X_{i_n})$ satisfying this inclusion tends to $0$ as $n \to \infty$.  Thus,
	\[
	N_n \mathop{\longrightarrow}\limits_{n\to \infty}^{(\BP)} N.
	\]
	This proves the first part of the proposition. The second part comes from the fact that the $\SCL^4(X_i)$'s for $1 \le i\le N$ encircle Euclidean balls centered at $X_i$ of random radii $R_i>0$. Then, by \eqref{eq increasing union loop}, we deduce that with probability $1-o(1)$ as $n\to \infty$, the $\SCL^{\kappa_n}(X_i)$'s surround such Euclidean balls, hence the desired result (for the finitely many $n$ below a large threshold, a small enough $\delta$ works by local finiteness of $\CLE_{\kappa_n}$).
\end{proof}

\subsubsection{Convergence of the crossings of the annulus}
\label{subsubsec:annulus_segments}

Next, let us prove the convergence of the crossings of the annulus $A$.
\begin{proposition}\label{prop cv crossing segments}
	The number $N^\kappa$ of crossings of $A$ by loops in $\Gamma^\kappa_A$ converges in law as $\kappa \uparrow 4$ to the number $N^4$ of crossings of $A$ by loops in $\Gamma^4_A$. More precisely, for all $\kappa_n \uparrow 4$ there exists an ordering $\eta_1^n, \ldots, \eta^n_{N^{\kappa_n}}$ of the crossings of $A$ by loops in $\Gamma^{\kappa_n}_A$ for which they converge in law to $\eta_1, \ldots, \eta_{N^4}$, which are the crossings of $A$ by loops in $\Gamma^4_A$, with respect to the Hausdorff topology.
\end{proposition}
\begin{proof}
	Let $\kappa_n \uparrow 4$.  Let us work under the same coupling as in the proof of~\Cref{prop:big_loops_are_fat_simple_cle}. For all $x \in \BD_\BQ$, recall that we denote by $\SCL^{\kappa_n}(x)$ (resp.\ $\SCL^4(x)$) the loop of $\Gamma^{\kappa_n}$ (resp.\ $\Gamma^{4}$) which encircles $x$. 
	
	Note that the loops of $\Gamma^{\kappa_n}$ crossing $A$ have diameter at least $t-s$. Let $\varepsilon'>0$ and let $\delta>0$ be given by~\Cref{prop:big_loops_are_fat_simple_cle} for $\varepsilon=t-s$, so that for all $n$, with probability at least $1-\varepsilon'$, the loops of $\Gamma^{\kappa_n}$ crossing $A$ encircle a point of $G\defeq(\delta/2)\BZ^2\cap\BD$; shrinking $\delta$, the same holds for $\Gamma^4$. For $x,y\in G$, a.s.\ for $n$ large, $\SCL^{\kappa_n}(x)$ crosses $A$ if and only if $\SCL^4(x)$ does, and $\SCL^{\kappa_n}(x)=\SCL^{\kappa_n}(y)$ if and only if $\SCL^4(x)=\SCL^4(y)$ (by~\eqref{eq increasing union loop} and~\Cref{rem:interior membership}, a.s.\ no loop of $\Gamma^4$ being tangent to $\partial A$). Hence, with probability at least $1-2\varepsilon'-o(1)$, the loops of $\Gamma^{\kappa_n}$ crossing $A$ are $\SCL^{\kappa_n}(X_1),\ldots,\SCL^{\kappa_n}(X_N)$, where $\SCL^4(X_1),\ldots,\SCL^4(X_N)$ are those of $\Gamma^4$ crossing $A$. Finally, the number of crossings of $A$ by $\SCL^{\kappa_n}(X_i)$ converges in law to that by $\SCL^4(X_i)$: along any subsequence, \Cref{lemma improved convergence loops} gives a further subsequence and a coupling in which $\SCL^{\kappa_m}(X_i)$ is, up to a piece of diameter at most $\varepsilon$, a uniformly converging curve, and $\SCL^4(X_i)$ is a.s.\ not tangent to $\partial A$. The second part stems from the Hausdorff convergence of the loops.
\end{proof}

\begin{proof}[Proof of~\Cref{prop:conformal_rectangles_tight}]
	We work using the same coupling as in the proof of~\Cref{prop cv crossing segments}. The first item is already shown in~\Cref{prop cv crossing segments}. Let us prove the convergence of $\overline{\bigcup_{\SCL \in \Gamma^{\kappa_n, \mathrm{out}}_A} \mathrm{int}(\SCL)}$; the second one is proved using the same reasoning. Let $\varepsilon>0$. Let $K^n_\varepsilon$ (resp.\ $K_\varepsilon$) be the closure of the union of the outer boundary of the annulus $A$ together with the $\mathrm{int}(\SCL)$'s for $\SCL \in \Gamma^{\kappa_n, \mathrm{out}}_A$ (resp.\ $\SCL \in \Gamma^{4, \mathrm{out}}_A$) which have a diameter at least $\varepsilon$. Note that the Hausdorff distance between $K^n_\varepsilon$ (resp.\ $K_\varepsilon$) and $\overline{\bigcup_{\SCL \in \Gamma^{\kappa_n, \mathrm{out}}_A}\mathrm{int}(\SCL)}$ (resp.\ $\overline{\bigcup_{\SCL \in \Gamma^{4, \mathrm{out}}_A} \mathrm{int}(\SCL)}$) is at most $\varepsilon$. Therefore, it suffices to show that $K^n_\varepsilon \to K_\varepsilon$ in probability.
	
	By~\Cref{prop:big_loops_are_fat_simple_cle}, we know that for all $\varepsilon'>0$, with probability at least $1- \varepsilon'$, all the loops with diameter at least $\varepsilon$ encircle a Euclidean ball of radius $\delta$. By considering a grid $(\delta/2)\BZ^2 \cap \BD$ and using~\Cref{lemma cv Hausdorff simple loops} (together with the same coupling as in~\Cref{prop cv crossing segments}), we obtain the Hausdorff convergence of all the loops with diameter at least $\varepsilon$ in probability. We also get the convergence $K^n_\varepsilon \to K_\varepsilon$ in probability and this ends the proof of the second item of the proposition. Finally, (iii) follows from (i) and (ii): a compact subset of a limiting rectangle is at positive distance from the limiting bounding sets, hence lies in a prelimit rectangle for $n$ large, and a connected open set lying in prelimit rectangles for infinitely many $n$ meets none of the limiting bounding sets; the corners converge with the crossings. The moduli are tight since a.s.\ each limiting rectangle contains a channel of fixed width joining its top and bottom sides and one joining its left and right sides; slightly shortened, these lie in the prelimit rectangles for $n$ large and still reach their sides, and their extremal lengths bound the moduli.
\end{proof}

\subsection{Coupling of the GFF with simple and non-simple CLEs}\label{sec:cle-percolation}

Before proceeding to the proof of~\Cref{prop:thin-loops}, let us review the coupling between the simple and non-simple CLEs and the GFF that we will consider. We start by recalling the coupling between BCLEs and GFF which is detailed in \cite[Section 8]{CLEPerc} and which is summarized in \cite[Tables 1 and 2]{CLEPerc}.

Let $\kappa \in (2, 4]$ and $\rho \in (-2, \kappa - 4]$, where $\rho = \kappa-4$ is understood in the degenerate sense that the $\cwBCLE_\kappa(\kappa-4)$ consists of the single true loop given by the domain boundary (cf.~\Cref{re:critical_bcle}). Let $\Gamma$ be a branching $\SLE_\kappa(\rho; \kappa - 6 - \rho)$ process starting from $0$ and targeting all other boundary points. If we equip $\Gamma$ with the orientation from $0$ toward other boundary points, then $\Gamma$ is referred to as a \emph{$\cwBCLE_\kappa(\rho)$}. Moreover, the clockwise (resp.\ counterclockwise) loops formed from $\Gamma$ are referred to as the \emph{true} loops (resp.\ \emph{false} loops) of the $\cwBCLE_\kappa(\rho)$. We shall refer to a $\cwBCLE_\kappa(\kappa - 6 - \rho)$ as a \emph{$\ccwBCLE_\kappa(\rho)$}; to the true loops (resp.\ false loops) of $\cwBCLE_\kappa(\kappa - 6 - \rho)$ as the \emph{false} loops (resp.\ \emph{true} loops) of $\ccwBCLE_\kappa(\rho)$.

\begin{remark}\label{re:GFF-coupling-simple-BCLE}
	Let $\Psi$ be a GFF on $\BH$ with boundary values $-\lambda(1 + \rho)$ (resp.\ $\lambda(1 + (\kappa - 6 - \rho)) = -\lambda(1 + \rho) - 2\pi\chi$) on $\BR_{<0}$ (resp.\ $\BR_{>0}$), where $\lambda=\pi/\sqrt{\kappa}$ and $\chi= 2/\sqrt{\kappa} - \sqrt{\kappa}/2$. Then the branching flow line $\Gamma$ of $\Psi$ starting from $0$ and targeting all other boundary points is a branching $\SLE_\kappa(\rho; \kappa - 6 - \rho)$ process starting from $0$ and targeting all other boundary points, hence a $\cwBCLE_\kappa(\rho)$. Moreover, for any true (resp.\ false) loop $\SCL$, if we write $x$ for the first (resp.\ last) point on $\SCL$ and $y$ for the last (resp.\ first) point on $\SCL$ visited by the loop from $0$ to $0$ that traces the domain boundary $\partial \BH =\BR \cup \{\infty\}$ counterclockwise and $\varphi$ for a conformal mapping from $\BH$ onto the interior of $\SCL$ with $\varphi(0) = x$ and $\varphi(\infty)= y$, then $\Psi \circ \varphi - \chi\arg(\varphi^\prime)$ has the law of a GFF on $\BH$ with boundary values
	\begin{center}
		\begin{tabular}{c|c|c}
			$\SCL$ & $\BR_{<0}$ & $\BR_{>0}$ \\
			\hline
			true & $\lambda$ & $\lambda - 2\pi\chi$ \\
			\hline
			false & $-\lambda + 2\pi\chi$ & $-\lambda$
		\end{tabular}
	\end{center}
\end{remark}

In a similar vein, let $\kappa^\prime \in (4, 8)$ and $\rho^\prime \in [\kappa^\prime / 2 - 4, \kappa^\prime / 2 - 2]$. Let $\Gamma^\prime$ be a branching $\SLE_{\kappa^\prime}(\rho^\prime; \kappa^\prime - 6 - \rho^\prime)$ process starting from $0$ and targeting all other boundary points. If we equip $\Gamma^\prime$ with the orientation from $0$ toward other boundary points, then we shall refer to $\Gamma^\prime$ as a \emph{$\cwBCLE_{\kappa^\prime}(\rho^\prime)$}. Moreover, the clockwise (resp.\ counterclockwise) loops formed from $\Gamma^\prime$ are referred to as the \emph{true} loops (resp.\ \emph{false} loops) of the $\cwBCLE_{\kappa^\prime}(\rho^\prime)$. We shall refer to a $\cwBCLE_{\kappa^\prime}(\kappa^\prime - 6 - \rho^\prime)$ as a \emph{$\ccwBCLE_{\kappa^\prime}(\rho^\prime)$}; to the true loops (resp.\ false loops) of $\cwBCLE_{\kappa^\prime}(\kappa^\prime - 6 - \rho^\prime)$ as the \emph{false} loops (resp.\ \emph{true} loops) of $\ccwBCLE_{\kappa^\prime}(\rho^\prime)$.

\begin{remark}\label{re:GFF-coupling-nonsimple-BCLE}
	Let $\Psi$ be a GFF on $\BH$ with boundary values $\lambda^\prime(1 + \rho^\prime)$ (resp.\ $-\lambda^\prime(1 + (\kappa' - 6 - \rho^\prime)) = \lambda^\prime(1 + \rho^\prime) - 2\pi\chi$) on $\BR_{<0}$ (resp.\ $\BR_{>0}$), where $\lambda'= \pi/\sqrt{\kappa'}$ and we recall that $\chi= 2/\sqrt{\kappa}-\sqrt{\kappa}/2$ with $\kappa=16/\kappa'$. Then the branching counterflow line $\Gamma^\prime$ of $\Psi$ starting from $0$ and targeting all other boundary points is a branching $\SLE_{\kappa^\prime}(\rho^\prime; \kappa^\prime - 6 - \rho^\prime)$ process starting from $0$ and targeting all other boundary points, hence a $\cwBCLE_{\kappa^\prime}(\rho^\prime)$. Moreover, write $\eta^\prime$ for the loop from $0$ to $0$ that traces the domain boundary counterclockwise except that each time we encounter a true loop of $\Gamma^\prime$ for the first time, we trace the entire loop clockwise before continuing. Then for any connected component $D$ of the interior of a true (resp.\ false) loop $\SCL^\prime$, if we write $x$ for the first point (resp.\ last point) on $\partial D$ visited by $\eta^\prime$, and let $\varphi$ be a conformal mapping from $\BH$ onto $D$ which maps $0$ to $x$ and $\infty$ to any other fixed point $y \in \partial D \setminus \{x\}$, then $\Psi \circ \varphi - \chi\arg(\varphi^\prime)$ has the law of a GFF on $\BH$ with boundary values 
	\begin{center}
		\begin{tabular}{c|c|c}
			$\SCL^\prime$ & $\BR_{<0}$ & $\BR_{>0}$ \\
			\hline
			true & $-\lambda^\prime$ & $-\lambda^\prime - 2\pi\chi$ \\
			\hline
			false & $\lambda^\prime + 2\pi\chi$ & $\lambda^\prime$
		\end{tabular}
	\end{center}
\end{remark}

Note that by ``forgetting the orientations'', a $\cwBCLE_\kappa(\rho)$ and a $\ccwBCLE_\kappa(\rho)$ have the same law, which we shall simply denote by \emph{$\BCLE_\kappa(\rho)$}. In a similar vein, we shall use the notation \emph{$\BCLE_{\kappa^\prime}(\rho^\prime)$}.

We shall use the following result.

\begin{theorem}(\cite[Theorem~7.4]{CLEPerc} and \cite[Theorem~1.6]{SimCLELQG})\label{A021}
	Let $\kappa \in (8/3, 4)$ and set $\kappa^\prime \defeq 16/\kappa \in (4,6)$. Let $\rho^\prime \in [\kappa^\prime - 6, 0]$. Set $\rho_\BR \defeq -(\kappa / 4)(\rho^\prime + 2)$ and $\rho_\rL \defeq \kappa / 2 - 4 - \rho_\BR$. Let $\Gamma^\prime$ be a $\cwBCLE_{\kappa^\prime}(\rho^\prime)$. Inside of each connected component of each true (resp.\ false) loop of $\Gamma^\prime$, we sample an independent $\ccwBCLE_\kappa(\rho_\BR)$ (resp.\ $\cwBCLE_\kappa(\rho_\rL)$), and we write $\Gamma$ for the collection of all the true loops of all these simple BCLEs. Write $\eta^\prime$ for the loop from $0$ to $0$ that traces the domain boundary counterclockwise except that each time we encounter a true loop of $\Gamma^\prime$ for the first time, we trace the entire loop before continuing. Write $\eta$ for the loop from $0$ to $0$ that traces $\eta^\prime$ except that each time we encounter a true loop of $\Gamma$ for the first time, we trace the entire loop (clockwise if the loop lies to the left of $\eta^\prime$ and counterclockwise if it lies to the right of $\eta^\prime$) before continuing. For each $x \in \partial\BH$, write $\eta_x^\prime$ (resp.\ $\eta_x$) for the curve $\eta^\prime$ (resp.\ $\eta$) parameterized by the half-plane capacity seen from $x$ (note that $\eta'_x$ and $\eta_x$ are then curves from $0$ to $x$). Let $\beta = \beta(\kappa^\prime, \rho^\prime) \in [-1, 1]$ be such that
	\begin{equation*}
		\frac{1 - \beta}2 = \frac{\sin(\pi\rho^\prime / 2)}{\sin(\pi\rho^\prime / 2) - \sin(\pi(\kappa^\prime - \rho^\prime) / 2)}.
	\end{equation*}
	Then for each $x \in \partial\BH$, $\eta_x$ has the law of an $\SLE_\kappa^\beta(\kappa - 6)$ targeting $x$ and $\eta_x^\prime$ is the trunk of $\eta_x$; moreover, $\eta_x$ is equal to the path that traces $\eta_x^\prime$ except that each time we encounter a true loop of $\Gamma$ for the first time, we trace the entire loop before continuing. Furthermore, we have that the map $\beta \mapsto \rho'(\beta,\kappa')$ is an increasing bijection from $[-1,1]$ onto $[\kappa'-6,0]$. In particular, an $\SLE_\kappa^\beta(\kappa - 6)$ is a.s.\ generated by a continuous curve, and its trunk has the law of an $\SLE_{\kappa^\prime}(\rho^\prime; \kappa^\prime - 6 - \rho^\prime)$. 
\end{theorem}

\begin{remark}\label{re:GFF-coupling-0}
	With the notation of~\Cref{re:GFF-coupling-simple-BCLE,re:GFF-coupling-nonsimple-BCLE,A021}, let $\Psi$ be a GFF on $\BH$ with boundary values $\lambda^\prime(1 + \rho^\prime)$ (resp.\ $\lambda^\prime(1 + \rho^\prime) - 2\pi\chi$) on $\BR_{<0}$ (resp.\ $\BR_{>0}$) such that $\Gamma^\prime$ is generated as the branching counterflow line of $\Psi$ starting from $0$ and targeting all other boundary points (cf.~\Cref{re:GFF-coupling-nonsimple-BCLE}). Write 
	\begin{equation*}
		c_\BR \defeq -\lambda(1 + (\kappa - 6 - \rho_\BR)) + \lambda^\prime = \lambda(1 + \rho_\BR) + 2\pi\chi + \lambda^\prime; \quad c_\rL \defeq -\lambda(1 + \rho_\rL) - 2\pi\chi - \lambda^\prime. 
	\end{equation*}
	Then by~\cite[Section~8]{CLEPerc}, $\Gamma$ can be coupled with $\Psi$ so that conditionally on $\Gamma'$, inside each true (resp.\ false) loop $\SCL^\prime$ of $\Gamma^\prime$, the restriction of $\Gamma$ is generated as the branching flow line of $\Psi + c_\BR$ (resp.\ $\Psi + c_\rL$) starting from the first (resp.\ last) point on $\SCL^\prime$ visited by the loop from $0$ to $0$ that traces the domain boundary counterclockwise, and targeting all other boundary points (cf.~\Cref{re:GFF-coupling-simple-BCLE}). 
\end{remark}

\begin{corollary}\label{A025}
	Let $\kappa \in (8/3, 4)$ and $\beta \in [-1, 1]$. Let $\rho^\prime$, $\rho_\BR$, and $\rho_\rL$ be as in~\Cref{A021} such that $\beta = \beta(\kappa^\prime, \rho^\prime)$. Consider the following algorithm:
	\begin{itemize}
		\item One samples a $\cwBCLE_{\kappa^\prime}(\rho^\prime)$ in $\BH$. Then, inside each connected component of each true (resp.\ false) loop of this $\cwBCLE_{\kappa^\prime}(\rho^\prime)$, one samples an independent $\ccwBCLE_\kappa(\rho_\BR)$ (resp.\ $\cwBCLE_\kappa(\rho_\rL)$).
		\item Inside each (both true and false) loop of the $\ccwBCLE_\kappa(\rho_\BR)$'s and $\cwBCLE_\kappa(\rho_\rL)$'s of the previous step, one samples an independent $\cwBCLE_{\kappa^\prime}(\rho^\prime)$. Then, inside each connected component of each true (resp.\ false) loop of these $\cwBCLE_{\kappa^\prime}(\rho^\prime)$'s, one samples an independent $\ccwBCLE_\kappa(\rho_\BR)$ (resp.\ $\cwBCLE_\kappa(\rho_\rL)$). Then, we iterate this step.
	\end{itemize}
	Then the collection of all the true loops of all the simple BCLEs has the law of a nested $\CLE_\kappa^\beta$ in $\BH$.
\end{corollary}

\begin{proof}
	This follows immediately from~\Cref{A021} (see also the proofs of \cite[Theorems~7.8, 7.9]{CLEPerc}).
\end{proof}

\begin{remark}\label{re:critical_bcle}
	We note that $\cwBCLE_\kappa(\kappa -4)$ corresponds to a single loop which is equal to the domain boundary. Thus, in the special case that $\beta = 1$ (equivalently, $\rho_{\mathrm{L}} = \kappa-4$) in~\Cref{A025}, at every step of the iteration, we sample independent $\cwBCLE_{\kappa^\prime}(0)$s inside the false loops of the $\cwBCLE_{\kappa^\prime}(0)$s sampled at the previous step of the iteration. This is exactly the result of \cite[Theorem~7.8]{CLEPerc}.
\end{remark}

\begin{remark}\label{re:GFF-coupling-1}
	With the notation of~\Cref{A025}, let $\Psi$ be a GFF on $\BH$ with boundary values $\lambda^\prime(1 + \rho^\prime)$ (resp.\ $\lambda^\prime(1 + \rho^\prime) - 2\pi\chi$) on $\BR_{<0}$ (resp.\ $\BR_{>0}$). Then it follows immediately from~\Cref{re:GFF-coupling-0} that we may couple the simple and non-simple BCLEs of~\Cref{A021} with $\Psi$ so that they arise as branching flow lines and branching counterflow lines of $\Psi$, respectively.
\end{remark}

\subsection{First part of the proof of~\Cref{prop:thin-loops}: strategy and construction of the exploration}
\label{sec:proof_of_main_result}

In this subsection, we are going to construct an exploration in order to prove~\Cref{prop:thin-loops}. Let us first give a brief overview of the main strategy of the proof.

By covering $\BD$ using Euclidean annuli centered at a sufficiently dense set of points, it suffices to prove that the following is true (this covering argument is carried out in the proof of~\Cref{lem:big_loops_surround_other_loops}). Fix $0<s<t$ and $z \in \BD$, and consider the annulus $A_{s,t}(z) = B_t(z) \setminus  \overline{B_s(z)} \subseteq \BD$. First, we note that for every $\varepsilon>0$, by taking $t$ sufficiently close to $s$, we obtain that with probability at least $1-\varepsilon$ for all $n$ large enough (on the event that there is no loop in the $\CLE_{\kappa_n}$ surrounding $\partial B_s(z)$), there exists a loop in the $\CLE_{\kappa_n}$ crossing $A_{s,t}(z)$. Therefore, it suffices to show that with probability tending to $1$ as $n \to \infty$, we have that whenever a $\CLE_{\kappa_n'}$ loop crosses $A_{s,t}(z)$, it has to intersect one of the $\CLE_{\kappa_n}$ loops crossing $A_{s,t}(z)$.

Let $V_1^n,\ldots,V_{N_n}^n$ denote the conformal rectangles obtained by removing from $A_{s,t}(z)$ the loops in the $\CLE_{\kappa_n}$ intersecting $\partial B_s(z) \cup \partial B_t(z)$. Note that we have already shown in~\Cref{prop:conformal_rectangles_tight} that $V_j^n$ converges to $V_j$ as $n \to \infty$ in the Carath\'eodory sense, where $V_1,\ldots,V_N$ denote the conformal rectangles obtained by removing from $A_{s,t}(z)$ the loops in the $\CLE_4$ intersecting $\partial B_s(z) \cup \partial B_t(z)$. 

We will show in~\Cref{subsec:unlikeliness_of_crossing} (see~\Cref{lem:unlikeliness}) that the following is true.  With probability tending to $1$ as $n \to \infty$, we have that a loop of a $\CLE_{\kappa_n'}$ in $V_j^n$ intersecting the top and bottom boundaries of $V_j^n$ has to intersect either the left or the right boundary of $V_j^n$. In particular, it intersects some loop in the original $\CLE_{\kappa_n}$ in $\BD$. We would like to use this result to prove~\Cref{prop:thin-loops}. However, the restriction of the original $\CLE_{\kappa_n'}$ in $\BD$ to $V_j^n$ does not have the law of a $\CLE_{\kappa_n'}$ in $V_j^n$. To get around that issue, we will introduce in~\Cref{subsec:cpi_exploration,subsec:discovering_all_loops} a Markovian way to explore the loops of the $\CLE_{\kappa_n}$ in $\BD$ that intersect $\partial B_s(z) \cup \partial B_t(z)$, which has the following properties.
	\begin{enumerate}
		\item At each step, all the simple loops contained in the unexplored region are part of the original $\CLE_{\kappa_n}$ in $\BD$.
		\item Every segment of a loop in the original $\CLE_{\kappa_n'}$ in $\BD$ which is contained in the unexplored region is a segment of a loop in a $\CLE_{\kappa_n'}$ in the unexplored region (see~\Cref{lem:excursion}).
\end{enumerate}

The above properties of the exploration imply that any segment of a loop in the original $\CLE_{\kappa_n'}$ in $\BD$ that makes a crossing between the top and bottom boundaries of the conformal rectangle $V_j^n$ is part of a loop of a $\CLE_{\kappa_n'}$ in $V_j^n$. Therefore, we will complete the proof of~\Cref{prop:thin-loops} in~\Cref{subsec:completion_of_the_proof} by combining with~\Cref{lem:unlikeliness}, the Carath\'eodory convergence of $V_j^n$ to $V_j$ as $n \to \infty$ (\Cref{prop:conformal_rectangles_tight}), and the fact that in our coupling between $\CLE_{\kappa_n}$ and $\CLE_{\kappa_n'}$, we have that the following holds a.s.\ (see~\Cref{lem:loop_surrounds_another_loop}). Every loop in the $\CLE_{\kappa_n'}$ intersecting some loop in the $\CLE_{\kappa_n}$ has to surround it.

\subsubsection{Unlikeliness of non-simple CLE loops crossing a rectangle}
\label{subsec:unlikeliness_of_crossing}

Next, we show that as $\kappa' \downarrow 4$, it is very unlikely that there exists a loop $\SCL$ in a $\CLE_{\kappa'}$ in a rectangle $(0,L) \times (0,1)$ for a fixed $L>0$, such that $\SCL$ makes a crossing of the rectangle between its top and bottom boundaries without intersecting either its left or right boundaries. In particular, we will prove the following.

\begin{lemma}\label{lem:unlikeliness}
	Let $L_n \to L > 0$ and $\kappa_n^\prime \downarrow 4$. For each $n$, write $R_n \defeq \{z \in \BC: 0 < \Re(z) < L_n \text{ and } 0 < \Im(z) < 1\}$ and let $\Gamma_n^\prime$ be a $\CLE_{\kappa_n^\prime}$ in $R_n$. Then the probability that there exists a segment of a loop of $\Gamma_n^\prime$ that touches the top and bottom sides of $R_n$ but neither the left nor the right side of $R_n$ goes to $0$ as $n \to \infty$. 
\end{lemma}

Equivalently, by conformal invariance, it suffices to show that the following is true: Let $I_n$, $J_n$ be connected arcs of $\partial\BD$ such that $I_n \to I$ and $J_n \to J$ in the Hausdorff sense with $\dist(I, J) > 0$. Let $\Gamma_n^\prime$ be a $\CLE_{\kappa_n^\prime}$ in $\BD$. Then the probability that there exists a segment of a loop of $\Gamma_n^\prime$ that touches both $I_n$ and $J_n$ but not $\partial\BD \setminus (I_n \cup J_n)$ goes to $0$ as $n \to \infty$. 

One of the main ingredients in the proof of~\Cref{lem:unlikeliness} is the following lemma. It states that with probability tending to $1$ as $n \to \infty$, we have that every loop in $\Gamma_n'$ intersecting $\partial \BD$ has to stay close to $\partial \BD$ in the Hausdorff sense.

\begin{lemma}\label{lem:unlikeliness-proof}
	For each $n$, write $B_n$ for the closure of the union of the loops of $\Gamma_n^\prime$ that touch $\partial\BD$. Then $B_n$ converges in probability to $\partial\BD$ with respect to the Hausdorff topology. That is to say, for each $\varepsilon > 0$, it holds with probability tending to one as $n \to \infty$ that $B_n \subset B_\varepsilon(\partial\BD)$. 
\end{lemma}

\begin{proof}
	We will use the same notation as in~\Cref{prop subsequential limit boundary}. First, we note that~\Cref{prop BCLE rho to zero} implies that $\SCB_t(\partial \BD)$ converges in probability to $\partial \BD$ as $t \to 0$ with respect to the Hausdorff metric induced by the Euclidean metric.  Moreover,~\Cref{prop subsequential limit boundary} (applied along a subsequence of any given subsequence of $(\kappa'_n)$, which suffices) implies that $\SCB_{\ka_{\kappa_n} t}^{\kappa_n}(\partial \BD)$ converges in distribution to $\SCB_t(\partial \BD)$ as $n \to \infty$ with respect to the Hausdorff metric; since $\{K : d_{\mathrm H}(K,\partial\BD) < \varepsilon\}$ is open, the Portmanteau theorem gives $\liminf_n \BP[\SCB^{\kappa_n}_{\ka_{\kappa_n}t}(\partial\BD) \subseteq B_\varepsilon(\partial\BD)] \ge \BP[d_{\mathrm H}(\SCB_t(\partial\BD),\partial\BD)<\varepsilon]$. Therefore, we can conclude the proof of the lemma by combining with the fact that $B_n \subseteq \SCB_{\ka_{\kappa_n} t}^{\kappa_n}(\partial \BD)$ as soon as $\ka_{\kappa_n}t\ge 1$, which holds for $n$ large enough for any fixed $t>0$.
\end{proof}

Thus, it remains to rule out the bad situations illustrated in~\Cref{fig:bad}. To do this, we will construct ``shields'' of flow lines of the GFF used to generate $\Gamma_n'$ which prevent the pathological cases illustrated in~\Cref{fig:bad} from occurring. Note that the local behavior of flow lines is a.s.\ determined by the GFF generating them. Hence, to prove that we have sufficiently many ``shields'' of flow lines with high probability, we will use the following lemma.

\begin{figure}[ht!]
	\centering
	\includegraphics[width=.49\linewidth]{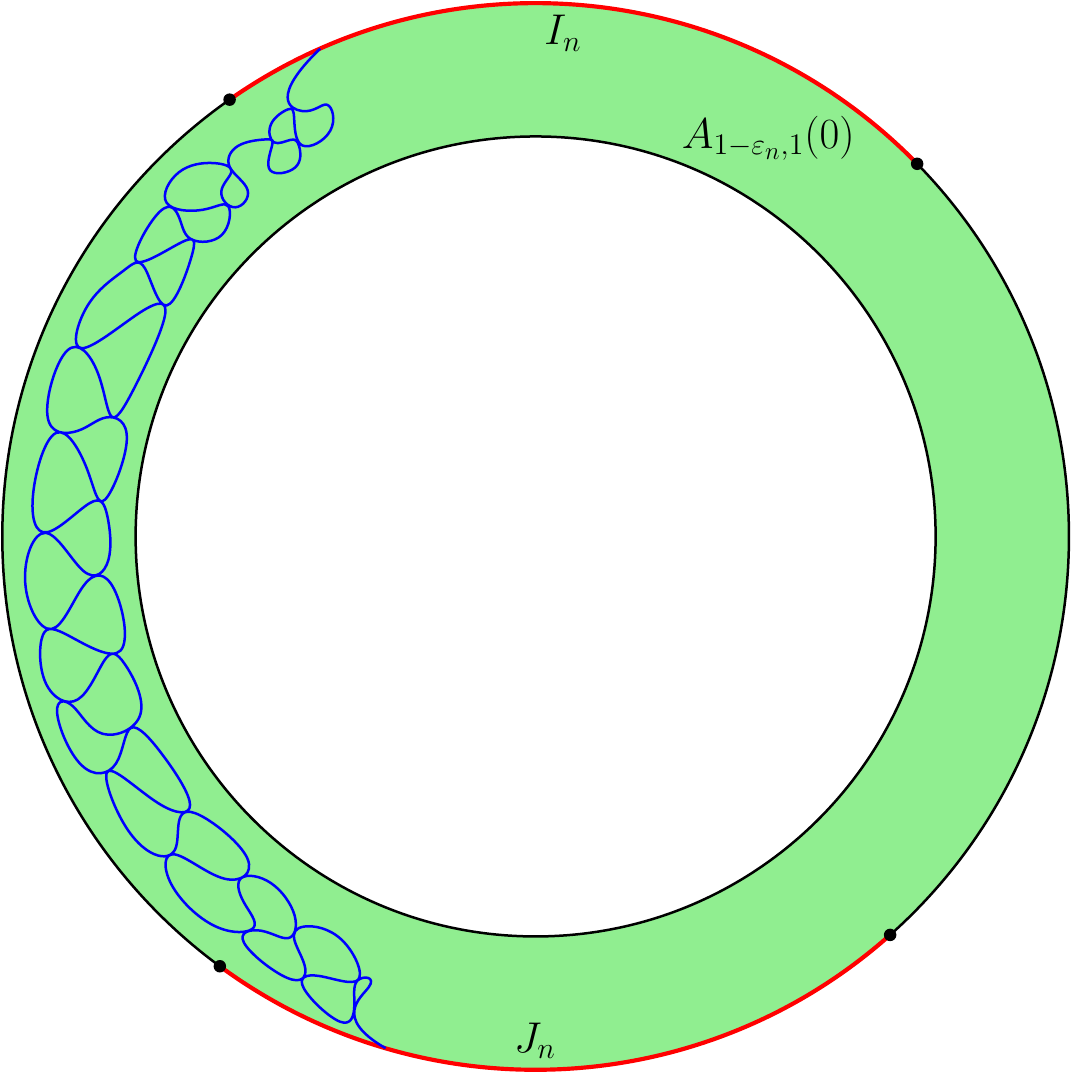}
	\includegraphics[width=.49\linewidth]{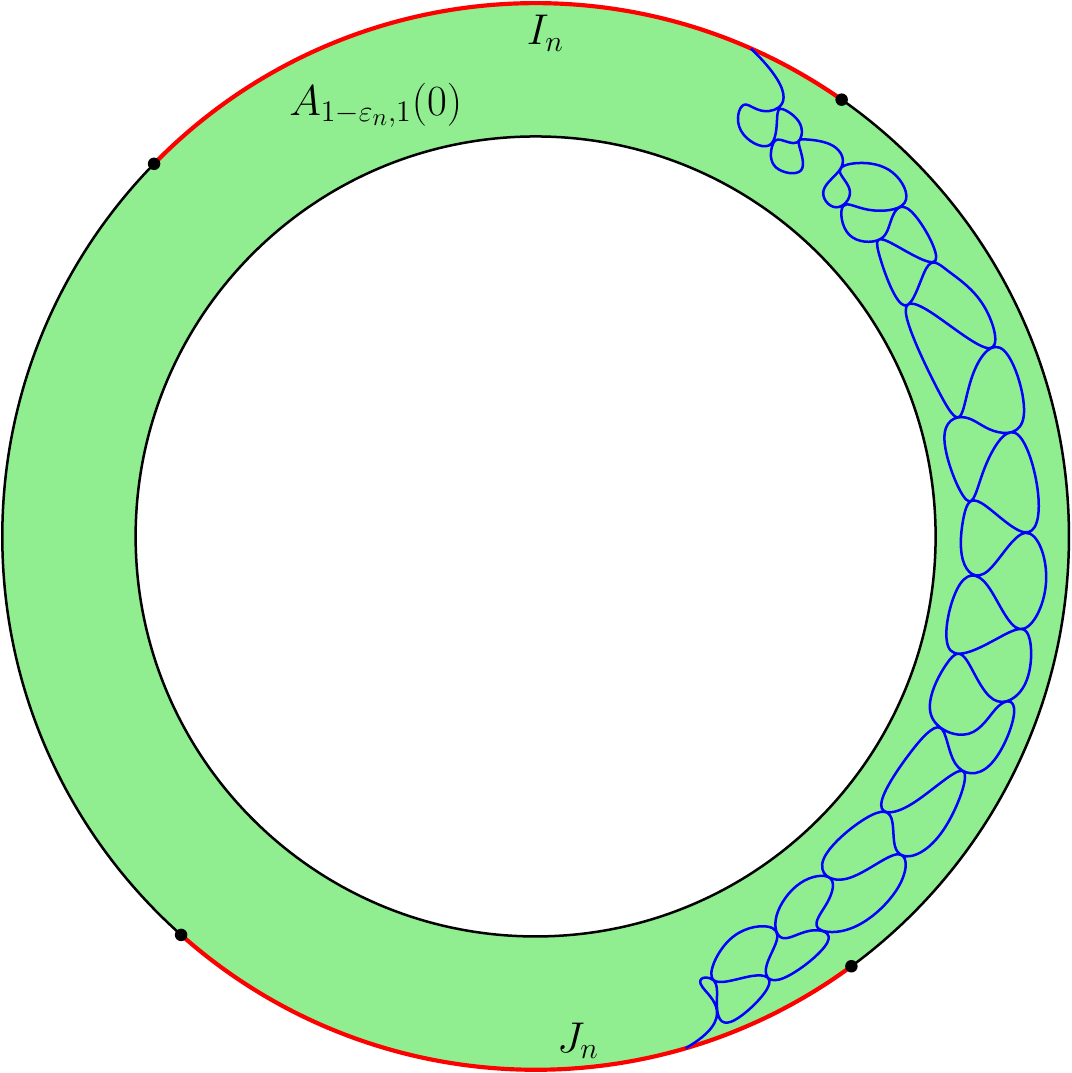}
	\caption{Illustration of the bad scenarios which are ruled out in the proof of~\Cref{lem:unlikeliness}. We have shown in~\Cref{lem:unlikeliness-proof} that non-simple loops of $\Gamma'_n$ that intersect $\partial \BD$ have to stay close to $\partial \BD$ when $n\to \infty$. Therefore, the two bad scenarios that have to be dealt with here (as in the proof of \cite[Lemma~2.7]{ExUniLQG}) are the events that a segment of a non-simple loop joins $I_n$ to $J_n$ through the left or right-hand side of the green annulus $A_{1-\varepsilon_n, 1}(0)$.}
	\label{fig:bad}
\end{figure}

\begin{lemma}\label{lem:independence-across-disjoint-balls}
For each $M>0$, $p,q \in (0,1)$, and $s>0$, there exists $n_* = n_*(M,p,q,s) \in \BN$ such that the following is true. Let $x_1,\ldots,x_n \in \partial \BD$ with $n \geq n_*$. Let $r>0$. Suppose that $|x_j - x_k| \geq 2(s+1) r$ for all distinct $j,k \in [1,n]_{\BZ}$. Let also $\Psi$ be a GFF on $\BD$ with boundary conditions given by a measurable function bounded by $M$. Let $E_1,\ldots,E_n$ be a family of events such that each $E_j$ is measurable with respect to the $\sigma$-algebra generated by $\Psi|_{B_r(x_j) \cap \BD}$. Suppose that $\BP[E_j] \geq p$ for all $j \in [1,n]_{\BZ}$. Then
	\begin{equation*}
		\BP\!\left\lbrack\bigcup_{j = 1}^n E_j\right\rbrack \ge q. 
\end{equation*} 
\end{lemma}

\begin{proof}
This follows from the spatial Markov property of the GFF, using an argument similar to the one applied in the proof of \cite[Lemma~2.7]{ExUniLQG}. Let $\widetilde U_j \defeq B_{(1+s)r}(x_j)\cap\BD$ and $U_j \defeq B_r(x_j)\cap\BD$. By the Markov property, $\Psi|_{\widetilde U_j} = h_j + \varphi_j$ with $\varphi_j$ harmonic on $\widetilde U_j$ and the $h_j$ independent zero-boundary GFFs on $\widetilde U_j$, independent of $(\varphi_j)_j$. By scaling and the buffer of width $sr$ between $U_j$ and $\partial \widetilde U_j\setminus\partial\BD$, the 
random variable $S_j \defeq \sup_{B_{(1+s/2)r}(x_j)\cap\BD}|\varphi_j - \BE\varphi_j|$ has tails bounded uniformly in $j$ and $r$, so we may fix $C = C(q,s)$ with $\BP[G_j] \ge 1-(1-q)/4$, where $G_j \defeq \{S_j \le C\}$. If $\varphi,\varphi'$ both satisfy this bound, the laws of $(h_j+\varphi)|_{U_j}$ and $(h_j+\varphi')|_{U_j}$ are mutually absolutely continuous with a Radon--Nikodym derivative of second moment bounded in terms of $C$ and $s$ (Cameron--Martin, applied to $(\varphi-\varphi')\chi$ with $\chi$ equal to $1$ on $U_j$ and supported in $B_{(1+s/2)r}(x_j)$). Since $\BP[E_j] \ge p$, Cauchy--Schwarz gives $\BP[E_j \mid \varphi_j] \ge c(p,q,s) > 0$ on $G_j$. The $E_j$ being conditionally independent given $(\varphi_j)_j$, and at least $n/2$ of the $G_j$ occurring with probability at least $1-(1-q)/2$, we get $\BP[\bigcup_j E_j] \ge 1-(1-q)/2-(1-c)^{n/2} \ge q$ for $n$ large.
\end{proof}

Let us now mention a useful consequence of \cite[Proposition~2.6]{Leh23}.
\begin{lemma}\label{lemma cv SLE kappa rho Caratheodory}
	Let $\kappa^n \to \kappa>0$, $\rho^n_1\to \rho_1>-2$ and $\rho^n_2\to \rho_2>-2$ with $\kappa^n>0, \rho^n_1>-2$ and $\rho^n_2>-2$ for all $n\ge 1$. Let $f^n_t\colon \BH \setminus K^n_t \to \BH$ (resp.\ $f_t\colon \BH \setminus K_t \to \BH$) for $t\ge 0$ be the Loewner chain associated with a chordal $\SLE_{\kappa^n}(\rho^n_1;\rho^n_2)$ (resp.\ $\SLE_{\kappa}(\rho_1;\rho_2)$) from $0$ to $\infty$ with the force points located at $0^-$ and $0^+$ respectively, where $K^n_t$ (resp.\ $K_t$) is the associated compact hull. Then, there is a coupling of the $(f^n_t)_{t\ge 0}$'s with $(f_t)_{t\ge 0}$ such that a.s., 
	\[
	f^n \mathop{\longrightarrow}\limits_{n\to \infty} f
	\]
	uniformly on $[0, T] \times (\BH \setminus B_\varepsilon(K_T))$ for all $T,\varepsilon>0$. 

\end{lemma}
\begin{proof}
	The result from \cite[Proposition~2.6]{Leh23} establishes that as $\kappa^n \to \kappa$ and $\rho^n_i \to \rho_i$, the driving functions $W^n$ of the $\SLE_{\kappa^n}(\rho^n_1; \rho^n_2)$ processes converge in distribution in the local uniform topology to the driving function $W$ of the $\SLE_{\kappa}(\rho_1; \rho_2)$ process. By Skorokhod's representation theorem, we can construct a coupling of these processes such that this local uniform convergence of the driving functions holds a.s. Under this coupling, we have almost sure uniform convergence of $W^n \to W$ on $[0,T]$ for any $T > 0$. According to \cite[Proposition~4.47]{lawler2008conformally}, the uniform convergence of driving functions on compact time intervals implies that the corresponding Loewner maps $f^n_t$ converge to $f_t$ uniformly on $[0, T] \times (\BH \setminus B_\varepsilon(K_T))$ for any $T,\varepsilon>0$. This completes the proof.
\end{proof}
	
In the following, write $\kappa_n \defeq 16/\kappa_n^\prime$. Thus, $\kappa_n \uparrow 4$. Let us now define the main event involving flow lines that we are going to consider and prove that it occurs with positive probability.
	
\begin{lemma}\label{lem:likeliness}
Fix distinct $x,y \in \partial \BD$ and let $\varepsilon \in (0,1)$ be such that $\varepsilon < \min\{1/2, |x-y| / 3\}$. For each $n \in \BN$, let $\eta_n$ be a chordal $\SLE_{\kappa_n}(-1;-1)$ in $\BD$ starting from $x$ and targeting $y$, where the force points are located at $x^-$ and $x^+$ respectively. Then, there exists a universal constant $p \in (0,1)$ such that the following holds for all $n \in \BN$. With probability at least $p$, we have that $\eta_n$ hits $B_{\varepsilon}((1-2\varepsilon)x)$ before exiting $B_{2\varepsilon}(x)$.   
\end{lemma}

\begin{proof}
 Let $\eta$ denote a chordal $\SLE_4(-1;-1)$ process in $\BD$ from $x$ to $y$ with the force points located at $x^-$ and $x^+$ respectively, and let $\varepsilon \in (0,1)$ be as in the lemma statement. Note that combining \cite[Lemma~2.8]{miller2017intersections} with the conformal invariance of the laws of $\SLE_{\kappa}(\rho^{\mathrm{L}}; \rho^{\mathrm{R}})$ processes and the Koebe distortion theorem (the conformal map to $\BH$ sending $x$ to $0$ and $y$ to $\infty$ distorts the balls $B_{2\varepsilon}(x)$, $B_\varepsilon((1-2\varepsilon)x)$ by a bounded factor, uniformly in $\varepsilon$), we obtain that it suffices to prove the claim in the lemma statement in the case that $\varepsilon = 1/3$. Henceforth, assume that $\varepsilon = 1/3$.

For each $n$, we parameterize the curve $\eta_n$ according to half-plane capacity as seen from $y$, and let $(K_t^n)_{t \geq 0}$ denote the family of the corresponding hulls in $\BD$. We parameterize $\eta$ and define the family of hulls $(K_t)_{t \geq 0}$ for $\eta$ similarly. Let $f_t\colon \BD\setminus K_t\to \BD$ (resp.\ $f^n_t\colon \BD\setminus K^n_t\to \BD$) be the Loewner chain associated with $(K_t)_{t\ge 0}$ (resp.\ $(K^n_t)_{t\ge 0}$). Let $\sigma$ be the hitting time of $B_{1/6}(x/3)$ by $\eta$ and $E$ the event that $\sigma$ precedes the exit time of $B_{2/3}(x)$; by \cite[Lemmas~2.3 and~2.8]{miller2017intersections} and the conformal invariance of the law of $\eta$, $\BP[E] \ge p$ for some universal $p \in (0,1)$. Couple $(\eta_n)_{n \in \BN}$ and $\eta$ as in~\Cref{lemma cv SLE kappa rho Caratheodory}. On $E$, $K_\sigma$ is a compact subset of $B_{2/3}(x)$ meeting $\overline{B_{1/6}(x/3)}$, so the 
convergence $f^n\to f$, which gives $K^n_\sigma \subseteq B_\delta(K_\sigma)$ for every $\delta>0$, yields for $n$ large that $K^n_\sigma \subseteq B_{2/3}(x)$ and $K^n_\sigma\cap B_{1/3}(x/3) \ne \emptyset$, i.e.\ $\eta_n$ hits $B_{1/3}(x/3)$ before exiting $B_{2/3}(x)$. For the finitely many remaining $n$, the probability is positive by the same lemmas.
\end{proof}

We are now ready to complete the proof of~\Cref{lem:unlikeliness}.

\begin{proof}[Proof of~\Cref{lem:unlikeliness}]

 \stepx{step:unlik-setup}{Setup} Suppose that we have the setup described just after the statement of~\Cref{lem:unlikeliness}. Fix $z \in I_n$ and $w \in J_n$. We may assume without loss of generality that $z \in I_n$ and $w \in J_n$ for all $n \in \BN$. Also, for each $n$, we let $\Psi_n$ be a GFF on $\BD$ rooted at $z$ that generates $\Gamma_n'$ (see~\Cref{re:GFF-coupling-nonsimple-BCLE}). That is to say, if $\phi_n: \BH \to \BD$ is a conformal transformation mapping $0$ to $z$, $\infty$ to $w$, and the clockwise (resp.\ counterclockwise) arc of $\partial \BD$ to $\BR_-$ (resp.\ $\BR_+$), then the boundary conditions of $\Psi_n \circ \phi_n - \chi_n \arg(\phi_n')$ are given by $\lambda_n'$ on $(-\infty,0)$ and $\lambda_n' - 2\pi \chi_n$ on $(0,\infty)$, where $\lambda_n^\prime \defeq \pi\sqrt{\kappa_n}/4$ and $\chi_n \defeq 2/\sqrt{\kappa_n} - \sqrt{\kappa_n}/2$. Moreover, the branching tree of counterflow lines of $\Psi_n$ centered at $z$ and targeted at any $u \in \overline{\BD}$ is used to construct $\Gamma_n'$ as in \cite[Section~4]{TreeCLE}.

\stepx{step:unlik-shields}{Construction of the shields} By~\Cref{lem:unlikeliness-proof}, it suffices to verify that for each sequence $\varepsilon_n \to 0$ as $n \to \infty$, the bad scenarios illustrated in~\Cref{fig:bad} occur with probability tending to zero. To show the latter, for each $n$, we may choose points $x_1^n,\ldots,x_{N_n}^n$ and $y_1^n,\ldots,y_{N_n}^n$ on the two connected components of $\partial \BD \setminus (I_n \cup J_n)$, respectively, such that all of $x_1^n,\ldots,x_{N_n}^n$ (resp.\ $y_1^n,\ldots,y_{N_n}^n$) lie on the clockwise (resp.\ counterclockwise) arc of $\partial \BD$ from $z$ to $w$, and such that $|x_j^n - x_k^n| \wedge |y_j^n - y_k^n| \geq 6\varepsilon_n$ for all distinct $j,k \in [1,N_n]_{\BZ}$ and $N_n$ is of order $\varepsilon_n^{-1}$ as $n \to \infty$.

Let $\eta_j^n$ (resp.\ $\widetilde{\eta}_j^n$) denote the flow line of $\Psi_n$ of angle $-\lambda_n' / \chi_n$ (resp.\ $2\pi - \lambda_n' / \chi_n$) starting from $x_j^n$ (resp.\ $y_j^n$) and targeted at $z$, so that $\eta_j^n$ (resp.\ $\widetilde{\eta}_j^n$) has the law of a chordal $\SLE_{\kappa_n}(-1;-1)$ on $\BD$ from $x_j^n$ (resp.\ $y_j^n$) to $z$ with the force points located at $(x_j^n)^-$ and $(x_j^n)^+$ (resp.\ $(y_j^n)^-$ and $(y_j^n)^+$) (these are the zero-angle flow lines of $\Psi_n-\lambda'_n$ and $\Psi_n+2\pi\chi_n-\lambda'_n$, whose boundary values vanish near $x^n_j$, $y^n_j$). Consider the event $E_j^n$ that $\eta_j^n$ hits $B_{\varepsilon_n}((1-2\varepsilon_n) x_j^n)$ before exiting $B_{2\varepsilon_n}(x_j^n)$ (which is measurable with respect to $\Psi_n|_{B_{2\varepsilon_n}(x^n_j)\cap\BD}$, flow lines being local~\cite{IG1}). Define $F_j^n$ in a similar way but with $y_j^n$ in place of $x_j^n$. Then, it follows from~\Cref{lem:independence-across-disjoint-balls,lem:likeliness} (the former applied separately to $(E^n_j)_j$ and $(F^n_k)_k$ with $r=2\varepsilon_n$, $s=1/2$) that it holds with probability tending to $1$ as $n \to \infty$ that $E_j^n \cap F_k^n$ occurs for some $j$ and $k$.

\stepx{step:unlik-conclusion}{Conclusion} Recall that counterflow lines do not cross flow lines generated by the same random field a.s. Indeed, it follows from \cite[Theorem~1.13]{IG4} that the left and right outer boundaries of counterflow lines can be described in terms of flow lines with angles $\pi / 2$ and $-\pi / 2$ respectively. Thus, in order for a counterflow line $\eta'$ to cross a flow line $\eta$, we need to have that both the flow lines with angles $\pi/2$ and $-\pi/2$ forming the left and right boundaries of $\eta'$ cross $\eta$. But this does not happen a.s.\ due to the flow line interaction rules (see \cite[Theorem~1.7]{IG4}). Therefore, since every segment of a loop in $\Gamma_n'$ is part of a counterflow line of $\Psi_n$, combining with the considerations of the previous paragraph, we obtain that the bad situations illustrated in~\Cref{fig:bad} occur with probability tending to zero as $n \to \infty$. Indeed, on $E^n_j\cap F^n_k$ the stopped flow lines join $\partial\BD$ to depth at least $\varepsilon_n$ on each of the two arcs of $\partial\BD\setminus(I_n\cup J_n)$, so that a segment of a loop of $\Gamma'_n$ contained in $B_{\varepsilon_n}(\partial\BD)$ and joining $I_n$ to $J_n$ would have to cross one of them.
\end{proof}

\subsubsection{CPI exploration}\label{subsec:cpi_exploration}

In this section, we will describe in a more detailed way the exploration of simple $\CLE_{\kappa}$ loops introduced in the statement of~\Cref{A021}, and we will state and prove two useful consequences of it (\Cref{lem:markovian_cpi_exploration,lem:excursion}).

By setting $\beta = 1$ in~\Cref{A025}, we obtain that the following is true. Let $\kappa \in (8/3, 4)$. Consider the following algorithm:
\begin{itemize}
\item One samples a $\cwBCLE_{\kappa^\prime}(0)$. Then, inside each connected component of each true loop of this $\cwBCLE_{\kappa^\prime}(0)$, one samples an independent $\ccwBCLE_\kappa(-\kappa/2)$. 

\item Inside each connected component of each false loop of the $\cwBCLE_{\kappa^\prime}(0)$'s and $\ccwBCLE_\kappa(-\kappa/2)$'s of the previous step, one samples an independent $\cwBCLE_{\kappa^\prime}(0)$. Then, inside each connected component of each true loop of these $\cwBCLE_{\kappa^\prime}(0)$'s, one samples an independent $\ccwBCLE_\kappa(-\kappa/2)$. Then, iterate this step.
\end{itemize}
Then the collection of all the true loops of all the $\ccwBCLE_\kappa(-\kappa/2)$'s has the law of a non-nested $\CLE_\kappa$. Moreover, we note that the collection of all the outermost true loops of all the $\cwBCLE_{\kappa^\prime}(0)$'s has the law of a non-nested $\CLE_{\kappa^\prime}$.

By iterating this construction inside each loop of the non-nested $\CLE_\kappa$, we obtain a nested $\CLE_\kappa$. But, for our purpose, we will only use the non-nested $\CLE_\kappa$.

Recall from~\Cref{re:GFF-coupling-0} that, once a root on the domain boundary is chosen, all the $\ccwBCLE_\kappa(-\kappa/2)$'s and $\cwBCLE_{\kappa^\prime}(0)$'s above may be coupled with a GFF $\Psi$ with boundary conditions given by $\lambda^\prime$ on $(-\infty, 0)$ and $\lambda^\prime - 2\pi\chi$ on $(0, \infty)$ (after a conformal transformation) as flow lines and counterflow lines, respectively. 

\textbf{CPI exploration}. Let us recall the exploration described in~\Cref{A021}. Suppose that we start with the $\cwBCLE_{\kappa^\prime}(0)$/$\ccwBCLE_\kappa(-\kappa/2)$ iteration on $\BD$ rooted at some point $x \in \partial \BD$. Then, we can define a path $\eta$ as follows. First, we let $\eta'$ denote the path that traces the domain boundary in a counterclockwise way except that each time it encounters a true loop of the first $\cwBCLE_{\kappa^\prime}(0)$ used in the above construction, it traverses the entire loop clockwise before continuing. Let $\Gamma$ denote the collection of all true loops in the $ \ccwBCLE_\kappa(-\kappa/2)$ process that we have considered inside the true loops of that $\cwBCLE_{\kappa^\prime}(0)$. Then, we define the path $\eta$ by following $\eta'$ except that each time we first encounter a point on a loop in $\Gamma$ we traverse the loop in a counterclockwise way (recall that $\beta=1$, so that all of these loops lie on the same side of $\eta'$; cf.~\Cref{A021}). Then, the path $\eta$ traces the true loops of the $\cwBCLE_{\kappa^\prime}(0)$'s and the $ \ccwBCLE_\kappa(-\kappa/2)$'s used in the first level of the iteration.

To trace the loops of the $\cwBCLE_{\kappa^\prime}(0)$'s and the $ \ccwBCLE_\kappa(-\kappa/2)$'s used in the second level of the iteration, we construct a branching collection of exploration paths as follows. Each time the path $\eta$ finishes tracing a false loop at some point $y$ of the $\cwBCLE_{\kappa^\prime}(0)$ or one of the $ \ccwBCLE_\kappa(-\kappa/2)$'s used in the first level of the iteration, we construct a path inside the domain surrounded by that loop in the exact same way that we constructed $\eta$ except that the root of the new loop is given by $y$. Then, we iterate the above construction inside the false loops of the $\cwBCLE_{\kappa^\prime}(0)$'s  and the $ \ccwBCLE_\kappa(-\kappa/2)$'s used in the second level of the iteration in the exact same way as we described in the previous two steps.

We call the branching collection of paths described in the previous two paragraphs the \emph{CPI (conformal percolation interface) exploration} of~\cite{CLEPerc} of the iterated $\cwBCLE_{\kappa^\prime}(0)$/$\ccwBCLE_\kappa(-\kappa/2)$ construction.

\textbf{Obtaining the entire  iterated $\cwBCLE_{\kappa^\prime}(0)$/$\ccwBCLE_\kappa(-\kappa/2)$ construction from the CPI exploration}. For all $z \in \partial \BD \setminus \{x\}$, we let $\eta_z$ denote the branch of the CPI exploration targeted at $z$ and parameterized by capacity as seen from $z$. Then, the collection of the true loops of the $\ccwBCLE_\kappa(-\kappa/2)$'s used in the first level of iteration will be given by the loops attached to the paths $\eta_z$'s for $z \in \partial \BD \setminus \{x\}$. Moreover, the boundary branching exploration tree used to construct the $\cwBCLE_{\kappa^\prime}(0)$ in the first level of iteration will be given by the tree formed by the trunks of the $\eta_z$'s for $z \in \partial \BD \setminus \{x\}$. This gives the $\cwBCLE_{\kappa^\prime}(0)$ and the $\ccwBCLE_\kappa(-\kappa/2)$'s used in the first level of iteration. By iterating the same procedure using the branching property of the CPI exploration inside the false loops of the above $\cwBCLE_{\kappa^\prime}(0)$'s and  $\ccwBCLE_\kappa(-\kappa/2)$'s, we obtain the rest of the $\cwBCLE_{\kappa^\prime}(0)$ and  $\ccwBCLE_\kappa(-\kappa/2)$'s used in the next levels of iteration.

\textbf{Re-rooting the iterated $\cwBCLE_{\kappa^\prime}(0)$/$\ccwBCLE_\kappa(-\kappa/2)$ construction}. Let us now explain how to re-root the above $\cwBCLE_{\kappa^\prime}(0)$/$\ccwBCLE_\kappa(-\kappa/2)$ iteration. Fix $y \in \partial \BD \setminus \{x\}$ and re-parameterize the path $\eta$ defined above so that it starts and ends at $y$. Then, inside each connected component $U$ of $\BD \setminus \eta$ where there is branching of $\eta$, we re-parameterize the path that discovers the true loops of the $\ccwBCLE_\kappa(-\kappa/2)$ in the second level of iteration which is contained in $U$, so that the path starts and ends at the last point on $\partial U$ hit by $\eta$. Then, we iterate the above in the remaining connected components to obtain a branching exploration tree of paths. Hence, we can construct an iterated process of loops as explained in the previous paragraph by using the aforementioned exploration tree.

It follows from the results in~\cite{CLEPerc} that the iterated process of loops described in the above paragraph has the same law as the iterated $\cwBCLE_{\kappa^\prime}(0)$/$\ccwBCLE_\kappa(-\kappa/2)$ described above, and the collection of true loops in the $  \ccwBCLE_\kappa(-\kappa/2)$'s in the original iterated $\cwBCLE_{\kappa^\prime}(0)$/$\ccwBCLE_\kappa(-\kappa/2)$ construction (before re-rooting) is the same as the true loops in the $  \ccwBCLE_\kappa(-\kappa/2)$'s in the new iterated $\cwBCLE_{\kappa^\prime}(0)$/$\ccwBCLE_\kappa(-\kappa/2)$ construction (after re-rooting).

Next, we state and prove the following useful Markovian property of the CPI exploration.

\begin{lemma}\label{lem:markovian_cpi_exploration}
	Fix $z \in \overline{\BD}$ and let $\eta_z$ denote the branch of the CPI exploration targeted at $z$ of an iterated $\cwBCLE_{\kappa^\prime}(0)$/$\ccwBCLE_\kappa(-\kappa/2)$ construction. Let also $\tau$ be a stopping time for $\eta_z$ such that $\eta_z(\tau)$ is not on a simple loop a.s. Let $D_{\tau}$ denote the connected component of $\BD \setminus \eta_z([0,\tau])$ containing $z$. Then, we have that the iterated process of loops in $D_{\tau}$ obtained from the restriction to $D_{\tau}$ of the CPI exploration tree rooted at $\eta_z(\tau)$ has the law of an iterated $\cwBCLE_{\kappa^\prime}(0)$/$\ccwBCLE_\kappa(-\kappa/2)$ construction on $D_{\tau}$ conditionally on $\eta_z|_{[0,\tau]}$. Moreover, the simple true loops obtained from the new construction are the same as the simple loops of the original $\cwBCLE_{\kappa^\prime}(0)$/$\ccwBCLE_\kappa(-\kappa/2)$ construction which are contained in $D_{\tau}$.
	\end{lemma}

	\begin{proof}
Let $\Psi$ denote the GFF on $\BD$ used to generate the $\cwBCLE_{\kappa^\prime}(0)$'s and the $  \ccwBCLE_\kappa(-\kappa/2)$'s in the statement of the lemma. It follows from~\Cref{sec:cle-percolation} and especially the proof of~\Cref{A021} (cf.~\cite[Figure~9.2]{CLEPerc}) that the boundary conditions of $\Psi$ on $D_{\tau}$ are, up to an additive constant $c$ (the constant by which $\Psi$ is shifted at the level of the iteration being explored at time $\tau$, cf.~\Cref{re:GFF-coupling-0}), given by $\lambda'$ on $(-\infty,0)$ and $\lambda' - 2\pi \chi$ on $(0,\infty)$ (after a conformal transformation onto $\BH$) (by the strong Markov property of $\Psi$ at $\tau$, $\eta_z$ being a local set, and the boundary values of $\Psi$ along flow and counterflow lines~\cite{IG1,IG4}). Furthermore, the iterated construction of loops in $D_{\tau}$ described in the statement of the lemma can be obtained as an iterated $\cwBCLE_{\kappa^\prime}(0)$/$\ccwBCLE_\kappa(-\kappa/2)$ construction on $D_{\tau}$ using the field $\Psi|_{D_{\tau}}-c$. This proves the first claim in the statement of the lemma. Finally, since the simple loops in the latter construction are constructed using branching flow lines of $\Psi|_{D_{\tau}}$ (which are also parts of branching flow lines of $\Psi$) and no flow line of $\Psi$ can cross $\partial D_{\tau}$ (due to the flow line interaction rules, see, e.g.,~\cite[Theorem~1.7]{IG4}), we obtain the second claim of the lemma (both inclusions, since by the description given below the loops of the iterations inside the first level $\cwBCLE_{\kappa'}(0)$ of $D_\tau$ are exactly the original ones). This completes the proof of the lemma.
\end{proof}

Let us now give a more detailed description of the relation  between the $\cwBCLE_{\kappa^\prime}(0)$/$\ccwBCLE_\kappa(-\kappa/2)$ iteration in the remaining domain and the original $\cwBCLE_{\kappa^\prime}(0)$/$\ccwBCLE_\kappa(-\kappa/2)$ iteration in the context of~\Cref{lem:markovian_cpi_exploration}.

The relation between the $\cwBCLE_{\kappa^\prime}(0)$/$\ccwBCLE_\kappa(-\kappa/2)$ iteration in the remaining domain and the original $\cwBCLE_{\kappa^\prime}(0)$/$\ccwBCLE_\kappa(-\kappa/2)$ iteration can be described as follows (see~\Cref{fig:markov-01,fig:markov-23} for an illustration): Let $x$ be a point that is on the boundary of the remaining domain. Since we always finish tracing the entire loop at the first time we encounter a true loop of some $\ccwBCLE_\kappa(-\kappa/2)$, and $\eta_z(\tau)$ is not on a loop of some $\ccwBCLE_\kappa(-\kappa/2)$, we see that $x$ is either 
\begin{itemize}
	\item on the false $\ccwBCLE_\kappa(-\kappa/2)$ loop that surrounds the $\cwBCLE_{\kappa^\prime}(0)$ that we were tracing at time $\tau$ (by convention, the domain boundary is also viewed as a false $\ccwBCLE_\kappa(-\kappa/2)$ loop);
	\item or a false loop of the $\ccwBCLE_\kappa(-\kappa/2)$ that is on the right-hand side of $\eta_z(\tau)$ (note that $\eta_z(\tau)$ is on both a true and a false loop of the $\cwBCLE_{\kappa^\prime}(0)$ that we were tracing at time $\tau$; the true one is on the right-hand side of $\eta_z(\tau)$; the false one is on the left);
	\item or the false $\cwBCLE_{\kappa^\prime}(0)$ loop that we were tracing at time $\tau$ which is on the left-hand side of $\eta_z(\tau)$.
\end{itemize}
Then the branch of the first level $\cwBCLE_{\kappa^\prime}(0)$ in the remaining domain from $\eta_z(\tau)$ to $x$ can be described as follows: 
\begin{itemize}
	\item in the first case, we follow the branch of the $\cwBCLE_{\kappa^\prime}(0)$ that we were tracing at time $\tau$ targeting $x$; 
	\item in the second case, the root of the false $\ccwBCLE_\kappa(-\kappa/2)$ loop that contains $x$ is on the true $\cwBCLE_{\kappa^\prime}(0)$ loop that we were tracing at time $\tau$ and is to the right of $\eta_z(\tau)$; then we follow this true $\cwBCLE_{\kappa^\prime}(0)$ loop until the root of the false $\ccwBCLE_\kappa(-\kappa/2)$ loop that contains $x$, then we trace the branch of the $\cwBCLE_{\kappa^\prime}(0)$ inside this false $\ccwBCLE_\kappa(-\kappa/2)$ loop from its root to $x$; 
	\item in the third case, we follow the false $\cwBCLE_{\kappa^\prime}(0)$ loop to the left of $\eta_z(\tau)$ until we close it, then we follow the branch of the $\cwBCLE_{\kappa^\prime}(0)$ inside this false $\cwBCLE_{\kappa^\prime}(0)$ loop targeting $x$. 
\end{itemize}

Once this first level $\cwBCLE_{\kappa^\prime}(0)$ is constructed, by the description above, we observe that each connected component of each true (resp.~false) loop of this $\cwBCLE_{\kappa^\prime}(0)$ is exactly the connected component of some true (resp.~false) loop of some original $\cwBCLE_{\kappa^\prime}(0)$. Thus, the second level $\ccwBCLE_\kappa(-\kappa/2)$ and the third level $\cwBCLE_{\kappa^\prime}(0)$, and so on, inside them, are exactly given by the original ones inside them.

\begin{figure}[ht!]
\centering
\includegraphics[width=0.485\linewidth]{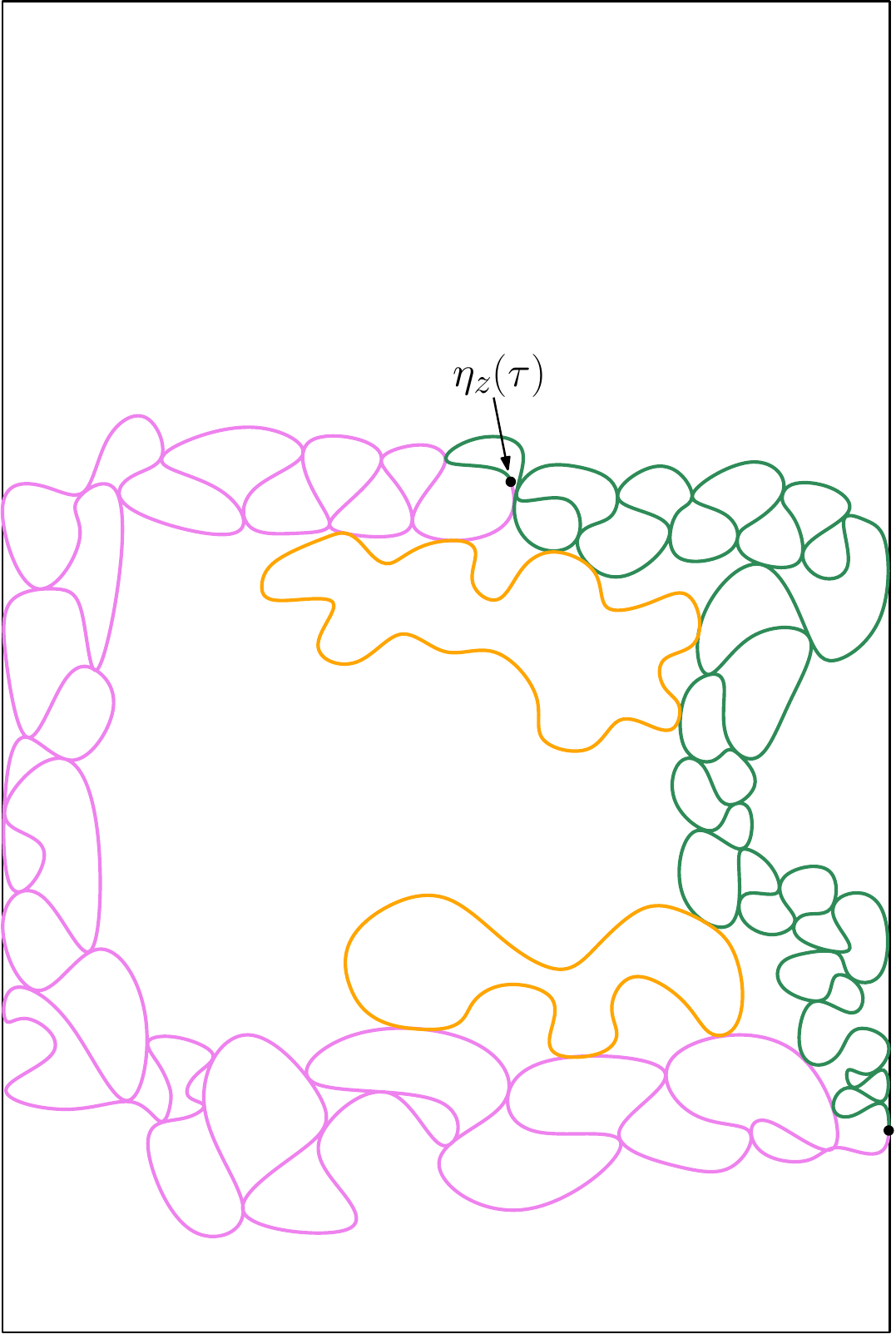}
\includegraphics[width=0.50\linewidth]{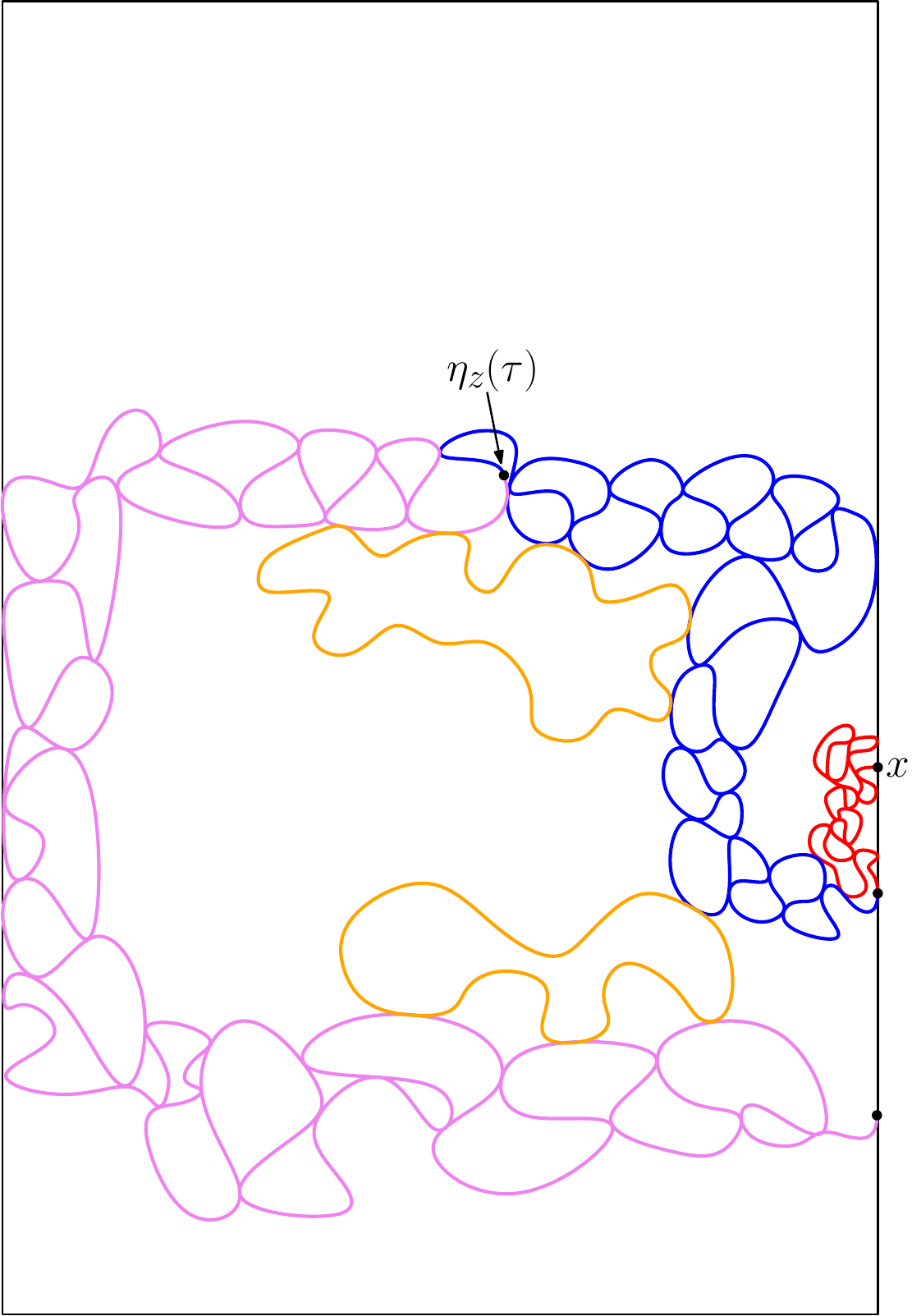}

\caption{Illustration of the exploration and the $\cwBCLE_{\kappa^\prime}(0)$/$\ccwBCLE_\kappa(-\kappa/2)$ iteration construction in the remaining domain when we stopped at time $\tau$. {\bfseries Left:} The purple depicts the $\cwBCLE_{\kappa^\prime}(0)$ branch that we were tracing at time $\tau$. One true loop of this $\cwBCLE_{\kappa^\prime}(0)$ is to the right of $\eta_z(\tau)$ and one false loop is to the left. The light green depicts the remainder of this true $\cwBCLE_{\kappa^\prime}(0)$ loop. The orange loops are two false $\ccwBCLE_\kappa(-\kappa/2)$ loops inside this true $\cwBCLE_{\kappa^\prime}(0)$ loop. {\bfseries Right:} The first case of $x$, where $x$ is on the false $\ccwBCLE_\kappa(-\kappa/2)$ loop that surrounds the $\cwBCLE_{\kappa^\prime}(0)$ that we were tracing at time $\tau$ (including the case where $x$ is on the domain boundary). In this case, the $\cwBCLE_{\kappa^\prime}(0)$ branch in the remaining domain from $\eta_z(\tau)$ to $x$ is given by following the branch of the original $\cwBCLE_{\kappa^\prime}(0)$ that we were tracing at time $\tau$ targeting $x$.}
\label{fig:markov-01}
\end{figure}

\begin{figure}[ht!]
\centering
\includegraphics[width=0.49\linewidth]{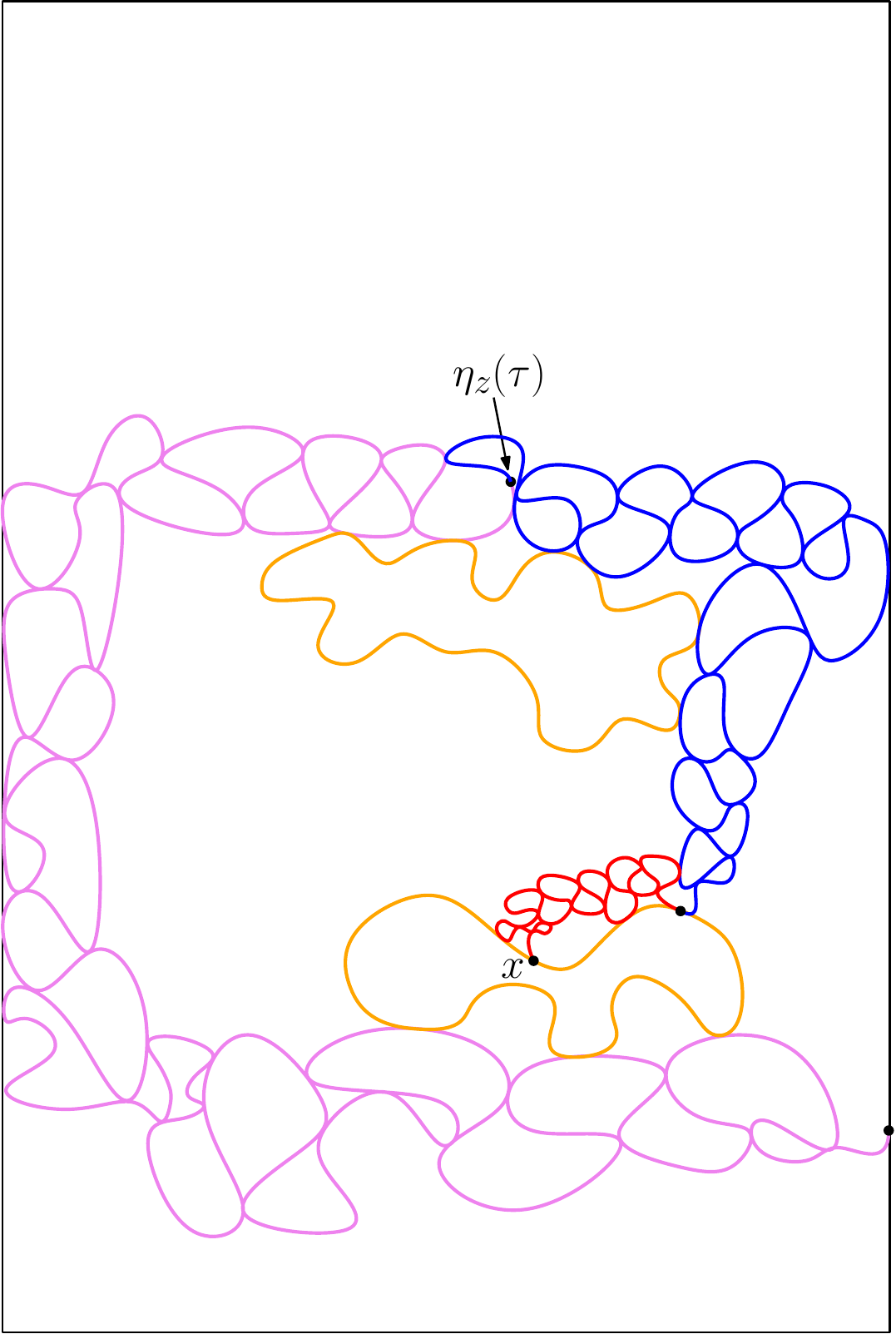}
\includegraphics[width=0.492\linewidth]{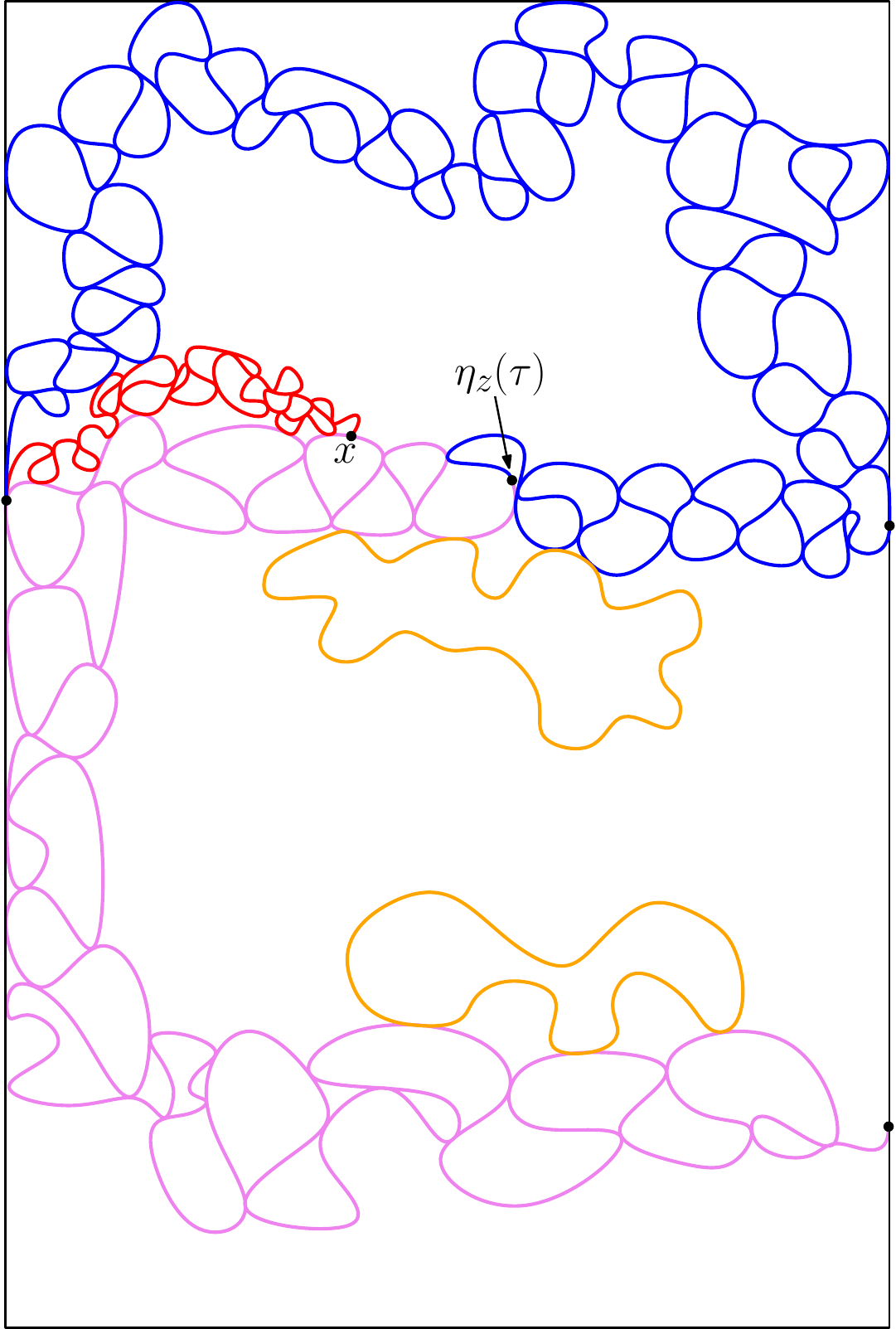}
\caption{{\bfseries Left:} The second case of $x$, where $x$ is on a false loop of the $\ccwBCLE_\kappa(-\kappa/2)$ inside the true $\cwBCLE_{\kappa^\prime}(0)$ loop that is on the right-hand side of $\eta_z(\tau)$. In this case, the $\cwBCLE_{\kappa^\prime}(0)$ branch in the remaining domain from $\eta_z(\tau)$ to $x$ is given by following the true $\cwBCLE_{\kappa^\prime}(0)$ loop that is on the right-hand side of $\eta_z(\tau)$ until closing the false $\ccwBCLE_\kappa(-\kappa/2)$ loop that contains $x$, and then follow the branch of the $\cwBCLE_{\kappa^\prime}(0)$ inside this false $\ccwBCLE_\kappa(-\kappa/2)$ loop from its root targeting $x$. {\bfseries Right:} The third case of $x$, where $x$ is on the false $\cwBCLE_{\kappa^\prime}(0)$ loop that we were tracing at time $\tau$ which is on the left-hand side of $\eta_z(\tau)$. In this case, the $\cwBCLE_{\kappa^\prime}(0)$ branch in the remaining domain from $\eta_z(\tau)$ to $x$ is given by following the $\cwBCLE_{\kappa^\prime}(0)$ that we were tracing at time $\tau$ until we close this false $\cwBCLE_{\kappa^\prime}(0)$ loop, then we follow the branch of the $\cwBCLE_{\kappa^\prime}(0)$ inside this false $\cwBCLE_{\kappa^\prime}(0)$ loop targeting $x$.}
\label{fig:markov-23}
\end{figure}

Note that the loops of the $\cwBCLE_{\kappa^\prime}(0)$'s in the remaining domain are not necessarily loops of the original $\cwBCLE_{\kappa^\prime}(0)$'s. However, we have the following.

\begin{lemma}\label{lem:excursion}

Suppose that we have the same setup as in~\Cref{lem:markovian_cpi_exploration}. Recall that the set of outermost true loops of all the $\cwBCLE_{\kappa^\prime}(0)$'s has the law of a non-nested $\CLE_{\kappa^\prime}$. Let $\SCI$ be a connected arc of some loop of this non-nested $\CLE_{\kappa^\prime}$. Suppose that $\SCI$ is contained in $D_{\tau}$ (hence it does not intersect $\eta_z([0,\tau])$). Then $\SCI$ is contained in some loop of the non-nested $\CLE_{\kappa^\prime}$ associated with the iterated $\cwBCLE_{\kappa^\prime}(0)$/$\ccwBCLE_\kappa(-\kappa/2)$ construction in $D_{\tau}$. 
\end{lemma}

\begin{figure}[ht!]
\centering
\includegraphics[width=0.49\linewidth]{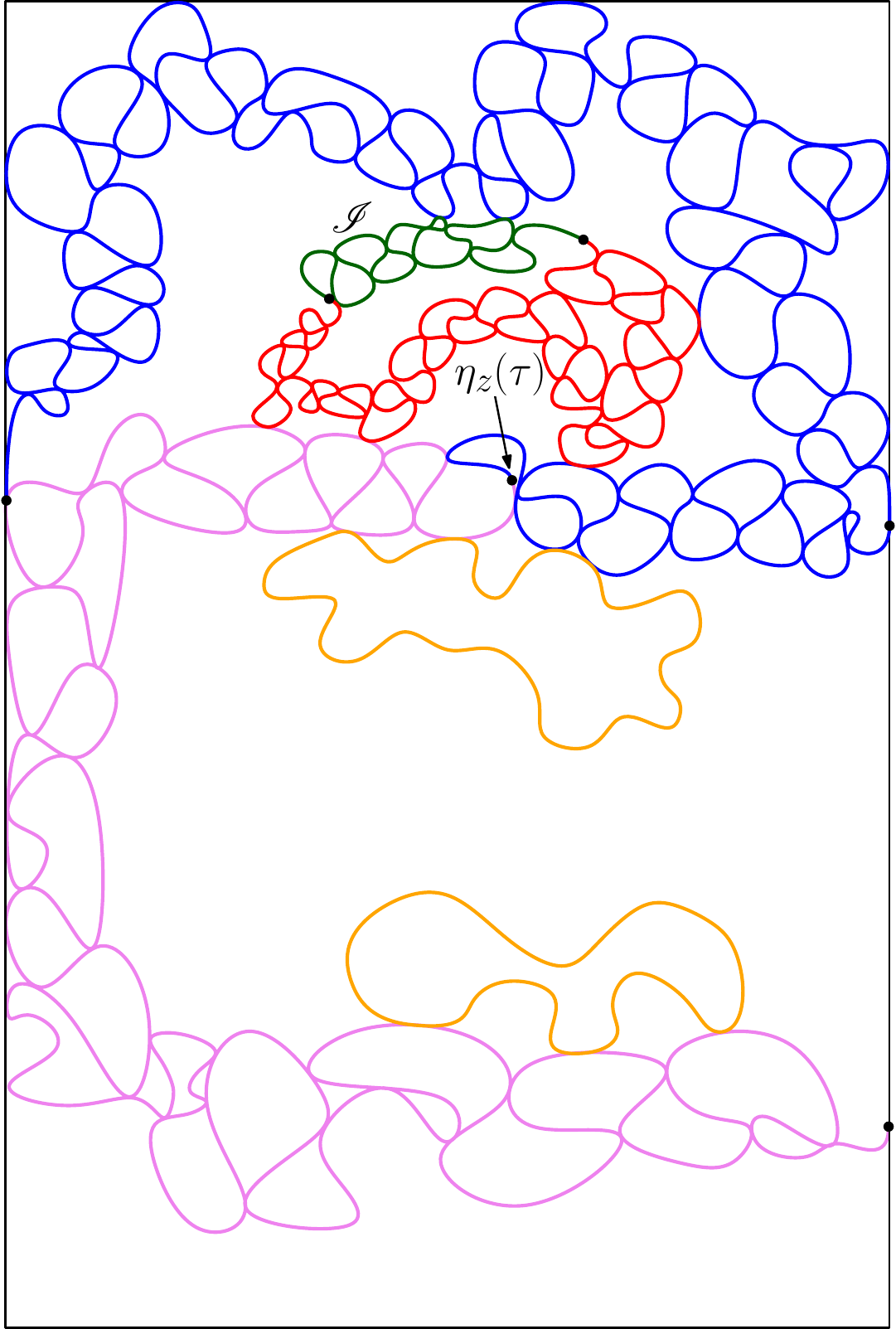}
\includegraphics[width=0.49\linewidth]{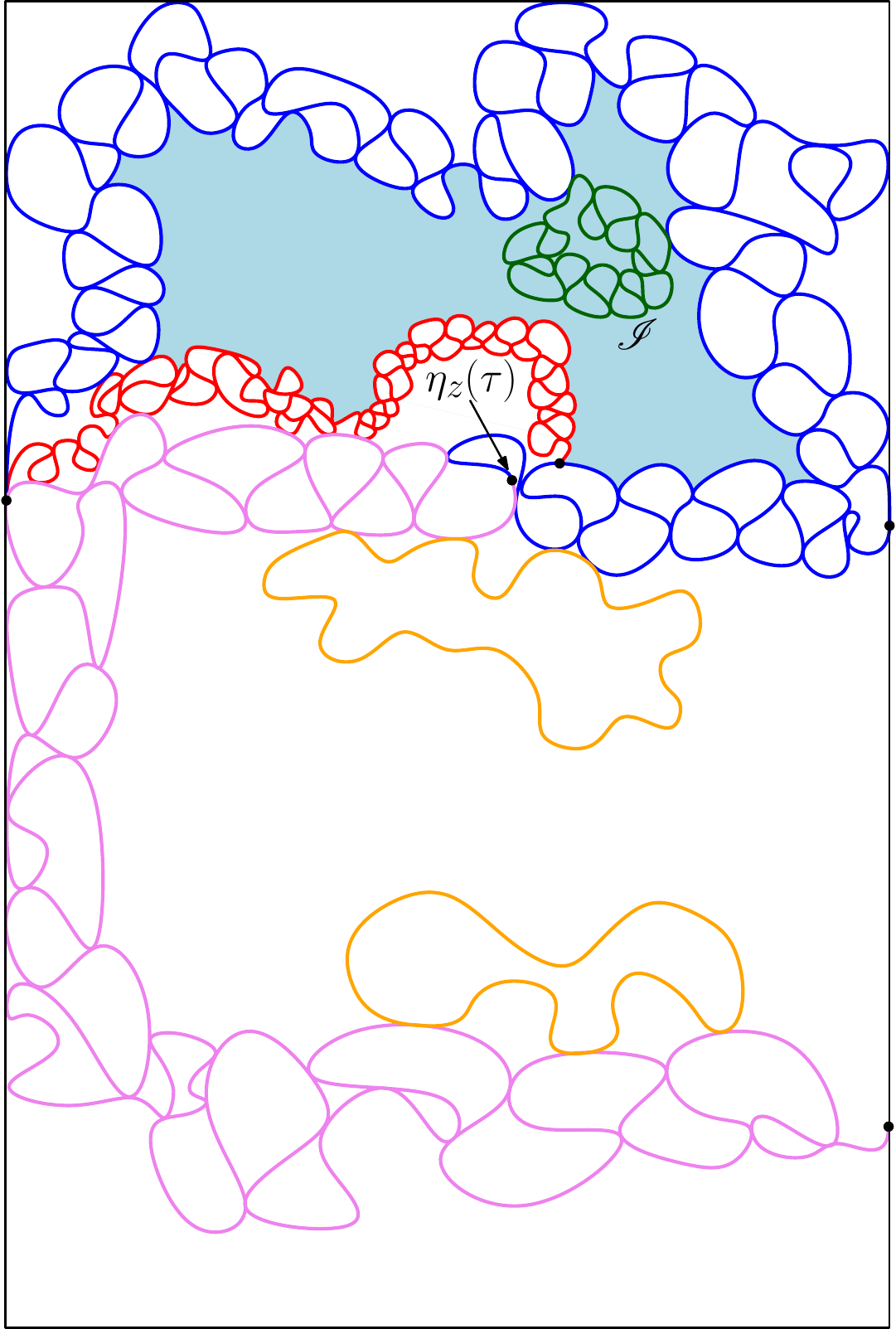}
\caption{Illustration of one situation of the proof of~\Cref{lem:excursion}, where $\SCI$ is contained in the false $\cwBCLE_{\kappa^\prime}(0)$ loop that lies immediately to the left of $\eta_z(\tau)$. The connected arc $\SCI$ is depicted in green. {\bfseries Left:} The original loop that contains $\SCI$ (the union of the red and green arcs) touches $\eta_z|_{[0, \tau]}$, in which case this loop is again a true loop of the first level $\cwBCLE_{\kappa^\prime}(0)$ in the remaining domain. {\bfseries Right:} The original loop that contains $\SCI$ does not touch $\eta_z|_{[0, \tau]}$. Then the light blue region is a connected component of a false loop of the first level $\cwBCLE_{\kappa^\prime}(0)$ in the remaining domain, and the original loop that contains $\SCI$ is a true loop of the $\cwBCLE_{\kappa^\prime}(0)$ inside it.}
\label{fig:excursion}
\end{figure}

\begin{proof}
\stepx{step:exc-cases}{The possible positions of $\SCI$} Consider the $\cwBCLE_{\kappa^\prime}(0)$ that we were tracing at time $\tau$. Since $\SCI$ is contained in the remaining domain and $\SCI$ is assumed to be contained in an outermost true loop of some $\cwBCLE_{\kappa^\prime}(0)$, it follows that this $\cwBCLE_{\kappa^\prime}(0)$ that we were tracing at time $\tau$ is not surrounded by some other true $\cwBCLE_{\kappa^\prime}(0)$ loop. Moreover, since $\SCI$ must be contained in a outermost true $\cwBCLE_{\kappa^\prime}(0)$ loop, it cannot be contained in any true or false loop of any $\ccwBCLE_\kappa(-\kappa/2)$ inside the true loops of the current $\cwBCLE_{\kappa^\prime}(0)$ that we were tracing at time $\tau$. Thus, either
	\begin{itemize}
		\item $\SCI$ is contained in the true loop of the $\cwBCLE_{\kappa^\prime}(0)$ that we were tracing at time $\tau$, or
		\item $\SCI$ is contained in the false loop of the $\cwBCLE_{\kappa^\prime}(0)$ that we were tracing at time $\tau$ (in a true loop of the outermost $\cwBCLE_{\kappa^\prime}(0)$ inside it, or deeper in the iteration), or
		\item $\SCI$ is outside of the true $\cwBCLE_{\kappa^\prime}(0)$ that we were tracing at time $\tau$ and outside of the false loop of the $\cwBCLE_{\kappa^\prime}(0)$ that we were tracing at time $\tau$.
	\end{itemize}
	Consider the domain that contains the $\cwBCLE_{\kappa^\prime}(0)$ that we were tracing at time $\tau$. (We may assume without loss of generality that this is the initial domain.)

\stepx{step:exc-first}{The first case} Consider the first case above. In this case, $\SCI$ must be part of the branch of the $\cwBCLE_{\kappa^\prime}(0)$ that we were tracing at time $\tau$ from $\eta_z(\tau)$ and targeting a point that is both on the domain boundary and the boundary of the remaining domain. Moreover, by the previous discussions, this branch must also be a branch of some outermost true $\cwBCLE_{\kappa^\prime}(0)$ loop in the remaining domain (cf.~\Cref{fig:markov-23} (left)). Since a branch of a BCLE can only separate into distinct loops at the boundary (every loop of a BCLE being a pocket cut out between two consecutive boundary hits of a branch), and $\SCI$ does not touch the boundary of the remaining domain, it follows that $\SCI$ is contained in some true loop of the outermost $\cwBCLE_{\kappa^\prime}(0)$ in the remaining domain.

\stepx{step:exc-second}{The false loop to the left of $\eta_z(\tau)$} Next, we consider the case that $\SCI$ is contained in the false loop that is immediately to the left of $\eta_z(\tau)$. See~\Cref{fig:excursion} for an illustration of this case. Note that it is possible that $\SCI$ touches the false loop that is immediately to the left of $\eta_z(\tau)$ (but not the portion of it that is contained in $\eta_z|_{[0, \tau]}$). Note that there are again two subcases: either the original $\CLE_{\kappa^\prime}$ loop that contains $\SCI$ touches $\eta_z(\tau)$ (cf.~\Cref{fig:excursion} (left)) or not (cf.~\Cref{fig:excursion} (right)). In the first subcase, by the previous discussions (cf.~\Cref{fig:markov-23} (right)), this loop is a true loop of the outermost $\cwBCLE_{\kappa^\prime}(0)$ in the remaining domain. In the second subcase, it is a true loop of some $\cwBCLE_{\kappa^\prime}(0)$ inside a false loop of the outermost $\cwBCLE_{\kappa^\prime}(0)$ in the remaining domain.

\stepx{step:exc-rest}{The remaining cases} Finally, we observe that, for the $\cwBCLE_{\kappa^\prime}(0)$ that we were tracing at time $\tau$, except for the true loop that is immediately to the right of $\eta_z(\tau)$ and the false loop that is immediately to the left of $\eta_z(\tau)$, all other loops (true or false) of this $\cwBCLE_{\kappa^\prime}(0)$ that are contained in the remaining domain are also loops of the outermost $\cwBCLE_{\kappa^\prime}(0)$ in the remaining domain, and the iteration constructions inside them are exactly the same for the original domain and for the remaining domain. Thus, except for the cases discussed in the preceding two paragraphs, other cases of $\SCI$ follow by this observation. This completes the proof of~\Cref{lem:excursion}.
\end{proof}

\subsubsection{The exploration gives all the loops}
\label{subsec:discovering_all_loops}

As explained at the beginning of~\Cref{sec:proof_of_main_result}, one of the main ingredients of the proof of~\Cref{prop:thin-loops} is a certain Markovian exploration of the loops of a simple (non-nested) $\CLE_{\kappa}$ in $\BD$ that intersect the inner and outer boundaries of a fixed Euclidean annulus. In this subsection, we are going to describe the aforementioned exploration in a detailed way and prove that it discovers all the simple loops intersecting the inner and outer boundaries of the annulus a.s.\ (\Cref{lemma all the loops are discovered}). We will also prove the following, for each fixed $\kappa'$ (\Cref{lem:segments_of_loops_inside_rectangles}). Every segment of a loop in the $\CLE_{\kappa'}$ in $\BD$ (which is coupled with the simple $\CLE_{\kappa}$ as in~\Cref{A021}) that is included in one of the conformal rectangles formed by the crossings of the annulus by $\CLE_{\kappa}$ loops (on the event that such crossings exist) is part of a loop of a $\CLE_{\kappa'}$ in the above conformal rectangle.

We continue to use the notation of the preceding subsection. Let $\Gamma$ denote the non-nested $\CLE_\kappa$ consisting of the true loops of all the $\ccwBCLE_\kappa(-\kappa/2)$'s. 

In the present subsection, we fix an annulus $A_{s,t}(z) = B_t(z) \setminus \overline{B_s(z)}$ inside $\BD$. On the event that there exists a loop of $\Gamma$ crossing between the inner and outer boundaries of $A_{s,t}(z)$, we define an exploration that explores all the loops of $\Gamma$ that intersect the boundary of $A_{s,t}(z)$, and such that the restriction of the $\cwBCLE_{\kappa^\prime}(0)$'s and $\ccwBCLE_\kappa(-\kappa/2)$'s in the remaining domains again has the law of the $\cwBCLE_{\kappa^\prime}(0)$/$\ccwBCLE_\kappa(-\kappa/2)$ iteration construction in the remaining domains. 

Recall that the collection of all the outermost true loops of all the $\cwBCLE_{\kappa^\prime}(0)$'s forms a non-nested $\CLE_{\kappa^\prime}$. Let us denote it by $\Gamma^\prime$. 

More precisely, the exploration is defined in the following order. Let $(Z_n)_{n\ge 1}$ be a sequence of i.i.d.\ uniform random variables on $\BD$. Then the exploration is defined inductively. 
We start with the iterated $\cwBCLE_{\kappa^\prime}(0)$/$\ccwBCLE_\kappa(-\kappa/2)$ construction used to generate $\Gamma$ and $\Gamma'$, and re-root it (in the sense of~\Cref{subsec:cpi_exploration}) at some point $w_1$ chosen (independently of everything else) according to the harmonic measure seen from $Z_1$ in $\partial \BD \setminus A_{s,t}(z)$ (which is actually $\partial \BD$ since $A_{s,t}(z)\subset \BD$). Then, we consider the branch of the CPI exploration associated with the iterated $\cwBCLE_{\kappa^\prime}(0)$/$\ccwBCLE_\kappa(-\kappa/2)$ construction re-rooted at $w_1$, started at $w_1$ and targeted at $Z_1$, and we stop it at the first time that it discovers a loop in $\Gamma$ that intersects $A_{s,t}(z)$ , or when it hits $\overline{A_{s,t}(z)}$ (but it might be the case that the branch of the CPI exploration is not stopped, in which case we draw the full branch to $Z_1$).

Assume that the exploration is constructed until step $n \geq 1$. The unexplored region at step $n$ is defined as the connected component containing $Z_{n+1}$ of the complement in $\BD$ of the union of all branches drawn during the first $n$ steps of the exploration and of the interiors of the $\CLE_{\kappa}$ loops discovered in the first $n$ steps of the exploration. Recall that~\Cref{lem:markovian_cpi_exploration} implies that the iterated process of loops constructed using the restriction to the unexplored region at step $n$ of the original CPI exploration in $\BD$ has the law of an iterated $\cwBCLE_{\kappa^\prime}(0)$/$\ccwBCLE_\kappa(-\kappa/2)$ construction in the unexplored region.

On the event that $Z_{n+1}$ does not belong to $A_{s,t}(z)$ and that $Z_{n+1}$ is in the unexplored region, we sample a point $w_{n+1}$ (independently of everything else) according to the harmonic measure seen from $Z_{n+1}$ on the subset of the boundary of the unexplored region which is not included in $A_{s,t}(z)$. Then, we consider the branch of the CPI exploration associated with the iterated $\cwBCLE_{\kappa^\prime}(0)$/$\ccwBCLE_\kappa(-\kappa/2)$ construction in the unexplored region re-rooted at $w_{n+1}$, started at $w_{n+1}$ and targeted at $Z_{n+1}$, and stopped at the first time that it discovers a loop in $\Gamma$ intersecting $A_{s,t}(z)$, or when it hits $\overline{A_{s,t}(z)}$. 

On the event that $Z_{n+1}$ belongs to $A_{s,t}(z)$ or is not in the unexplored region, we do nothing. 

Note that~\Cref{lem:markovian_cpi_exploration} implies that all the simple true loops discovered during the steps of the above exploration are loops in $\Gamma$.
Also, for each $n$, the conditional law of the path that we draw at step $n$ (on the event that $Z_{n}$ does not belong to $A_{s,t}(z)$ and is in the unexplored region) given all the paths that we have drawn during the first $n-1$ steps is that of a radial $\SLE_{\kappa}(\kappa-6)$ with $\beta=1$ in the unexplored region at step $n$, by~\Cref{lem:markovian_cpi_exploration} together with~\Cref{A021}.

Let us check that the above exploration indeed discovers all the loops intersecting the boundary of $A_{s,t}(z)$.
\begin{lemma}\label{lemma all the loops are discovered}
	Almost surely, the above exploration indeed discovers all the loops of $\Gamma$ intersecting the outer boundary of $A_{s,t}(z)$. The same is true for the inner boundary on the event that there is a $\CLE_\kappa$ loop which crosses $A_{s,t}(z)$.
	\end{lemma}
	\begin{proof}
\stepx{step:disc-reduction}{Reduction to hitting a fixed ball} Let us first deal with the loops intersecting the outer boundary of $A_{s,t}(z)$. Let $z' \in \BD \setminus \overline{A_{s,t}(z)}$ be a point that is not disconnected from $\partial \BD$ by $A_{s,t}(z)$ and let $\varepsilon>0$. It suffices to show that the loop $\SCL$ of $\Gamma$ surrounding $B_{\varepsilon}(z')$ (when it exists) will be a.s.\ discovered by the above exploration (a.s.\ no loop of $\Gamma$ is tangent to $\partial B_s(z)$ or $\partial B_t(z)$, so every loop meeting $\partial A_{s,t}(z)$ has interior points outside $\overline{A_{s,t}(z)}$). Let $\delta \in (0, \varepsilon)$ such that $B_{2\delta}(z') \cap A_{s,t}(z) = \emptyset$.

\stepx{step:disc-corridor}{Construction of a corridor of definite width} For all $j\ge 1$, let $n_j$ be the $j$-th integer $n$ such that $Z_n \in B_\delta(z')$. Let $U_j$ denote the connected component of the unexplored region at step $n_j-1$ containing $Z_{n_j}$. Let $f_j \colon U_j \to \BD$ be the conformal map from the connected component of the unexplored region containing $Z_{n_j}$ to $\BD$ such that $f_j(Z_{n_j})=0$ and $f_j'(Z_{n_j})>0$. Note that since $Z_{n_j} \in B_\delta(z')$ and $B_{2\delta}(z') \cap A_{s,t}(z) = \emptyset$, the probability that a planar Brownian motion starting from $Z_{n_j}$ hits $\partial \BD$ before $A_{s,t}(z)$ is bounded from below by a constant $p>0$ which does not depend on $j$. In particular, the probability that a planar Brownian motion starting from $Z_{n_j}$ hits the explored region before hitting $A_{s,t}(z)$ is bounded from below by $p$. Therefore, the probability that a planar Brownian motion starting from $0$ hits $\partial \BD$ before $f_j(A_{s,t}(z) \cap U_j)$ is bounded from below by $p$. Moreover, note that the number of connected components of $f_j(A_{s,t}(z)\cap U_j)$ is at most one plus the number of $\CLE_\kappa$ loops crossing $A_{s,t}(z)$, so that the number of connected components of $f_j(A_{s,t}(z)\cap U_j)$ is bounded by a random variable $N$ which does not depend on $j$. Thus, on the  event that $N \leq m$, for some fixed and deterministic $m \in \BN$, there exists $\varepsilon'\in (0, \delta)$ depending only on $p,\delta$, and $m$ which does not depend on $j$, and a path $\gamma_j$ from $\partial \BD$ to $0$ such that $B_{2\varepsilon'}(\gamma_j) \subset \BD \setminus f_j(A_{s,t}(z)\cap U_j)$ (by Beurling's estimate: if every path from $0$ to $\partial\BD$ in this complement had a bottleneck of width smaller than $\varepsilon'$, then, the complement having at most $m+1$ boundary components, the harmonic measure of $\partial \BD$ from $0$ in it would be smaller than $p$). For any path $\gamma$ in $\BD$, there exists a path $\widetilde{\gamma}$ in the grid $(\varepsilon'/8)\BZ^2$ which is $\varepsilon'/2$-close to $\gamma$ for the Hausdorff distance so actually we may divide $\varepsilon'$ by two to assume that $\gamma_j$ is the intersection with $\BD$ of a path in the grid $(\varepsilon'/8)\BZ^2$. Note that there are finitely many such simple paths.

\stepx{step:disc-trunk}{The trunk follows the corridor} Next, we apply \cite[Lemma 2.3]{miller2017intersections} to the trunk of a radial $\SLE_\kappa(\kappa-6)$ starting from a  random point on $\partial \BD \cap B_{\varepsilon'}(\gamma_j)$ sampled according to the Lebesgue measure and independently of everything else, and targeting $0$. The trunk is a radial $\SLE_{\kappa'}(0;\kappa'-6)$ targeting $0$ by \cite[Theorem~1.6]{SimCLELQG} and one can see that $\kappa'-6>-2$. 
	So, by  \cite[Lemma 2.3]{miller2017intersections}, with probability at least $p_0(\varepsilon')>0$, where $p_0(\varepsilon')$ is deterministic and only depends on $\varepsilon'$, the trunk $\eta$ of a radial $\SLE_\kappa(\kappa-6)$ starting from a random point on $\partial \BD \cap B_{\varepsilon'}(\gamma_j)$ sampled according to the Lebesgue measure and independently of everything else, and targeting $0$, stopped when it hits $f_j(B_\delta(z'))$, is included in $B_{\varepsilon'}(\gamma_j)$. Note that on this event the trunk stopped when it hits $f_j(B_\delta(z'))$ can be coupled with one chordal $\SLE_{\kappa'}(\kappa'-6)$ targeting a fixed point in $\partial \BD \setminus B_{\varepsilon'}(\gamma_j)$ \cite{SLE-CoC} so that we only apply \cite[Lemma 2.3]{miller2017intersections} to one chordal curve. The fact that $p_0(\varepsilon')$ does not depend on the choice of the path $\gamma_j$ stems from the fact that there are finitely many paths in the grid $(\varepsilon'/8)\BZ^2$ that are included in the unit disk.

\stepx{step:disc-loops}{The attached loops stay near the trunk} Recall that~\Cref{A021} implies that the radial $\SLE_\kappa(\kappa-6)$ is obtained by attaching $\ccwBCLE_\kappa(-\kappa/2)$ loops on the right-hand side of the trunk. These loops are traced by a chordal $\SLE_\kappa(-\kappa/2; 3\kappa/2 -6)$ from the right-hand side of the starting point of the trunk to the tip of the stopped trunk $\eta$. As in the previous paragraph, we conformally map the unexplored component containing $Z_{n_j}$ to the unit disk and $Z_{n_j}$ is sent to $0$ and we denote by $\widetilde{\eta}$ the image of the chordal $\SLE_\kappa(-\kappa/2; 3\kappa/2 -6)$. Recall that $\eta$ is the trunk stopped when it hits $f_j(B_\delta(z'))$ for the first time. So, conditionally on $\eta$, the curve $\widetilde{\eta}$ is a chordal $\SLE_\kappa(-\kappa/2; 3\kappa/2 -6)$ from the right-hand side of $\eta(0)$ to the endpoint of $\eta$, where its force points are located immediately to the left and right of $\eta(0)$ respectively \cite{CLEPerc}.

Let $V_j$ denote the connected component of $\BD \setminus \eta$ containing $0$, and let $g_j: V_j \to \BH$ be the conformal transformation mapping $V_j$ onto $\BH$ such that it maps the rightmost point of $\eta \cap \partial \BD$ to $0$, the tip of $\eta$ to $1$, the leftmost point of $\eta \cap \partial\BD$ to $\infty$, and the right-hand side of $\eta$ onto $[0,1]$. Then, $g_j(\widetilde{\eta})$ is a chordal $\SLE_\kappa(-\kappa/2; 3\kappa/2 -6)$ from $0$ to $1$ with the force points located at $0^-$ and $0^+$ respectively. Furthermore, note that there exists some deterministic constant $c(\varepsilon',\delta)>0$ depending only on $\varepsilon'$ and $\delta$ such that  with probability at least $c(\varepsilon', \delta)>0$, a planar Brownian motion starting from $0$ stopped when it hits $\partial \BD$ for the first time stays in the union $\overline{B_\delta(0)}\cup B^{\mathrm{R}}_{\varepsilon'}(\eta)$, where $B^{\mathrm{R}}_{\varepsilon'}(\eta)$ is the set of points of $\BD$ at distance smaller than $\varepsilon'$ from the right-hand side of $\eta$ (and the same is true if we replace the right-hand side of $\eta$ by the left-hand side of $\eta$). In particular, there exists $\varepsilon''>0$ depending only on $\varepsilon', \delta$ such that, noting that $[0,1]$ is the image by $g_j$ of the right-hand side of $\eta$, we have
	\[B_{\varepsilon''}([0,1]) \subset g_j(B^{\mathrm{R}}_{\varepsilon'}(\eta)).\]

Moreover, \cite[Lemma~A.1]{CoInCLERiemSph} implies that there exists some deterministic constant $p_1(\varepsilon'')>0$ depending only on $\varepsilon''$ such that a.s., with conditional probability at least $p_1(\varepsilon'')$, we have that $g_j(\widetilde{\eta}) \subseteq B_{\varepsilon''}([0,1])$. In particular, going back to the unit disk, we have that a.s., with conditional probability at least $p_1(\varepsilon'')$, we have that $\widetilde{\eta}$ does not hit $f_j(A_{s,t}(z) \cap U_j)$.	

\stepx{step:disc-outer}{Conclusion for the outer boundary} Thus, conditionally on the explored region at step $n_j-1$, with probability at least $p'>0$ for some $p'>0$ deterministic which does not depend on $j\ge 1$, the radial $\SLE_{\kappa}(\kappa-6)$ will hit $B_\delta(z')$. So, a.s., the exploration will hit $B_\delta(z')$ in finitely many steps 
. In particular, a.s., the exploration will discover the loop $\SCL$ when it exists.

\stepx{step:disc-inner}{The inner boundary} Finally, let us deal with the loops intersecting the inner boundary of $A_{s,t}(z)$ on the event that there is a $\CLE_\kappa$ loop crossing $A_{s,t}(z)$. Let $z'$ be in the connected component of $\BD \setminus A_{s,t}(z)$ which is disconnected from $\partial \BD$. Let $\delta>0$ such that $B_{2\delta}(z') \cap A_{s,t}(z) = \emptyset$. On the event that there is a $\CLE_\kappa$ loop crossing $A_{s,t}(z)$, we have already shown that such a loop will be explored a.s. Let $n_0$ be the step at which the first loop crossing $A_{s,t}(z)$ is discovered. Then, note that for all $n> n_0$ such that $Z_n \in B_\delta(z')$, the probability that a planar Brownian motion starting from $Z_n$ hits the explored region before $A_{s,t}(z)$ is bounded from below by some (random) constant $p''>0$ which does not depend on $n$. We can then apply the same reasoning as in \Crefrange{step:disc-corridor}{step:disc-outer}.
\end{proof}

Recall that, on the event that there is a loop of $\Gamma$ crossing the annulus, the loops of $\Gamma$ intersecting the boundary of $A_{s,t}(z)$ bound a finite number of conformal rectangles. Each rectangle is a connected component of $A_{s,t}(z)$ minus the closure of the union of the domains encircled by the loops of $\Gamma$ which intersect the boundary of $A_{s,t}(z)$. Let us also check that the connected component of the unexplored region at step $n$ of the exploration defined at the beginning of the current subsection containing the conformal rectangle converges to that conformal rectangle. More precisely, let $R_1, \ldots, R_k$ be the conformal rectangles bounded by the loops of $\Gamma$ intersecting $\partial A_{s,t}(z)$. Let $R^n_1, \ldots, R^n_k$ be the connected components of the unexplored region at step $n$ of the exploration containing respectively $R_1, \ldots, R_k$ (note that these connected components are not necessarily distinct).
\begin{lemma}\label{lemma cv Caratheodory unexplored components to rectangles}
On the event that there is a loop of $\Gamma$ crossing $A_{s,t}(z)$, a.s., for each conformal rectangle bounded by the loops of $\Gamma$ intersecting $\partial A_{s,t}(z)$, the connected component of the unexplored region at step $n$ converges as $n\to \infty$ in the Carath\'eodory sense to this rectangle.
\end{lemma}
\begin{proof}
Let $j \in [1,k]_\BZ$. By~\Cref{lemma all the loops are discovered}, we know that for any loop of $\Gamma$ intersecting the boundary of $A_{s,t}(z)$ there exists $n\ge 1$ such that the loop has been discovered at step $n$.

Let $K \subset R_j$ be a compact subset, then we automatically have $K \subset R^n_j$ for all $n\ge 1$ since $R_j \subseteq R^n_j$ by construction (the branches are stopped upon hitting $\overline{A_{s,t}(z)}$, and the interiors of the discovered loops do not meet the rectangles).

Let $w \in R_j$. Let $U\ni w$ be a connected open subset of $\BD$ such that $U \subseteq R^n_j$ for infinitely many $n\ge 1$. Let us show that $U \subseteq R_j$. Let $\SCL$ be a loop of $\Gamma$ which intersects $U$ and intersects $\partial A_{s,t}(z)$, then $\SCL$ will be discovered eventually by the exploration by~\Cref{lemma all the loops are discovered}, which is impossible since $U \subseteq R^n_j$ for infinitely many $n\ge 1$. Therefore, the loops of $\Gamma$ that intersect $U$ are either contained in $A_{s,t}(z)$ or in $\BD \setminus \overline{A_{s,t}(z)}$. Therefore, since $U$ is open and connected, we see that $U$ is included in $\BD$ minus the closure of the union of the loops of $\Gamma$ which intersect $\partial A_{s,t}(z)$. Thus, by definition of $R_j$ and by connectedness of $U$, we deduce that $U\subseteq R_j$ (here we use that a.s.\ every point of $\partial A_{s,t}(z)$ lies in the closure of the union of the domains encircled by the loops of $\Gamma$ meeting $\partial A_{s,t}(z)$, so that $U$ cannot leave $A_{s,t}(z)$).
\end{proof}
From the above lemma, we get in particular the convergence of the $\CLE_{\kappa'}$'s in the connected components of the unexplored regions $R^n_j$ containing the conformal rectangles $R^n_j$ for $j \in [1, k]_\BZ$ toward $\CLE_{\kappa'}$'s in the $R_j$'s.
\begin{lemma}\label{lemma cv CLE in the rectangles}
On the event that there is a loop of $\Gamma$ crossing $A_{s,t}(z)$ and that there are $k\ge 1$ conformal rectangles $R_1, \ldots, R_k$, for all $j \in [1, k]_\BZ$, let $\Gamma'_{R^n_j}$ be the non-nested $\CLE_{\kappa'}$ associated with the $\cwBCLE_{\kappa^\prime}(0)$/$\ccwBCLE_\kappa(-\kappa/2)$ iteration construction (see~\Cref{A025}) in the connected component of the unexplored region $R^n_j$ containing $R_j$. For all $x\in R_j \cap \BQ^2$, consider the loop $\SCL_j^n(x)$ of $\Gamma'_{R^n_j}$ surrounding $x$. Then, for all simply connected open subset $U \subset R_j$ such that $\overline{U} \subset R_j$ there exists a sequence $(U_n)_{n\ge 1}$ of simply connected open subsets such that $\overline{U_n} \to \overline{U}$ a.s.\@ in the Hausdorff topology, $U_n \to U$ a.s.\@ in the Carathéodory sense, and such that the conditional laws given $R^n_j$ of the $\overline{U_n}\cap \SCL^n_j(x)$'s for $x\in R_j \cap \BQ^2$ converge a.s.\ weakly, in the Hausdorff topology, to the conditional law given $R_j$ of the $\overline{U} \cap \SCL_j(x)$'s for $x \in R_j \cap \BQ^2$, where the $\SCL_j(x)$'s are the loops of a $\CLE_{\kappa'}$ in $R_j$. 
\end{lemma}
\begin{proof}
The Carath\'eodory convergence of~\Cref{lemma cv Caratheodory unexplored components to rectangles} can be written as the almost sure uniform convergence on compact subsets of the conformal mapping $f^n_j\colon \BD \to R^n_j$ such that $f^n_j(0)=z_j$ and $(f^n_j)'(0)>0$ to the conformal mapping $f_j \colon \BD \to R_j$ such that $f_j(0)=z_j$ and $f_j'(0)>0$, where $z_j$ is a point of $R_j$ chosen arbitrarily in a deterministic way. Combining this convergence with the fact that $(f^n_j)^{-1}(\Gamma'_{R^n_j})$ is a $\CLE_{\kappa'}$ in the unit disk (the unexplored components being simply connected), we get the desired result, taking $U_n=  f^n_j(f_j^{-1}(U))$. Let $z'_j \in U$. Taking a conformal mapping $\phi\colon \BD \to f_j^{-1}(U)$ such that $(\phi \circ f_j)(0)=z'_j$ and $(\phi \circ f_j)'(0)>0$ we see that $\phi \circ f^n_j$ converges a.s.\ to $\phi \circ f_j$ uniformly on compact subsets of $\BD$ so that $U_n \to U$ almost surely in the Carathéodory sense.
\end{proof}
Combining the above lemma with~\Cref{lem:excursion}, we obtain the following lemma.
\begin{lemma}\label{lem:segments_of_loops_inside_rectangles}
	We work on the event that there is a loop of $\Gamma$ crossing $A_{s,t}(z)$. Recall that we denote by $R_1, \ldots, R_k$ be the conformal rectangles bounded by the loops of $\Gamma$ intersecting $\partial A_{s,t}(z)$. Then, there is a coupling with loop ensembles $\Gamma'_{R_1}, \ldots, \Gamma'_{R_k}$ in the respective conformal rectangles $R_1, \ldots, R_k$ that are $\CLE_{\kappa'}$'s conditionally on $R_1, \ldots, R_k$ and such that the following holds almost surely. For all segment $\SCI$ of a loop of $\Gamma'$ which is included a conformal rectangle $R_j$ for $j \in [1, k]_\BZ$, $\SCI$ is a segment of a loop of $\Gamma'_{R_j}$.
\end{lemma}
\begin{proof}
	Let $j \in [1,k]_\BZ$. Let $U\subset R_j$ by a simply connected open subset such that $\overline{U} \subset R_j$. By~\Cref{lemma cv CLE in the rectangles}, we know that there exists a sequence of simply connected open subsets $U_n\subset R^n_j$ such that $U_n \to U$ a.s.\@ in the Carathéodory sense, such that the conditional laws given $R^n_j$ of the $\overline{U_n}\cap \SCL^n_j(x)$'s for $x\in R_j \cap \BQ^2$ converge a.s.\ weakly, in the Hausdorff topology, to the conditional law given $R_j$ of the $\overline{U} \cap \SCL_j(x)$'s for $x \in R_j \cap \BQ^2$, where the $\SCL_j(x)$'s are the loops of a $\CLE_{\kappa'}$ in $R_j$ that we denote by $\Gamma'_{R_j}$. 
	
	Furthermore let $\SCI_1, \ldots, \SCI_r$ be segments of loops of $\Gamma'$ that are included in $U$. For all $n\ge 1$, by~\Cref{lem:excursion}, applied inductively at the first $n$ steps, using the notation of~\Cref{lemma cv CLE in the rectangles}, we know that $\SCI_{i}$ is a segment of a loop of $\Gamma'_{R^n_j}$ for all $i \in[1, r]_\BZ$. By taking the first part of the proof into account, we see that there is a coupling with $\Gamma'_{R_j}$ such that for all $i \in[1, r]_\BZ$, the segment $\SCI_i$ is a segment of a loop of $\Gamma'_{R_j}$.
\end{proof}

\subsection{Completion of the proof of~\Cref{prop:thin-loops}}
\label{subsec:completion_of_the_proof}

In this last subsection, $\Gamma_n$ is a (non-nested) $\CLE_{\kappa_n}$ and $\Gamma'_n$ is a (non-nested) $\CLE_{\kappa'_n}$. We would like to combine~\Cref{prop:big_loops_are_fat_simple_cle} with~\Cref{lem:unlikeliness,lem:segments_of_loops_inside_rectangles} to complete the proof of~\Cref{prop:thin-loops}. However, there are two main issues remaining. First, we must ensure that whenever a loop in $\Gamma_n'$ intersects a loop in $\Gamma_n$, it actually surrounds it, so that we can use~\Cref{prop:big_loops_are_fat_simple_cle}. Second,~\Cref{lem:segments_of_loops_inside_rectangles} requires the annulus $A_{s,t}(z)$ to be strictly contained in $\BD$, which means we must control the behavior of loops near $\partial \BD$.  To resolve these issues, we proceed in three steps:
\begin{enumerate}[label=(\roman*)]
\item In~\Cref{subsubsec:loop_surrounding}, we show that, under the setup of~\Crefrange{subsec:cpi_exploration}{subsec:discovering_all_loops}, whenever a loop in $\Gamma_n'$ intersects a loop in $\Gamma_n$, it necessarily surrounds it.
\item  In~\Cref{subsubsec:big_loops_do_not_intersect_the_boundary}, we establish that with probability tending to $1$ as $n \to \infty$, macroscopic loops in $\Gamma_n'$ do not intersect $\partial \BD$.
\item In~\Cref{subsubsec:completion_of_proof}, we combine these ingredients to complete the proof of~\Cref{prop:thin-loops}.
\end{enumerate}

\subsubsection{Every loop in the non-simple CLE that intersects a loop in the simple CLE has to surround it}
\label{subsubsec:loop_surrounding}

Fix $\kappa' \in (4,8)$ and let $\Gamma'$ (resp.\ $\Gamma$) denote a $\CLE_{\kappa'}$ (resp.\ $\CLE_{\kappa}$) in $\BD$ coupled as in \Crefrange{subsec:cpi_exploration}{subsec:discovering_all_loops} with $\kappa = 16 / \kappa'$. We will prove the following (thereby  resolving one of the two issues described above).

\begin{lemma}\label{lem:loop_surrounds_another_loop}
It is a.s.\ the case that the following holds. If $\SCL'$ is a loop in $\Gamma'$ that intersects some loop $\SCL$ in $\Gamma$, then $\SCL'$ surrounds $\SCL$.
\end{lemma}

The main ingredient in the proof of~\Cref{lem:loop_surrounds_another_loop} is the following lemma which says that any two distinct loops of any $\CLE_{\kappa'}$ for $\kappa' \in (4,8)$ a.s.\ do not intersect at points which are cut points of either of the loops,  where a point $x$ on a loop $\SCL$ is called a \emph{cut point} of $\SCL$ if $x$ lies on the intersection of the outer boundary of $\SCL$ (i.e., the boundary of the unbounded connected component of $\BC \setminus \SCL$) with the boundary of some connected component of $\BC \setminus \SCL$ which is surrounded by $\SCL$.

\begin{lemma}\label{lemma no pinched point}
Let $\Gamma'$ be a $\CLE_{\kappa'}$ in $\BD$ with $\kappa' \in (4,8)$.  Then,  it is a.s.\ the case that the following is true.  Let $\SCL,  \SCL'$ be two distinct loops in $\Gamma'$ such that $\SCL \cap \SCL' \neq \emptyset$.  Then,  $\SCL \cap \SCL'$ does not contain any cut points of $\SCL$ or of $\SCL'$.
\end{lemma}

Lemma~\ref{lemma no pinched point} will follow from combining Lemmas~\ref{lem:sle_bubbles} and~\ref{lem:cut_points_not_on_the_boundary} below.  Lemma~\ref{lem:sle_bubbles} states that the complementary connected components of a chordal $\SLE_{\kappa'}(\kappa'-6)$ in $\BD$ that are surrounded by the curve do not intersect $\partial \BD$ a.s.  Lemma~\ref{lem:cut_points_not_on_the_boundary} states that the loops in a $\CLE_{\kappa'}$ in $\BD$ do not have cut points on $\partial \BD$ a.s.

\begin{lemma}\label{lem:sle_bubbles}
Fix $\kappa' \in (4,8)$ and let $\eta'$ be an $\SLE_{\kappa'}(\kappa'-6)$ curve in $\BD$ from $-\ri$ to $\ri$ with the force point located at $(-\ri)^+$.  Then,   it is a.s.\ the case that if $U$ is any connected component of $\BD \setminus \eta'$ such that $\partial U \subseteq \eta'$,  we have that $\partial U \cap \partial \BD = \emptyset$.
\end{lemma}

\begin{proof}
Fix $z \in \BD_{\BQ}$ and let $U$ denote the connected component of $\BD \setminus \eta'$ containing $z$.  For the rest of the proof,  we will assume that we are working on the event that $\partial U \subseteq \eta'$.
Suppose first that $\eta'$ draws $\partial U$ in the clockwise way.  Then,  it was shown in Step 2 in the proof of \cite[Proposition~3.1]{doherty2025square} that it is a.s.\ the case that $\partial U$ does not intersect the clockwise arc of $\partial \BD$ from $-\ri$ to $\ri$.

We will show that $\partial U$ does not intersect the counterclockwise arc of $\partial \BD$ from $-\ri$ to $\ri$ and hence conclude that $\partial U \cap \partial \BD = \emptyset$.  Suppose that this does not hold and let $t_1$ (resp.\ $t_2$) be the first (resp.\ last) time that $\eta'$ intersects $\partial U$ and note that $w = \eta'(t_1) = \eta'(t_2)$.  Let $s \in [t_1,t_2]$ denote the last time before $t_2$ that $\eta'$ intersects the counterclockwise arc of $\partial \BD$ from $-\ri$ to $\ri$.  Note that it is explained in the proof of \cite[Theorem~5.4]{TreeCLE} that it is a.s.\ the case that $\SLE_{\kappa'}$ hits each point on the boundary of the domain that it is defined at most once. Thus, since an $\SLE_{\kappa'}(\kappa'-6)$ curve targeted at its force point has the law of an $\SLE_{\kappa'}$ curve (and, by target invariance~\cite{TreeCLE}, $\eta'$ agrees with the branch targeted at the force point up to the time at which the target points are disconnected),  we obtain that it is a.s.\ the case that $\eta'$ hits each point on $\partial \BD$ at most once.  In particular,  we have that $t_1 < s < t_2$ and $w \notin \partial \BD$.

Let $\tau$ be the last time before $t_1$ that $\eta'$ intersects the counterclockwise arc of $\partial \BD$ from $-\ri$ to $\ri$ and let $V$ denote the connected component of $\BD \setminus \eta'([0,s])$ whose boundary contains the counterclockwise arc of $\partial \BD$ from $\eta'(\tau)$ to $\eta'(s)$.  Then,  since $\partial U$ is drawn by $\eta'$ in the clockwise way and $\eta'$ does not trace itself and hits $w$ after time $s$,  we obtain that there exists $t \in (s,t_2)$ such that $\eta'(t) \in V$.  In particular,  we have that $\partial V$ disconnects $\eta'(t)$ from $\ri$ but this contradicts the fact that $\eta'$ does not cross itself a.s.  Therefore,  we obtain that $\partial U \cap \partial \BD = \emptyset$ on the event that $\eta'$ draws $\partial U$ in the clockwise way.

A similar argument shows that $\partial U \cap \partial \BD = \emptyset$ a.s.\ on the event that $\eta'$ draws $\partial U$ in the counterclockwise way.  Therefore,  this completes the proof of the lemma since $z \in \BD_{\BQ}$ was arbitrary and $\eta'$ does not hit fixed points a.s.
\end{proof}

\begin{lemma}\label{lem:cut_points_not_on_the_boundary}
Let $\Gamma'$ be a $\CLE_{\kappa'}$ in $\BD$ with $\kappa' \in (4,8)$.  Then,  a.s.,  the loops of $\Gamma'$ do not have any cut points on $\partial \BD$.
\end{lemma}

\begin{proof}
The claim in the statement of the lemma follows from Lemma~\ref{lem:sle_bubbles} combined with the $\BCLE_{\kappa'}(0)$ construction of the boundary touching loops of a $\CLE_{\kappa'}$ (the branches of the $\BCLE_{\kappa'}(0)$ being $\SLE_{\kappa'}(0;\kappa'-6)$ processes, i.e.\ the $\SLE_{\kappa'}(\kappa'-6)$ processes of~\Cref{lem:sle_bubbles}, a force point of weight $0$ having no effect).
\end{proof}

\begin{proof}[Proof of Lemma~\ref{lemma no pinched point}.]
Fix $z,w \in \BD_{\BQ}$ distinct points and let $\SCL(z)$ (resp.\ $\SCL(w)$) denote the loop in $\Gamma'$ surrounding $z$ (resp.\ $w$).  Suppose that we are working on the event that $\SCL(z) \neq \SCL(w)$.  Let $\mathcal{P}$ denote the collection of all simple polygonal paths $\gamma : [0,1] \rightarrow \overline{\BD}$ such that $\gamma(0) = e^{i\theta}$ for some $\theta \in [0,2\pi) \cap \BQ$,  $\gamma(1) = w$,  $\gamma((0,1)) \subseteq \BD$,  and $\gamma|_{(0,1)}$ consists of a finite union of segments with endpoints in $\BD_{\BQ}$.  Note that since $\SCL(z) \neq \SCL(w)$,  we have that $\SCL(z)$ does not disconnect $w$ from $\partial \BD$ and so there exists a path $\gamma$ in $\mathcal{P}$ such that $\gamma \cap \SCL(z) = \emptyset$.

Fix $\gamma \in \mathcal{P}$ and suppose that we are working on the event that $\gamma \cap \SCL(z) = \emptyset$.  Let $K$ denote the closure of the union of $\gamma$ and of the loops in $\Gamma'$ that intersect $\gamma$ and let $(V_j)_{j \geq 1}$ denote the collection of the connected components of $\BD \setminus K$.  Then,  since $K \cap \partial \BD \neq \emptyset$ and $K$ is connected (each of these loops meeting $\gamma$),  we have that $V_j$ is simply connected for all $j$.  Moreover,  the locality property of $\CLE_{\kappa'}$ \cite{TreeCLE,CoInCLERiemSph} implies that conditioned on $K$,  the conditional law of the restrictions of $\Gamma'$ to the $V_j$'s is that of independent copies of $\CLE_{\kappa'}$'s in each component.  

Since $\SCL(z) \cap \gamma = \emptyset$,  we have that there exists $n \in \BN$ such that $\SCL(z) \subseteq \overline{V_n}$ (the loop $\SCL(z)$ being connected, disjoint from $\gamma$, and not crossing the loops of $\Gamma'$).  But then,  combining the locality property of $\CLE$ with Lemma~\ref{lem:cut_points_not_on_the_boundary},  we obtain that it is a.s.\ the case that there are no cut points of $\SCL(z)$ which are contained in $\partial V_n$.  Therefore,  since $\SCL(w) \subseteq K$ and hence $\SCL(z) \cap \SCL(w) \subseteq \partial V_n$,  we obtain that it is a.s.\ the case that there are no cut points of $\SCL(z)$ which are contained in $\SCL(w)$.

Finally,  the proof of the lemma is complete since $z,w \in \BD_{\BQ}$ and $\gamma \in \mathcal{P}$ were arbitrary,  and since for every fixed and deterministic point in $\BD$,  it is a.s.\ the case that there exists a loop in $\Gamma'$ surrounding it.
\end{proof}

\begin{proof}[Proof of~\Cref{lem:loop_surrounds_another_loop}]
Recall that $\Gamma'$ corresponds to the outermost true  $\cwBCLE_{\kappa^\prime}(0)$ loops in the $\cwBCLE_{\kappa^\prime}(0)$/$\ccwBCLE_\kappa(-\kappa/2)$ construction described in~\Cref{A021}. In particular, if $\SCL,\SCL'$ are as in the statement of the lemma, and $\SCL'$ does not surround $\SCL$, then we have that the outermost true  $\cwBCLE_{\kappa^\prime}(0)$ loop $\SCL''$ surrounding $\SCL$ must intersect $\SCL'$ and be distinct from $\SCL'$. But then, we have that every point on $\SCL \cap \SCL'$ is a cut point in $\SCL''$ (the loops of $\Gamma'$ do not cross each other, so $\SCL \cap \SCL' \subseteq \SCL''$), and we have already shown in~\Cref{lemma no pinched point} that it is a.s.\ the case that there are no cut points of $\SCL''$ that are contained in $\SCL'$.  Therefore, we obtain that $\SCL'$ surrounds $\SCL$ a.s. This completes the proof of the lemma.
\end{proof}

\subsubsection{Large loops in $\Gamma_n'$ do not intersect the boundary with high probability when $n$ is large}
\label{subsubsec:big_loops_do_not_intersect_the_boundary}

Next, we will prove that with probability tending to $1$ as $n \to \infty$, there are no loops in $\Gamma_n'$ with large Euclidean diameter that intersect $\partial \BD$ (thereby resolving the second main issue mentioned at the beginning of~\Cref{subsec:completion_of_the_proof}). Formally, we will show the following.
\begin{lemma}\label{lem:big_loops_do_not_intersect_the_boundary}
	Fix $\varepsilon,p \in (0,1)$. Then, there exists $n_0 \in \BN$ such that the following is true for all $n \geq n_0$. With probability at least $p$, we have that there is no loop in $\Gamma_n'$ intersecting $\partial \BD$ whose diameter is at least $\varepsilon$.
	\end{lemma}
	
Before we prove~\Cref{lem:big_loops_do_not_intersect_the_boundary}, we state and prove two lemmas. The following lemma states that with probability tending to $1$ as $n \to \infty$, there are no loops in $\Gamma_n'$ with large Euclidean diameter which stay close to $\partial \BD$.

\begin{lemma}\label{lem:big_loops_staying_close_to_the_boundary_have_to_intersect_it}
Fix $\varepsilon>0$ and $p \in (0,1)$. Then, there exists $\delta \in (0,1)$ depending only on $p$ and $\varepsilon$ such that the following holds for all $n \in \BN$ sufficiently large (the threshold depending only on $\varepsilon$ and $p$). With probability at least $p$, there is no loop $\SCL$ in $\Gamma_n'$ with diameter at least $\varepsilon$ such that $\SCL \subseteq \overline{\BD} \setminus B_{1-\delta}(0)$.
\end{lemma}

\begin{figure}[ht!]
	\centering
	\includegraphics[width=0.6\linewidth]{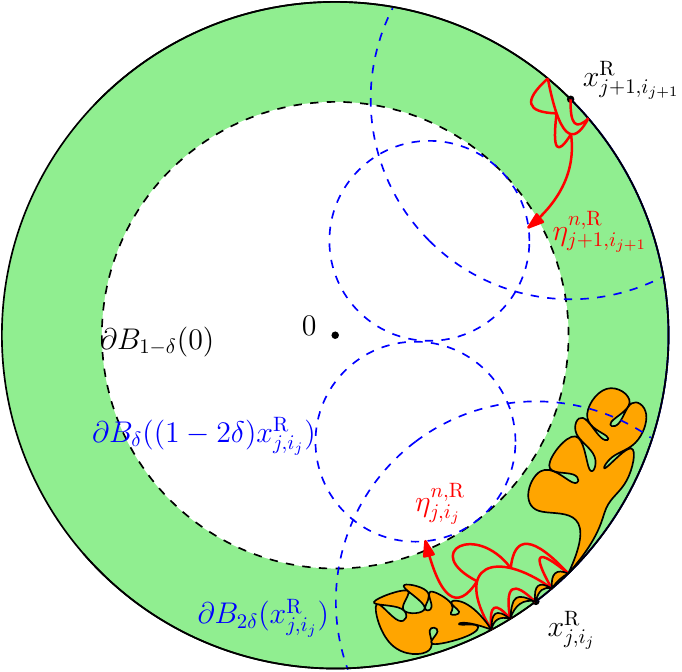}
	\caption{Illustration of the proof of~\Cref{lem:big_loops_staying_close_to_the_boundary_have_to_intersect_it}. The orange loop contained in the green annulus $\overline{\BD}\setminus B_{1-\delta}(0)$ cannot cross the flow line $\eta^{n,\mathrm{R}}_{j,i_j}$, so that all the points of $\eta^{n,\mathrm{R}}_{j,i_j}\cap \partial \BD$ are hit at least twice by the loop (except if the orange loop surrounds $B_{1-\delta}(0)$) which is impossible.}
	\label{fig:pinched}
\end{figure}

\begin{proof}
\stepx{step:close-setup}{Setup and the shield events} As in the proof of~\Cref{lem:unlikeliness}, for each $n$, we let $\Psi_n$ be a GFF on $\BD$ rooted at $-\ri$ that generates $\Gamma_n'$ as in~\Cref{re:GFF-coupling-nonsimple-BCLE}. Then, the branching tree of counterflow lines of $\Psi_n$ centered at $-\ri$ and targeted at any $z \in \overline{\BD}$ is used to construct $\Gamma_n'$. 

We partition the counterclockwise (resp.\ clockwise) arc of $\partial \BD$ from $-\ri$ to $\ri$ into arcs $I_1^{\mathrm{R}},\ldots,I_{m_{\varepsilon}}^{\mathrm{R}}$ (resp.\ $I_1^{\mathrm{L}},\ldots,I_{m_{\varepsilon}}^{\mathrm{L}}$) of length at most $\varepsilon / 2$, where $m_{\varepsilon} \in \BN$ depends only on $\varepsilon$. Fix $\delta \in (0,\varepsilon)$ sufficiently small (depending only on $p$ and $\varepsilon$) (only the smallness of $\delta$ is used, both here and at the end of the proof). By shrinking $\delta$ if necessary, we may assume without loss of generality that $\varepsilon / (10 \delta) \in \BN$. For each $j \in \{1,\ldots,m_{\varepsilon}\}$, we choose points $x_{j,1}^{\mathrm{R}},\ldots,x_{j,k_{\delta}}^{\mathrm{R}}$ (resp.\ $x_{j,1}^{\mathrm{L}},\ldots,x_{j,k_{\delta}}^{\mathrm{L}}$) on $I_j^{\mathrm{R}}$ (resp.\ $I_j^{\mathrm{L}}$) ordered in the counterclockwise (resp.\ clockwise) way such that for all $m \in \{1,\ldots,k_{\delta}-1\}$, the counterclockwise (resp.\ clockwise) arc of $\partial \BD$ from $x_{j,m}^{\mathrm{R}}$ to $x_{j,m+1}^{\mathrm{R}}$ (resp.\ from $x_{j,m}^{\mathrm{L}}$ to $x_{j,m+1}^{\mathrm{L}}$) has length equal to $5 \delta$, where $k_{\delta} = \varepsilon / (10 \delta) - 1 \in \BN$. Also, the counterclockwise (resp.\ clockwise) arcs of $I_j^{\mathrm{R}}$ (resp.\ $I_j^{\mathrm{L}}$) between one endpoint of $I_j^{\mathrm{R}}$ (resp.\ $I_j^{\mathrm{L}}$) and $x_{j,1}^{\mathrm{R}}$ (resp.\ $x_{j,1}^{\mathrm{L}}$), and the other endpoint of $I_j^{\mathrm{R}}$ (resp.\ $I_j^{\mathrm{L}}$) and $x_{j,k_{\delta}}^{\mathrm{R}}$ (resp.\ $x_{j,k_{\delta}}^{\mathrm{L}}$) both have length $5 \delta$.

Furthermore, for all $k \in \{1,\ldots,k_{\delta}\}$, we let $\eta_{j,k}^{n,\mathrm{R}}$ (resp.\ $\eta_{j,k}^{n,\mathrm{L}}$) denote the flow line of $\Psi_n$ of angle $2\pi - \lambda_n' / \chi_n$ (resp.\ $-\lambda_n' / \chi_n$) starting from $x_{j,k}^{\mathrm{R}}$ (resp.\ $x_{j,k}^{\mathrm{L}}$) and targeted at $-\ri$, where we recall that $\lambda_n' = \pi \sqrt{\kappa_n} / 4$ and $\chi_n = 2 / \sqrt{\kappa_n} - \sqrt{\kappa_n} / 2$. Also, for $q \in \{\mathrm{L}, \mathrm{R}\}$ and $k \in \{1,\ldots,k_{\delta}\}$, we let $E_{j,k}^{n,q}$ denote the event that $\eta_{j,k}^{n,q}$ hits $B_{\delta}((1-2\delta) x_{j,k}^q)$ before exiting $B_{2\delta}(x^q_{j,k})$, and note that~\Cref{lem:likeliness} implies that there exists some universal constant $\widetilde{p} \in (0,1)$ such that $\BP[E_{j,k}^{n,q}] \geq \widetilde{p}$.

\stepx{step:close-many}{Many shields occur} Therefore,~\Cref{lem:independence-across-disjoint-balls} implies that we can choose $\delta \in (0,1)$ sufficiently small (in a way that depends only on $p$ and $\varepsilon$) such that for all $n \in \BN$, the following holds with probability at least $p$. For all $j \in \{1,\ldots,m_{\varepsilon}\}$, there exist $i_j^{\mathrm{L}}, i_j^{\mathrm{R}} \in \{1,\ldots,k_{\delta}\}$ such that $E_{j,i_j^{\mathrm{L}}}^{n,\mathrm{L}} \cap E_{j,i_j^{\mathrm{R}}}^{n,\mathrm{R}}$ occurs. Suppose that we are working on that event and let $\SCL$ be a loop in $\Gamma_n'$ with diameter at least $\varepsilon$ such that $\SCL \subseteq \overline{\BD} \setminus B_{1-\delta}(0)$. Then, there exist $j \in \{1,\ldots,m_{\varepsilon}\}, q \in \{\mathrm{L},\mathrm{R}\}$ such that $\SCL$ makes a crossing between the left and right boundaries of the conformal rectangle bounded by $I_j^{q}, \partial B_{1-\delta}(0)$, and the two segments connecting the two endpoints of $I_j^{q}$ to $\partial B_{1-\delta}(0)$ (indeed, $\SCL$ joins two points at distance at least $\varepsilon$ while staying in $\overline{\BD}\setminus B_{1-\delta}(0)$, hence crosses one of these $2m_\varepsilon$ rectangles from side to side). Also, there exists $k \in \{1,\ldots,k_{\delta}\}$ (specifically, $k = i_j^q$) such that $E_{j,k}^{n,q}$ occurs.

\stepx{step:close-conclusion}{Conclusion} Note that the construction of $\Gamma_n'$ using the counterflow lines of $\Psi_n$ combined with the fact that counterflow lines do not cross flow lines of the same field implies that $\SCL$ cannot cross $\eta_{j,k}^{n,q}$ a.s. Therefore, we obtain that $\SCL$ has to contain all the points of $\partial \BD \cap \eta_{j,k}^{n,q}$. Note that $\SCL$ has to surround $B_{1-\delta}(0)$, otherwise the points of $\partial \BD \cap \eta_{j,k}^{n,q}$ will be hit at least twice by $\SCL$ (see~\Cref{fig:pinched}), which is impossible by \cite[Theorem 5.4]{TreeCLE}. In particular $\SCL$ is the origin-containing loop of $\Gamma'_n$ and does not meet $B_{1-\delta}(0)$. Since $\{K : K \cap B_{1-\delta}(0) = \emptyset\}$ is closed for the Hausdorff topology and, by~\eqref{eq cv subsequence}, the origin-containing loop of $\Gamma'_n$ converges in law to that of a $\CLE_4$, the Portmanteau theorem bounds the $\limsup_n$ of the probability of the above event by the probability that the origin-containing loop of a $\CLE_4$ does not meet $B_{1-\delta}(0)$, which tends to zero as $\delta \to 0$ (that loop being a.s.\ at positive distance from $\partial \BD$). Thus one can choose $\delta$ small enough (which is also admissible above) for the lemma to hold.
\end{proof}

\begin{proof}[Proof of~\Cref{lem:big_loops_do_not_intersect_the_boundary}]
Fix $\varepsilon,p \in (0,1)$. 
By~\Cref{lem:big_loops_staying_close_to_the_boundary_have_to_intersect_it}, there exist $\delta>0$ and $n_0 \in \BN$ depending only on $\varepsilon$ and $p$ such that the following holds for all $n \geq n_0$. With probability at least $1 - (1-p) / 100$, we have that there is no loop $\SCL$ of $\Gamma'_n$ with diameter at least $\varepsilon$ such that $\SCL \subseteq \overline{\BD}\setminus B_{1-\delta}(0)$. 
Let $F_n$ denote that event.

Next,~\Cref{lem:unlikeliness-proof} implies that there exists $n_1 \geq n_0$ depending only on $p$ and $\delta$, such that the following holds for all $n \geq n_1$. With probability at least $1 - (1-p) /100$, we have the following. For all $\SCL \in \Gamma_n'$ such that $\SCL \cap \partial \BD \neq \emptyset$, we have that $\SCL \subseteq \overline{\BD} \setminus B_{1-\delta}(0)$. Let $G_n$ denote that event.

Fix $n \geq n_1$ and suppose that $F_n \cap G_n$ occurs; then there is no loop in $\Gamma'_n$ intersecting $\partial \BD$ with diameter at least $\varepsilon$. Moreover, note that $\BP[ F_n \cap G_n] \geq p$. This ends the proof of the lemma. 
\end{proof}

\subsubsection{Proof of~\Cref{prop:thin-loops}}
\label{subsubsec:completion_of_proof}

Finally, we are ready to prove~\Cref{prop:thin-loops}. It will follow from~\Cref{lem:loops_crossing_annuli,lem:big_loops_surround_other_loops} below combined with~\Cref{prop:big_loops_are_fat_simple_cle}.

\begin{lemma}
\label{lem:loops_crossing_annuli}
Fix $z \in \BD$ and $0<s<t$ such that $\overline{A_{s,t}(z)} \subseteq \BD$. The probability of the event that there are loops in $\Gamma_n$ crossing $A_{s,t}(z)$ and that some loop of $\Gamma_n'$ crossing $A_{s,t}(z)$ does not surround any loop of $\Gamma_n$ crossing $A_{s,t}(z)$ tends to $0$ as $n \to \infty$.
\end{lemma}
\begin{proof}
Suppose that we are working on the event that there are loops in $\Gamma_n$ crossing $A_{s,t}(z)$. These loops divide the annulus $A_{s,t}(z)$ into a set of conformal rectangles. By~\Cref{prop:conformal_rectangles_tight}, as $n \to \infty$, these conformal rectangles converge in distribution to those formed by the crossings of $A_{s,t}(z)$ by a $\CLE_4$ $\Gamma^4$, which implies in particular that their moduli are stochastically bounded away from $0$ and $\infty$. 

Let $\SCL'$ be a loop of $\Gamma_n'$ crossing $A_{s,t}(z)$ and $\SCI$ one of its crossing segments. Either $\SCI$ meets one of the loops of $\Gamma_n$ intersecting $\partial A_{s,t}(z)$, or it avoids all of them, hence is contained in one of the conformal rectangles and crosses it between its top and bottom boundaries. In the latter case, by~\Cref{lem:segments_of_loops_inside_rectangles} applied in a rectangle $R_j$ (which depends on $n$ since the conformal rectangles are defined via $\Gamma_n$)
, $\SCI$ is a segment of a loop of a $\CLE_{\kappa_n'}$ in $R_j$.

By~\Cref{lem:unlikeliness}, the probability that a $\CLE_{\kappa_n'}$ loop in such a conformal rectangle crosses between the top and bottom boundaries without intersecting the left or right boundaries (which are formed by the crossings of $\Gamma_n$) tends to $0$ as $n \to \infty$ (we apply~\Cref{lem:unlikeliness} in $R_j$, using the convergence of the rectangles $R_j$ as $n\to \infty$ and of their moduli given by~\Cref{prop:conformal_rectangles_tight}, together with a union bound over the number of rectangles, which is tight).

Therefore, off an event whose probability tends to $0$ as $n \to \infty$, any loop in $\Gamma_n'$ crossing $A_{s,t}(z)$ must intersect one of the loops of $\Gamma_n$ meeting $\partial A_{s,t}(z)$, hence surround it by~\Cref{lem:loop_surrounds_another_loop}. As $\SCL'$ crosses $A_{s,t}(z)$, that loop cannot be an excursion loop of $A_{s,t}(z)$ unless $\SCL'$ also surrounds a crossing loop; in either case $\SCL'$ surrounds a loop of $\Gamma_n$ crossing $A_{s,t}(z)$, which completes the proof.
\end{proof}

\begin{lemma}\label{lem:big_loops_surround_other_loops}
Fix $\varepsilon, p \in (0,1)$. Then, there exist $\widetilde{\varepsilon} \in (0,1)$ and $n_0 \in \BN$ depending only on $\varepsilon$ and $p$, such that the following holds for all $n \geq n_0$. With probability at least $p$, every loop in $\Gamma_n'$ with diameter at least $\varepsilon$ surrounds a loop in $\Gamma_n$ with diameter at least $\widetilde{\varepsilon}$.
\end{lemma}

\begin{proof}
Combining~\Cref{lem:big_loops_do_not_intersect_the_boundary,lem:big_loops_staying_close_to_the_boundary_have_to_intersect_it}, we obtain that there exist $\delta \in (0,\varepsilon/100), n_0 \in \BN$ depending only on $\varepsilon$ and $p$ such that the following is true for all $n \geq n_0$. With probability at least $1 - (1-p) / 100$, we have that for all $\SCL \in \Gamma_n'$ such that $\mathrm{diam}(\SCL) \geq \varepsilon$, there exists $x \in \SCL \cap B_{1-\delta}(0)$. Let $A_n$ denote that event.

Suppose that $A_n$ occurs and let $\SCL \in \Gamma_n'$ be such that $\mathrm{diam}(\SCL) \geq \varepsilon$. Then, there exists $z \in B_{1-\delta}(0)$ such that $\SCL$ crosses the annulus $A_{\delta / 4,\delta}(z)$. In particular, there exists $z \in (\frac{\delta}{100} \BZ^2) \cap B_{1-\delta}(0)$ such that $\SCL$ crosses $A_{\delta / 3,\delta/2}(z)$. Moreover, using the Brownian loop soup construction of $\CLE_{\kappa}$ for $\kappa \in (8/3,4]$ (see~\cite{CLE}), we obtain that the following is true. There exists $s \in (\delta / 4, \delta/2)$ such that the following holds for all $n \geq n_0$. With probability at least $1 - (1-p) / 100$, we have that for all $z \in (\frac{\delta}{100} \BZ^2) \cap B_{1-\delta}(0)$, there exists $\SCL \in \Gamma_n$ such that $\SCL$ crosses $A_{s,\delta/2}(z)$ or such that $\SCL$ encircles $A_{s, \delta/2}(z)$ (in the increasing loop-soup coupling, a $\CLE_{\kappa_1}$ loop crossing the annulus is surrounded by a $\CLE_{\kappa_n}$ loop, which therefore crosses or encircles it). We denote that event by $B_n$.

Note that $\BP[A_n \cap B_n] \geq 1 - (1-p)/50$ for all $n \geq n_0$. Fix $n \geq n_0$ and suppose that $A_n \cap B_n$ occurs. Then, for all $\SCL' \in \Gamma_n'$ with $\mathrm{diam}(\SCL') \geq \varepsilon$, there exists $z \in (\frac{\delta}{100} \BZ^2) \cap B_{1-\delta}(0)$ and $\SCL \in \Gamma_n$ such that both $\SCL'$ and $\SCL$ cross $A_{s,\delta/2}(z)$ (since on the event that $\SCL$ encircles $A_{s,\delta/2}(z)$, the loop $\SCL'$ cannot cross $A_{s,\delta/2}(z)$ in the coupling because it would have to intersect $\SCL$, and hence surround it by~\Cref{lem:loop_surrounds_another_loop}, or else lie inside $\SCL$, which is impossible since the loops of $\Gamma'_n$ are sampled inside the \emph{false} loops of the $\ccwBCLE_\kappa(-\kappa/2)$'s, which are disjoint from the interiors of the true ones). Note that $\mathrm{diam}(\SCL) \geq \delta / 2 - s >0$. Therefore, by~\Cref{lem:loops_crossing_annuli} applied to the finitely many annuli $A_{s,\delta/2}(z)$, $z \in (\frac{\delta}{100}\BZ^2) \cap B_{1-\delta}(0)$ (with a union bound and $n_0$ enlarged so that the total exceptional probability is at most $(1-p)/100$), the loop $\SCL'$ surrounds a loop of $\Gamma_n$ crossing $A_{s,\delta/2}(z)$, which completes the proof of the lemma.
\end{proof}

\begin{proof}[Proof of~\Cref{prop:thin-loops}]
Fix $\varepsilon \in (0,1)$ and let $p = 1 - \varepsilon / 2$. By~\Cref{lem:big_loops_surround_other_loops}, there exist $\widetilde{\varepsilon} \in (0,1)$ and $n_0 \in \BN$ depending only on $\varepsilon$ and $p$ such that the following holds for all $n \geq n_0$. With probability at least $p$, every loop in $\Gamma_n'$ with diameter at least $\varepsilon$ surrounds a loop in $\Gamma_n$ with diameter at least $\widetilde{\varepsilon}$. 
Furthermore, by~\Cref{prop:big_loops_are_fat_simple_cle} applied to $\widetilde{\varepsilon}$, there exists $\delta > 0$ such that for all $n \in \BN$, with probability at least $p$, every loop in $\Gamma_n$ with diameter at least $\widetilde{\varepsilon}$ surrounds a Euclidean ball of radius $\delta$.

Let $E_n$ be the event that every loop in $\Gamma_n'$ with diameter at least $\varepsilon$ surrounds a loop in $\Gamma_n$ with diameter at least $\widetilde{\varepsilon}$, and let $F_n$ be the event that every loop in $\Gamma_n$ with diameter at least $\widetilde{\varepsilon}$ surrounds a Euclidean ball of radius $\delta$.  Then, for all $n \geq n_0$, a union bound gives $\BP[E_n \cap F_n] \geq 1 - \BP[E_n^c] - \BP[F_n^c] \geq 1 - 2(1-p) = 1-\varepsilon$.

Suppose that $E_n \cap F_n$ occurs. Then, any loop in $\Gamma_n'$ with diameter at least $\varepsilon$ surrounds a loop in $\Gamma_n$ with diameter at least $\widetilde{\varepsilon}$, which in turn surrounds a Euclidean ball of radius $\delta$. Therefore, on $E_n \cap F_n$, any loop in $\Gamma_n'$ with diameter at least $\varepsilon$ surrounds a Euclidean ball of radius $\delta$.
To complete the proof, it suffices to note that for each fixed $n < n_0$, $\Gamma_n'$ is a non-nested $\CLE_{\kappa_n'}$, which has finitely many loops of diameter at least $\varepsilon$, each of which bounds an open domain and therefore surrounds a Euclidean ball. Hence, for each such $n$ there is $\delta_n>0$ such that with probability at least $1-\varepsilon$ the loops of $\Gamma_n'$ of diameter at least $\varepsilon$ all surround a Euclidean ball of radius $\delta_n$; replacing $\delta$ by $\delta \wedge \min_{n<n_0}\delta_n$ gives the claim for all $n$. This implies the statement of the proposition.
\end{proof}

\section{On the distance between two segments}\label{sec: distance between two segments}
Now, let us pick up the notation used in \Crefrange{sec:intro}{sec:boundary not a point}, with the sequence $\kappa_n \downarrow 4$. This section aims to prove the following lemma which states that one can find two disjoint segments that are close to each other for the $\CLE_{\kappa_n}$ graph distance.

	\begin{lemma}\label{lemma distance from two long segments is small}
		For all $\varepsilon>0$, there exist two disjoint segments $\alpha$ and $\beta$ of $\partial \BD$ such that for all $n$ large enough,
		\[
		\BP\!\left[ \ka_{\kappa_n}^{-1}D^{\kappa_n} (\alpha, \beta) \ge \varepsilon \right] \le \varepsilon.
		\]
		More precisely, there is $M>0$ such that for all $M' \ge M$ and $n$ large, if one considers a $\CLE_{\kappa_n}$ in $[0,M']\times[0,1]$, then $\CLE_{\kappa_n}$ graph distance between the top and bottom sides of $[0,M']\times[0,1]$ is less than $\varepsilon \ka_{\kappa_n}$ with probability at least $1-\varepsilon$; the statement follows by taking for $\alpha$, $\beta$ the images of these sides under a conformal map onto $\BD$.
	\end{lemma}

Throughout this section we use the following form of the Markov property of $\CLE_{\kappa}$, $\kappa \in [4,8)$ \cite{TreeCLE, CoInCLERiemSph}: for a deterministic closed set $A \subseteq \overline{\BD}$ so that $A \cup \overline{\BD}$ is connected, conditionally on the loops meeting $A$, the remaining loops form independent $\CLE_\kappa$'s in the components of the complement of the closure of the union of $A$ and the domains encircled by those loops.

Let us start by a simple remark.
\begin{remark}\label{remark monotonicity rectangle}
	Let $\alpha$ and $\beta$ be two disjoint segments of $\partial \BD$. If $\alpha'\subset \alpha$ is another segment, then $D^{\kappa_n} (\alpha, \beta) \le D^{\kappa_n}(\alpha', \beta)$. Indeed, any path of loops of $\Gamma^{\kappa_n}$ from $\alpha'$ to $\beta$ is a path of loops from $\alpha$ to $\beta$.
\end{remark}

The proof of~\Cref{lemma distance from two long segments is small} relies on the following key steps:
\begin{enumerate}
	\item We first show in \Cref{lemma distance between segments is large} that the $\CLE_{\kappa_n}$ graph distance between two macroscopic segments on the boundary of the unit disk is bounded from below by a small constant times $\ka_{\kappa_n}$ with high probability.
	\item We then prove in \Cref{lemma the geodesic does not touch the line} that, with positive probability, the distance from $\partial \BD$ to $\SCL^{\kappa_n}(0)$ can be arbitrarily small, and the shortest path achieving this distance can avoid a given line segment (such as $[0,1]$). This relies on the fact that $\SCL(0)$ is macroscopically close to $\partial \BD$ with positive probability, combined with the lower bound from the previous lemma to force the shortest path to avoid the segment.
	\item Next, in \Cref{cor long segments can be close} we show that by cutting the unit disk along the line segment from the previous step and applying a conformal map, for every $\varepsilon>0$ there exist $L,p>0$ such that, for all $n$ large enough, the $\CLE_{\kappa_n}$ graph distance between the top and bottom sides of $[0,L] \times [0,1]$ is at most $\varepsilon \ka_{\kappa_n}$ with probability at least $p$ (``long'' meaning that the extremal length of the left--right crossings is large).
	\item Finally, to complete the proof of \Cref{lemma distance from two long segments is small}, on the event that $\SCL(0)$ is close to $\partial \BD$, we construct consecutive disjoint, long conformal rectangles connecting $\partial \BD$ to $\SCL^{\kappa_n}(0)$. By the domain Markov property, the $\CLE_{\kappa_n}$ distances across these rectangles are conditionally independent. Applying~\Cref{cor long segments can be close} to each rectangle and using independence allows us to deduce that the distance across one of these rectangles is small with probability at least $1-\varepsilon$. We then conclude using the fact that, conditionally on the loops outside the region between the outermost cuts, the restriction of the $\CLE_{\kappa_n}$ to this region is a $\CLE_{\kappa_n}$ there, whose top-to-bottom distance is therefore not affected by the conditioning.
\end{enumerate}

\begin{lemma}\label{lemma distance between segments is large}
	For all $\varepsilon>0$, there exist $\delta, c>0$ such that if $\alpha= \{ e^{\ri \theta}: \theta \in [0, \delta]\}$ and $\beta =  \{ e^{\ri \theta}: \theta \in [\pi, \pi+ \delta]\}$, then, for all $n\ge 1$ large enough,
	\[
	\BP\!\left[ D^{\kappa_n}(\alpha, \beta) \ge c \ka_{\kappa_n} \right] \ge 1-\varepsilon.
	\]
\end{lemma}
\begin{proof}
	Let $\varepsilon>0$. We know that with probability $1-1/\ka_{\kappa_n} \to 1$, the loop $\SCL^{\kappa_n}(0)$ does not touch $\partial \BD$. By Skorokhod's representation theorem, we may assume that~\eqref{eq cv subsequence boundary} and the convergence of~\Cref{lemma cv Hausdorff union of loops} with $K=[0,1]$ hold a.s. Let $F_n$ (resp.\ $F$) be the closure of the union of $[0,1]$ and the domains encircled by the loops of $\Gamma^{\kappa_n}$ (resp.\ $\Gamma$) meeting $[0,1]$, so that $F_n \to F$ and the components of $\BD \setminus F_n$ are simply connected. Fix $z \in \BD_\BQ \setminus F$ and let $R_n$ be the component of $\BD \setminus F_n$ containing $z$ (well defined for $n$ large). With probability $1-o(1)$, $R_n$ is a conformal rectangle with bottom boundary $B_n \subseteq \partial \BD$, top boundary $T_n \subseteq \SCL^{\kappa_n}(0)$, and left and right boundaries in loops of $\Gamma^{\kappa_n} \setminus \{\SCL^{\kappa_n}(0)\}$.

	By the Hausdorff convergence of $F_n$, we see that the probability that a planar Brownian motion started from $z$ stopped when it exits $R_n$ hits $B_n$ (resp.\ $T_n$) is a.s.\ bounded from below by a positive random variable, and the same is true for the left and right sides of the rectangle. In other words, if we let $\varphi_n\colon R_n \to \BD$ be the unique conformal mapping such that $\varphi_n(z)=0$ and $\varphi_n'(z)>0$, the lengths of the segments $\alpha_n \defeq \varphi_n(B_n)$ and $\beta_n \defeq \varphi_n(T_n)$ are bounded from below by a positive random variable, and the same is true for the images of the left and right sides. In particular, after taking a conformal mapping $\psi_n: \BD \to \BD$ sending the left endpoint of $\alpha_n$ to $1$, the right endpoint of $\alpha_n$ to $e^{\ri \mathrm{len}(\alpha_n)}$ and the right endpoint of $\beta_n$ to $-1$ (by possibly replacing $\alpha_n$ by a segment of length $\min(\mathrm{len}(\alpha_n), \pi/2)$, we may assume that the length of $\alpha_n$ is smaller than $\pi$ so that $\psi_n$ exists), letting $\alpha'_n\defeq\psi_n(\alpha_n)$ and $\beta'_n \defeq \psi_n(\beta_n)$ one can take $\delta>0$ small enough that with probability at least $1-\varepsilon$, for all $n$ large enough, we have $\mathrm{len}(\alpha_n'), \mathrm{len}(\beta_n')\ge \delta$.
	
	Note that conditionally on $F_n$ (and on the event that $z \notin F_n$), the loop ensemble $(\psi_n \circ \varphi_n)(\Gamma^{\kappa_n}\vert _{ R_n}) $ has the same law as a $\CLE_{\kappa_n}$ in $\BD$. Thus, on the above event, the conditional law given $F_n$ of $D^{\kappa_n}(\alpha, \beta)$ stochastically dominates that of $D^{\kappa_n}(\partial \BD, \SCL^{\kappa_n}(0))$ by~\Cref{remark monotonicity rectangle}. Finally, $\ka_{\kappa_n}^{-1} D^{\kappa_n}(\partial \BD, \SCL^{\kappa_n}(0))$ converges in law to $D(\partial \BD, \SCL(0))>0$ by~\eqref{eq cv subsequence boundary}, so $\BP[D^{\kappa_n}(\partial \BD, \SCL^{\kappa_n}(0)) \ge c \ka_{\kappa_n}] \ge 1-\varepsilon$ for some $c>0$ and $n$ large, which ends the proof (with $2\varepsilon$ in place of $\varepsilon$).
\end{proof}

The following lemma controls the loops of a $\CLE_4$ in a thin annular region (see~\cite{SimCLEDblConnDom} for $\CLE_4$ in doubly connected domains).
\begin{lemma}\label{lemma loops in thin annulus are small}
For all $\eta>0$, there exists $\delta>0$ such that, for every simply 
 connected domain $U \subseteq \BD$ whose complement 
 contains $B_{1-\delta}(0)$ and every $\CLE_4$ $\Gamma_U$ in $U$, with probability at least $1-\eta$ all the loops of $\Gamma_U$ have diameter at most $\eta$. The same is true for doubly a connected domain $U \subseteq \BD$ whose bounded complementary component 
 contains $B_{1-\delta}(0)$.
\end{lemma}
\begin{proof}
By 
	the Brownian loop soup construction, $\Gamma_U$ consists of the outer boundaries of the outermost clusters of a Brownian loop soup in $U$ with intensity $c=1$. 
	By the restriction property, this soup consists of the loops contained in $U$ of a loop soup $S$ in $\BD$ with $c=1$, so its clusters lie in clusters of $S$. By \cite[Theorem~1.6]{CLE}, the clusters of $S$ of diameter at least $\eta_0 \defeq \min(\eta,1)$ lie inside the finitely many $\CLE_4$ loops of diameter at least $\eta_0$, which do not meet $\partial \BD$; hence the probability $p_\delta$ that such a cluster meets $\{|z| \ge 1-\delta\}$ tends to $0$ as $\delta \to 0$. 

In the case where $U$ is doubly connected, we reason as follows. By \cite[Section~3]{SimCLEDblConnDom}, $\Gamma_U$ consists of the outer boundaries of the outermost clusters of a Brownian loop soup in $U$ with intensity $c=1$, conditioned on the event $G_U$ that no cluster surrounds the bounded complementary component of $U$. By the restriction property, this soup consists of the loops contained in $U$ of a loop soup $S$ in $\BD$ with $c=1$, so its clusters lie in clusters of $S$. By \cite[Theorem~1.6]{CLE}, the clusters of $S$ of diameter at least $\eta_0 \defeq \min(\eta,1)$ lie inside the finitely many $\CLE_4$ loops of diameter at least $\eta_0$, which do not meet $\partial \BD$; hence the probability $p_\delta$ that such a cluster meets $\{|z| \ge 1-\delta\}$ tends to $0$ as $\delta \to 0$. On the complementary event, all the clusters of the soup in $U \subseteq \{|z| > 1-\delta\}$ have diameter less than $\eta_0$, so $G_U$ holds (a cluster surrounding the hole would have diameter at least $2(1-\delta)$) and all the loops of $\Gamma_U$ have diameter at most $\eta$. Hence the probability that some loop of $\Gamma_U$ has diameter larger than $\eta$ is at most $p_\delta/(1-p_\delta) \le \eta$ for $\delta$ small, uniformly in $U$.
\end{proof}

\begin{lemma}\label{lemma the geodesic does not touch the line}
	For all $\varepsilon'>0$ and $n\ge 1$, let $\mathcal{A}_n(\varepsilon')$ be the event that $D^{\kappa_n}(\partial \BD, \SCL^{\kappa_n}(0)) < \varepsilon' \ka_{\kappa_n}$. Let $F_n$ be the closure of the union of the domains encircled by the loops of $\Gamma^{\kappa_n} \setminus \{\SCL^{\kappa_n}(0)\}$ meeting $[0,1]$. On the event that $\SCL^{\kappa_n}(0) \cap \partial \BD = \emptyset$ and $F_n \subseteq B_{1/2}(1)$, let $R_n$ be the component of $\BD \setminus (F_n \cup \overline{\mathrm{int}(\SCL^{\kappa_n}(0))})$ whose closure contains $\partial \BD \setminus B_{1/2}(1)$ (well defined, a thin collar along this arc lying in one component), and let $\mathcal{E}_n$ be the event that this holds and that some shortest path from $\partial \BD$ to $\SCL^{\kappa_n}(0)$ consists of loops of $\Gamma^{\kappa_n}\vert_{R_n}$ (which do not touch $[0,1]$).
Then, for all $\varepsilon' >0$, there exists $\delta>0$ such that for all $n$ large enough, we have
	\[
	\BP\!\left[ \mathcal{E}_n \cap \mathcal{A}_n(\varepsilon')\right]\ge \delta.
	\]
\end{lemma}
\begin{proof}
	\stepx{step:geod-event}{Construction of the event $\mathcal{C}_N$} Let $\varepsilon'>0$ and $N\ge 1$. As in the proof of the previous lemma, by Skorokhod's representation theorem, we assume that the convergence~\eqref{eq cv subsequence boundary} and the convergences given by~\Cref{lemma cv Hausdorff union of loops} (for the compact sets used below, with the loop surrounding $0$ excluded) hold a.s. Let us first construct an event $\mathcal{C}_N$ of positive probability on which $D(\partial \BD, \SCL(0)) \le \varepsilon'/2$ and $d_\mathrm{H}(\SCL(0), \partial \BD)\le 1/N$. 
	
	For all $\varepsilon>0$, consider the time $\tau_\varepsilon$ of the first $\SLE_4$ bubble $\gamma_{\tau_\varepsilon}$ associated with an excursion of height at least $\varepsilon$. Note that $\tau_\varepsilon$ converges to zero as $\varepsilon\to 0$ in probability. In particular, as in~\eqref{eq continuite a droite explo unif}, thanks to \cite[Proposition~6.1]{MS16QLE}, the domain $C_{\tau_\varepsilon -}(0)$ converges in the Carath\'eodory sense toward $\BD$. Indeed, $C_{\tau_\varepsilon -}(0)$ has the same law as $\widetilde{C}_{\tau_\varepsilon }(0)$, where $(\widetilde{C}_{t}(0))_{t\ge 0}$ is independent of $({C}_{t}(0))_{t\ge 0}$ and defined in the same way as $({C}_{t}(0))_{t\ge 0}$ except that we remove from the driving measure the excursions of height at least $\varepsilon$, so that its driving measure restricted to $[0,\tau_\varepsilon]$ tends to zero as $\varepsilon \to 0$, and we can also apply \cite[Proposition~6.1]{MS16QLE} which says that the convergence of the driving measure implies the Carath\'eodory convergence of the associated Loewner chain. Recall that the Carath\'eodory convergence of $C_{\tau_\varepsilon -}(0)$ to $\BD$ can be rephrased as the uniform convergence on compact subsets (and in particular on $\overline{B_{1-1/(2N)}(0)}$) of the associated conformal maps from $\BD$ to $C_{\tau_\varepsilon -}(0)$ toward the identity function. Moreover, the probability that $d_\mathrm{H}(\gamma_{\tau_\varepsilon}, \partial \BD)\le 1/(2N)$ and $\gamma_{\tau_\varepsilon}$ surrounds $0$ is positive and does not depend on $\varepsilon \le 2\pi$ (a bubble surrounding $0$ has excursion height $2\pi$): since $\SCL(0)$ is the image of the first bubble surrounding $0$ under a conformal map from $\BD$ onto $C_{\tau_0-}(0)$ fixing $0$, which does not increase the modulus (Schwarz lemma), it is at least $\BP[d_\mathrm{H}(\SCL(0), \partial \BD)<1/(2N)]>0$ 
	. Let $\mathcal{C}_N$ be the event that $\tau_\varepsilon<\varepsilon'/2$, that the associated conformal map from $\BD$ to $C_{\tau_\varepsilon -}(0)$ is within uniform distance $1/(2N)$ of the identity on $\overline{B_{1-1/(2N)}(0)}$, that all the loops discovered before time $\tau_\varepsilon$ by the exploration targeted at $0$ have diameter at most $1/N$, and that $d_\mathrm{H}(\gamma_{\tau_\varepsilon}, \partial \BD)\le 1/(2N)$ and $\gamma_{\tau_\varepsilon}$ surrounds $0$. The first three conditions have probability tending to $1$ as $\varepsilon \to 0$ (for the third, $\tau_\varepsilon \to 0$ in probability and a.s.\ no loop of diameter at least $1/N$ is discovered before some positive time, such loops being finitely many with positive labels). Using the independence of $\gamma_{\tau_\varepsilon}$ and of the exploration before time $\tau_\varepsilon$, we deduce that for all $\varepsilon>0$ small enough, $\BP[\mathcal{C}_N] > 0$. On $\mathcal{C}_N$, $D(\partial \BD, \SCL(0)) \le \varepsilon'/2$ (as $\tau_\varepsilon<\varepsilon'/2$ and $\gamma_{\tau_\varepsilon}$ surrounds $0$) and $d_\mathrm{H}(\SCL(0), \partial \BD)\le 1/N$ (by Rouch\'e's theorem the image of $B_{1-1/(2N)}(0)$ contains $B_{1-1/N}(0)$, so by injectivity $\SCL(0)$, the image of $\gamma_{\tau_\varepsilon}$, surrounds $B_{1-1/N}(0)$ without intersecting it).
	
	Recall that we have the a.s.\ Hausdorff convergence of $\SCL^{\kappa_n}(0)$ toward $\SCL(0)$. In particular, a.s.\ for all $n$ large enough, the loop $\SCL^{\kappa_n}(0)$ does not touch $\partial \BD$.
	
	\stepx{step:geod-cuts}{The cuts have small diameter} Let $K\ge1$ be an integer to be chosen later. For all $k \in [0,K-1]_\BZ$, let $F_k$ be the closure of the union of the domains encircled by loops in $\Gamma \setminus \{\SCL(0)\}$ which intersect the segment $I_k\defeq\{r e^{2\ri \pi k/K}: r \in [0,1]\}$. We define similarly $F^{\kappa_n}_k$ the closure of the union of the domains encircled by loops in $\Gamma^{\kappa_n} \setminus \{ \SCL^{\kappa_n}(0)\}$ which intersect $I_k$. For each $k\in [1,K]_\BZ$ (with the convention $F^{\kappa_n}_K = F^{\kappa_n}_0$), the sets $F_{k-1}^{\kappa_n}$ and $F_k^{\kappa_n}$ act as thickened radial cuts. Let $w_k \defeq e^{\ri \pi (2k-1)/K}$. When $w_k \notin F_{k-1}^{\kappa_n} \cup F_k^{\kappa_n} \cup \overline{\mathrm{int}(\SCL^{\kappa_n}(0))}$, let $R_k^{\kappa_n}$ be the component of $\BD \setminus (F_{k-1}^{\kappa_n} \cup F_k^{\kappa_n} \cup \SCL^{\kappa_n}(0))$ whose closure contains $w_k$; we check below that, with probability tending to $1$ as $N \to \infty$, it is a conformal rectangle with bottom, top, left and right boundaries in $\partial \BD$, $\SCL^{\kappa_n}(0)$, $F_{k-1}^{\kappa_n}$ and $F_k^{\kappa_n}$ respectively.
	
	Let us show that, conditionally on the event $\mathcal{C}_N$ (where $\varepsilon>0$ is chosen to be sufficiently small in a way that depends only on $N$ such that $\BP[\mathcal{C}_N] > 0$), $\max_{k}\mathrm{diam}(F_k) \to 0$ in probability as $N \to \infty$. On $\mathcal{C}_N$, $\tau_\varepsilon = \tau_0$ and the loops of $\Gamma \setminus \{\SCL(0)\}$ are either discovered before $\tau_\varepsilon$, hence of diameter at most $1/N$ and contained in $\BD \setminus C_{\tau_\varepsilon-}(0)$, or contained in the simply 
	connected domain $U \defeq C_{\tau_\varepsilon-}(0) \setminus \overline{\mathrm{int}(\SCL(0))}$ (bubbles do not surround other loops), where they form a $\CLE_4$ given the exploration up to time $\tau_\varepsilon$ by the strong Markov property for the uniform exploration
	. Since $C_{\tau_\varepsilon-}(0)$ and $\mathrm{int}(\SCL(0))$ contain $B_{1-1/N}(0)$ on $\mathcal{C}_N$, \Cref{lemma loops in thin annulus are small} shows that, conditionally on $\mathcal{C}_N$, with probability $1-o(1)$ as $N \to \infty$, these loops all have diameter at most some $\eta_N \to 0$. Lying in $\{|z| \ge 1-1/N\}$, those meeting $I_k$ are within distance $\eta_N$ of $I_k \cap \{|z| \ge 1-1/N\}$, so that $F_k \subseteq \overline{B_{1/N + \eta_N}(e^{2 \ri \pi k/K})}$ and $\mathrm{diam}(F_k) \le 2/N + 2\eta_N$. Therefore, by~\Cref{lemma cv Hausdorff union of loops} (with the loop surrounding $0$ excluded), a.s.\ $\max_{k} \mathrm{diam}(F^{\kappa_n}_k) \to \max_{k} \mathrm{diam}(F_k)$ as $n \to \infty$, which tends to zero in probability as $N \to \infty$ conditionally on $\mathcal{C}_N$.
	
	\stepx{step:geod-thin}{The rectangles $R^{\kappa_n}_k$ are thin} In particular, conditionally on $\mathcal{C}_N$, with probability $1-o(1)$ as $N\to \infty$, the conformal rectangles $R^{\kappa_n}_k$'s are well-defined for all $n$ large enough. But one can conformally map the conformal rectangles $R_k^{\kappa_n}$ for $k \in [1,K]_\BZ$ to some rectangles $[0,1] \times [0,L_k^{\kappa_n}]$. Then, conditionally on $\mathcal{C}_N$,
	\[
	\limsup_{n\to \infty} \max_{1 \le k \le K} L_k^{\kappa_n} \mathop{\longrightarrow}\limits_{N\to \infty}^{(\BP)}0
	\]
	where again we emphasize that in the above, first the $\limsup$ as $n \to \infty$ is taken and then the limit as $N \to \infty$. Let $\delta_0,c>0$ be given by~\Cref{lemma distance between segments is large} (associated to $\varepsilon=1/2$ for instance). One can then see that by choosing $N$ large enough (depending on $\delta_0$ and $K$), conditionally on $\mathcal{C}_N$, with probability $1-o(1)$ as $N\to \infty$, for all $n$ large enough, one can map the conformal rectangles $R^{\kappa_n}_k$'s for $k \in [1, K]_\BZ$ to the unit disk so that the left and right sides are sent to segments included in respectively $\{e^{\ri \theta}: \theta \in [0, \delta_0]\}$ and $\{e^{\ri \theta}: \theta \in [\pi,\pi+\delta_0]\}$. In particular, by~\Cref{lemma distance between segments is large}, conditionally on $\Gamma^{\kappa_n}\vert_{\BD \setminus \overline{R^{\kappa_n}_k}}$, with (conditional) probability at least $1/2$, the graph distance given by the restriction $\Gamma^{\kappa_n}\vert_{R^{\kappa_n}_k}$ between the left-hand side and the right-hand side of $R^{\kappa_n}_k$ is at least $c\ka_{\kappa_n}$.
	
	\stepx{step:geod-cross}{Crossings of the wide rectangles are unlikely} Let us write $K= K_1 \cdot K_2$ with $K_1, K_2\ge 1$ some integers to be chosen later. For all $j \in [1, K_1]$, we denote by $\widetilde{R}^{\kappa_n}_j$ the conformal rectangle obtained as the connected component of $\BD \setminus (F^{\kappa_n}_{(j-1)K_2} \cup  F^{\kappa_n}_{jK_2} \cup \SCL^{\kappa_n}(0) )$, where we recall that by convention $F^{\kappa_n}_K = F^{\kappa_n}_0$, such that the bottom and the top sides of $\widetilde{R}_j^{\kappa_n}$ are the sides included respectively in $\partial \BD$ and $\SCL^{\kappa_n}(0)$, while the left and right-hand sides are respectively included in $F_{(j-1)K_2}^{\kappa_n} $ and $ F_{jK_2}^{\kappa_n} $. Then, since a loop of $\Gamma^{\kappa_n}\vert_{\widetilde{R}^{\kappa_n}_j}$ meeting two of the $R^{\kappa_n}_k$'s crosses some $I_k$ and hence belongs to the family forming $F^{\kappa_n}_k$, a path of loops from the left to the right side of $\widetilde{R}^{\kappa_n}_j$ contains such a path in each of the $K_2$ rectangles $R^{\kappa_n}_k$. Applying the reasoning of~\Cref{step:geod-thin} successively to these rectangles, conditionally on $\mathcal{C}_N$, with probability $1-o(1)$ as $N\to \infty$, for all $n$ large enough, for all $j \in [1, K_1]_\BZ$, conditionally on $\Gamma^{\kappa_n}\vert_{\BD \setminus \overline{\widetilde{R}^{\kappa_n}_j}}$, the graph distance in $\Gamma^{\kappa_n}\vert_{\widetilde{R}^{\kappa_n}_j}$ between the left and right-hand sides of $\widetilde{R}^{\kappa_n}_j$ stochastically dominates $c\ka_{\kappa_n} X_j$, where $X_j$ is a binomial random variable of parameters $K_2$ and $1/2$.
		Let $p\in (0,1)$.  Set $K_1=4$. By taking $K_2\ge 1, N\ge 2$ large enough, we deduce that conditionally on $\mathcal{C}_N$, with probability at least $p$, for all $n$ large enough, for all $j \in \{1, \ldots, 4\}$, the graph distance in $\Gamma^{\kappa_n}\vert_{\widetilde{R}^{\kappa_n}_j}$ between the left and right-hand sides of $\widetilde{R}^{\kappa_n}_j$ is strictly larger than the graph distance in $\Gamma^{\kappa_n}$ between $\partial \BD$ and $\SCL^{\kappa_n}(0)$.
	
	\stepx{step:geod-conclusion}{Conclusion: a shortest path stays in one rectangle} In particular, on this event, the union of the loops in a shortest path from $\partial \BD$ to $\SCL^{\kappa_n}(0)$ cannot cross the rectangles $\widetilde{R}^{\kappa_n}_j$'s from left to right (or right to left). Since the loops do not cross each other, a loop in none of the four families forming the cuts $F^{\kappa_n}_{jK_2}$, $j \in [0,3]_\BZ$, lies in the closure of a single component of the complement of these cuts and of $\SCL^{\kappa_n}(0)$, and the components other than the $\widetilde{R}^{\kappa_n}_j$'s are bounded by loops of a single family (with $\partial \BD$ and $\SCL^{\kappa_n}(0)$). Hence a path of loops from $\partial \BD$ to $\SCL^{\kappa_n}(0)$ containing loops of two families crosses some $\widetilde{R}^{\kappa_n}_j$ from left to right, so on the above event every shortest path contains loops of at most one family. For $j \in [0,3]_\BZ$, let $M^{\kappa_n}_j$ be the component of $\BD \setminus (F^{\kappa_n}_{jK_2} \cup \overline{\mathrm{int}(\SCL^{\kappa_n}(0))})$ whose closure contains $\partial \BD \setminus B_{1/2}(e^{\ri \pi j/2})$, well defined (conditionally on $\mathcal{C}_N$, with probability $1-o(1)$ as $N \to \infty$, for $n$ large) since $\SCL^{\kappa_n}(0) \cap \partial \BD = \emptyset$ and $F^{\kappa_n}_{jK_2} \subseteq B_{1/2}(e^{\ri \pi j/2})$ by the above diameter bound; note that $M^{\kappa_n}_0 = R_n$
	. Let $P$ be a shortest path from $\partial \BD$ to $\SCL^{\kappa_n}(0)$ and $x \in \partial \BD$ a point of its first loop. As $x$ lies in at most one ball $B_{1/2}(e^{\ri \pi j/2})$ and $P$ contains loops of at most one family, there are at least two indices $j$ with $x \notin B_{1/2}(e^{\ri \pi j/2})$ and $P$ containing no loop of the family $F^{\kappa_n}_{jK_2}$; for such $j$, the interior of the first loop of $P$ contains points of the collar near $x$, hence lies in $M^{\kappa_n}_j$, and then all the loops of $P$ belong to $\Gamma^{\kappa_n}\vert_{M^{\kappa_n}_j}$ (touching loops outside the family have their interiors in the same component, as loops do not cross). Note that the conditioning event $\mathcal{C}_N$ is rotationally invariant (the exploration targeted at $0$ is rotationally invariant in law and each condition defining $\mathcal{C}_N$ is rotation invariant; applying an independent uniform rotation to the coupling, we may assume that the joint law of $(\Gamma^{\kappa_n})_{n \ge 1}$ and of the exploration of $\Gamma$ is rotationally invariant). Thus, the four events that some shortest path from $\partial \BD$ to $\SCL^{\kappa_n}(0)$ consists of loops of $\Gamma^{\kappa_n}\vert_{M^{\kappa_n}_j}$ have the same conditional probability given $\mathcal{C}_N$, and at least two of them occur on the above event; hence conditionally on $\mathcal{C}_N$, with probability at least $p/2$, for all $n$ large enough, $\mathcal{E}_n$ holds. Moreover, on $\mathcal{C}_N$, $\ka_{\kappa_n}^{-1} D^{\kappa_n}(\partial \BD, \SCL^{\kappa_n}(0)) \to D(\partial \BD, \SCL(0)) \le \varepsilon'/2$ by~\eqref{eq cv subsequence boundary}, so $\mathcal{A}_n(\varepsilon')$ holds for $n$ large. As the event that $\mathcal{E}_n \cap \mathcal{A}_n(\varepsilon')$ holds for all $n \ge n_0$ increases with $n_0$, we get $\BP[\mathcal{E}_n \cap \mathcal{A}_n(\varepsilon')] \ge \widetilde{\delta} \defeq \frac{p}{4}\BP[\mathcal{C}_N] > 0$ for all $n$ large enough.
\end{proof}
From the above lemma, we deduce that two long segments can be close with positive probability.
\begin{corollary}\label{cor long segments can be close}
	For all $\varepsilon>0$, there exist $L>0$, $p>0$, and $n_0 \in \BN$ such that for each $n \ge n_0$, a $\CLE_{\kappa_n}$ in $[0,L]\times [0,1]$ has graph distance between its top and bottom sides less than $\varepsilon \ka_{\kappa_n}$ with probability at least $p$ (a statement about each fixed $n$ separately).
\end{corollary}

\begin{proof}
	Let us start with the $\CLE_{\kappa_n}$ in $\BD$. By~\Cref{lemma the geodesic does not touch the line}, we know that for all $\varepsilon>0$, there exists $\delta>0$ such that for all $n$ large enough, with probability at least $\delta$, we have $D^{\kappa_n}(\partial \BD, \SCL^{\kappa_n}(0))< \varepsilon \ka_{\kappa_n}$ and $\mathcal{E}_n$ holds, with the notation $F_n$, $R_n$ of that lemma.
	
	On this event, the loops of the shortest path given by $\mathcal{E}_n$ belong to $\Gamma^{\kappa_n}\vert_{R_n}$ and join $\partial \BD$ to $\SCL^{\kappa_n}(0)$; with probability $1-o(1)$, $R_n$ is a conformal rectangle with top boundary in $\SCL^{\kappa_n}(0)$, bottom boundary in $\partial \BD$, left and right boundaries in $F_n$, and top-to-bottom graph distance with respect to $\Gamma^{\kappa_n}\vert_{R_n}$ less than $\varepsilon \ka_{\kappa_n}$. Note that $R_n$ is a function of $F_n \cup \SCL^{\kappa_n}(0)$ and that, by the domain Markov property of the $\CLE_{\kappa_n}$, conditionally on $F_n \cup \SCL^{\kappa_n}(0)$, the restriction $\Gamma^{\kappa_n}\vert_{R_n}$ is an independent $\CLE_{\kappa_n}$ in $R_n$.

We conformally map $R_n$ onto the rectangle $V_{L_n} \defeq [0,L_n] \times [0,1]$ for some (random) $L_n \in (0,\infty)$ such that the top (resp.\ bottom) boundary of $R_n$ is mapped to the top (resp.\ bottom) boundary of $V_{L_n}$, and note that there exists $L \in (0,\infty)$ such that for all $n \in \BN$ sufficiently large, we have, on an event with probability at least $\delta / 2$, that the above holds and $L_n \leq L$ ($L_n^{-1}$ is the extremal length of the curves joining the top and bottom boundaries of $R_n$, which is at least that of the curves joining $\SCL^{\kappa_n}(0)$ to $\partial \BD$, and $\SCL^{\kappa_n}(0) \to \SCL(0)$, which is at positive distance from $\partial \BD$; hence the $L_n$ are tight). Moreover, on $\{L_n \le L\}$, the top-to-bottom distance of a $\CLE_{\kappa_n}$ in $V_{L_n}$ stochastically dominates that of a $\CLE_{\kappa_n}$ in $V_L \defeq [0,L] \times [0,1]$ (map both rectangles onto $\BD$ with centers sent to $0$ and symmetry axes to the coordinate axes: the images of the top and bottom sides are arcs centered at $\pm \ri$ whose lengths increase with the length of the rectangle; conclude by~\Cref{remark monotonicity rectangle}). Hence, on $\{L_n \le L\}$, the conditional probability given $F_n \cup \SCL^{\kappa_n}(0)$ that the top-to-bottom distance of $R_n$ is less than $\varepsilon \ka_{\kappa_n}$ is at most $p_n(L)$, the corresponding probability in $V_L$, so that $p_n(L) \ge \delta/2$ for $n$ large.	
\end{proof}

We are now in position to prove~\Cref{lemma distance from two long segments is small}.
\begin{proof}[Proof of~\Cref{lemma distance from two long segments is small}]
	\stepx{step:long-setup}{Setup} As usual, by Skorokhod's representation theorem, let us assume that the convergence~\eqref{eq cv subsequence boundary} and the convergences given by~\Cref{lemma cv Hausdorff union of loops} (for the compact sets used below, with the loop surrounding $0$ excluded) hold a.s. Let $\varepsilon \in (0,1/2)$. Let $L,p>0$ be given by~\Cref{cor long segments can be close} for this $\varepsilon$, let $K \ge 1$ with $(1-p)^K \le \varepsilon/3$ and $d \defeq \pi/(16K)$.
	
	\stepx{step:long-positive}{The origin-containing loop is close to $\partial \BD$ with positive probability} Let $\delta\in (0,1)$. We know that
	\[
	\BP[d_\mathrm{H}(\partial \BD, \SCL(0)) < \delta]>0.
	\]
	This can be seen using the Brownian loop soup construction: the event $E_\delta$ that the loop soup contains a Brownian loop which is contained in $\BD \setminus \overline{B_{1-\delta/2}(0)}$ and which surrounds $B_{1-\delta/2}(0)$ has non-zero probability. On $E_\delta$ the cluster of that loop, hence its outer boundary, surrounds $B_{1-\delta/2}(0)$, so that $\SCL(0)$ surrounds $B_{1-\delta/2}(0)$ and lies in the annulus $\BD \setminus \overline{B_{1-\delta/2}(0)}$, which gives $d_\mathrm{H}(\partial \BD, \SCL(0)) \le \delta/2 < \delta$.
	
	Unlike in the previous proofs, we condition not on the event $\{d_\mathrm{H}(\partial \BD, \SCL(0)) < \delta\}$ but on an event of comparable probability measurable with respect to the loops of $\Gamma^{\kappa_n}$ meeting $[-1,1]$; by the Markov property, the law of the restriction of $\Gamma^{\kappa_n}$ to the region between the cuts along $[0,1]$ and $[-1,0]$ is then unaffected.
	
	\begin{figure}[ht!]
		\centering
		\includegraphics[width=0.6\linewidth]{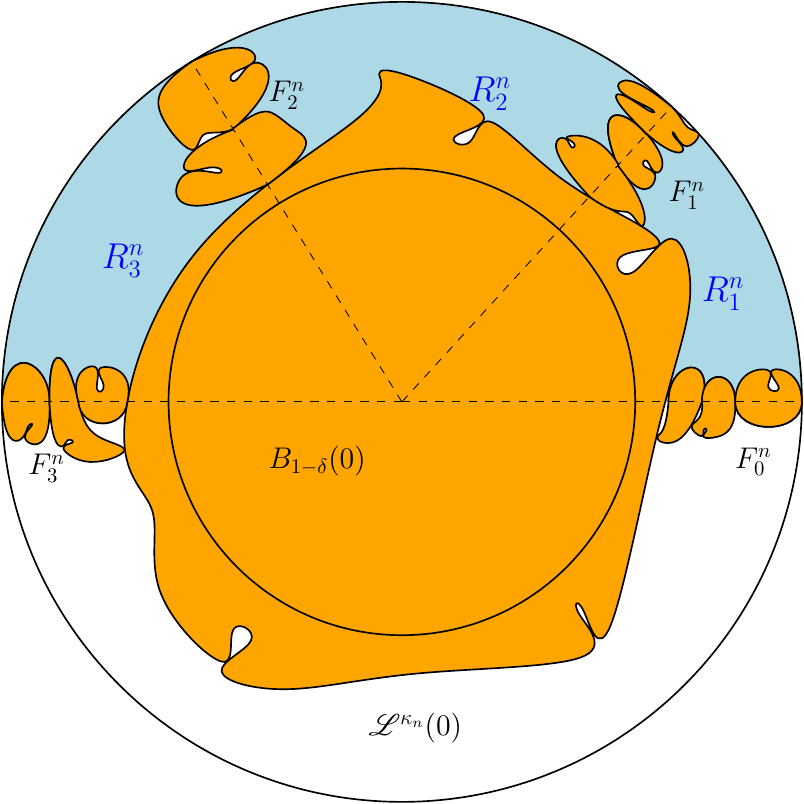}
		\caption{Illustration of the proof of~\Cref{lemma distance from two long segments is small}. The interiors of the loop encircling the origin $\SCL^{\kappa_n}(0)$ and of the loops forming the $F^n_k$ for $k \in [0, K]_\BZ$ (with $K=3$ in this figure) are in orange. The conformal rectangles $R^n_k$ for $k \in [1, K]_\BZ$ are in blue.}
		\label{fig:long_rectangles}
	\end{figure}
	
	\stepx{step:long-rectangles}{The cuts and the conformal rectangles} Let $\alpha_k \defeq k\pi/K$ for $k \in [0,K]_\BZ$.
	For all $k \in [0, K]_\BZ$, let $F_k$ be the closure of the union of the domains encircled by loops of $\Gamma\setminus \{\SCL(0)\}$ which intersect $\{r e^{\ri \alpha_k}: r \in [0,1]\}$. Similarly, let $F_k^n$ be the closure of the union of the domains encircled by loops of $\Gamma^{\kappa_n}\setminus \{\SCL^{\kappa_n}(0)\}$ which intersect $\{r e^{\ri \alpha_k}: r \in [0,1]\}$. Note that $F^n_k \cup \overline{\mathrm{int}(\SCL^{\kappa_n}(0))}$ contains the segment $\{r e^{\ri \alpha_k}: r \in [0,1]\}$ (almost every point of $\BD$ being surrounded by a loop), so $e^{\ri \alpha_k} \in F^n_k$. Let $I_k \defeq \{e^{\ri \theta} : \theta \in [\alpha_{k-1} + 5d, \alpha_k - 5d]\}$ for $k \in [1,K]_\BZ$. On the event that $\SCL^{\kappa_n}(0) \cap \partial \BD = \emptyset$ and all the $F^n_k$ have diameter less than $2d$ (so $F^n_k \subseteq B_{2d}(e^{\ri \alpha_k})$), let $R^n_k$ be the component of $\BD \setminus (F^n_{k-1} \cup F^n_k \cup \overline{\mathrm{int}(\SCL^{\kappa_n}(0))})$ whose closure contains $I_k$, and $R^n$ that of $\BD \setminus (F^n_0 \cup F^n_K \cup \overline{\mathrm{int}(\SCL^{\kappa_n}(0))})$ whose closure contains $I_1 \cup \dots \cup I_K$ (well defined by the collar argument of the proof of~\Cref{lemma the geodesic does not touch the line}); since the segments $\{r e^{\ri \alpha_k}\}$ outside $\mathrm{int}(\SCL^{\kappa_n}(0))$ lie in the cuts, $R^n_k$ lies in the sector between $\alpha_{k-1}$ and $\alpha_k$, so $R^n_k \subseteq R^n$. See Figure~\ref{fig:long_rectangles}. As before, these are conformal rectangles with top boundaries in $\SCL^{\kappa_n}(0)$, bottom boundaries in $\partial \BD$, and left and right boundaries in the cuts; we write $L^n_k$ (resp.\ $L^n$) for the modulus of $R^n_k$ (resp.\ $R^n$), so that $R^n_k$ is conformally equivalent to $[0,L^n_k] \times [0,1]$ with top and bottom boundaries sent to the top and bottom sides.

	\stepx{step:long-moduli}{Bounds on the moduli} Let us bound these moduli when moreover $d_\mathrm{H}(\partial \BD, \SCL^{\kappa_n}(0)) < \delta \le 1/4$. For $\theta \in [\alpha_{k-1}+5d, \alpha_k - 5d]$, the radial segment from the last point of $\SCL^{\kappa_n}(0)$ on $\{r e^{\ri \theta}\}$ to $e^{\ri \theta}$ lies in $\overline{R^n_k}$ (it avoids the cuts, since $|re^{\ri \theta} - e^{\ri \alpha_j}| > 2d$ for $r \ge 1-\delta$ and $|\theta - \alpha_j| \ge 5d$) and joins its top and bottom boundaries. These segments have length at most $\delta$ and span an angular sector of width $3\pi/(8K)$, so $(L^n_k)^{-1} \le 8K \log(1/(1-\delta))/(3\pi) \le 16 K \delta/(3\pi)$, i.e.\ $L^n_k \ge 3\pi/(16K\delta)$. On the other hand, with $\rho_n \defeq \dist(\SCL^{\kappa_n}(0), \partial \BD)$, the curves joining the top and bottom boundaries of $R^n$ cross the annulus $\{1-\rho_n < |z| < 1\}$, so $L^n \le 2\pi/\log(1/(1-\rho_n))$.

	\stepx{step:long-parameters}{Choice of the parameters} Note that conditionally on $d_\mathrm{H}(\partial \BD, \SCL(0)) < \delta'$, the diameters of $F_k$ for $k \in [0, K]_\BZ$ go to zero in probability as $\delta' \to 0$. Indeed, conditionally on $\SCL(0)$, $\Gamma \setminus \{\SCL(0)\}$ is a $\CLE_4$ in the annulus $\BD \setminus \overline{\mathrm{int}(\SCL(0))}$ \cite[Proposition~3.5]{SimCLEDblConnDom}, and on $\{d_\mathrm{H}(\partial \BD, \SCL(0)) < \delta'\}$ the loop $\SCL(0)$ surrounds $B_{1-\delta'}(0)$; by~\Cref{lemma loops in thin annulus are small}, conditionally on this event, with probability $1-o(1)$ as $\delta' \to 0$, all these loops have diameter at most some $\eta_{\delta'} \to 0$, and, lying in $\{|z| > 1-\delta'\}$, they give as in the proof of~\Cref{lemma the geodesic does not touch the line} that $\mathrm{diam}(F_k) \le 2\delta' + 2\eta_{\delta'}$ for all $k$. Let $\delta_1>0$ be such that for all $\delta' \le \delta_1$, conditionally on $d_\mathrm{H}(\partial \BD, \SCL(0)) < \delta'$, the probability that some $F_k$ has diameter at least $d$ is at most $\varepsilon^2/100$. Fix $\delta \in (0, \min(\delta_1/2, 1/4, 3\pi/(16KL)))$ with $\BP[d_\mathrm{H}(\partial \BD, \SCL(0)) = \delta] = 0$, then fix $\rho_0>0$ and $\delta' \in (\delta, \delta_1]$ such that, with $B \defeq \{d_\mathrm{H}(\partial \BD, \SCL(0)) < \delta\}$ and $\rho \defeq \dist(\SCL(0), \partial \BD)$, $\BP[\rho < \rho_0 \mid B] \le \varepsilon/6$ and $\BP[d_\mathrm{H}(\partial \BD, \SCL(0)) < \delta'] \le 2\BP[B]$.
	
	\stepx{step:long-event}{The conditioning event $E_n$} Let $\mathcal{G}_n$ be the $\sigma$-algebra generated by the loops of $\Gamma^{\kappa_n}$ meeting $[-1,1]$ (namely $\SCL^{\kappa_n}(0)$ and the loops forming $F^n_0$ and $F^n_K$), and $E_n \in \mathcal{G}_n$ the event that $d_\mathrm{H}(\partial \BD, \SCL^{\kappa_n}(0)) < \delta$, $\mathrm{diam}(F^n_0), \mathrm{diam}(F^n_K) < 2d$ and $\rho_n \ge \rho_0$. By the a.s.\ convergences of $\SCL^{\kappa_n}(0)$, $F^n_0$, $F^n_K$, almost surely,
	\[
	\liminf_{n\to \infty} {\bf 1}_{E_n} \ge {\bf 1}_{B \cap \{\mathrm{diam}(F_0), \mathrm{diam}(F_K) \le 2d\} \cap \{\rho \ge \rho_0\}},
	\]
	
	and the expectation of the right-hand side is at least $(1-\varepsilon^2/100-\varepsilon/6)\BP[B]$, so $\BP[E_n] \ge \BP[B]/2$ for $n$ large. Moreover, a.s.\ for $n$ large, $E_n \subseteq \{d_\mathrm{H}(\partial \BD, \SCL(0)) < \delta'\}$ and $\mathrm{diam}(F^n_k) \ge 2d$ implies $\mathrm{diam}(F_k) \ge d$; hence $\BP[E_n \cap \{\exists k: \mathrm{diam}(F^n_k) \ge 2d\}] \le \varepsilon^2 \BP[B]/25$ and the conditional probability given $E_n$ that some $F^n_k$ has diameter at least $2d$ is at most $\varepsilon/10$.
	
	\stepx{step:long-conclusion}{Conclusion} By the Markov property recalled above, conditionally on $\mathcal{G}_n$, $\Gamma^{\kappa_n}\vert_{R^n}$ is a $\CLE_{\kappa_n}$ in $R^n$. For $\ell>0$, let $q_n(\ell)$ be the probability that the top-to-bottom graph distance of a $\CLE_{\kappa_n}$ in $[0,\ell]\times[0,1]$ is less than $\varepsilon \ka_{\kappa_n}$; it is nondecreasing in $\ell$ (see the proof of~\Cref{cor long segments can be close}) and $q_n(L) \ge p$ for $n$ large by that corollary. On $E_n$, $L^n \le M \defeq 2\pi/\log(1/(1-\rho_0))$, so the conditional probability given $\mathcal{G}_n$ that the top-to-bottom distance of $R^n$ with respect to $\Gamma^{\kappa_n}\vert_{R^n}$ is less than $\varepsilon \ka_{\kappa_n}$, namely $q_n(L^n)$ by conformal invariance, is at most $q_n(M)$. On the other hand, conditionally on $\mathcal{G}_n$ and on the loops of $\Gamma^{\kappa_n}\vert_{R^n}$ meeting the segments $\{r e^{\ri \alpha_k}\}$, $k \in [1,K-1]_\BZ$ (which determine the $F^n_k$), the $\Gamma^{\kappa_n}\vert_{R^n_k}$ are independent $\CLE_{\kappa_n}$'s, and on $E_n \cap \{\forall k : \mathrm{diam}(F^n_k) < 2d\}$ we have $L^n_k \ge 3\pi/(16K\delta) > L$; hence, on this event, with conditional probability at least $1-(1-p)^K \ge 1-\varepsilon/3$, the top-to-bottom distance of some $R^n_k$, and thus that of $R^n$ (as $R^n_k \subseteq R^n$), is less than $\varepsilon \ka_{\kappa_n}$. Combining, $q_n(M) \ge (1-\varepsilon/3)(1-\varepsilon/10) \ge 1-\varepsilon$ for $n$ large, hence $q_n(M') \ge 1-\varepsilon$ for all $M' \ge M$. Finally, $D^{\kappa_n}(\alpha,\beta)$ for $\alpha,\beta$ the images of the top and bottom sides of $[0,M]\times[0,1]$ under a conformal map onto $\BD$ has the law of the top-to-bottom distance, which concludes the proof.
\end{proof}

\section{Checking the axioms} 
\label{sec:checking_the_axioms}

In this section, we prove that the subsequential limit $D$ satisfies the axioms listed in~\Cref{def:weak_axioms2}. Together with~\Cref{prop subsequential limit,prop subsequential limit boundary}, this proves~\Cref{thm:convergence_nonsimple_cle}.

\subsection{Axiom~\eqref{it:weak_axiom_geodesic2} (weak geodesic metric)}
Axiom~\eqref{it:weak_axiom_geodesic2},~\eqref{it:weak_axiom_geodesic_02} (connectedness of balls) is satisfied since by~\eqref{eq cv Skorokhod}, $\SCB_t(\SCL)$ is the Hausdorff limit of $\SCB^{\kappa_n}_{\ka_{\kappa_n}t}(\SCL^{\kappa_n})$ (for fixed $t$, see~\Cref{subsec:skorokhod}), which is connected (each loop at graph distance $k$ touches one at distance $k-1$), and by Lemma~\ref{lemma right-continuity at rational times and density of loops},  we have that $\SCB_t(\SCL)$ is the closure of the union of the regions encircled by the loops of $\Gamma$ at $D$ distance at most $t$ from $\SCL$. For a segment $\alpha\subseteq \partial \BD$, the same argument gives that $\SCB_t(\alpha)$ is connected.

Furthermore, Axiom~\eqref{it:weak_axiom_geodesic2},~\eqref{it:weak_axiom_geodesic_12} (right-continuity and left limits of balls) follows from~\eqref{eq cv Skorokhod}, which was obtained using~\Cref{lemma Skorokhod tightness}, and from~\Cref{lem:density_of_loops_in_balls} for the left limits (the balls being nondecreasing in $t$).

Axiom~\eqref{it:weak_axiom_geodesic2},~\eqref{it:weak_axiom_geodesic_42} (geodesic sets) stems from~\Cref{prop subsequential limit boundary}: $\widetilde{K}_{\SCL(x),\SCL(y)}$ is a Hausdorff limit of connected sets containing both loops, and its inclusion in $K_{\SCL(x),\SCL(y)}$ follows from the convergence of the balls and~\Cref{prop hitting balls} below.

We postpone the proof of Axiom~\eqref{it:weak_axiom_geodesic2},~\eqref{it:weak_axiom_geodesic_22} to the end of the section.

\subsection{Axiom~\eqref{it:weak_axiom_locality2} (locality)}
Let $V \subseteq \BD$ be a simply connected subdomain and let $I \in \SCS$ be a segment of $\partial \BD$. Let $V^{\star, \kappa}$ be the open set obtained by removing from $V$ the closure of the union of the domains encircled by loops of $\Gamma^\kappa$ which are not contained in $V$,  where $\kappa \in (4,8)$ is fixed.

	Let $V_j^\kappa$ for $j\ge 1$ be the connected components of $V^{\star, \kappa}$, in decreasing order of Lebesgue measure (we set $V^\kappa_j=\emptyset$ for $j$ larger than the number of connected components if there are only finitely many).  Note that since $V$ is simply connected (so that $\BC \setminus V$ is connected),  we have that $V_j^{\kappa}$ is simply connected for all $j\geq 1$. For all $j\ge 1$, let $Z_j^\kappa$ be an independent random point in $V_j^\kappa$ sampled uniformly with respect to Lebesgue measure. Let $\phi_j^\kappa$ be the unique conformal mapping from $\BD$ to $V_j^\kappa$ sending $0$ to $Z_j^\kappa$ and such that $(\phi_j^\kappa)'(0)>0$. 
	
	Here and below, $\Gamma^\kappa\vert_{W}$ denotes the loops of $\Gamma^\kappa$ contained in $\overline{W}$. For all $j\ge 1$ such that $V_j^\kappa \neq \emptyset$, let $D^\kappa_j$ be the graph distance on $(\phi_j^\kappa)^{-1}(\Gamma^\kappa\vert_{V^\kappa_j})$ (where, as for $D^\kappa$, two loops are adjacent if they intersect). Note that for all $\SCL_1, \SCL_2 \in \Gamma^\kappa\vert_{V^\kappa_j}$, we have
\begin{equation}\label{eq D kappa restricted is larger}
D^{\kappa}_j((\phi^\kappa_j)^{-1}(\SCL_1), (\phi^\kappa_j)^{-1}(\SCL_2)) \ge D^\kappa(\SCL_1, \SCL_2).
\end{equation}
Moreover, for all $\SCL \in \Gamma^\kappa\vert_{V^\kappa_j}$, for all $t \in [0, D^\kappa(\SCL, \partial V^\kappa_j)]$, the image under $\phi_j^\kappa$ of the metric ball for $D_j^\kappa$ of radius $t$ centered at $(\phi_j^\kappa)^{-1}(\SCL)$ coincides with $\SCB^\kappa_t(\SCL)$. This comes from the fact that $\Gamma^\kappa\vert_{V^\kappa_j}$ equipped with the adjacency relation can be seen as a subgraph of $\Gamma^\kappa$, and from the fact that a path of loops leaving $V^\kappa_j$ contains a loop intersecting $\partial V^\kappa_j$.

Here, for $A \subseteq \overline{\BD}$, we set $D^\kappa(\SCL, A) \defeq \inf\{t : \SCB^\kappa_t(\SCL) \cap A \neq \emptyset\}$ as in~\Cref{def:weak_axioms2}. Furthermore, given $\{\SCL \in \Gamma^\kappa: \SCL \not\subseteq V\}$,
$\{\SCB^\kappa_t(\SCL): \SCL \in \Gamma^\kappa, \  \SCL \not\subseteq V, t \in [0, D^{\kappa}(\SCL, \partial V^{\star, \kappa})]\}$ and $\{\SCB^{\kappa}_t(I): t \in [0, D^{\kappa}(I, \partial V^{\star, \kappa})]\},$ the conditional law of 	
\begin{equation*}
\left( (\phi^\kappa_j)^{-1}(\Gamma^\kappa \vert_{V_j^\kappa}), D^\kappa_j
 \right)
\end{equation*}
for $j\ge 1$ are independent with the law of $(\Gamma^\kappa, D^\kappa)$. Indeed, loops contained in $V$ touch loops not contained in $V$ only on $\partial V^{\star,\kappa}$, so the above balls and cutoff times are functions of the loops not contained in $V$, and the statement follows from the Markov property of~\Cref{sec: distance between two segments} with $A = \overline{\BD} \setminus V$ (the second coordinate being a function of the first). We call this statement the locality of the $\CLE_\kappa$.

Let us prove the analogous statement for $\kappa=4$ in our subsequential limit. 

\begin{proposition}\label{prop locality}
	Let $\Gamma, D$ be the random variables appearing in the subsequential limits~\eqref{eq cv ball} and~\eqref{eq cv subsequence boundary}. Let $V \subseteq \BD$ be a deterministic simply connected subdomain and $I \in \SCS$ a deterministic segment. Let $V^\star$ be the open set obtained by removing from $V$ the closure of the union of the domains encircled by the loops of $\Gamma$ which are not included in $V$. Let $V_j$ for $j\ge 1$ be the connected components of $V^\star$ in decreasing order of Lebesgue measure. For all $j\ge 1$, let $Z_j$ be an independent random point in $V_j$ sampled uniformly with respect to Lebesgue measure. Let $\phi_j$ be the unique conformal mapping from $\BD$ to $V_j$ sending $0$ to $Z_j$ and such that $\phi_j'(0)>0$.

	Then, there exist random distances $D_j$ on $\phi_j^{-1}(\Gamma\vert_{V_j})$, coupled with $(\Gamma, D)$, such that
	\begin{itemize}
		\item Conditionally on $\{ \SCL \in \Gamma: \SCL \not \subseteq V\}$, 
		$\{\SCB_t(\SCL): \SCL \in \Gamma, \  \SCL \not\subseteq V, t \in [0, D(\SCL, \partial V^{\star}))\}$ and $\{\SCB_t(I): t \in [0, D(I, \partial V^\star))\}$, the $(\phi_j^{-1}(\Gamma\vert_{V_j}), D_j)$'s are i.i.d.\ and their conditional law is that of $(\Gamma, D)$.
		\item For all $j$, for all $\SCL_1, \SCL_2 \in \Gamma\vert_{V_j}$, we have
		\[
		D(\SCL_1, \SCL_2) \le D_j(\phi_j^{-1}(\SCL_1), \phi_j^{-1}(\SCL_2)).
		\]
		\item For all $j$, for all $\SCL \in \Gamma\vert_{V_j}$, for all $t \in [0, D(\SCL, \partial V_j))$, the image under $\phi_j$ of the metric ball of radius $t$ for $D_j$ centered at $\phi_j^{-1}(\SCL)$ is $\SCB_t(\SCL)$.
	\end{itemize}
\end{proposition}
\begin{proof}
	\stepx{step:loc-setup}{Almost sure convergence of the restricted configurations} By Skorokhod's representation theorem (applied once to the joint law of all the random variables considered below), we may assume that~\eqref{eq cv subsequence boundary},~\eqref{eq cv Skorokhod} hold a.s., together with, for all $j\ge 1$, the almost sure convergence of $V_j^{\kappa_n}$ to $V_j$ in the Carath\'eodory sense (which stems from~\Cref{lemma cv Hausdorff union of loops} applied with $K = \overline{\BD} \setminus V$ (the loops not contained in $V$ are those intersecting $K$), because the Hausdorff convergence of the closure of the union of the domains encircled by the loops not contained in $V$ implies the Carath\'eodory convergence of the connected components of its complement in $V$; their areas converge as well (as in the proof of~\Cref{cor uniform exploration from x}), which allows us to match the labels $j$ (the areas of the $V_j$ being a.s.\ distinct) and to couple $Z_j^{\kappa_n}$ with $Z_j$), the almost sure convergence of $Z_j^{\kappa_n}$ to $Z_j$, yielding the almost sure convergence of $\phi_j^{\kappa_n}$ toward $\phi_j$ uniformly on compact subsets of $\BD$ (which entails the convergence of $(\phi_j^{\kappa_n})^{-1}$ toward $\phi_j^{-1}$ uniformly on compact subsets of $V_j$). Using the fact that the $D_j^{\kappa_n}$'s for $j\ge 1$ have the same law as $D^{\kappa_n}$, we may also assume that along the same subsequence, we also have the almost sure convergence \[\ka_{\kappa_n}^{-1}D_j^{\kappa_n}\left((\phi_j^{\kappa_n})^{-1}(\SCL^{\kappa_n}(\phi_j^{\kappa_n}(x))), (\phi_j^{\kappa_n})^{-1}(\SCL^{\kappa_n}(\phi_j^{\kappa_n}(y)))\right)
		\mathop{\longrightarrow}\limits_{n\to \infty} D_j\left(\phi_j^{-1}(\SCL(\phi_j(x))), \phi_j^{-1}(\SCL(\phi_j(y)))\right)
		\]
	 for all $x,y \in \BD_{\BQ}$, where $D_j$ is a random distance on $\phi_j^{-1}(\Gamma \vert_{V_j})$ such that $(\phi_j^{-1}(\Gamma \vert_{V_j}), D_j)$ has the same law as $(\Gamma,D)$, and we can also assume that the metric balls from loops for the distance $D_j^{\kappa_n}$ converge to the metric balls from loops for the distance $D_j$ in the sense of Skorokhod. Thus, we get the convergence
	 \begin{multline}\label{eq cv balls locality}
\left(
(\phi^{\kappa_n}_j)^{-1}(\SCB_{\ka_{\kappa_n}t}^{\kappa_n}(\SCL^{\kappa_n}(x)))
\right)_{x \in \BD_\BQ\cap V^{\kappa_n}_j, \ t \in 
[0, D( \SCL(x), \partial V_j) ) }
\\
\mathop{\longrightarrow}\limits_{n\to \infty}^{\mathrm{a.s.}}
\left( (\phi_j)^{-1}(\SCB_t(\SCL(x))) \right)_{ x \in \BD_\BQ \cap V_j, \ t \in [0, D( \SCL(x), \partial V_j) ) }
	 \end{multline}
	 in the sense of the J$_1$ Skorokhod topology for all $x \in \BD_\BQ \cap V_j$ (processes stopped at their terminal times; the limits are identified since the loops of $\Gamma\vert_{V_j}$ and the balls $\SCB_t(\SCL(x))$, $t<D(\SCL(x),\partial V_j)$, are compact in $V_j$, where $(\phi^{\kappa_n}_j)^{-1} \to \phi_j^{-1}$ uniformly, and the loop surrounding $\phi_j^{\kappa_n}(x)$ is eventually that surrounding a fixed rational point of $\mathrm{int}(\SCL(\phi_j(x)))$; this also gives $\liminf_n \ka_{\kappa_n}^{-1}D^{\kappa_n}(\SCL^{\kappa_n}(x),\partial V^{\kappa_n}_j) \ge D(\SCL(x),\partial V_j)$), as well as the convergence
	\begin{multline*}
		X_j^{\kappa_n} \defeq
		\left( 
		\begin{matrix}
			(\phi^{\kappa_n}_j)^{-1}(\Gamma^{\kappa_n} \vert_{V_j^{\kappa_n}}) \\
			\left(\ka_{\kappa_n}^{-1}D_j^{\kappa_n}\left((\phi_j^{\kappa_n})^{-1}(\SCL^{\kappa_n}(\phi_j^{\kappa_n}(x))), (\phi_j^{\kappa_n})^{-1}(\SCL^{\kappa_n}(\phi_j^{\kappa_n}(y)))\right)\right)_{x,y \in \BD_\BQ}
		\end{matrix}
		\right)
		\\
		\mathop{\longrightarrow}\limits_{n\to \infty}^{\mathrm{a.s.}}
		X_j \defeq 
		\left(
		\begin{matrix}
			\phi_j^{-1}(\Gamma \vert_{V_j})\\
			 \left(D_j\left(\phi_j^{-1}(\SCL(\phi_j(x))), \phi_j^{-1}(\SCL(\phi_j(y)))\right)\right)_{x,y \in \BD_\BQ}
		\end{matrix}
		\right),
	\end{multline*}
	where the first coordinate converges in the sense that the loop surrounding $x$ converges in the Hausdorff topology for all $x  \in \BD_\BQ$, and the second coordinate converges for the product topology. 

	\stepx{step:loc-items}{The second and third items} Note that the second item of the proposition already follows from~\eqref{eq D kappa restricted is larger} and from the above convergence.
	
	Moreover, the convergence~\eqref{eq cv balls locality} and the convergence of metric balls for $D^{\kappa_n}_j$, together with the fact that for all $\SCL \in \Gamma^{\kappa_n}\vert_{V^{\kappa_n}_j}$, for all $t \in [0, D^{\kappa_n}(\SCL, \partial V^{\kappa_n}_j)]$, the image under $\phi_j^{\kappa_n}$ of the metric ball for $D_j^{\kappa_n}$ of radius $t$ centered at $(\phi_j^{\kappa_n})^{-1}(\SCL)$ coincides with $\SCB^{\kappa_n}_t(\SCL)$, yield the third item of the proposition (the cutoff times being handled by the $\liminf$ bound above).
	
	\stepx{step:loc-discretization}{Discretization of the conditioning} The rest of the proof is devoted to the proof of the first item in the statement of the proposition. Let $(x_i)_{i\ge 1}$ be an enumeration of $\BD_\BQ$. For all $k\ge 1$, for all $\SCL^\kappa \in \Gamma^\kappa$ (resp.~$\SCL \in \Gamma$), let $\SCL^\kappa_k$ (resp.~$\SCL_k$) be the union of the squares $[ j 2^{-k}, (j+1)2^{-k}] \times [\ell 2^{-k}, (\ell+1)2^{-k}]$ for $j, \ell \in \BZ$ which intersect $\SCL^\kappa \in \Gamma^\kappa$ (resp.~$\SCL \in \Gamma$). We define similarly $\SCB^\kappa_{k,t}(\SCL^\kappa)$, $\SCB_{k,t}(\SCL)$, $\SCB_{k,t}^\kappa(I)$ and $\SCB_{k,t}(I)$ from $\SCB^\kappa_{t}(\SCL^\kappa)$, $\SCB_{t}(\SCL)$, $\SCB_{t}^\kappa(I)$ and $\SCB_{t}(I)$ for all $t\ge 0$, $\SCL^\kappa \in \Gamma^\kappa$ and $\SCL \in \Gamma$. We also let $V^{\star,\kappa}_k$ be the union of the $[ j 2^{-k}, (j+1)2^{-k}] \times [\ell 2^{-k}, (\ell+1)2^{-k}]$ for $j, \ell \in \BZ$ which intersect $V^{\star, \kappa}$ and we define similarly $V^{\star}_k$ from $V^\star$. Here, the grid is shifted by an independent uniform random vector in $[0,2^{-k})^2$, so that a.s.\ no compact set considered below meets a closed square without meeting its interior (which makes the discretizations continuous).
	
	Let \begin{equation*}Y^\kappa_k
	\defeq
	\left(
	\begin{matrix}
	\left\{ \SCL^\kappa_k(x_i) \right\}_{1 \le i \le k, \ \SCL^\kappa(x_i) \not\subseteq V } \\
	{\left\{\SCB_{k, \ka_{\kappa} 2^{-k} \ell }^{\kappa}(\SCL^{\kappa}(x_i)): 1 \le \ell \le k2^k, \ 1\le i\le k, \ \text{such that } { \SCB_{k, \ka_{\kappa} 2^{-k} \ell }^{\kappa}(\SCL^{\kappa}(x_i)) \cap V^{\star,\kappa}_k = \emptyset} \right\}}\\
	{
	\left\{\SCB_{k, \ka_{\kappa} 2^{-k} \ell }^{\kappa}(I): 1 \le \ell \le k2^k,  \ \text{such that } { \SCB_{k, \ka_{\kappa} 2^{-k} \ell }^{\kappa}(I)}\cap V^{\star,\kappa}_k = \emptyset \right\}} \\
	V^{\star,\kappa}_k
	\end{matrix}
	\right)
	\end{equation*}
	and 
	\begin{equation*}Y_k
	\defeq
	\left(
	\begin{matrix}
	\left\{\SCL_k(x_i) \right\}_{1 \le i \le k, \ \SCL(x_i) \not\subseteq V }\\
	\left\{\SCB_{k, 2^{-k}\ell }(\SCL(x_i)): 1 \le \ell \le k2^k, \ 1\le i\le k, \ \text{such that } {\SCB_{k, 2^{-k}\ell }(\SCL(x_i)) \cap V^\star_k =\emptyset}  \right\}\\
	{
	\left\{\SCB_{k, 2^{-k}\ell }(I): 1 \le \ell \le k2^k, \ \text{such that }  {\SCB_{k, 2^{-k}\ell }(I)} \cap V^\star_k =\emptyset \right\}} \\
	V^\star_k
	\end{matrix}
	\right).
	\end{equation*}
	Note that the above random variables take a finite number of possible values (uniformly in $\kappa$). Note also that the balls in $Y^\kappa_k$ are disjoint from $V^{\star,\kappa}_k$, hence are pre-cutoff balls of loops not contained in $V$, so that $Y^\kappa_k$ is a function of the locality data, and that $\sigma(Y_k)$ is nondecreasing in $k$. Moreover, by~\eqref{eq cv subsequence boundary} and the beginning of the proof, we know that $Y^{\kappa_n}_k \rightarrow Y_k$ a.s.
	
	\stepx{step:loc-prelimit}{Conditional independence in the prelimit} By the locality for the $\CLE_\kappa$ (and the previous remark), we know that conditionally on $Y^{\kappa_n}_k$, the $X^{\kappa_n}_j$'s for $j\ge 1$ are i.i.d.~with the same law as $X^{\kappa_n} \defeq (\Gamma^{\kappa_n}, (D^{\kappa_n}(\SCL^{\kappa_n}(x), \SCL^{\kappa_n}(y)))_{x,y \in \BD_\BQ})$. As a result, for all $d\ge 1$, for all possible values $y$ of $Y_k^{\kappa_n}$, for all bounded continuous functions $g_1, \ldots, g_d$,
	\begin{equation*}
	\BE[g_1(X^{\kappa_n}_1)\cdots g_d(X^{\kappa_n}_d) \mid Y^{\kappa_n}_k=y]
	=
	\prod_{j=1}^d \BE[ g_j(X^{\kappa_n})].
	\end{equation*}
	\stepx{step:loc-limit}{Passing to the limit} Moreover, by the convergences of $X^{\kappa_n}_j, Y^{\kappa_n}_k$ toward $X_j, Y_k$ (for the above product topology, $Y_k$ taking finitely many values) we see that
	\begin{equation*}
	\BE[ g_1(X^{\kappa_n}_1)\cdots g_d(X^{\kappa_n}_d) \mid Y^{\kappa_n}_k=y] \mathop{\longrightarrow}\limits_{n\to \infty}
	\BE[ g_1(X_1)\cdots g_d(X_d) \mid Y_k=y].
	\end{equation*}
	Furthermore, using the convergence of $X^{\kappa_n}$ to $X \defeq (\Gamma,(D(\SCL(x),\SCL(y)))_{x,y \in \BD_\BQ}
	)$, we conclude that
	\begin{equation*}
	\BE[ g_1(X_1)\cdots g_d(X_d) \mid Y_k=y]
	=
	\prod_{j=1}^d \BE[ g_j(X)].
	\end{equation*}
	Since the $\sigma(Y_k)$ are nondecreasing and generate, up to null sets, the conditioning $\sigma$-algebra of the first item (balls at all times $t < D(\cdot, \partial V^\star)$ being determined by those at dyadic times by right-continuity), the martingale convergence theorem concludes.
\end{proof}

\subsection{Axiom~\eqref{it:weak_axiom_conformal_invariance2} (conformal invariance)} This follows from the conformal invariance of the $\CLE_\kappa$ combined with the convergence in~\eqref{eq cv subsequence boundary}: $(\phi(\Gamma^{\kappa_n}), \phi_\ast D^{\kappa_n})$ has the law of $(\Gamma^{\kappa_n}, D^{\kappa_n})$, the graph distance being intrinsic, and this passes to the limit since the loop surrounding $\phi(x)$ is eventually that surrounding a fixed rational point of $\mathrm{int}(\SCL(\phi(x)))$.

\subsection{Axiom~\eqref{it:weak_axiom_uniform_exploration2} (uniform exploration)} This condition comes from combining ~\Cref{cor cv Hausdorff ball boundary} with Remark~\ref{remark removing the boundary increases the distance}: $\SCB_t(\partial \BD) = \SCB^\partial_t(\partial \BD)$ by the latter, this process is the explored region of a uniform exploration by the former, and $D(\SCL, \partial \BD)$ is the first time $\SCL \subseteq \SCB_t(\partial \BD)$ (\Cref{lem:density_of_loops_in_balls}), i.e.\ the discovery time of $\SCL$.

\subsection{Axiom~\eqref{it:weak_axiom_geodesic2},~\eqref{it:weak_axiom_geodesic_22}}
Finally, we consider Axiom~\eqref{it:weak_axiom_geodesic2},~\eqref{it:weak_axiom_geodesic_22}. Let us show the following. 
\begin{proposition}\label{prop hitting balls}
	Almost surely, for all $x, y \in \BD_\BQ$, for all $t_1 \le D(\SCL(x), \SCL(y)) $,
	\begin{equation*}
		D(\SCL(x), \SCL(y))  = t_1 + \inf \{ t \ge 0: \SCB_t(\SCL(y)) \cap \SCB_{t_1}(\SCL(x)) \neq \emptyset\}.
	\end{equation*}
	The same identity holds if we replace one of the loops or both of them by segments.
\end{proposition}
Not only does the above proposition prove Axiom~\eqref{it:weak_axiom_geodesic2},~\eqref{it:weak_axiom_geodesic_22}, but it also shows that the distances and balls from segments defined via the subsequential limit satisfy the definition given in Axiom~\eqref{it:weak_axiom_geodesic2}. Indeed, by taking $t_1=0$, we get $D(\alpha, \SCL(x))= \inf \{ t\ge 0: \SCB_t(\SCL(x)) \cap \alpha \neq \emptyset \}$.

Before proving \Cref{prop hitting balls}, let us state a useful lemma on conformal rectangles in the unit disk.

\begin{lemma}\label{lemma extremal distance rectangle}
	For all $l, \delta>0$, there exists $\varepsilon_0>0$ depending only on $l$ and $\delta$ such that the following holds for all $\varepsilon \in (0, \varepsilon_0)$. Let $R \subset \BD$ be a conformal rectangle. Let $L>0$ be the unique real number such that the conformal rectangle $R$ can be conformally mapped to $(0, L) \times (0, 1)$ (so that the top, bottom, left and right sides of $R$ are mapped to the corresponding sides of $(0, L) \times (0,1)$). Let $z \in \BD$ and assume that there exists some connected component $\widetilde{R}$ of $R \cap (B_{\delta}(z) \setminus \overline{B_{\varepsilon}(z)})$ such that $\widetilde{R}$ is a conformal rectangle whose left (resp.\ right) boundary is a subarc of $\partial B_{\delta}(z)$ (resp.\ $\partial B_{\varepsilon}(z)$) and its top (resp.\ bottom) boundary is a subarc of the top (resp.\ bottom) boundary of $R$,  and $\partial \widetilde{R}$ does not intersect the left and right boundaries of $R$.  Then $L\ge l$.
\end{lemma}

\begin{proof}
\stepx{step:extr-component}{Restriction to an annular component} Let $\widetilde{R}$ be a connected component of $R \cap (B_{\delta}(z) \setminus \overline{B_{\varepsilon}(z)})$ satisfying the properties in the statement of the lemma.  Note that there exists a closed subarc $I$ of $\partial B_{\sqrt{\varepsilon \delta}}(z)$ such that 
$I$ crosses $\widetilde{R}$ between the top and bottom boundaries of $\widetilde{R}$,  and such that $I$ intersects $\partial \widetilde{R}$ only at its endpoints (the circle separates the left and right boundaries of $\widetilde{R}$, so one of its crosscuts in $\widetilde{R}$ does).

\stepx{step:extr-beurling}{A Beurling estimate} Fix $w \in I \cap \widetilde{R}$.  Then,  in order for a planar Brownian motion $\widetilde{B}$ starting from $w$ to exit $R$ for the first time either on the left or right boundaries of $R$,  it must exit $\widetilde{R}$ for the first time either on the left or right boundaries of $\widetilde{R}$,  i.e., at some point on $\partial B_{\delta}(z) \cup \partial B_{\varepsilon}(z)$.  Note that $\text{dist}(w,\partial \widetilde{R}) \leq 2 \sqrt{\varepsilon \delta}$ and that in order for $\widetilde{B}$ to exit $\widetilde{R}$ for the first time on the left boundary of $\widetilde{R}$,  it must traverse distance at least $\delta- \sqrt{\delta \varepsilon}\ge (\sqrt{\delta} - \sqrt{\varepsilon}) \text{dist}(w,\partial \widetilde{R}) / (2\sqrt{\varepsilon})$
.  Thus \cite[Proposition~3.79]{lawler2008conformally} implies that this occurs with probability at most $\widetilde{C} \varepsilon^{1/4} (\sqrt{\delta} - \sqrt{\varepsilon})^{-1/2}$ for some universal constant $\widetilde{C} \in (0,\infty)$. 

Similarly,  in order for $\widetilde{B}$ to exit $\widetilde{R}$ for the first time on the right boundary of $\widetilde{R}$,  it must make a crossing between the inner and outer boundaries of the annulus $B_{\sqrt{\varepsilon \delta}}(z) \setminus \overline{B_{\varepsilon}(z)}$ without intersecting the top and bottom boundaries of $\widetilde{R}$ (which we recall that they both cross $B_{\sqrt{\varepsilon \delta}}(z) \setminus \overline{B_{\varepsilon}(z)}$).  It follows from \cite[Corollary~ 3.78]{lawler2008conformally} that by possibly taking $\widetilde{C}$ to be larger (in a universal way),  we have that the probability that $\widetilde{B}$ exits $\widetilde{R}$ for the first time on the right boundary of $\widetilde{R}$ is at most $\widetilde{C} \varepsilon^{1/4} \delta^{-1/4}$.  Combining with the previous paragraph,  we obtain that there exists a constant $C = C(\delta) \in (0,\infty)$ depending only on $\delta$ such that the following holds for all $w \in I \cap \widetilde{R}$ and all $\varepsilon \in (0,\delta/2)$.  The probability that a planar Brownian motion starting from $w$ exits $R$ for the first time either on its left or right boundaries is at most $C \varepsilon^{1/4}$.

\stepx{step:extr-conclusion}{Conclusion} Let $\phi$ be the conformal transformation mapping $R$ onto $(0,L) \times (0,1)$ such that the top (resp.\ bottom) boundary of $R$ is mapped onto $(0,L) \times \{1\}$ (resp.\ $(0,L) \times \{0\}$) and the left (resp.\ right) boundary of $R$ is mapped onto $\{0\} \times (0,1)$ (resp.\ $\{L\} \times (0,1)$).  If $(L/2) + \ri(1/2) \in \phi(I)$,  then \cite[Proposition ~3.69]{lawler2008conformally} combined with the previous paragraph imply that if we take $\varepsilon \in (0,\delta / 2)$ sufficiently small (in a way that depends only on $\delta$ and $l$),  we have that $L \ge l$.  If $(L/2) + \ri(1/2)$ is not in $\phi(I)$,  then $\phi(I)$ disconnects $(L/2) + \ri(1/2)$ either from the left or the right boundary of $(0,L) \times (0,1)$.  Without loss of generality,  we can assume that the latter holds since the exact same argument works in the former case.  Then,  by the Markov property of planar Brownian motion,  we have that the probability that a planar Brownian motion $B$ exits $(0,L) \times (0,1)$ for the first time on $\{L\} \times (0,1)$ is at most
\begin{align*}
\sup_{w \in \phi(I)} \BP_w[B\,\,\text{exits}\,\,(0,L) \times (0,1) \,\,\text{for the first time on} \,\,  \{L\} \times (0,1)].
\end{align*}
But \Cref{step:extr-beurling} implies that the latter quantity is at most $C \varepsilon^{1/4}$ and so using again \cite[Proposition~3.69]{lawler2008conformally},  we obtain that $L \ge l$ (if we take $\varepsilon \in (0,\delta/2)$ sufficiently small depending only on $\delta$ and $l$).  This completes the proof of the lemma.
\end{proof}

We are now in position to prove \Cref{prop hitting balls}.

\begin{proof}[Proof of~\Cref{prop hitting balls}]
	It is enough to show that for all $t_1, t_2\in  \BQ_{\ge 0}$, we have a.s.
	\begin{equation}\label{eq axiom iii}
	\SCB_{t_1}(\SCL(x)) \cap \SCB_{t_2}(\SCL(y)) \neq \emptyset \iff
	D(\SCL(x), \SCL(y)) \le t_1 + t_2.
	\end{equation}
	(The general case follows by monotonicity and right-continuity of the balls.)
	\stepx{step:hb-disjoint}{The case $\SCB_{t_1}(\SCL(x)) \cap \SCB_{t_2} (\SCL(y))= \emptyset$} Fix $t_1, t_2 \ge 0$ and $x, y \in \BQ^2$ such that this holds. By right-continuity of $t \mapsto \SCB_{t}(\SCL(x))$, there exists $\varepsilon>0$ such that $\SCB_{t_1+\varepsilon}(\SCL(x)) \cap \SCB_{t_2} (\SCL(y))= \emptyset$. We aim to show that $D(\SCL(x), \SCL(y))\ge t_1+t_2+\varepsilon$.
	By Skorokhod's representation theorem, we may assume that the convergences~\eqref{eq cv subsequence boundary} and~\eqref{eq cv Skorokhod} hold a.s. 
	
	Then, $\SCB_{t_1}(\SCL(x))$ and $ \SCB_{t_2} (\SCL(y))$ are at a non-zero Euclidean distance from each other since they are compact and disjoint. But, by the Hausdorff convergence~\eqref{eq cv subsequence boundary}, we deduce that a.s.\ $\SCB^{\kappa_n}_{\ka_{\kappa_n} (t_1+\varepsilon)}(\SCL^{\kappa_n}(x))\cap \SCB^{\kappa_n}_{\ka_{\kappa_n}t_2} (\SCL^{\kappa_n}(y)) =\emptyset$ for all $n$ large enough. In other words, \begin{equation*}D^{\kappa_n}(\SCL^{\kappa_n}(x), \SCL^{\kappa_n}(y)) > \ka_{\kappa_n}(t_1+t_2+\varepsilon),\end{equation*} up to an additive error of at most $2$ due to rounding, negligible after dividing by $\ka_{\kappa_n}$. So by~\eqref{eq cv subsequence boundary}, we deduce that $D(\SCL(x), \SCL(y))\ge t_1+t_2+\varepsilon$, so that $D(\SCL(x), \SCL(y))>t_1+t_2$. This proves the first part of \eqref{eq axiom iii}.
	
	\begin{figure}[ht!]
		\centering
		\includegraphics[width=0.6\linewidth]{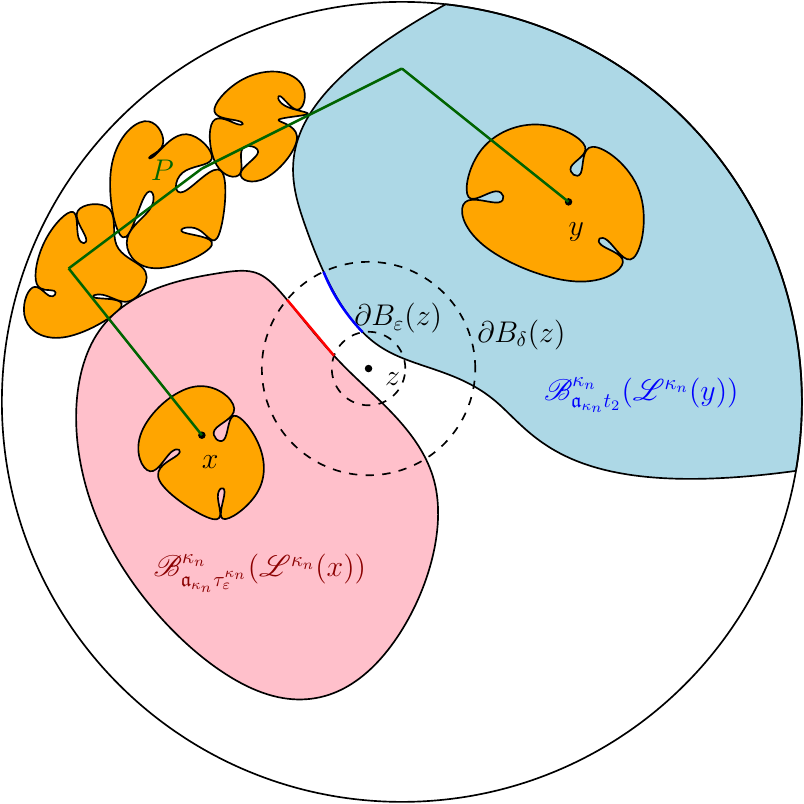}
		\caption{Illustration of the proof of~\Cref{prop hitting balls}. The case represented is Case 2, where $\SCB^{\kappa_n}_{\ka_{\kappa_n} \tau^{\kappa_n}_\varepsilon}(\SCL^{\kappa_n}(x)) \cap \partial \BD = \emptyset$ and $\SCB^{\kappa_n}_{\ka_{\kappa_n} t_2}(\SCL^{\kappa_n}(y)) \cap \partial \BD \neq \emptyset$. The ball $\SCB^{\kappa_n}_{\ka_{\kappa_n} \tau^{\kappa_n}_\varepsilon}(\SCL^{\kappa_n}(x))$ is sketched in pink and $\SCB^{\kappa_n}_{\ka_{\kappa_n} t_2}(\SCL^{\kappa_n}(y)) $ is in light blue. A polygonal path $P$ between $x$ and $y$ is drawn in green and some loops of $\Gamma^{\kappa_n}$ that intersect $P$ and form the set $F_n$ are drawn in orange. The conformal rectangle $R_n$ is the connected component of the white region whose closure intersects both metric balls. The red crossing of $B_\delta(z)\setminus \overline{B_\varepsilon(z)}$ is part of the bottom boundary of $R_n$ and the blue crossing is part of the top boundary of $R_n$. Note that, in this drawing, the annulus $B_\delta(z)\setminus \overline{B_\varepsilon(z)}$ is included in $\BD$, but this is not at all required by the proof.}
		\label{close_balls_new_proof}
	\end{figure}
	
	\stepx{step:hb-intersect}{The case $\SCB_{t_1}(\SCL(x)) \cap \SCB_{t_2} (\SCL(y)) \neq \emptyset$} Fix $t_1, t_2 \ge 0$ and $x, y \in \BQ^2$ such that this holds. We aim to show that $D(\SCL(x), \SCL(y))\le t_1+t_2$. If $y \in \SCB_{t_1}(\SCL(x))$,  then Lemma~\ref{lem:density_of_loops_in_balls} implies that $D(\SCL(x) , \SCL(y)) \leq t_1$,  and so the claim is obviously true. Hence,  we assume that $y \notin \SCB_{t_1}(\SCL(x))$. 
	
	Let $\varepsilon>0$. For all $n\ge 1$, let $\tau_\varepsilon^{\kappa_n}$ be the first time $t\ge 0$ such that $\mathrm{dist} (\SCB^{\kappa_n}_{\ka_{\kappa_n} t}(\SCL^{\kappa_n}(x)) ,\allowbreak \SCB^{\kappa_n}_{\ka_{\kappa_n} t_2}(\SCL^{\kappa_n}(y)) ) \le \varepsilon$. By~\eqref{eq cv Skorokhod}, a.s., for all $n$ large enough, we have $\tau_\varepsilon^{\kappa_n} \le t_1$.
	
	If $\SCB^{\kappa_n}_{\ka_{\kappa_n} \tau^{\kappa_n}_\varepsilon}(\SCL^{\kappa_n}(x)) \cap  \SCB^{\kappa_n}_{\ka_{\kappa_n} t_2}(\SCL^{\kappa_n}(y)) \neq \emptyset $ for infinitely many $n$, then we have
	\[
	D^{\kappa_n}(\SCL^{\kappa_n}(x), \SCL^{\kappa_n}(y))\le \ka_{\kappa_n}(\tau^{\kappa_n}_\varepsilon + t_2) +1\le \ka_{\kappa_n}(t_1+t_2)+1
	\]
	for infinitely many $n$ (by local finiteness such a ball is the union of the closed interiors of its loops, so the common point lies in two of them, which therefore intersect), so that $D(\SCL(x), \SCL(y))\le t_1+t_2$.
	
	From now on, we focus on the case where $\SCB^{\kappa_n}_{\ka_{\kappa_n} \tau^{\kappa_n}_\varepsilon}(\SCL^{\kappa_n}(x)) \cap  \SCB^{\kappa_n}_{\ka_{\kappa_n} t_2}(\SCL^{\kappa_n}(y)) =\emptyset $ for all $n$ large enough. We distinguish three cases.
	
	\emph{Case 1:} For infinitely many $n$, we have $\SCB^{\kappa_n}_{\ka_{\kappa_n} \tau^{\kappa_n}_\varepsilon}(\SCL^{\kappa_n}(x)) \cap \partial \BD \neq \emptyset$ and $\SCB^{\kappa_n}_{\ka_{\kappa_n} t_2}(\SCL^{\kappa_n}(y)) \cap \partial \BD \neq \emptyset$. Let $\mathcal{E}_1^n$ be the event that $\SCB^{\kappa_n}_{\ka_{\kappa_n} \tau^{\kappa_n}_\varepsilon}(\SCL^{\kappa_n}(x)) \cap  \SCB^{\kappa_n}_{\ka_{\kappa_n} t_2}(\SCL^{\kappa_n}(y)) =\emptyset $, that $\SCB^{\kappa_n}_{\ka_{\kappa_n} \tau^{\kappa_n}_\varepsilon}(\SCL^{\kappa_n}(x)) \cap \partial \BD \neq \emptyset$ and $\SCB^{\kappa_n}_{\ka_{\kappa_n} t_2}(\SCL^{\kappa_n}(y)) \cap \partial \BD \neq \emptyset$. On the event $\mathcal{E}^n_1$, the connected component $R_n$ of $\BD \setminus ( \SCB^{\kappa_n}_{\ka_{\kappa_n} t_2}(\SCL^{\kappa_n}(y)) \cup \SCB^{\kappa_n}_{\ka_{\kappa_n} \tau^{\kappa_n}_\varepsilon}(\SCL^{\kappa_n}(x)) )$ whose closure intersects both $\SCB^{\kappa_n}_{\ka_{\kappa_n} t_2}(\SCL^{\kappa_n}(y)) $ and $ \SCB^{\kappa_n}_{\ka_{\kappa_n} \tau^{\kappa_n}_\varepsilon}(\SCL^{\kappa_n}(x)) $ and which contains the point $z$ chosen below is simply connected (as $\partial \BD$ and the two balls form a connected set). We see $R_n$ as a conformal rectangle, whose top boundary is  $\partial R_n \cap \SCB^{\kappa_n}_{\ka_{\kappa_n} t_2}(\SCL^{\kappa_n}(y)) $ and whose bottom boundary is $\partial R_n \cap \SCB^{\kappa_n}_{\ka_{\kappa_n} \tau^{\kappa_n}_\varepsilon}(\SCL^{\kappa_n}(x)) $. Let $\phi_n$ be the conformal mapping from $R_n$ to the rectangle $(0, L_n) \times (0,1)$ for some random $L_n$ such that the top (resp.\ bottom) boundary of $R_n$ is sent to the top (resp.\ bottom) boundary of $(0, L_n) \times (0,1)$.
	
	Let $l_0>0$ and $\eta>0$. Let $\delta>0$ small enough that with probability at least $1-\eta$, the diameters of the loops $\SCL(x)$ and $\SCL(y)$ are larger than $4\delta$. Recall that $\SCL^{\kappa_n}(x)$ and $\SCL^{\kappa_n}(y)$ converge toward $\SCL(x)$ and $\SCL(y)$ respectively, so that for all $n$ large enough, with probability at least $1-\eta$, we have \[\mathrm{diam}(\SCB^{\kappa_n}_{\ka_{\kappa_n} \tau^{\kappa_n}_\varepsilon}(\SCL^{\kappa_n}(x)))\ge 3\delta
	\text{ and }
	\mathrm{diam}(\SCB^{\kappa_n}_{\ka_{\kappa_n} t_2}(\SCL^{\kappa_n}(y)) ) \ge 3\delta.
	\]
	Let $\mathcal{E}_2^n$ be the above event.
	
	 By taking $\varepsilon \in (0, \delta)$,  and since $\mathrm{dist} (\SCB^{\kappa_n}_{\ka_{\kappa_n} \tau^{\kappa_n}_\varepsilon}(\SCL^{\kappa_n}(x)), \SCB^{\kappa_n}_{\ka_{\kappa_n} t_2}(\SCL^{\kappa_n}(y))  )\le \varepsilon$, on the event $\mathcal{E}_1^n\cap \mathcal{E}^n_2$, there exists $z\in \BD_{\BQ} \setminus (\SCB^{\kappa_n}_{\ka_{\kappa_n} \tau^{\kappa_n}_\varepsilon}(\SCL^{\kappa_n}(x)) \cup \SCB^{\kappa_n}_{\ka_{\kappa_n} t_2}(\SCL^{\kappa_n}(y)))$ (near the midpoint of a segment realizing the distance between the two connected balls) such that $\SCB^{\kappa_n}_{\ka_{\kappa_n} \tau^{\kappa_n}_\varepsilon}(\SCL^{\kappa_n}(x))$ and $ \SCB^{\kappa_n}_{\ka_{\kappa_n} t_2}(\SCL^{\kappa_n}(y))$ make a crossing of $B_\delta(z) \setminus \overline{B_\varepsilon(z)}$. 
Note that both the left and right boundaries of $R_n$ are contained in $\partial \BD$ and that the sets $\SCB_{\ka_{\kappa_n} \tau_{\varepsilon}^{\kappa_n}}^{\kappa_n}(\SCL^{\kappa_n}(x)) \cap \partial \BD$ and $\SCB_{\ka_{\kappa_n} t_2}^{\kappa_n}(\SCL^{\kappa_n}(y)) \cap \partial \BD$ have empty interiors
.  Therefore,  both the top and bottom boundaries of $R_n$ cross the annulus $B_{\delta}(z) \setminus \overline{B_{\varepsilon}(z)}$,  and so the conditions in the statement of \Cref{lemma extremal distance rectangle} are satisfied for the conformal rectangle $R_n$.  It follows that we can choose $\varepsilon \in (0,\delta)$ sufficiently small (in a way that depends only on $\delta$ and $l_0$) such that on the event $\mathcal{E}_1^n \cap \mathcal{E}_2^n$,  we have that $L_n > l_0$.
	 
	 Moreover, by definition of $\tau^{\kappa_n}_\varepsilon$ and by the Markov property of~\Cref{sec: distance between two segments} applied layer by layer to the two balls (the second one explored until the stopping rule defining $\tau^{\kappa_n}_\varepsilon$), $z$ and $R_n$ being functions of the balls, conditionally on $\SCB^{\kappa_n}_{\ka_{\kappa_n} \tau^{\kappa_n}_\varepsilon}(\SCL^{\kappa_n}(x)), \SCB^{\kappa_n}_{\ka_{\kappa_n} t_2}(\SCL^{\kappa_n}(y)) $, the image $\phi_n(\Gamma^{\kappa_n}\vert_{R_n})$ is a $\CLE_{\kappa_n}$ in $(0, L_n) \times (0,1)$. By the rectangle form of~\Cref{lemma distance from two long segments is small} applied with $\varepsilon = \eta$, and by choosing $l_0 \ge M$, for all $n$ large enough, on the event $\mathcal{E}_1^n\cap \mathcal{E}^n_2$, conditionally on the metric balls $\SCB^{\kappa_n}_{\ka_{\kappa_n} \tau^{\kappa_n}_\varepsilon}(\SCL^{\kappa_n}(x)), \SCB^{\kappa_n}_{\ka_{\kappa_n} t_2}(\SCL^{\kappa_n}(y)) $, with probability at least $1-\eta$, the $\CLE_{\kappa_n}$ graph distance between the top and bottom boundaries of $R_n$ is at most $\ka_{\kappa_n} \eta$. In particular, for all $n$ large enough, on the event $\mathcal{E}_1^n\cap \mathcal{E}^n_2$, with probability at least $1-\eta$,
	\[
	D^{\kappa_n}(\SCL^{\kappa_n}(x), \SCL^{\kappa_n}(y))\le \ka_{\kappa_n}(\tau^{\kappa_n}_\varepsilon + t_2 + \eta) + 2.
	\] 
	Thus, by \eqref{eq cv subsequence boundary}, on the event that the diameters of the loops $\SCL(x)$ and $\SCL(y)$ are larger than $2\delta$, that $\mathcal{E}^n_1$ holds for infinitely many $n$ and that $\SCB_{t_1}(\SCL(x)) \cap \SCB_{t_2} (\SCL(y)) \neq \emptyset$, with probability at least $1-\eta$, we have
	\[
	D(\SCL(x), \SCL(y)) \le t_1+t_2+\eta.
	\]
	By letting $\eta \downarrow 0$ and then $\delta \downarrow 0$, we deduce that $D(\SCL(x), \SCL(y))\le t_1+t_2$ a.s.\ on the event that $\mathcal{E}^n_1$ holds for infinitely many $n$ and that $\SCB_{t_1}(\SCL(x)) \cap \SCB_{t_2} (\SCL(y)) \neq \emptyset$.
	
	\emph{Case 2:} For infinitely many $n$, we have $\SCB^{\kappa_n}_{\ka_{\kappa_n} \tau^{\kappa_n}_\varepsilon}(\SCL^{\kappa_n}(x)) \cap \partial \BD = \emptyset$ and $\SCB^{\kappa_n}_{\ka_{\kappa_n} t_2}(\SCL^{\kappa_n}(y)) \cap \partial \BD \neq \emptyset$. Let $\delta>0$. Let $P$ be a polygonal path from $x$ to $y$ made of finitely many segments whose endpoints have rational coordinates. Let $z \in \BD_{\BQ}$. 
	
	Let $\tau_\varepsilon^\infty$ be the limit of $\tau^{\kappa_n}_\varepsilon$ along a subsequence. Let $K_1=\SCB_{\tau_\varepsilon^\infty}(\SCL(x))$ and $K_2=\SCB_{t_2}(\SCL(y))$. Note that by \eqref{eq cv subsequence boundary}, $K_2$ is the almost sure limit  of $\SCB^{\kappa_n}_{\ka_{\kappa_n} t_2}(\SCL^{\kappa_n}(y))$ as $n\to \infty$. Moreover, by \eqref{eq cv Skorokhod} for all $\rho>0$, a.s.\ for all $n$ large enough, $K_1 \subset B_\rho(\SCB^{\kappa_n}_{\ka_{\kappa_n} \tau^{\kappa_n}_\varepsilon}(\SCL^{\kappa_n}(x)))$.
	
	Let $F$ be the closure of the union of the domains encircled by the loops of $\Gamma$ that intersect $P$. We work on the event that $F$ does not intersect $\overline{B_\delta(z)}$ and that $\mathrm{dist}(z, K_1)< \varepsilon$ and $\mathrm{dist}(z,K_2 )<\varepsilon$. Note that, by the convergences in the previous paragraph and by the Hausdorff convergence of $\SCL^{\kappa_n}(x)$ and $\SCL^{\kappa_n}(y)$ to $\SCL(x)$ and $\SCL(y)$ respectively, there exists such a path $P$ and such a point $z$ with probability arbitrarily close to one uniformly in $\varepsilon\in (0, \delta)$ by taking $\delta$ small enough. For all $n\ge 1$, let $F_n$ be the closure of the union of the domains encircled by the loops of $\Gamma^{\kappa_n}$ that intersect $P$. Let $R_n$ be the connected component of $\BD \setminus (\SCB^{\kappa_n}_{\ka_{\kappa_n} \tau^{\kappa_n}_\varepsilon}(\SCL^{\kappa_n}(x)) \cup \SCB^{\kappa_n}_{\ka_{\kappa_n} t_2}(\SCL^{\kappa_n}(y)) \cup F_n)$ whose closure intersects $\SCB^{\kappa_n}_{\ka_{\kappa_n} \tau^{\kappa_n}_\varepsilon}(\SCL^{\kappa_n}(x)) $ and $ \SCB^{\kappa_n}_{\ka_{\kappa_n} t_2}(\SCL^{\kappa_n}(y))$ and contains $z$. We view $R_n$ as a conformal rectangle whose bottom (resp.\ top) boundary is a connected component of $(\partial R_n )\cap  \SCB^{\kappa_n}_{\ka_{\kappa_n} \tau^{\kappa_n}_\varepsilon}(\SCL^{\kappa_n}(x)) $ (resp.\ $(\partial R_n) \cap \SCB^{\kappa_n}_{\ka_{\kappa_n} t_2}(\SCL^{\kappa_n}(y))$) that crosses the annulus $B_\delta(z)\setminus \overline{B_\varepsilon(z)}$, its left and right boundaries being the parts of $\partial R_n$ contained in $F_n$. 
	
	Let $l_0, \eta>0$. Let $\phi_n$ be the conformal mapping from $R_n$ to the rectangle $(0, L_n) \times (0,1)$ for some random $L_n$ such that the top (resp.\ bottom) boundary of $R_n$ is sent to the top (resp.\ bottom) boundary of $(0, L_n) \times (0,1)$. Next, note that by \Cref{lemma cv Hausdorff union of loops}, a.s., for all $n$ large enough, the set $F_n$ does not intersect $\overline{B_\delta(z)}$. By taking $\varepsilon$ small enough and applying \Cref{lemma extremal distance rectangle}, with probability at least $1-\eta$, we obtain that $L_n \ge l_0$. We conclude as in Case 1.
	
	\emph{Case 3:} For infinitely many $n$, we have $\SCB^{\kappa_n}_{\ka_{\kappa_n} \tau^{\kappa_n}_\varepsilon}(\SCL^{\kappa_n}(x)) \cap \partial \BD \neq \emptyset$ and $\SCB^{\kappa_n}_{\ka_{\kappa_n} t_2}(\SCL^{\kappa_n}(y)) \cap \partial \BD = \emptyset$. We do exactly the same proof as in Case 2, exchanging the roles of $(x,t_1)$ and $(y,t_2)$ (\eqref{eq axiom iii} being symmetric in these pairs).
	
	\emph{Case 4:} For all $n$ large enough, we have $\SCB^{\kappa_n}_{\ka_{\kappa_n} \tau^{\kappa_n}_\varepsilon}(\SCL^{\kappa_n}(x)) \cap \partial \BD = \emptyset$ and $\SCB^{\kappa_n}_{\ka_{\kappa_n} t_2}(\SCL^{\kappa_n}(y)) \cap \partial \BD = \emptyset$. This case is similar except that we draw a polygonal path $P_1$ from $x$ to $y$ and a polygonal path $P_2$ from $x$ (or $y$) to $\partial \BD$.
	
	These four cases are exhaustive (if none of Cases 1--3 occurs, Case 4 does for all $n$ large).

	In all the above cases, we get that $D(\SCL(x), \SCL(y)) \le t_1+ t_2$, thus finishing the proof of \eqref{eq axiom iii}. When one of the loops (or both) is replaced by a segment of $\partial \BD$, the same proof works.
\end{proof}

\bibliographystyle{alpha}
\bibliography{references}

\end{document}